\documentclass[a4paper,10pt]{amsart}
\usepackage[latin1]{inputenc}
\usepackage{dsfont}
\usepackage{amsxtra}
\usepackage{amsmath}
\usepackage{amscd}
\usepackage{amssymb}
\usepackage{amsfonts}
\usepackage[all]{xy}
\usepackage{mathrsfs}
\usepackage{amsthm}
\usepackage{epsfig,psfrag}
\usepackage{mathbbol}

\def\ind(E){{\lim\limits_{\buildrel{\longleftarrow}\over{E}}}}

\def\scrA{{{\mathscr A}}}
\def\scrK{{{\mathscr K}}}

\def\C{\mathbb{C}}
\def\H{\mathbf{H}}
\def\HT{{R}}
\def\cb{\mathbf{c}}

\def\fb{\mathbf{f}}
\def\xb{\mathbf{x}}

\def\Uen{{{\mathfrak{U}}}}
\def\scrU{{{\mathscr U}}}
\def\Wen{{{\mathfrak{W}}}}

\def\Sen{{{\mathfrak{S}}}}
\def\Heis{{{\mathscr{H}}}}

\def\PW{{{\widetilde W}}}
\def\eps{{\varepsilon}}
\def\ben{{\mathfrak{b}}}
\def\gen{{\mathfrak{g}}}

\def\hen{{\mathfrak{h}}}
\def\men{{\mathfrak{m}}}
\def\den{{\mathfrak{d}}}
\def\uen{{\mathfrak{u}}}
\def\pen{{\mathfrak{p}}}
\def\sen{{\mathfrak{s}}}
\def\V{\mathbf{V}}
\def\len{{\mathfrak{l}}}
\def\nen{{\mathfrak{n}}}

\def\Dc{{\mathcal D}}

\def\Lc{{\mathcal L}}
\def\Oc{{\mathcal O}}

\def\dim{\mathrm{dim}}
\def\Hom{\mathrm{Hom}}

\def\det{\mathrm{det}}

\def\End{\mathrm{End}}

\def\cdeg{\operatorname{cdeg}}
\def\lg{G}
\def\ad{{\mathrm{ad}}}

\def\Ind{\mathrm{Ind}}

\def\Z{\mathbb{Z}}

\def\E{{X}}

\def\H{\mathbf{H}}
\def\Cb{\mathbf{C}}
\def\kb{\mathbf{k}}

\def\Ecb{\pmb{\mathcal E}}

\def\B{\mathbb{A}}

\def\Sc{\mathcal{S}}
\def\Hc{\mathcal{H}}

\def\f{\mathbf{f}}
\def\Hin{\mathrm{Hilb}_n}
\def\Hilb{\mathrm{Hilb}}

\def\Hi{\mathrm{Hilb}}

\def\U{{\mathbf U }}
\def\Sb{{\mathbf{S}}}
\def\SH{\mathbf{S}{\mathbf{H}}}
\def\SHo{\mathbf{S}{\mathbf{H}}}

\def\SHoo{{\SH^0}}
\def\SC{\mathbf{S}{\mathbf{C}}}
\def\qed{\hfill $\sqcap \hskip-6.5pt \sqcup$}
\def\N{\mathbb{N}}

\def\a{a}
\def\b{b}

\def\cc{\mathbf{c}}
\def\CC{\mathbb{C}}
\def\Am{\mathbb{A}}
\def\Bm{\mathbb{B}}
\def\cm{\mathbb{c}}
\def\Km{\mathbb{K}}
\def\Lm{\mathbb{L}}

\def\Hm{\mathbb{H}}
\def\Sm{\mathbb{S}}

\def\ZZ{\mathbb{Z}}

\def\Lb{{\mathbf{L}}}
\def\Wb{{\mathbf{W}}}
\def\E{{\mathbf{E}}}

\def\Shb{\mathbf{Sh}}

\def\mod{{\text{mod}}}

\def\eu{\operatorname{eu}\nolimits}
\def\op{{\operatorname{op}\nolimits}}
\def\tor{{\operatorname{tor}\nolimits}}

\def\cc{\operatorname{c}\nolimits}
\def\Hom{\operatorname{Hom}\nolimits}
\def\mod{\operatorname{mod}\nolimits}
\def\End{\operatorname{End}\nolimits}
\def\Aut{\operatorname{Aut}\nolimits}
\def\Diff{\operatorname{Diff}\nolimits}
\def\Ker{\operatorname{Ker}\nolimits}
\def\Im{\operatorname{Im}\nolimits}
\def\dim{{\operatorname{dim}\nolimits}}
\def\Frac{{\operatorname{Frac}\nolimits}}
\def\SYM{\operatorname{SYM}\nolimits}

\theoremstyle{plain}
\newtheorem{theo}{Theorem}[section]
\newtheorem{lem}[theo]{Lemma}
\newtheorem{prop}[theo]{Proposition}
\newtheorem{cor}[theo]{Corollary}
\newtheorem*{claim}{Claim}
\newtheorem{conj}[theo]{Conjecture}
\newtheorem{slem}[theo]{Claim}

\newtheorem*{theoa}{Theorem A}
\newtheorem*{theob}{Theorem B}
\newtheorem*{theoc}{Theorem C}
\newtheorem*{theo*}{Theorem}
\newtheorem*{prop*}{Proposition}
\newtheorem{conj*}{Conjecture}

\theoremstyle{definition}
\newtheorem{df}[theo]{Definition}

\theoremstyle{remark}
\newtheorem{ex}[theo]{Example}
\newtheorem{rem}[theo]{Remark}

\numberwithin{equation}{section}

\title[Degenerate DAHA, W-algebras and instantons]
{Cherednik algebras, W-algebras and 
the equivariant cohomology of the moduli space of instantons on $\mathbb{A}^2$}
\author{O. Schiffmann, E. Vasserot}

\begin{document}

\begin{abstract} 
We construct a representation of the affine $W$-algebra 
of $\gen\len_r$ on the equivariant
homology space of the moduli space of $U_r$-instantons, 
and we identify the corresponding module.
As a corollary, we give a proof of a version of the AGT conjecture 
concerning pure $N=2$
gauge theory for the group $SU(r)$.  Another proof  has been announced by Maulik and Okounkov.
Our approach uses a deformation of the universal enveloping algebra of $W_{1+\infty}$,
which acts on the above homology space and which 
specializes to $W(\gen\len_r)$ for all $r$. This deformation
is constructed from a limit, as $n$ tends to $\infty$, 
of the spherical degenerate double affine Hecke algebra of $GL_n$.
\end{abstract}

\maketitle

%{{\flushright{\textit{To Igor Frenkel, for his sixtieth birthday.}}}}

\maketitle

\setcounter{tocdepth}{2}

\tableofcontents

\section{Introduction}

\vspace{.15in}

In their recent study of $N=2$ super-symmetric gauge theory in dimension four, 
the authors of \cite{AGT} observed a
striking relation with two-dimensional Conformal Field Theory. 
More precisely, they observed in some examples and
conjectured in many other an equality between the conformal blocks of Liouville theory associated with a punctured Riemann
surface and the group $U_r$ on the one hand and the instanton part of the 
Nekrasov partition function for a suitable
four-dimensional gauge theory associated with the group $U_r$ on the other hand. 
Numerous partial results in this direction
have been obtained in the physics litterature,
see e.g., \cite{FL} and the references therein. 
In mathematical
terms, the AGT conjecture suggests in particular the existence of a representation of the 
affine $W$-algebra of $G$ on the equivariant
intersection cohomology of the moduli space of $G^L$-instantons on 
$\mathbb{R}^4$ satisfying some extra properties (relating the fundamental class and the Whittaker vector),
see \cite{BFRF} and \cite{Gaiotto}. Here $G,$ $G^L$ are a pair of
complex reductive groups which are dual to each other in the sense of Langlands.
For the gauge group $G=G^L=GL_r,$ a construction of this action 
will be given by Maulik and Okounkov. 
It uses ideas from symplectic geometry, see, e.g., Okounkov's talk in Jerusalem in December 2010.
The purpose of this paper is to give, again for $G=GL_r$, an 
alternative construction of this action which is inspired by our previous work
\cite{SV2}. It is based on degenerate double affine Hecke algebras.

Let us describe our main result more precisely. Let $M_r=\bigsqcup_{n \geqslant 0} M_{r,n}$ be the moduli 
space of rank $r$ torsion free coherent sheaves
on $\mathbb{P}^2$, equipped with a framing along 
$\mathbb{P}^1_{\infty} \subset \mathbb{P}^2$.
It is a smooth symplectic variety of dimension
$2rn$. It is acted upon by an $r+2$-dimensional torus $\widetilde{D}=( \CC^\times)^2 \times D$ 
where $(\CC^\times)^2$ acts on
$\mathbb{P}^2$ and $D=(\CC^\times)^r$ acts on the framing. When $r=1$, 
the moduli space $M_{1,n}$ is isomorphic to the Hilbert scheme 
$Hilb_n$ of $n$ points on $\CC^2$. 
In the mid $90$s, Nakajima constructed a 
representation of the rank one Heisenberg 
algebra on the space
$$\widetilde\Lb^{(1)}=\bigoplus_{n \geqslant 0} H_*(Hilb_n)$$
by geometric methods, which 
identifies it with the standard level
one Fock space, see \cite{NakHilb} and \cite{Groj}. 
The case of the equivariant Borel-Moore homology 
$$\Lb^{(1)}=\bigoplus_{n \geqslant 0} H_*^{\widetilde D}(Hilb_n)$$ 
was considered later in \cite{V}.
For $r \geqslant 1$ there is still a representation of a rank one Heisenberg algebra
on the space 
$$\mathbf{L}^{(r)}=\bigoplus_{n \geqslant 0} H^{\widetilde{D}}_*(M_{r,n}),$$
but it is neither irreducible nor cyclic, see \cite{Baranov}.
A construction of a representation of a $r$-dimensional
Heisenberg algebra on $\mathbf{L}^{(r)}$ has also been given in
\cite{LicataSavage}.
Now, let 
$$
R_r=\CC[x,y,e_1, \ldots, e_r],\qquad
K_r=\CC(x,y,e_1, \ldots, e_r)
$$
be the cohomology ring of the
classifying space of $\widetilde{D}$ and its fraction field. 
The space $\Lb^{(r)}$ is an
$R_r$-module. Set $\Lb^{(r)}_K=\Lb^{(r)}\otimes_{R_r}K_r$.
We'll abbreviate 
$$\kappa=-y/x,\qquad \varepsilon_i=e_i/x,\qquad \xi=1-\kappa,\qquad i\in[1,r].$$ 
Let $W_k(\gen\len_r)$ be the
level $k$ affine $W$-algebra of $\gen\len_r$.
Recall that the cup product in equivaraint cohomology yields a bilinear map
$$(\bullet,\bullet):\Lb^{(r)}_K\times\Lb^{(r)}_K\to K_r$$
called the \emph{intersection pairing}.
Set $\vec{e}=(e_1,e_2,\dots,e_r)$, $\vec\varepsilon=\vec{e}/x$ and
$\rho=(0,-1, -2, \ldots, 1-r)$.
Here is the main result of this paper. 

\vspace{.1in}

\begin{theo*} 
(a) There is a representation of  $W_k(\gen\len_r)$ of 
level $k=\kappa-r$ on $\Lb^{(r)}_K$, 
identifying it with the Verma module $M_\beta$ of highest weight
$\beta=-(\vec\varepsilon+\xi\rho)/\kappa.$

(b) This action is quasi-unitary with respect to the intersection pairing 
on $\Lb^{(r)}_K$.

(c) The \emph{Gaiotto state} $G=\sum_{n \geqslant 0}G_n$, $G_n= [M_{r,n}],$
is a Whittaker vector of $M_\beta$.
\end{theo*}

\vspace{.1in}

Parts $(a)$ and $(b)$ are proved in Theorem \ref{Theo:Main}
and part $(c)$ is proved in Proposition \ref{prop:8.12}.
Note that $W_k(\gen\len_1)$ is a
Heisenberg algebra of rank one. So the theorem may be seen as a generalization 
to higher ranks of the representation of the Heisenberg algebra
on the equivariant cohomology of the Hilbert scheme. 
For instance, for $r=2$ we get an action of the
Virasoro algebra on the cohomology of the moduli space of 
$U_2$-instantons on $\mathbb{R}^4$.  
The relation with the AGT conjecture for the
pure $N = 2$ supersymmetric
gauge theory is the following.
Recall that Nekrasov's partition function is the generating function of the integral
of the equivariant cohomology class $1\in H^{\widetilde D}_*(M_{r,n})$, i.e., we have
$$Z(x,y, \vec{e}\,;\,q)=\sum_{n\geqslant 0} q^n\big([M_{r,n}],[M_{r,n}]\big).$$
The element $G$ belongs to the completed Verma module
$$\widehat M_\beta=\prod_{n\geqslant 0}M_{\beta,n},\qquad
M_{\beta,n}=H_*^{\widetilde D}(M_{r,n})\otimes_{R_r}K_r.$$
Let $\{W_{d,l}\,;\,l\in\Z,\,d\in[1,r]\}$ be the set of the
Fourier modes of the generating fields of $W_k(\gen\len_r)$.
Then $M_\beta$ has a unique 
bilinear form $(\bullet,\bullet)$ such that 
the highest weight vector has norm 1 and the
adjoint of $W_{d,-l}$ is $W_{d,l}$ for $l\geqslant 0$ (up to a sign). 
Then, the element $G$ is uniquely determined by the
Whittaker condition and we have
$$Z(x,y, \vec{e}\,;\,q)=\sum_{n\geqslant 0}q^n(G_n,G_n).$$

Let us now explain the main steps of the proof. 
Since $W$-algebras do not possess, 
beyond the case of $\gen\len_3$,  a presentation by
generators and relations, we cannot hope to construct directly the action of $W_k(\gen\len_r)$
on $\mathbf{L}^{(r)}_K$ by some correspondences. 
Our approach relies instead on an intermediate 
algebra $\SH^{\cb}$, defined over 
the field $F=\CC(\kappa),$ which is interesting in its own right, 
and which does act on $\mathbf{L}^{(r)}_K$ by some correspondences. 
The actual definition of $\SH^{\cb}$ is rather involved. 
Its main properties are summarized below. 
Let $\SH_n$ denote the 
spherical degenerate double affine Hecke algebra of $GL_n$. Let 
$\Lambda=F[p_l\,;\,l\geqslant 1]$. Let
$$\mathscr{H}=\langle \cb_0, b_l \;;\; l \in \Z \rangle$$ be the Heisenberg 
algebra of central charge $\cb_0/\kappa$. 
First, we prove the following in Section 1 and Appendix F.

\begin{prop*}
(a) The algebra $\SH^{\cb}$ is $\ZZ$-graded, $\N$-filtered and has a 
triangular decomposition 
$$\SH^{\cb}= \SH^> \otimes \SH^{\cb,0} \otimes \SH^<,\qquad
\SH^{\cb,0}=F[\mathbf{c}_l\,;\,l\geqslant 1]\otimes F[D_{0,l}\,;\,l\geqslant 1].$$
Here $F[\mathbf{c}_l\,;\,l\geqslant 1]$ is a central subalgebra.
The Poincar\'e polynomials of $\SH^<$ and $\SH^>$ are 
$$P_{\SH^{>}}(t,q)=\prod_{r>0}\prod_{l \geqslant 0} \frac{1}{1-t^rq^l}, \qquad 
P_{\SH^{<}}(t,q)=\prod_{r<0}\prod_{l \geqslant 0} \frac{1}{1-t^rq^l}.$$

(b) Let $\SHo$ be the specialization of $\SH^{\cb}$ at 
$\mathbf{c}_0=0$ and $\mathbf{c}_l=-\kappa^l\omega^l$ for $l\geqslant 1$. 
For  $n\geqslant 1$ there is a surjective algebra
homomorphism $\Psi_n :\SHo \to \SH_n$ with
$\bigcap_n \Ker\Psi_n=\{0\}$. 

\vspace{1mm}

(c) The part of order $\leqslant 0$ for the $\N$-filtration
is $\SH^\cb[\leqslant\!\! 0] = \mathscr{H}\otimes F[\cb_l\,;\,l\geqslant 2]$.
The algebra
$\SH^{\cb}$ is generated by
$\SH^{\cb}[\leqslant 0]$ and $D_{0,2}$.

\vspace{1mm}

(d) Let $\SH^{(1,0, \ldots)}$ be the specialization of $\SH^{\cb}$ at 
$\mathbf{c}_0=1$ and $\mathbf{c}_l=0$ for $l\geqslant 1$. It has a 
faithful representation in $\Lambda$ 
such that $\mathscr{H}$ acts in the standard 
way and $D_{0,2}$ acts as
the Laplace-Beltrami (or Calogero-Sutherland) operator
$$D_{0,2}=\kappa \,\square=\frac{1}{2}\kappa(1-\kappa)
\sum_{l\geqslant 1}(l-1)b_{-l}b_l+ 
\frac{1}{2}\kappa^2\sum_{l, k \geqslant 1}
\big(b_{-l-k}b_l b_k + b_{-l}b_{-k} b_{l+k}\big).$$

(e) The specialization of $\SH^{(1,0, \ldots)}$ at $\kappa=1$ is isomorphic to 
the universal enveloping algebra of the Witt algebra 
$W_{1+\infty}$.

\end{prop*}

\vspace{.1in}

We do not give a presentation of $\SH^{\cb}$ by generators and relations.
We do not need it. 
However, the subalgebras $\SH^{>}$ and $\SH^{<}$ have realizations as 
shuffle algebras, see Theorem \ref{thm:shuffleC} and 
Corollary \ref{cor:SC/SH}. The central subalgebra 
$F[\mathbf{c}_l\,;\,l\geqslant 0]$ is not finitely generated,
but  only two of the generators are essential, i.e.,
the rest may be split off. 

A construction of a similar limit $\Sm\Hm^\cm$ of the 
spherical double affine Hecke algebras of $GL_n$ as $n$ 
tends to infinity appears
in \cite{SV1}. The algebra $\Sm\Hm^\cm$ depends on two 
parameters $t,q$, and $\SH^{\cb}$ may be obtained by degeneration of
$\Sm\Hm^\cm$ as $t \mapsto 1$ and $q \mapsto 1$ with $t=q^{-\kappa}$,
in much the same way as the trigonometric Cherednik algebra is
obtained by degeneration of the elliptic Cherednik algebra, see Section 7. 
In \cite{SV2} it was shown that $\Sm\Hm^\cm$ 
acts on the space $\bigoplus_{n \geqslant 0} K^{\widetilde{D}}(M_{r,n})$, where
$K^{\widetilde{D}}$ is the equivariant algebraic K-theory. Adapting the 
arguments of \textit{loc.~ cit.} to the equivariant cohomology setup, we
prove the following in Theorem \ref{thm:SH/Ur}, Proposition \ref{prop:faithfulr} and
Lemma \ref{lem:cyclicr}.
Let $\SH^{(r)}_K$ be
the specialization of $\SH^{\cb} \otimes K_r$ to $\mathbf{c}_0=r$ and
$\mathbf{c}_i=p_i(e_1/x, \ldots, e_r/x)$.

\vspace{.1in}

\begin{theoa} 
There is a faithful representation $\rho^{(r)}$ of $\SH^{(r)}_K$ on 
$\Lb^{(r)}_K$ such that 
$\Lb^{(r)}_K$ is generated by the fundamental class of $M_{r,0}$.
\end{theoa}

\vspace{.05in}

This representation is given by convolution with 
correspondences supported on the nested instanton 
spaces. In the proof of Theorem A an important role is 
played by the commuting varieties
$$C_n=\{ (u,v) \in (\gen\len_n)^2\,;\, [u,v]=0\}$$
and the \emph{cohomological Hall algebra}, which is 
an associative algebra on
$$\Cb'=\bigoplus_{n \geqslant 0}H^T(C_n),\qquad T=(\CC^\times)^2.$$
Let $\Uen(W_k(\gen\len_r))$ be the current algebra of $W_k(\gen\len_r)$.
We'll use a quotient $\scrU(W_k(\gen\len_r))$ of $\Uen(W_k(\gen\len_r))$ 
whose definition is given in Section \ref{sec:miura}. 
It is a $\ZZ$-graded, $\mathbb{N}$-filtered, degreewise topological
associative algebra with 1 which is degreewise complete. Let $\Uen(\SH^{(r)}_K)$ be the degreewise
completion of $\SH^{(r)}_K$, which is defined in Definition \ref{df:degreewise}. 
Our main theorem is a consequence of the 
following results, proved in Theorem \ref{8.7:thm} and Corollary \ref{8.7:cor}, and in
Theorem \ref{Theo:Main} and Proposition \ref{prop:8.12}. 
Put $k=\kappa-r$. First, we have

\vspace{.1in}

\begin{theob}
There is an embedding of graded and filtered algebras
$$\Theta: \SH^{(r)}_K \longrightarrow \scrU(W_k(\gen\len_r))$$
which extends to a surjective morphism 
$\Uen(\SH^{(r)}_K) \longrightarrow \scrU(W_k(\gen\len_r))$.
The map $\Theta$ induces an equivalence between the categories of admissible  
$\SH^{(r)}_K$
and $\scrU(W_k(\gen\len_r))$ modules. 
\end{theob}

\vspace{.1in}

This allows us to regard $\Lb^{(r)}_K$ as a $W_k(\gen\len_r)$-module. Then, we have the following.

\vspace{.1in}

\begin{theoc}
The representation $\mathbf{L}_K^{(r)}$ of  $W_k(\gen\len_r)$ is a Verma module.
It is quasi-unitary with respect to the intersection pairing.
The element $G$ is a Whittaker vector for $W_k(\gen\len_r)$.
\end{theoc}

Theorem C is proved by some simple explicit calculation. 
Let us briefly indicate how we prove Theorem B. Our approach
rests upon the following crucial fact proved in Theorem \ref{7.7:thm2}.

\vspace{.1in}

\begin{prop*} The algebra
$\SH^{\cb}$ is equipped with a topological Hopf algebra structure.
The comultiplication is
uniquely determined by the following formulas 
$$\pmb\Delta(\cb_l)= \cb_l\otimes 1 + 1 \otimes \cb_l, \qquad l\geqslant 0,$$
$$\pmb\Delta(b_0)=b_0 \otimes 1 + 1 \otimes b_0 + \xi \cb_0 \otimes \cb_0,
\qquad\pmb\Delta(b_{l})=b_{l} \otimes 1 + 1 \otimes b_l, \qquad l \neq 0,$$
$$\pmb\Delta(D_{0,2})=D_{0,2} \otimes 1 + 1 \otimes D_{0,2} 
+\kappa\xi\sum_{l\geqslant 1}lb_{l}\otimes b_{-l}.$$
\end{prop*}

Using this coproduct, we equip the category of admissible $\SH^{\cb}$-modules with a monoidal 
structure. In particular $(\Lb^{(1)}_K)^{\otimes r}$ is equipped with a faithful representation of
$\SH^{(r)}_K$. We call it the \emph{free field realization} representation. 
We then compare this free field representation of $\SH^{(r)}_K$ 
with the free field representation of $W_k(\gen\len_r)$
using some explicit computations in the cases $r=1,2$, the coassociativity of $\pmb\Delta$ and the
fundamental result of Feigin and Frenkel  \cite{FF0}, \cite{FF} which characterizes $W_k(\gen\len_r)$ 
as the intersection of some \emph{screening operators}.

One remark about the Hopf algebra structure on $\SH^{\cb}$ is in order. 
It was observed in \cite{VVcop} that, under Nakajima's realization of affine quantum groups
in terms of equivariant K-theory of quiver varieties, the coproduct of the quantum groups
could be constructed geometrically using some fixed subsets of the quiver varieties.
In later works, a geometric construction of tensor products of representations
in terms of both cohomology and K-theory of some quiver varieties
was given in \cite{Nakcop}, \cite{Malkin}.
In this paper, we do not give a geometric interpretation of our map $\pmb\Delta$. 
In fact, we obtain it by degenerating a similar coproduct on the
algebra $\Sm\Hm^\cm$. The existence of a Hopf algebra structure
on $\Sm\Hm^\cm$, in 
equivariant $K$-theory, is not more natural than on $\SH^{\cb}$,
in equivariant cohomology. However, since $\Sm\Hm^\cm$ is
identified with a central extension of the Drinfeld
double of the spherical Hall algebra $\widehat{\Ecb}$ of an elliptic curve 
over a finite field, see \cite{SV2}, and since this Hall algebra has
a coproduct\footnote{The correct choice of coproduct on 
$\widehat{\Ecb}$ here is not the standard one, but rather the standard one
twisted by a \textit{Fourier transform}, see \eqref{sigma}.},
the algebra $\Sm\Hm^\cm$ is also equipped with a comultiplication.
We do not know, however, 
of a similar isomorphism involving $\SH^{\cb}$ which would
give directly the comultiplication.

Some of the methods and results of this paper generalize to 
the case of the moduli spaces of instantons on resolutions of simple Kleinian 
singularities, equivalently, the Nakajima quiver varieties attached to 
affine Dynkin diagrams. We'll come back to this question elsewhere.

To finish, let us say a few words concerning the organization of this paper. 
The construction and properties of the algebra $\SH^{\cb}$
are given in Section 1. In Sections 2 and 3 we define some convolution 
algebra acting on the space $\Lb^{(r)}_K$ and state our first main result,
Theorem \ref{thm:SH/Ur}, which claims that this algebra is isomorphic to 
$\SH^{(r)}_K$. In Section 4 we introduce the commuting variety and its
convolution algebra, the so-called \emph{cohomological Hall algebra}. 
The proof of Theorem \ref{thm:SH/Ur} is given in Sections 5 and 6.
Section 7 is devoted to the construction of the Hopf algebra structure on 
$\SH^{\cb}$. Section 8 discusses the free field realizations of
$\SH^{(r)}_K$ and $W_k(\gen\len_r)$, and compares them 
(first for $r=1$ then $r=2$ and then for arbitrary $r$). Theorem \ref{8.7:thm}
is proved in Section 8.9, and part $(a)$ of our main Theorem is proved in 
Section 8.11, see Theorem \ref{Theo:Main}. Finally,
Section 9 is devoted to the Whittaker property of the Gaiotto state, with 
respect to both $\SH^{(r)}_K$ and $W_k(\gen\len_r)$.
Several technical lemmas are postponed to the appendices. 
In particular, the relation with $W_{1+\infty}$ is explained in Appendix F.

\vspace{.15in}

\subsection{Notation} 
\label{sec:notation}
We'll use the continental way of drawing a partition 
$\lambda=(\lambda_1 \geqslant \lambda_2 \geqslant \ldots)$, 
with rows going from the bottom up of successive
length $\lambda_1, \lambda_2,$ etc. 
If $s$ is a box in the diagram of a partition $\lambda$, we denote by $x(s),y(s), l(s), a(s)$ the number
of boxes lying \textit{strictly} to the west, resp. south, resp. north, 
resp. east, of the box $s$. 

\begin{ex}
For the box $s$ in the partition $(5,4^2,2,1)$ depicted below

\vspace{.1in}

\centerline{
\begin{picture}(60,60)
\put(0,0){\line(0,1){60}}
\put(12,0){\line(0,1){60}}
\put(24,0){\line(0,1){48}}
\put(36,0){\line(0,1){36}}
\put(48,0){\line(0,1){36}}
\put(60,0){\line(0,1){12}}
\put(0,0){\line(1,0){60}}
\put(0,12){\line(1,0){60}}
\put(0,24){\line(1,0){48}}
\put(0,36){\line(1,0){48}}
\put(0,48){\line(1,0){24}}
\put(0,60){\line(1,0){12}}
\put(40,4){$s$}
\end{picture}}

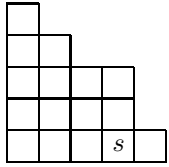
\centerline{Figure 1. The partition $(5,4^2,2,1)$ and a box in it}

\vspace{.1in}

\noindent
we have $x(s)=3$, $y(s)=0$, $l(s)=2$ and $a(s)=1$. 
\end{ex}

When we need to stress the dependance on the partition $\lambda$
we will write $a_{\lambda}(s)$ and $l_{\lambda}(s)$. 
This notation extends in an obvious way to boxes $s$ 
which might lie outside of $\lambda$ (in which case, $a_{\lambda}(s)$ or $l_{\lambda}(s)$ could be negative). 
For instance, if $\lambda=(5,4^2,2,1)$ as in Figure 1 above and $x(s)=4, y(s)=2$ then $a_{\lambda}(s)=-1$ 
and $l_{\lambda}(s)=-2$. We will occasionaly refer to a box through its coordinates $s=(x(s)+1, y(s)+1)$.
As usual, the length of a partition $\lambda$ is denoted $l(\lambda)$, 
and the conjugate partition is denoted 
$\lambda'$. Finally, if $s$ is a box of a partition $\lambda$ 
then we denote by $R_s$ and $C_s$, the set of 
all boxes of $\lambda$ in the same row and same column respectively, 
as $s$, with $s$ excepted.
We call {\it $r$-partition of $n$} a $r$-tuple of partitions 
with total weight $n$. Given two $r$-partitions
$\lambda=(\lambda^{(1)},\lambda^{(2)},\dots\lambda^{(r)})$ 
and $\mu=(\mu^{(1)},\mu^{(2)},\dots\mu^{(r)})$ we write
$$\lambda\subset\mu\iff\lambda^{(a)}\subset\mu^{(a)},\,\forall a.$$

For any commutative ring $A$ we set 
\begin{equation}
\Lambda_{n,A}=A[X_1, \ldots, X_n]^{\mathfrak{S}_n} ,\qquad
\Lambda_A=A[X_1, X_2, \ldots]^{\mathfrak{S}_\infty}. 
\end{equation}
Note that $\Lambda_A$ is the Macdonald algebra of symmetric functions.
Let $\pi_n$ be the obvious projection 
\begin{equation}
\pi_n: \Lambda_A\to \Lambda_{n,A}. 
\end{equation}
For any ring $A$ let $\delta$ be the map $A\to A\otimes A$ given by
\begin{equation}
\delta(a)=a\otimes 1 + 1 \otimes a.
\end{equation}
For $r\geqslant 1$ let $\delta^{r-1}:A\to A^{\otimes r}$ be the map obtained by 
iterating $r-1$ times the map $\delta$.
Let 
\begin{equation}
e_l=e_l(X_1, X_2, \ldots),\qquad
p_l=p_l(X_1, X_2,\ldots), \qquad m_\lambda=m_\lambda(X_1,X_2,\dots)
\end{equation}
be the $l$-th elementary symmetric function,  
the $l$-th power sum polynomial  
and the monomial 
symmetric function,
see e.g., \cite[chap.~I]{Mac}. 
Let 
\begin{equation}
e_l^{(n)}=e_l(X_1, \ldots, X_n),\qquad
p_l^{(n)}=p_l(X_1, \ldots, X_n),\qquad
m_\lambda^{(n)}=m_\lambda(X_1,X_2,\ldots,X_n)
\end{equation}
be the corresponding functions in $\Lambda_{A,n}$.
If no confusion is possible we abbreviate
\begin{equation}
e_l=e_l^{(n)},\qquad p_l=p_l^{(n)},\qquad
m_\lambda=m_\lambda^{(n)}.
\end{equation}
We write also
\begin{equation}
\gathered
\Z^2_0=\Z^2\setminus (0,0),\\
\N^2_0=\N^2\setminus (0,0),\\
\mathscr E=\{(\epsilon,l)\;;\;\epsilon=-1,0,1,\;
l\geqslant 0\}\setminus (0,0),\\
\mathscr E^+=\{(\epsilon,l)\in\mathscr E\;;\;\epsilon\geqslant 0\},\\
\mathscr E^-=\{(\epsilon,l)\in\mathscr E\;;\;\epsilon\leqslant 0\},\\
\endgathered
\end{equation}
\vspace{.2in}

\section{The algebra $\SH^{\cb}$}

\vspace{.15in}

\subsection{The DDAHA}
We define
\begin{equation}
\gathered
G=GL_n,\qquad H=(\CC^\times)^n,\qquad \hen=Lie(H),\\
\CC[H]=\CC[X_1^{\pm 1}, \ldots, X_n^{\pm 1}],\quad
\CC[\hen]=\CC[x_1, \ldots, x_n],\quad\CC[\hen^*]=\CC[y_1, \ldots, y_n].
\endgathered
\end{equation}
Here $(y_1, \ldots, y_n)$ is the basis dual to 
$(x_1, \ldots, x_n)$.
The symmetric group $\Sen_n$ 
acts on $H$, $\hen$ and $\hen^*$. 
Let $s_1, \ldots, s_{n-1}$ be the standard generators of $\Sen_n$.
For $i \neq j$ let $s_{ij}$ be the transposition $(ij)$.
Finally, set $F=\C(\kappa)$ and $A=\C[\kappa]$.
The \emph{degenerate double affine Hecke algebra} (=DDAHA) 
of $G$ is the associative $F$-algebra
${\H}_n$  generated by $F[H],$ $ F[\hen^*]$ 
and $F[\Sen_n]$ subject to the following set of relations
\begin{equation}\label{E:1}
sX_i^{\pm 1}=X_{s(i)}^{\pm 1}\, s,\qquad s \in \Sen_n,
\end{equation}
\begin{equation}\label{E:2}
s_iy=s_i(y)s_i-\kappa \langle x_i-x_{i+1}, y\rangle,
\qquad y\in\mathfrak h^*,
\end{equation}
\begin{equation}\label{E:3}
[y_i,X_j]=\begin{cases} -\kappa X_i s_{ij}\qquad &\text{if}\; i < j,\\
X_i +\kappa \big( \sum_{k <i} X_k s_{ik} + \sum_{k>i} X_i s_{ik}\big) \qquad &\text{if}\; i=j,\\
-\kappa X_j s_{ij}\qquad &\text{if}\; i > j.
\end{cases}
\end{equation}
Let 
$\Sb=\frac{1}{n!}\sum_{s \in \Sen_n} s$ 
be the complete idempotent in $\CC[\Sen_n]$. 
The \emph{spherical DDAHA} of $G$ is 
\begin{equation}
\mathbf{S}{\H}_n=\Sb\cdot  {\H}_n \cdot \Sb.
\end{equation}
Let ${\H}^{+}_n\subset \H_n$ be the $F$-subalgebra  
generated by $\Sen_n$ and $\{ y_i , X_i \;;\; i\in[1,n]\}$. 
This is a deformation of the algebra of polynomial
differential operators on $H$. 
Similarly, let ${\H}^-_n\subset \H_n$ be the subalgebra 
generated by $\Sen_n$ and $\{y_i, X_i^{-1}\;;\; i\in[1,n]\}$.
Write 
\begin{equation}
\SH^\pm_n=\Sb \cdot {\H}^\pm_n\cdot \Sb,\qquad\SH^0_n=\Sb F[\hen^*]  \Sb.
\end{equation}

\vspace{.1in}

\begin{rem}
Formally setting $\kappa =0$ in the relations of $\H_n$ yields a presentation 
of the crossed product $\Diff(H) \rtimes\Sen_n$, 
with $y_i$ degenerating to $X_i\partial_{X_i}$. 
The spherical DAHA is a deformation of the ring $\Diff(H)^{\Sen_n}$ 
of symmetric differential operators on the torus $H$.
\end{rem}

\vspace{.15in}

\subsection{Filtrations on $\H_n$ and $\SH_n$}

We define the \emph{order filtration} on $\H_n$ by  
letting $y_i$ be of order 1 and  $s,X_i^{\pm 1}$  be of
order 0. 
We define the \emph{rank grading} 
on $\H_n$ by giving to $s,y_i$ the degree 0 and to $X_i^{\pm 1}$ the 
degree $\pm 1$. 
Let $\H_n[r,\leqslant\! l]$ be the piece of $\H_n$ of degree $r$ 
and of order $\leqslant l$.
The piece of degree $r$ 
and of order $\leqslant l$ in $\SH_n$ is
\begin{equation}
\gathered
\SH_n[r,\leqslant\! l]=\Sb\cdot\H_n[r,\leqslant\! l]\cdot\Sb=
\SH_n\cap\H_n[r,\leqslant\! l].
\endgathered
\end{equation}
Similarly, we set
\begin{equation}
\gathered
\SH^+_n[r,\leqslant\! l]=\Sb\cdot\H^+_n[r,\leqslant\! l]\cdot\Sb=
\SH^+_n\cap\H_n[r,\leqslant\! l].
\endgathered
\end{equation}
All the constructions given above make sense over the ring $A$.
For instance, let 
${\H}_{n,A}\subset\H_n$ be the 
$A$-subalgebra generated by 
$\Sen_n$, $A[\hen^*]$ and $A[H]$, and put
\begin{equation}
\gathered
\SH_{n,A}=\Sb \cdot {\H}_{n,A}\cdot \Sb=\SH_n\cap\H_{n,A},\\
\H_{n,A}[r,\leqslant\! l]=\H_{n,A}\cap \H_n[r,\leqslant\! l],\\
\SH^+_{n,A}[r,\leqslant\! l]=
\Sb\cdot\H^+_{n,A}[r,\leqslant\! l]\cdot\Sb=
\SH^+_{n,A}\cap\H_{n}[r,\leqslant\! l].
\endgathered
\end{equation}
The PBW theorem for ${\H}_{n,A}$ implies that
any element of has a unique decomposition in the form
\begin{equation}
\sum_{s\in\Sen_n}h_s(X)g_s(y)s,\quad
g_s(y)\in A[\hen^*],\quad
h_s(X)\in A[H].
\end{equation}
Therefore, we have
$\H_{n,A}\otimes_AF=\H_n$. Since
$\SH_{n,A}$
is a direct summand of the $A$-module $\H_{n,A},$
we have also
$\SH_{n,A}\otimes_AF=\SH_n.$
A similar argument yields
\begin{equation}\label{1.4}
\H_{n,A}[r,\leqslant \! l]\otimes_AF=\H_n[r,\leqslant\! l],\qquad
\SH_{n,A}[r,\leqslant\! l]\otimes_AF=\SH_n[r,\leqslant\! l].
\end{equation}
Let $\overline\H_{n,A}$ and $\overline\SH_{n,A}$ be the graded $A$-algebras 
associated with  the order filtrations
on $\H_{n,A}$ and $\SH_{n,A}$ respectively.
Let us state some useful consequences of the PBW theorem.
Whenever this makes sense we may abbreviate $\ad(z)$ 
for the commutator with $z$.

\vspace{.1in}

\begin{prop}
\label{prop:coco}
(a)
An element $u\in\SH_{n,A}$ is of order $\leqslant k$ if and only if
\begin{equation}\label{E:newk1}
\ad(z_1) \circ \cdots \circ
\ad(z_k) (u)\in \Sb\cdot A[H]^{\Sen_n}\cdot\Sb,\qquad\forall z_1, \ldots, z_k \in 
\Sb\cdot A[H]^{\Sen_n}\cdot\Sb.
\end{equation}

(b)
The obvious maps yield $A$-algebra isomorphisms
$$A[H\times\hen^*]\rtimes\Sen_n=\overline \H_{n,A},
\qquad
A[H\times\hen^*]^{\Sen_n}\cdot\Sb=\Sb\cdot A[H\times\hen^*]\cdot\Sb
=\overline \SH_{n,A}.$$
\end{prop}

\begin{proof} Let $\SH_{n,A}[\leqslant\! k]$ be the space of
the elements of order $\leqslant\! k$ in $\SH_{n,A}$.
We have
\begin{equation}\label{E:newk2}
\H_{n,A}[\leqslant \! k]=
\Big\{\sum_s h_s(X)g_s(y)s\;;\; \deg(g_s) \leqslant k,\; \forall s\Big\}.
\end{equation} 
Let $U_k$ be the set of elements of $\SH_{n,A}$ satisfying \eqref{E:newk1}. 
The inclusion $\SH_{n,A}[\leqslant\! k] \subset U_k$
follows from (\ref{E:3}). We prove the reverse inclusion by induction. 
For $k=0$ there is nothing to prove, so let us assume
that $U_l \subset \SH_{n,A}[\leqslant \!l]$ for all $l <k$. 
We have $\ad(X_1 + \cdots + X_n) (y_i)=X_i$ for all $i$. From this
and (\ref{E:newk2}) we deduce that 
\begin{equation}
\{u \in \H_{n,A}\;;\; \ad(X_1 + \cdots + X_n)(u) \in \H_{n,A}[<\! j]\} 
\subset \H_{n,A}[\leqslant\! j].
\end{equation}
In particular, 
we have $U_k \subset \H_{n,A}[\leqslant\! k]$. We are done. 
\end{proof}

\vspace{.1in}

\begin{lem}\label{lem:uno} The $F$-algebra $\SH_n$ is generated by 
$\Sb F[\hen^*]\Sb$ and 
$\Sb F[H]\Sb$. The $F$-algebra $\SH_n^+$ 
is generated by $\Sb F[\hen^*]\Sb$ and $\Sb F[X_1,\dots,X_n]\Sb$.
\end{lem}

\begin{proof} 
First, we have  an isomorphism
\begin{equation}
\SH_{n,A}/(\kappa)\simeq \CC\big[X_1^{\pm 1}, X_1 \partial_{X_1}, \ldots, 
X^{\pm 1}_n, X_n \partial_{X_n}\big]^{\Sen_n}.
\end{equation}
A similar result holds for $\SH_{n,A}^+=\Sb \cdot {\H}_{n,A}^+ \cdot\Sb$. 
Next, the following is well-known.

\begin{claim}
The algebra
$\CC[X^{\pm 1}_1,X_1\partial_{X_1},\ldots,X^{\pm 1}_n,
X_n\partial_{X_n}]^{\Sen_n}$ 
is generated by 
$\CC[X^{\pm 1}_1, \ldots, X^{\pm 1}_n]^{\Sen_n}$ and 
$\CC[X_1 \partial_{X_1}, \ldots, X_n\partial_{X_n}]^{\Sen_n}$.
An similar result holds
% for $\SH^+_{n,A}$ and 
for
$\CC[X_1, X_1 \partial_{X_1},\ldots, X_n, X_n\partial_{X_n}]^{\Sen_n}$.
\end{claim}

\noindent
We now prove the second statement of Lemma~\ref{lem:uno}.
We have 
\begin{equation}
\SH_{n,A}^+=\bigoplus_{r \geqslant 0}
\bigcup_{l\geqslant 0}\SH_{n,A}^+[r,\leqslant\! l]
\end{equation}
and $\SH_{n,A}^+[r,\leqslant\! l]$ is a free $A$-module 
of finite rank such that 
\begin{equation}\label{SH+RF}\SH_{n,A}^+[r,\leqslant\! l]\otimes_AF=\SH_n^+[r,\leqslant\! l],
\end{equation}
because $\SH_{n,A}^+[r,\leqslant\! l]$ is a direct summand in the $A$-module
$\H_{n,A}^+$ and $\H_{n,A}^+\otimes_AF=\H_n^+$.
The claim above implies that $\SH_{n,A}^+[r,\leqslant\! l]/(\kappa)$ is linearly spanned by
a suitable set of monomials in the elements $\Sb \sum_i X_i^l \Sb$ and 
$\Sb \sum_{i} y_i^l \Sb$
 for $l \geqslant 0$.
Thus, by Nakayama's lemma and \eqref{SH+RF}, we have that
$\SH_n^+[r,\leqslant\! l]$ is linearly spanned over 
$F$ by the same set of monomials. 
This proves the second statement in the lemma. 
The first one is now for instance a consequence of 
the fact that any element of $\SH_n$ 
belongs to $(X_1 \cdots X_n)^{-l}\cdot\SH^+_n$ for $l$ big enough. 
\end{proof} 

\vspace{.1in}

The assignment $X_i \mapsto X_i^{-1}$, $y_i \mapsto y_i$, $s \mapsto s^{-1}$ 
extends to an algebra antiautomorphism
$\pi$ of ${\H}_{n,A}$, as may be directly seen from the defining relations. 
It restricts to an algebra antiautomorphism of
$\SH_{n,A}$ taking
$\SH_{n,A}^+$ to $\SH_{n,A}^-$. Thus $\SH_{n,A}^-$ may be identified with $(\SH_{n,A}^+)^{\op}$.

\vspace{.1in}
 
\begin{rem}\label{rem:uno} 
Let $A_1$  be the localization of $A$ at the ideal $(\kappa-1)$. We define
$\H_{n,A_1}$ and $\SH_{n,A_1}$ in the obvious way. Lemma \ref{lem:uno}
holds true with $F$ replaced by $A_1$. The proof is similar to the proof
of \cite[thm.~4.6]{BEG}. It suffices to observe that the specialization
of $\SH_{n,A_1}$ at $\kappa=1$ is a simple algebra, 
because it is a (Ore) localization
of a simple spherical rational DAHA by \cite[prop.~4.1]{Su}. 
\end{rem}

\vspace{.15in}

\subsection{The polynomial representation}\label{sec:pol}
The tautological representation of 
$\Diff(H) \ltimes {\Sen_n}$ on $\CC[H]$ 
can be deformed to a representation of  ${\H}_{n,A}$ on $A[H]$, see \cite{C2}. 
This representation is defined by the following explicit formulas
\begin{align}\label{E:dunkl}
&\rho_n (s)=s,\\
&\rho_n(X_i^{\pm 1}) =X_i^{\pm 1},\\
&\rho_n(y_i)=
X_i \partial_{X_i} + \kappa \sum_{k \neq i} \frac{1-s_{ik}}{1-X_k/X_i} + 
\kappa \sum_{k <i} s_{ik}.
\end{align}
From now on we'll write
\begin{equation}
\Lambda_{n,A}=A[X_1, \ldots, X_n]^{\Sen_n},\qquad
\Wb_{n,A}=A[X_1^{\pm 1}, \ldots, X_n^{\pm 1}],\qquad
\V_{n,A}=\Wb_{n,A}^{\Sen_n}.
\end{equation}
We'll abbreviate
\begin{equation}
\Lambda=\Lambda_F,\qquad
\Lambda_n=\Lambda_{n,F},\qquad
\Wb_n=\Wb_{n,F},\qquad
\V_n=\V_{n,F}.
\end{equation}

\vspace{.15in}

\begin{theo}[Cherednik] The assignment $\rho_n$ 
defines an embedding $\rho_n:{\H}_{n,A}\to\End(\Wb_{n,A})$ 
which takes
$\mathbf{S}{\H}_{n,A}$ into $\End (\V_ {n,A})$.
\end{theo}

\vspace{.15in}

The representation $\rho_n$ is called the
\emph{polynomial representation}.
The space $\Lambda_{n,A}$ is preserved by 
the action of the subalgebra $\SH_{n,A}^+$. 
Let $\rho_n^+$ denote the corresponding faithful
representation of $\SH_{n,A}^+$ on $\Lambda_{n,A}$.
We set 
\begin{equation}\label{E:defD}
\tilde D^{(n)}_{0,l}=\Sb\,p_l(y_1,\dots,y_n)\,\Sb/l,\qquad l\geqslant 1.
\end{equation}
The elements $\tilde D_{0,l}^{(n)}$ 
generate a commutative subalgebra called the algebra of \emph{Sekiguchi operators}. The joint 
spectrum in $\Lambda_{n,A}$ of the operators $\tilde D^{(n)}_{0,l}$ 
consists of the \textit{Jack polynomials} 
$J_{\lambda}^{(n)}$, for $\lambda$ a partition with at most $n$ parts, 
and their eigenvalues are as follows \cite{Mac}.
See Section \ref{sec:SH+infty} below for 
details on the notation for Jack polynomials.
Consider the generating function
\begin{equation}
\Delta_n(u)=
\Sb\prod_{i=1}^n(u + y_i)\Sb=\sum_{i=1}^n\Sb\,e_i\Sb\,u ^{n-i}.
\end{equation}

\vspace{.1in}

\begin{lem}[Macdonald]\label{lem:eigenvalues} 
For $l(\lambda)\leqslant n$  we have
$$\Delta_n(u) \cdot J_{\lambda}^{(n)}=
\prod_{i=1}^n \big( u + \lambda_i + \kappa(n-i) \big) \,
J_{\lambda}^{(n)}.$$
\end{lem}

\vspace{.1in}

The above lemma only gives the eigenvalues of the elements 
$\Sb\, e_i\Sb$ for 
$i\in[1,n],$ but this is enough to determine the eigenvalues of all 
the operators $\tilde D^{(n)}_{0,l}$.
In fact, Lemma~\ref{lem:eigenvalues} has the following immediate corollary.

\vspace{.1in}

\begin{cor}\label{C:Philemon} For $f\in\C[y_1, \ldots, y_n]^{\Sen_n}$
and $l(\lambda)\leqslant n$ we have
$$\Sb f\Sb\cdot J_{\lambda}^{(n)}=
f(\lambda_1+\kappa(n-1), \lambda_2+\kappa(n-2), 
\ldots, \lambda_n)\, J_{\lambda}^{(n)}.$$
\end{cor}

\vspace{.1in}

Since the joint spectrum of the $\tilde D^{(n)}_{0,l}$ is simple, the Jack 
polynomials $\{J_{\lambda}^{(n)}\}$ are completely determined 
by Lemma~\ref{lem:eigenvalues}, up to a scalar. Following Stanley 
\cite{Stanley} we normalize this scalar by requiring that . 

\begin{equation}
J_{\lambda}^{(n)} \in \bigoplus_{(1^n) < \mu \leqslant \lambda}
F m_{\mu} + |\lambda|!\, X_1\cdots X_n .
\end{equation}

For future use, we state here the 
Pieri rules for Jack polynomials \cite[thm~6.1]{Stanley}. 
For a pair of partitions $\mu \subset \lambda$ with $|\lambda|-|\mu|=1$ we write
\begin{equation}
\psi_{\lambda \backslash \mu}=\prod_{s \in C_{\lambda \backslash{\mu}}} \frac{h_{\mu}(s)}{h_{\lambda}(s)} 
\prod_{s \in R_{\lambda \backslash \mu}} \frac{h^{\mu}(s)}{h^{\lambda}(s)}
\end{equation}
where $C_{\lambda\backslash \mu}$ and $R_{\lambda \backslash \mu}$ are as in
Section \ref{sec:notation}.
Here, for any box $s$ of a 
partition $\lambda$, we have set
\begin{equation}
h_{\lambda}(s)=\kappa l_{\lambda}(s) + (a_{\lambda}(s) +1), 
\qquad h^{\lambda}(s)=\kappa(l_{\lambda}(s)+1) + a_{\lambda}(s).
\end{equation}

\vspace{.1in}

\begin{theo}[Stanley]\label{T:Pieri} 
For $l(\mu)\leqslant n$ we have
\begin{equation}\label{E:Pieri}
e_1\, J_{\mu}^{(n)}=
\sum_{\lambda} \psi_{\lambda \backslash \mu}\,
 J_{\lambda}^{(n)}, \end{equation}
where the sum ranges over all partitions $\lambda$ of length at most $n$ with $\mu \subset \lambda$ and 
$|\sigma|=|\mu| +1$.
\end{theo}

\vspace{.15in}

\subsection{The normalized Sekiguchi operators}
The eigenvalue of the operator
$\tilde D^{(n)}_{0,l}$ on the Jack polynomial 
$J_{\lambda}^{(n)}$ for $l(\lambda)\leqslant n$ 
do depend on $n$. In order to correct this, 
we will introduce a new set of diagonalisable operators $D^{(n)}_{0,l}$, whose 
eigenvalues on the 
$J_{\lambda}^{(n)}$'s are independent of $n$. We may think of these
new operators as normalized Sekiguchi operators. 
We'll use the following simple combinatorial lemma. 
Given a box $s$ in the diagram of a partition $\lambda$ we'll write
\begin{equation}\label{E:defc}
c(s)=x(s)-\kappa y(s).
\end{equation}

\vspace{.1in}

\begin{lem}\label{lem:Br} For $l\in \N$ there exists a 
unique element $B^{(n)}_l\in A[y_1, \ldots, y_n]^{\Sen_n}$
such that 
$$B^{(n)}_l(\lambda_1-\kappa,  \lambda_2-2\kappa, \ldots, \lambda_n-n\kappa)=
\sum_{ s \in \lambda} c(s)^l,
\qquad l(\lambda)\leqslant n.$$
\end{lem}

\begin{proof} There exists polynomials $T_{r,i}(z) \in A[z]$ 
such that for all $l \geqslant 0$ we have 
$$T_{r,i}(l) =\sum_{j=1}^{l} ((j-1)-\kappa (i-1))^r.$$ Then, for $l(\lambda)\leqslant n,$ we have
$$
\sum_{s \in \lambda} c(s)^r 
=\sum_{i=1}^{n} T_{r,i}(\lambda_i)
=\sum_{i=1}^{n} \tilde{T}_{r,i}(\lambda_i-\kappa i),
\qquad\tilde{T}_{r,i}(z)=T_{r,i}(z+\kappa i).$$
The existence of $B^{(n)}_r$ will be proved if we can show that 
$\tilde{T}_{r,i}(z) -\tilde{T}_{r,j}(z) \in A$ 
for any $i,j$ (as polynomials in $z$). 
For this, it is enough to show that for all $i,j$
\begin{equation}\label{E:aaatry}
\tilde{T}_{r,i}(z)-\tilde{T}_{r,i}(z-1)=\tilde{T}_{r,j}(z)-\tilde{T}_{r,j}(z-1).
\end{equation}
We have $T_{r,i}(z)-T_{r,i}(z-1)=(z-1-\kappa (i-1))^r,$ 
since this holds for any $z \in \N$. Therefore
$\tilde{T}_{r,i}(z)-\tilde{T}_{r,i}(z-1)=(z-1+\kappa)^r$ 
for any $i$, from which (\ref{E:aaatry}) is immediate. 
The unicity statement is clear. 
\end{proof}

\vspace{.15in}

Now,  we define the operators
\begin{equation}
D^{(n)}_{0,l} = \Sb B_{l-1}^{(n)}(y_1-n\kappa, y_2-n\kappa, 
\ldots, y_n-n\kappa) \Sb,\qquad l\geqslant 1.
\end{equation}
By Corollary~\ref{C:Philemon} 
and Lemma~\ref{lem:Br} we have, 
\begin{equation}\label{E:Horace}
D^{(n)}_{0,l} \cdot J_{\lambda}^{(n)}=
\sum_{s \in \lambda}c(s)^{l-1}J_{\lambda}^{(n)},\quad
 l(\lambda)\leqslant n.
\end{equation}
In particular, we have $D^{(n)}_{0,l}(1)=0$ and
the eigenvalues of  $D^{(n)}_{0,l}$ are 
independent of $n$.
It is easy to see from the 
proof of Lemma~\ref{lem:Br} that 
\begin{equation}
B_{l-1}^{(n)}=p_l/l + q_l,
\end{equation} with $q_l$ a symmetric 
function of degree $<\!l$. 
Thus
$\{B_0^{(n)}, \ldots, B_{n-1}^{(n)}\}$ is a system of generators of the 
$A$-algebra
$A[y_1, \ldots, y_n]^{\Sen_n}$. 
Hence, we have the following.

\vspace{.1in}

\begin{lem}
The $A$-algebra $\SH^0_{n,A}$ is generated by  $\{D^{(n)}_{0,l}\;;\; l \geqslant 1\}$.
\end{lem}

\iffalse%%%%%%%%%%%%%%%%%%%
\vspace{.1in}

\begin{ex}We have
$$\gathered
B_0^{(n)}=p_1^{(n)} + \kappa c_1(n),\\
B_1^{(n)}=p_2^{(n)} /2+\big(\kappa-1/2\big) p_1^{(n)} + 
\kappa\big(\kappa-1/2\big)c_1(n)-\kappa^2 c_2(n)/2,\\
D^{(n)}_{0,1}=\tilde D_{0,1}^{(n)} -\kappa n(n-1)/2,\\
D^{(n)}_{0,2}=\tilde D_{0,2}^{(n)} 
+\big( (1-n)\kappa -1/2\big)\tilde D_{0,1}^{(n)}+u,
\endgathered$$
where 
$$u=\kappa^2 n^3/2-\kappa(\kappa-1/2) c_1(n-1)-\kappa^2c_2(n)/2,\qquad
c_l(n)=p_l(1,2,\dots,n). $$
\end{ex}
\fi%%%%%%%%%%%%%%%%%%%%%%

\begin{rem}\label{rem:LBn}
For each partition $\lambda$ let  $\lambda'$ be the conjugate partition and set
$n(\lambda)=\sum_i\lambda'_i(\lambda'_i-1)/2$.
The formula \eqref{E:Horace} yields
$$D^{(n)}_{0,2}\cdot J_\lambda^{(n)}=
(n(\lambda')-\kappa n(\lambda))\,J_\lambda^{(n)}.$$ 
Thus, we have
$D^{(n)}_{0,2}=\kappa\,\square_n,$
where $\square_n$ is the
\textit{Laplace-Beltrami operator}. See e.g., 
\cite[chap.~VI, sec.~4, Ex.~3]{Mac} where $\square_n$ is denoted $\square^{\kappa^{-1}}_n$.
\end{rem}

\vspace{.15in}

\subsection{The algebras $\SH^+_{n}$ and $\SH^-_{n}$}
Our aim is to construct some limit of the algebra $\SH_n$ and of the 
representation $\rho_n$ as $n$ tends to infinity. 
The algebras $\SH_n$ do not seem to form a 
nice projective system. 
Instead, our method is as follows 
\begin{itemize}

\item first we define limits $\SH^{\pm}$ for the subalgebras $\SH^{\pm}_n$,

\item then we define $\SH$ as 
some amalgamated product of $\SH^+$ with $\SH^{-}$. 
\end{itemize}
For this we first need to understand some 
relations between $\SH^+_{n}$ and $\SH^-_{n}$ inside $\SH_n$. 
This is what we do in the present paragraph.
For $l\geqslant 1$ we set
\begin{equation}
\label{1.5:form}
\gathered
D^{(n)}_{0,0}=n\Sb,\\
D_{\pm l,0}^{(n)}=\Sb p_l^{(n)}(X^{\pm 1}_1, \ldots, X^{\pm 1}_n) \Sb, \\
D^{(n)}_{1,l}=[D^{(n)}_{0,l+1},D^{(n)}_{1,0}], \\ 
D^{(n)}_{-1,l}=[D_{-1,0}^{(n)}, D^{(n)}_{0,l+1}].
\endgathered\end{equation}
By Lemma \ref{lem:uno}, the $F$-algebra
$\SH^\pm_n$ is generated by 
$\{D_{0, l}^{(n)},\,D_{\pm l,0}^{(n)}\,;\,l\geqslant 1\}$.

\begin{df}
Let $\SH^{>}_n$ be the $F$-subalgebra of $\SH^+_n$ generated by 
$\{D_{1,l}^{(n)}\;;\; l \geqslant 0\}$.
We define the $F$-subalgebra $\SH^<_n$ of $\SH^-_n$ in a similar way. 
\end{df}

\begin{ex} The following identities hold
\begin{equation}\label{E:rel4}
\gathered
D^{(n)}_{1,1}=\Sb \big( \sum_i X_i y_i \big)\Sb- 
\kappa(n-1) D_{1,0}^{(n)}/2,\\
D^{(n)}_{-1,1}=\Sb \big( \sum_i y_i X_i^{-1}\big)\Sb-
\kappa(n-1)D_{-1,0}^{(n)}/2,\\
[D^{(n)}_{1,1}, D^{(n)}_{l,0}]=l D^{(n)}_{l+1,0}, \qquad [D^{(n)}_{-l,0},D^{(n)}_{-1,1}]=l D^{(n)}_{-l-1,0},
\qquad l\geqslant 0.
\endgathered
\end{equation}
\end{ex}

\vspace{.1in}

The following is immediate.

\begin{prop}\label{prop:1.16} For $l\geqslant 0$ the following hold

(a) $D_{l,0}^{(n)}\in\SH^>_n$  and $D_{-l,0}^{(n)}\in\SH^<_n$ for $l\neq 0$,

(b) for $l(\mu)\leqslant n$  we have
\begin{equation}\label{E:Pieri2}
D^{(n)}_{1,l} \cdot J_{\mu}^{(n)}=\sum_{\lambda} c(\lambda \backslash \mu)^l \,
\psi_{\lambda \backslash \mu}\, J_{\lambda}^{(n)}
\end{equation}
where 
the sum ranges over all partitions $\lambda$ with $l(\lambda)\leqslant n,$ 
$\mu \subset \lambda$ and 
$|\lambda|=|\mu| +1$.
\end{prop}

\vspace{.15in}

Note that \eqref{E:Pieri2} and \eqref{E:Horace} imply that \eqref{1.5:form} 
holds also for $l=0$.
The next result describes some of the relations between the three algebras 
$\SH^>_n$, $\SH^0_n$ and $\SH^<_n$. As we will see in Proposition 
\ref{P:triang} below, these relations (which, thanks to the 
introduction of $D^{(n)}_{0,0}$, do not depend on $n$) are the only ones which survive in the limit 
$n \to \infty$. 
For $l\geqslant 0$ we write
\begin{equation}\gathered
\xi=1-\kappa,\\
\lg_0(s)=-\log(s),\\
\lg_l(s)=(s^{-l}-1)/l,\qquad l\neq 0,\\
\varphi_l(s)=
\sum_{q=1,-\xi,-\kappa}s^l\bigl(\lg_l(1-qs)-\lg_l(1+qs)\bigr).
\endgathered
\end{equation}

\vspace{.1in}

\begin{prop}
\label{P:<0>} 
The following relations hold in $\SH_n$
\begin{equation}\label{E:rel'1}
[D^{(n)}_{0,l}, D^{(n)}_{1,k}]=D^{(n)}_{1,l+k-1},
\end{equation}
\begin{equation}\label{E:rel'2}
[D^{(n)}_{0,l}, D^{(n)}_{-1,k}]=-D^{(n)}_{-1,l+k-1},
\end{equation}
\begin{equation}\label{E:rel3}
[D_{-1,k}^{(n)},D_{1,l}^{(n)}]=E^{(n)}_{k+l}
\end{equation}
where the elements $E^{(n)}_{k+l}$ are determined through the formula
\begin{equation}\label{E:rel35}
\gathered
1+\xi \sum_{l\geqslant 0} E^{(n)}_{l}\,s^{l+1}=K(\kappa,D_{0,0}^{(n)},s) 
\;\exp\big(\sum_{l\geqslant 0} D^{(n)}_{0,l+1}\, \varphi_l(s)\big),\\
K(\kappa,\omega,s)
=\frac{\big( 1 + \xi s\big)\big( 1 + \kappa\omega s\big)}
{1 +\xi s+ \kappa \omega s}.
\endgathered
\end{equation}
\end{prop}

\begin{proof} The first two relations are easily deduced from (\ref{E:Horace}) 
and (\ref{E:Pieri2}), and from the faithfulness of the polynomial
representation $\Lambda_n$. The third relation is the result of a direct 
computation, see Appendix \ref{app:A}. 
\end{proof}

\vspace{.1in}

From now on we'll abbreviate $\otimes=\otimes_F$ 
(the tensor product of $F$-vector spaces).

\vspace{.1in}

\begin{prop}\label{P:SH>} 
The multiplication map induces isomorphisms
$$\SH^>_n \otimes \SH^0_n\to\SH^+_n,\qquad \SH^0_n \otimes \SH^<_n\to\SH^-_n.$$
\end{prop}

\begin{proof}  
By  Lemma \ref{lem:uno}, the algebra
$\SH^+_n$ is generated by the pair of subalgebras $\SH^>_n$, $\SH^0_n$. 
Next,  \eqref{E:rel'1} implies that
$[D^{(n)}_{0,l}, \SH^>_n] \subset \SH^>_n$ for $l\geqslant 0.$ 
Thus we have
$\SH^0_n \cdot \SH^>_n = \SH^>_n \cdot \SH^0_n$. 
The surjectivity of the multiplication map 
\begin{equation}
m:\SH^>_n \otimes \SH^0_n {\to} \SH^+_n
\end{equation}
follows. 
To show that $m$ is injective, we may use a degeneration 
argument similar to the one in 
Lemma~\ref{lem:uno}. 
We leave the details to the reader.
\end{proof}

\vspace{.1in}

\begin{cor}\label{C:triangn} 
The multiplication map induces a surjective map 
$\SH_n^> \otimes \SH^0_n \otimes \SH^<_n \to \SH_n$.
\end{cor}

\begin{proof} By Proposition~\ref{P:SH>} and Lemma \ref{lem:uno}, 
the $F$-algebra $\SH_n$ is generated by the triplet of subalgebras $\SH^>_n$, 
$ \SH^0_n$, $ \SH^>_n$, hence by the collection of generators $\{D_{1,l}^{(n)}, D_{0,l}^{(n)}, D_{-1,l}^{(n)}\}$. 
We must check that any monomial in these generators may be `straightened' 
into a linear combination of monomials in which the generators 
$\{D_{1,l}^{(n)}$, $D_{0,l}^{(n)}$, $D_{-1,l}^{(n)}\}$ appear in that fixed order.  It is not difficult to see that 
relations \eqref{E:rel'1}-\eqref{E:rel3} enable one to do this. \end{proof}

\vspace{.15in}

\subsection{The algebra $\SH^{+}$} 
\label{sec:SH+infty}
Let us now address the problem of constructing a 
limit $\SH^{+}$ of $\SH^+_n$. 
The following result is well-known, see e.g., \cite[Prop. 2.5]{Stanley}.

\begin{lem} For  $l(\lambda)\leqslant n$ 
and for any positive integer $m<n$ we have
$$J_{\lambda}^{(n)}(X_1, \ldots, X_m, 0, \ldots, 0)=
\begin{cases} J_{\lambda}^{(m)}(X_1, \ldots, X_m) \ &
 \text{if}\  l(\lambda) \leqslant m,\\ 0 \ & \text{if}\  l(\lambda) >m. \end{cases}$$
\end{lem}

\vspace{.1in}

This lemma allows one to define the limit of the symmetric polynomials
$J_{\lambda}^{(n)}$ as 
$n$ tends to infinity. We will write $J_\lambda=J_{\lambda}({X})$ 
for this limit. 
It is called the \textit{integral form of Jack's
symmetric function associated with the parameter $\alpha=1/\kappa$}.
It is denoted by the symbol $J_\lambda^{(1/\kappa)}$ in 
\cite[chap.~VI,(10.22-3)]{Mac}.
The family $\{J_{\lambda}\;;\; \lambda \in \Pi\}$ 
forms a $F$-basis of 
$\Lambda$,
see \cite[chap.~VI]{Mac}. 
The map
$\pi_n :\Lambda \to \Lambda_n$
is given by
$\pi_n(J_{\lambda})=J_{\lambda}^{(n)}$ if 
$l(\lambda) \leqslant n$ and $\pi_n(J_{\lambda})=0$ otherwise.
The operators $D^{(n)}_{l,0},$ 
for $l \in \N$, being the multiplication in $\Lambda_n$ 
by symmetric functions, obviously stabilize in the limit $\Lambda$, 
since $\Lambda$ is a ring. For instance $D^{(n)}_{1,0}$ 
is given by the Pieri formula 
(\ref{E:Pieri}), whose coefficients are independent of $n$. 
In other words, we have 
\begin{equation}
\pi_{n+1,n} \circ  D^{(n+1)}_{l,0}= D^{(n)}_{l,0} \circ \pi_{n+1,n}
\end{equation}
where we have denoted by 
\begin{equation}
\pi_{n+1,n} : \Lambda_{n+1} \to \Lambda_n
\end{equation}
the projection maps. The kernels of the maps $\pi_{n+1,n}$ 
are linearly spanned 
by Jack polynomials, and  the operators $D^{(n)}_{0,l}$ are diagonalisable on the basis of Jack polynomials 
with eigenvalues independent of $n$. 
This implies that for all $n,l\geqslant 1$ we have
\begin{equation}
\pi_{n+1,n} \circ  D^{(n+1)}_{0,l}= D^{(n)}_{0,l} \circ \pi_{n+1,n}.
\end{equation}
Since the polynomial representation is faithful and since 
the $F$-algebra $\SH^+_n$ is generated by 
\begin{equation}
\{D^{(n)}_{0,l}, \,D^{(n)}_{l,0}\;;\;l \geqslant 1\},
\end{equation}
we deduce that the assignement
\begin{equation}
D^{(n+1)}_{0,l} \mapsto D^{(n)}_{0,l}, \qquad 
D^{(n+1)}_{l,0} \mapsto D^{(n)}_{l,0}
\end{equation}
extends to a well-defined and surjective $F$-algebra homomorphism 
\begin{equation}
\Phi_{n+1,n}: \SH^+_{n+1} \to \SH^+_n.
\end{equation}
This allows us to consider the following algebra.

\vspace{.1in}

\begin{df}\label{df:1.1} 
We define $\SH^{+}$ to be the $F$-subalgebra of 
$\prod_{n\geqslant 1}\SH^+_n$ generated by the families 
$D_{0,l}=(D^{(n)}_{0,l})$ and
$D_{l,0}=(D^{(n)}_{l,0})$ with
$l \geqslant 1.$ 
\end{df}

\vspace{.1in}

By construction, there are surjective maps 
\begin{equation}
\Phi_n~: \SH^{+} \to \SH^+_n,\quad
D_{0,l}\mapsto D^{(n)}_{0,l},\quad
D_{l,0}\mapsto D^{(n)}_{l,0},\quad l\geqslant 1,
\end{equation}
such that $\bigcap_n \Ker(\Phi_n)=\{ 0\}$. 
Further, we have the following.

\vspace{.1in}

\begin{prop}\label{prop:rho+} 
There is a faithful representation
$\rho^+$ of $\SH^{+}$ on 
$\Lambda$ such that, for $l\geqslant 1,$
$$\rho^+(D_{0,l})(J_{\lambda})=\sum_{s \in \lambda} 
c(s)^{l-1} J_{\lambda},\qquad\rho^+(D_{l,0})= \text{multiplication\;by\;} p_l.$$
The map $\pi_n$ intertwines the representation
$\rho^+$ with the representation
$\rho_n^+$ of $\SH^+_n$ on $\Lambda_n$.
\end{prop}

\vspace{.1in}

Observe that $\{D_{0,l}\;; l \geqslant 1\}$ generates a free commutative 
algebra which is isomorphic to $\Lambda$.
The same holds for $\{D_{l,0}\;; l \geqslant 1\}$.
We define a $\N$-grading on $\SH^{+},$
called the \emph{rank grading}, 
by putting $D_{l,0}$ in degree $l$ and $D_{0,l}$ in degree $0$. We define
a $\N$-filtration on $\SH^{+},$ called the \emph{order filtration}, such that 
an element $u$ is of order $\leqslant k$ if 
\begin{equation}
\ad(z_1) \circ \cdots \circ
\ad(z_k) (u)\in F[D_{l,0}\;;\;l\in\N],\qquad\forall z_1, \ldots, z_k 
\in F[D_{l,0}\;;\;l\in\N].
\end{equation}
Let $\SH^{+}[r,\leqslant\! l]$ the piece of degree $r$ and order
$\leqslant l$. Note that any element of $\SH^+$ has indeed a finite order.
Consider the Poincar\'e polynomial 
\begin{equation}
P_{\SH^{+}}(t,q)=\sum_{r,l \geqslant 0} 
\dim\bigl(\overline\SH^{+}[r, l]\bigr) \, t^rq^l,\qquad
\overline\SH^{+}[r, l]=\SH^{+}[r,\leqslant\! l]
/\SH^{+}[r, <\!l].
\end{equation}

\begin{lem} \label{lem:poincare+}
The Poincar\'e polynomial  of
$\SH^{+}$ is given by
$$P_{\SH^{+}}(t,q)=
\prod_{r,l}\frac{1}{1-t^rq^l},\qquad (r,l)\in\N^2_0.$$
\end{lem}

\begin{proof} 
By Proposition \ref{prop:coco}$(b)$, the $F$-vector space
\begin{equation}
\overline\SH^+_{n}[r,l]=\SH^+_{n}[r,\leqslant\! l]/\SH^+_{n}[r, <\!l]
\end{equation}
is isomorphic to the subspace of polynomials in 
\begin{equation}
F[X_1,\dots,X_n,y_1,\dots,y_n]^{\Sen_n}
\end{equation}
of degree $r$ in the $X_i$'s and of degree $l$ in the $y_i$'s.
By Proposition \ref{prop:coco}$(a)$
we have 
\begin{equation}
\Phi_n(\SH^{+}[r,\leqslant \!l])=\SH^+_{n}[r,\leqslant \!l].
\end{equation}
Thus $\Phi_n$ induces a surjective map
\begin{equation}
\overline\Phi_n~:\overline \SH^{+} \to \overline\SH^+_n,\quad
\overline D_{0,l}\mapsto p_l^{(n)}(y_1,\dots,y_n)/l,\quad
\overline D_{l,0}\mapsto p_l^{(n)}(X_1,\dots,X_n),\quad l\geqslant 1.
\end{equation}
Thus $\overline{\SH}^{+}[r,l]$ is identified with the space of 
symmetric polynomials in infinitely many variables
\begin{equation}\label{toto1.6}
F[X_1,X_2,\dots,y_1, y_2\dots]^{\Sen_\infty}
\end{equation}
of degree $r$ in the $X_i$'s and degree $l$ in the $y_i$'s. 
By Weyl's theorem the $F$-algebra 
\eqref{toto1.6}
is freely generated by the invariants $\sum_{k\geqslant 1} X_k^ry_k^l$ for 
$r,l\geqslant 0$ and $(r,l)\neq(0,0)$.
The result easily follows.
\end{proof}

\vspace{.1in}

\begin{rem}
The order filtration on $\SH^+$ is \textit{not} the same as the
filtration given by putting  $D_{l,0}$ 
of order $0$ and $D_{0,l}$ of order $\leqslant l$ (see however Proposition~\ref{prop:filtration-SHc}). 
\end{rem}

\vspace{.1in}

\begin{rem}\label{rem:LBinfty} We have
$\rho^+(D_{0,2})=\kappa\,\square,$
where $\square$
is the Laplace-Beltrami operator in infinitely many variables,  i.e.,
$\square=\square^{\kappa^{-1}}=
\underset{\longleftarrow}{\lim} \;\square^{\kappa^{-1}}_n$ in Macdonald's notations,
see Remark\;\ref{rem:LBn}.
\end{rem}

\vspace{.1in}

\begin{rem}\label{rem1.26} There is a unique $F$-algebra homomorphism 
$\varepsilon^+:\SH^+\to F$ such that 
$\varepsilon^+(D_{0,l})=\varepsilon^+(D_{l,0})=0$ for  $l\geqslant 1$.
Indeed,  the sum of 
$\bigoplus_{r\geqslant 1}\SH^+[r]$ and of
the augmentation ideal of $F[D_{0,l}\,;\,l\geqslant 1]$ is a two-sided ideal of
$\SH^+$.
\end{rem}

\vspace{.15in}

\subsection{The algebra $\SHo$} 
\label{sec:SHinfty}
Our next objective is the construction of the  
limit of the whole algebra $\SH_n$.
We construct 
$\SHo$ by `gluing' together two copies 
of $\SH^{+}$, denoted $\SH^{+}$ 
and $\SH^{-}$, with 
$\SH^{-} = (\SH^{+})^{\op}$, 
along the subalgebra 
\begin{equation}
\SHoo=F[D_{0,l}\,;\,l\geqslant 0].
\end{equation}
The extra generator $D_{0,0}$, 
which accounts for the limit of the $D^{(n)}_{0,0}$'s, may be 
considered as a formal parameter. We'll write $\omega=D_{0,0}$.
For $l\geqslant 1$ let $D_{-l,0} \in \SH^{-}$ be the element 
mapping to $D_{-l,0}^{(n)}$ for any $n$.
Consider elements
\begin{equation}\label{E:defD1l}
D_{1,l}=[D_{0,l+1}, D_{1,0}] , 
\quad
D_{-1,l}=[D_{-1,0},D_{0,l+1}], \quad l \geqslant 0.
\end{equation}
Let $\SH^>$ be the $F$-subalgebra of $\SH^{+}$ 
generated by $\{D_{1,l}\;;\; l \geqslant 0\}$. This is 
the limit of $\SH^>_n$ as $n$ tends to infinity. Now put
$\SH^<=(\SH^>)^{\op}.$
We may view $\SH^{-}$ and $\SH^<$ as the limits
of $\SH^-_n$ and $\SH^<_n$ respectively. 
Note that $\SH^<$ is the $F$-subalgebra of $\SH^{-}$ 
generated by $\{D_{-1,l}\;;\; l \geqslant 0\}$.
We define
\begin{equation}
\label{1.7:filt}
\SH^>[r,\leqslant\! l]=\SH^>\cap\SH^{+}[r,\leqslant\! l],
\qquad
\SH^>[r,\leqslant\! l]=\SH^>\cap\SH^{+}[r,\leqslant\! l].
\end{equation}

\vspace{.1in}

\begin{df} Let $\SHo$ be the $F$-algebra 
generated by $\SH^>,$ $\SHoo$ and $\SH^<$ 
modulo the following set of relations
\begin{equation}\label{E:rel0bis}
\omega=D_{0,0} \; \text{is\;central,}
\end{equation}
\begin{equation}\label{E:rel1bis}
[D_{0,l}, D_{1,k}]=D_{1,l+k-1},\quad l\geqslant 1,
\end{equation}
\begin{equation}\label{E:rel2bis}
[D_{0,l}, D_{-1,k}]=-D_{-1,l+k-1},\quad l\geqslant 1,
\end{equation}
\begin{equation}\label{E:rel3bis}
[D_{-1,k},D_{1,l}]=E_{k+l},\quad l,k\geqslant 0,
\end{equation}
where the elements $E_{k+l}$ are determined through the formula
\begin{equation}\label{E:formulaklop2}
\gathered
1+\xi \sum_{l\geqslant 0} E^{}_{l}s^{l+1}=K(\kappa,\omega,s) 
\exp\big(\sum_{l\geqslant 0} D^{}_{0,l+1} \varphi_l(s)\big).
\endgathered
\end{equation}
\end{df}

\vspace{.1in}

\noindent 
By Proposition~\ref{P:<0>}, there are surjective maps 
\begin{equation}
\Phi_n: \SHo \to \SH_n,\quad
D_{0,l}\mapsto D^{(n)}_{0,l},\quad
D_{\pm l,0}\mapsto D^{(n)}_{\pm l,0},\quad 
\omega\mapsto nS,\quad l\geqslant 1.
\end{equation}
As above, for each $l \geqslant 0$ we write
$D_{0,l}$ and
$D_{\pm l,0}$ for the families
$(D^{(n)}_{0,l})$ and
$(D^{(n)}_{\pm l,0})$ in
$\prod_{n\geqslant 1}\SH_n$. 
The definition of $\SHo$ is justified by the following
result.

\vspace{.1in}

\begin{prop}\label{P:triang} 
(a) The multiplication map induces  isomorphisms 
$$\SH^> \otimes \SHoo\simeq 
\SH^{+}\otimes F[\omega],\quad
\SHoo \otimes \SH^< \simeq 
\SH^{-}\otimes F[\omega],\quad
\SH^> \otimes \SHoo \otimes 
\SH^< \simeq \SHo.$$

(b) The map $\prod_{n\geqslant 1}\Phi_n$ identifies
$\SHo$ with the $F$-subalgebra of 
$\prod_{n\geqslant 1}\SH_n$ generated by $D_{0,l}$ and
$D_{\pm l,0}$ with
$l \geqslant 0.$ 
\end{prop}

\begin{proof} The surjectivity statements in $(a)$
is proved as in Proposition~\ref{P:SH>} 
and Corollary~\ref{C:triangn}. To prove $(a)$
it thus remains to show that the multiplication map 
$$m: \SH^> \otimes \SHoo \otimes \SH^< \to \SHo$$
is injective.
Consider the following commutative diagram
\begin{equation}
\begin{split}
\xymatrix{\SH^> \otimes \SHoo \otimes \SH^< \ar[r]^-{m} \ar[d]_-{\Phi_n^{\otimes 3}} & \SHo \ar[d]^-{\Phi_n}\\
\SH^>_{n} \otimes \SH^{0}_n \otimes \SH^<_{n} \ar[r]^-{m} & \SH_n.
}
\end{split}
\end{equation}
Let $u \in \Ker(m)$ and assume that $u \neq 0$. 
There exists positive integers $r_1,r_3,l_1,l_2,l_3$ such that
$$u \in \SH^>[\leqslant\! r_1, \leqslant\! l_1] \otimes 
\SHoo[\leqslant\! l_2] \otimes \SH^<
[\leqslant\! r_3, \leqslant\! l_3].$$ 
By Definition \ref{df:1.1} we have 
\begin{equation}\label{inclusion11}
\SH^>\subset\prod_{n\geqslant 1}\SH_n^>,\qquad
\SH^<\subset\prod_{n\geqslant 1}\SH_n^<,\qquad
\SH^0\subset\prod_{n\geqslant 1}\SH_n^0.
\end{equation}
Since we have 
$$\SH_n^0=F[D_{0,l}^{(n)}\;;\;l\geqslant 1],\quad
\SHoo=F[D_{0,l}\;;\;l\geqslant 0],$$
we have also an inclusion
$\SHoo\subset\prod_{n\geqslant 1}\SH_n^0$
which identifies the element $\omega=D_{0,0}$ with the family $(nS)$.
Thus, for $n \gg 0$ we have $\Phi_n^{\otimes 3}(u) \neq 0$.
By passing to the associated graded and using the PBW theorem, 
we see that the restriction to
$$\SH^>_{n}[\leqslant\! r_1, \leqslant\! l_1] \otimes 
\SH^{0}_{n}[\leqslant\! l_2] 
\otimes \SH^<_{n}[\leqslant\! r_1, \leqslant\! l_3]$$
of the map $m$ is injective
for $n \gg 0$. But then $\Phi_n \circ m (u) \neq 0$, a contradiction. 
This shows that $\Ker(m)=\{0\}$. 
Our argument also implies that $\bigcap_n \Ker(\Phi_n \circ m)=\{0\}$.
Hence $\bigcap_n \Ker(\Phi_n)=\{0\}$ because $m$ is surjective.
This implies the part $(b)$.
 
\end{proof}

As a direct consequence of Lemma \ref{lem:poincare+} 
and Proposition~\ref{P:triang} $(a)$ 
we have the following.

\begin{cor}\label{P:poincareSH>} The Poincar\'e polynomials 
of $\SH^>$ and $\SH^<$ are
respectively given by
$$P_{\SH^>}(t,q)=
\prod_{r>0}\prod_{l \geqslant 0} \frac{1}{1-t^rq^l}, \qquad 
P_{\SH^<}(t,q)=\prod_{r<0}\prod_{l \geqslant 0} \frac{1}{1-t^rq^l}.$$
\end{cor}

\vspace{.1in}

For a future use, let us mention also the following basic facts.

\begin{prop}\label{prop:generators/involution} 
(a) The $F$-algebra $\SHo$ is generated by the elements 
$\omega$, $D_{1,0}$, $D_{-1,0}$, $D_{0,2}$.

(b) There is a unique anti-involution $\pi$ of 
$\SHo$ such that $\pi(D_{\pm 1,l})=D_{\mp 1,l}$,
$\pi(D_{0,l})=D_{0,l}$. 
\end{prop}

\begin{proof} From \eqref{E:rel1bis}-\eqref{E:rel2bis}
we see that $D_{\pm 1, l}$ is an iterated commutator 
of $D_{\pm 1, 0}$ and $D_{0,2}$. From \eqref{E:rel3bis} we see 
that $\SHoo$ is 
generated by the commutators $[D_{-1,k}, D_{1,l}]$ for $k, l \geqslant 0$.
This proves $(a)$. Part $(b)$ is obvious.
\end{proof}

\vspace{.1in}

%\begin{rem}
%The element $\omega$ of $\SHo$ is central. We denote by $\SH$ the 
%specialization of $\SHo$ at $\omega=0$. For $\omega=0$
%formula \eqref{E:formulaklop2}  becomes
%\begin{equation}\label{E:formulaklop3}
%1+\xi \sum_{l\geqslant 0} E_{l}\,s^{l+1}= \exp\big(\sum_{l\geqslant 0} D^{}_{0,l+1}\, \varphi_l(s)\big).
%\end{equation}
%\end{rem}

\begin{rem}
\label{rem:com}
Note that 
$\{D_{l,0}\;;\;l\in\Z\}$
generates a commutative subalgebra of $\SHo$
(use Proposition \ref{P:triang}$(b)$ and the commutativity of the elements
$D_{l,0}^{(n)}$, $l\in\Z$, in $\SH_n$). 
\end{rem}

\vspace{.1in}

\begin{rem}
In Corollary \ref{cor:SC/SH} we'll give an explicit description of the 
subalgebra $\SH^>$ of $\SH^{+}$.
\end{rem}

\vspace{.15in}

\subsection{The algebra $\SH^{\cb}$}
\label{sec:SHc}
Now, we define a central extension 
$\SH^{\cb}$ of $\SH$. 
To do this, we introduce a new
family $\cb=(\cb_0,\cb_1,\dots)$ of formal parameters, 
and for $l\geqslant 0$ we set
\begin{equation}
\label{1.68}
\phi_l(s)=s^l\lg_l(1+\xi s),\qquad 
\SH^{\cb,0}=F^\cb [D_{0,l}\;;\;l\geqslant 0],\qquad
F^\cb=F[\cb_{l}\;;\;l\geqslant 0].
\end{equation}

\vspace{.1in}

\noindent
\begin{df}
\label{1.8:def}
Let $\SH^{\cb}$ be the $F$-algebra generated by 
$\SH^>$, $\SH^{\cb,0}$, $\SH^<$ 
 modulo the following set of relations
\begin{equation}\label{E:rel0bis2}
\cb_l \; \text{is\;central,}
\end{equation}
\begin{equation}\label{E:rel1bis2}
[D_{0,l}, D_{1,k}]=D_{1,l+k-1},\quad l\geqslant 1,
\end{equation}
\begin{equation}\label{E:rel2bis2}
[D_{0,l}, D_{-1,k}]=-D_{-1,l+k-1},\quad l\geqslant 1,
\end{equation}
\begin{equation}\label{E:rel3bis2}
[D_{-1,k},D_{1,l}]=E_{k+l},\quad l,k\geqslant 0,
\end{equation}
where $D_{0,0}=0$ and the elements $E_{k+l}$ are determined through the formula
\begin{equation}
1+\xi\sum_{l\geqslant 0} E_{l}\,s^{l+1}=
\exp\bigl(\sum_{l\geqslant 0}(-1)^{l+1}\cb_l\phi_l(s)\bigr)\,
\exp\big(\sum_{l\geqslant 0} D_{0,l+1}\, \varphi_l(s)\big).
\end{equation}
\end{df}

\vspace{.1in}

\begin{rem} 
Given a family $c=(c_0,c_1,\dots)$ of elements in an extension of the field $F$,
let $\SH^{c}$ be the specialization of $\SH^{\cb}$ at $\cb=c$. 
The specialization at $\cb=0$ is canonically isomorphic to the specialization of
$\SHo$ at $\omega=0$. 
Next, a direct computation shows that
$$K(\kappa,\omega,s)=\exp\Big(\sum_{l\geqslant 0}(-1)^{l+1}(\delta_{l,0}-\kappa^l\omega^l)\phi_l(s)\Bigr).$$
Therefore, taking $c_0=0$ and $c_l=-\kappa^l\omega^l$ in $F(\omega)$ for $l\geqslant 1$,
we get an $F(\omega)$-algebra isomorphism
$\SH^{c}\to\SHo$ such that 
$D_{1,l}\mapsto D_{1,l}$ and $D_{-1,l}\mapsto D_{-1,l}$
for each $l\geqslant 0$. 
\end{rem}

\vspace{.1in}

\begin{rem}
We abbreviate $\SH^{\cb_0,\cb_1}$ for the algebra
associated with the familly of parameters $(\cb_0,\cb_1,0,\dots)$.
By Remark \ref{rem:combA} there is an algebra isomorphism
$\SH^\cb\to\SH^{\cb_0,\cb_1} \otimes F[\mathbf{c}_l\;;\; l \geqslant 2]$ 
such that 
$D_{1,l}\mapsto D_{1,l}$ and $D_{-1,l}\mapsto D_{-1,l}$
for each $l\geqslant 0$. 
In other words, the algebra $\SH^\cb$ depends only on the parameters
$\cb_0$, $\cb_1$ up to isomorphisms.
\end{rem}

\vspace{.1in}

\begin{prop} 
\label{prop:generators/involution2}
(a) The $F$-algebra $\SH^{\cb}$ is generated by 
$\cb_l$, $D_{1,0}$, $D_{-1,0}$, $D_{0,2}$.

(b) There is a unique anti-involution $\pi$ of
$\SH^{\cb}$ such that 
$\pi(\cb_l)=\cb_l,$
$\pi(D_{\pm 1,l})=D_{\mp 1,l}$
and
$\pi(D_{0,l})=D_{0,l}.$
\end{prop}

\vspace{.1in}

\begin{proof} 
Parts $(a)$, $(b)$ are proved as Proposition~\ref{prop:generators/involution}.
\end{proof}

\vspace{.1in}

The following specialization of the algebra $\SH^\cb$ will be important for us. 

\vspace{.05in}

\begin{df}\label{def:spe}
For a field extension $F\subset K$ and an integer $r>0$ let $K_r=K(\eps_1,\dots,\eps_r),$
where $\eps_1,\dots,\eps_r$ are new formal variables. Consider the algebra homomorphism
$F^\cb\to K_r$, $\cb_l\mapsto c_l=p_l(\eps_1,\dots,\eps_r)$.
We define the $K_r$-algebra $\SH^{(r)}_{K}=\SH^{\cb}\otimes_{F^\cb} K_r$.
We write also 
$\SH^{(r),>}_{K}=\SH^>\otimes K_r$ and
$\SH^{(r),<}_{K}=\SH^<\otimes K_r.$
\end{df}

\iffalse%%%%%%%%%%%%%%%%%%%%
\begin{rem}
Set $S_r=K(\eps_1,\dots,\eps_r)^{\Sen_r}.$
We define the $S_r$-algebra $\SH^{(r)}_{S_r}$ to be the specialization of
$\SH^{\cb}\otimes S_r$ at 
$\cb=(c_0,c_1,\dots)$ with $c_l=p_l(\eps_1,\dots,\eps_r)$ for
$l\geqslant 0.$ Then, we have
$$\SH^{(r)}_{K}=\SH^{(r)}_{S_r}\otimes_{S_r}K_r,\qquad
\SH^{(r)}_{S_r}=\big(\SH^{(r)}_{K}\big)^{\Sen_r}.$$
\end{rem}
\fi%%%%%%%%%%%%%%%%%%%%%%

\vspace{.1in}

We can now prove the following.

\vspace{.05in}

\begin{prop}
\label{1.8:prop1}
The multiplication map
$\SH^> \otimes \SH^{\cb,0} \otimes 
\SH^< \to \SH^{\cb}$
is an isomorphism. 
\end{prop}

\begin{proof}
The injectivity follows from Corollary \ref{cor:injectivity} and the commutativity of the diagram
\begin{equation}\label{diagram0}
\begin{split}
\xymatrix{
\SH_K^{(r),>}\otimes\SH_K^{(r),0}\otimes
\SH_K^{(r),<}\ar[r]&
\SH_K^{(r)}\\
(\SH^>\,\otimes\,\SH^{\cb,0}\,\otimes\,\SH^<)\otimes K_r\ar[r]\ar[u]
&\SH^\cb\otimes K_r\ar[u]
}
\end{split}
\end{equation}
for all $r$.
The vertical maps are given by the specialization and the horizontal ones by the multiplication.
The proof of the surjectivity
is similar to the proof of Corollary \ref{C:triangn}. 
Since $\SH^\cb$ is generated by 
$\SH^>,$ $ \SH^{\cb,0}$ and
$\SH^<$, and since these subalgebras are respectively generated by 
$\{D_{1,l}\;; \;l \geqslant 1\}$, $\{D_{0,l}\;;\; l \geqslant 0\}$ 
and $\{D_{-1,l}\;;\;l \geqslant 0\}$,
it suffices to prove that any monomial 
$D_{\xb_1} \cdots D_{\xb_s}$ may be expressed as a linear 
combination of monomials in which the generators
$D_{1,l}$, $D_{0,l}$, $D_{-1,l}$ appear in that fixed order. 
The relations \eqref{E:rel0bis2}-\eqref{E:rel3bis2}
allow one to do that.
\end{proof}

\vspace{.1in}

\begin{rem} \label{rem1.42} There is a unique $F$-algebra homomorphism
$\varepsilon:\SH^\cb\to F$ such that $\varepsilon|_{\SH^+}=\varepsilon^+$,
$\varepsilon|_{\SH^-}=\varepsilon^+\circ\pi$ and $\varepsilon(\cb_l)=0$ for each
$l\in \N$. Use Remark \ref{rem1.26} and Definition \ref{1.8:def}.
\end{rem}

\vspace{.15in}

\subsection{The order filtration on $\SH^{\cb}$}
In this section we extend the order filtration on $\SH^+$ to $\SH^{\cb}$.
Let $\SH^\cb[s,\leqslant\! l]$ be the image by the multiplication map of the 
$F$-vector space
\begin{equation}\label{E:filtration-SHc}
\sum_{s_1,s_3,l_1,l_2,l_3}\SH^> 
[s_1,\leqslant\! l_1]
\otimes \SH^{\cb,0}
[\leqslant\! l_2]
\otimes \SH^< 
[s_3,\leqslant\! l_3].
\end{equation}
The sum is over all tuples such that $s_1+s_3=s$ and $l_1+l_2+l_3=l$.
The $F$-subspaces
$\SH^>[s_1,\leqslant\! l_1]$ and $\SH^<[s_3,\leqslant\! l_3]$
are as in \eqref{1.7:filt}, and
$\SH^{\cb,0}[\leqslant\! l_2]$
is the $F^\cb$-subalgebra of
$\SH^{\cb,0}$
spanned by the polynomials in the elements
$D_{0,l}$ of order $\leqslant l_2$.
By Proposition \ref{1.8:prop1},
the $F$-algebra $\SH^\cb$ carries 
a $\Z$-grading and a $\N$-filtration
\begin{equation}
\gathered
\SH^\cb=\bigoplus_{s\in\Z}\SH^\cb[s],\qquad
\SH^\cb=\bigcup_{l\in\N}\SH^\cb[\leqslant\! l],\\
\SH^\cb[s]=\bigcup_{l\in\N}\SH^\cb[s,\leqslant\! l],
\qquad
\SH^\cb[\leqslant\!l]=\bigoplus_{s\in\Z}\SH^\cb[s,\leqslant\!l].
\endgathered
\end{equation}
We will prove  that $\SH^{\cb}$, with this
filtration, is a filtered algebra.
Hence, the associated graded $\overline\SH^\cb$ is
an algebra. 
Next, following \eqref{E:rel4}, we define inductively the element
$D_{l,0}\in\SH^{\cb}$ so that 
$D_{-1,0},$ $D_{1,0}$ are as above and
\begin{equation}
\label{1.8:Dl0} 
[D_{1,1}, D_{l,0}]=l D_{l+1,0}, \qquad [D_{-l,0},D_{-1,1}]=l D_{-l-1,0},
\qquad l\geqslant 0.
\end{equation}
Finally, for $l,r \geqslant 1$, we set
\begin{equation}
\label{E:defdrl}
D_{r,l}=[D_{0,l+1}, D_{r,0}], \qquad D_{-r,l}=[D_{-r,0}, D_{0,l+1}].
\end{equation}
This notation is compatible with the previous definition of 
$D_{\pm 1, l}$. The elements $D_{r,l}$ satisfy the following properties,
see Lemma \ref{lem:D5},
\begin{equation}
\gathered
D_{r,l} \in \SH^{>},\qquad
D_{-r,l} \in \SH^<,\qquad
\pi(D_{l,0})=D_{-l,0},\qquad
[D_{0,1},D_{l,0}]=lD_{l,0}.
\endgathered
\end{equation}

\vspace{.05in}

\begin{prop}\label{prop:filtration-SHc} 
(a) The element $D_{r,d}$ is of order $d$, i.e., we have
$$D_{r,d} \in \SH^>[\leqslant\! d] \;\backslash \;\SH^>[<\! d], \qquad 
D_{-r,d} \in \SH^<[\leqslant\! d] \;\backslash \;\SH^<[<\! d],\qquad
r\geqslant 1.$$
The symbols
of the elements $D_{\pm r,d}$ with $(r,d) \in \N^2_0$
freely generate $\overline{\SH}^\pm.$

(b) 
The order filtration on $\SH^{\cb}$ is determined by assigning to
$D_{r,d}$, $\cb_l$ the orders $d$ and zero.

(c) For $l_1,l_2 \geqslant 0$ we have
$\SH^{\cb}[\leqslant\! l_1] \cdot \SH^{\cb}[\leqslant l_2] \subseteq 
\SH^{\cb}[\leqslant l_1+l_2].$
\end{prop}

\begin{proof} It is known that $D^{(n)}_{0,d}$ is of order $d$ in $\SH^+_n$ 
for any $n$ hence $D_{0,d}$ is of order $d$ in $\SH^+$. Similarly, $D_{r,0}$ is of order zero. It follows
that $D_{r,d}$ is of order at most $d$. 
Let $\overline D_{\pm r,d}$ be the symbol of the element $D_{\pm r,d}$.
A direct computation shows that
\begin{equation}\label{Eq:starstar}
\overline{D}^{(n)}_{r,d}= 
c_{r,d} \sum_i \Sb\, \overline X_i^r \overline y_i^d\, \Sb\in \overline{\SH}^+_n,
\qquad c_{r,d} \in F^\times. 
\end{equation}
This means that $D^{(n)}_{r,d}$, 
hence also $D_{r,d}$, is of order $d$. Equation \eqref{Eq:starstar} 
also shows that the set
$\{\overline D^{(n)}_\xb\;;\; \xb \in \N^2_0\}$ 
generates $\overline{\SH}^+_n$, and therefore that 
$\{\overline{D}_\xb\;;\; \xb \in \N^2_0\}$ 
likewise generates $\overline{\SH}^+$. Comparing graded dimensions we get
that these same generators freely generate $\overline{\SH}^+$. This proves 
$(a)$ for $\SH^+$. The same proof works for $\SH^-$.

We now turn to part $(b)$. 
Let $\SH^{\cb}[\preccurlyeq\! l]$ temporarily denote 
the degree at most $l$ piece of $\SH^{\cb}$
with respect to the filtration defined by $(b)$. By $(a)$ we have 
$\SH^{\pm}[\preccurlyeq\! l] = \SH^{\pm}[\leqslant\! l]$ for any $l$, and
the same holds for $\SH^>$, $\SH^{\cb,0}$ and $\SH^<$. From the definition 
\eqref{E:filtration-SHc} we immediately have
$\SH^{\cb}[\leqslant\! l] \subseteq \SH^{\cb}[\preccurlyeq\! l]$. 
By construction, we have
\begin{equation}
\SH^{\cb}[\preccurlyeq\! l]=\{u_1 u_2 \cdots u_s\,;\,
u_i \in \SH^{\epsilon_i}[\leqslant\! l_i],\,\epsilon_i \in \{>,0,<\},\,l_1 + \cdots + l_s=l\}.
\end{equation} 
Thus, in order to 
show the 
inclusion $\SH^{\cb}[\preccurlyeq\! l] \subseteq \SH^{\cb}[\leqslant\! l]$, 
it is enough to prove that 
\begin{equation}\label{E:slurp0}
\SH^{\cb}[\leqslant\! l_1]\cdot\SH^{\cb}[\leqslant\! l_2]\subseteq
\SH^{\cb}[\leqslant l_1+l_2],
\end{equation}
which reduces to
\begin{equation}\label{E:slurp0.1}
\ad(D_{r,d})(\SH^{\cb}[\leqslant\! l]) \subset \SH^{\cb}[\leqslant\! l+d].
\end{equation}
Rather than using the elements $D_{r,d}$ we introduce a more convenient
set of elements. Define inductively, for $r \geqslant 2$ and $d \geqslant 1$,
\begin{equation}\label{yrd}
Y_{r,d}=\begin{cases} [D_{1,1}, Y_{r-1,d}] & \text{if}\; r -1\neq d\\ 
[D_{1,0}, Y_{r-1,d+1}] & \text{if}\; r-1=d, \end{cases} \qquad
Y_{-r,d}=\begin{cases} [D_{-1,1}, Y_{1-r,d}] & \text{if}\; r-1 \neq d\\ 
[D_{-1,0}, Y_{1-r,d+1}] & \text{if}\; r-1=d. \end{cases}
\end{equation}
We have $Y_{r,d} \in \SH^{>}$ 
and $Y_{-r,d} \in \SH^<$. One shows by arguments similar to
those used in (a) above that $Y_{r,d}$ is of order exactly $d$
and that the symbols $\overline Y_{r,d}$ freely generate $\overline{\SH}^{\pm}$.
We will now prove that
\begin{equation}\label{E:slurp1}
\ad(Y_{r,d})(\SH^{\cb}[\leqslant\! l]) \subset \SH^{\cb}[\leqslant\! l+d].
\end{equation}

Since $\ad(Y_{r,d})$ is an iterated commutator of operators 
$\ad(D_{0,l}),$ $\ad(D_{\pm 1, 1}),$ $\ad(D_{\pm 1, 0})$, it is enough to
prove \eqref{E:slurp1} for each of those. For $D_{0,l}$ this comes from the 
fact that $\SH^{\pm}$ are filtered algebras. For the
others it is enough to show that
\begin{equation}\label{E:slurp2}
\ad(D_{\pm 1 , 1}) (Y_{\pm r,l}) \in \SH^{\pm}[\leqslant\! l], \quad  
\ad(D_{\pm 1 , 0}) (Y_{\pm r,l}) \in \SH^{\pm}[<\! l] \quad 
r \geqslant 0,\quad l \geqslant 0,
\end{equation}
\begin{equation}\label{E:slurp3}
\ad(D_{\pm 1 , 1}) (Y_{\mp r,l}) \in \SH^{\mp}[\leqslant\! l], \quad  
\ad(D_{\pm 1 , 0}) (Y_{\mp r,l}) \in \SH^{\mp}[<\! l] \quad 
r > 0,\quad l \geqslant 0.
\end{equation}
Both \eqref{E:slurp2} and \eqref{E:slurp3} easily follow from the inductive 
definition of $Y_{r,d}$ and from the relations 
$$[D_{-1,1}, D_{1,1}]=E_2,\qquad 
[D_{-1,1}, D_{1, 0}]=[D_{-1, 0}, D_{1,1}]=E_1.$$
Statement $(c)$ was proved on the way. 
\end{proof}

\vspace{.15in}

\subsection{Wilson operators on $\SH^>$}
\label{sec:WilsonSH}
Recall that $\SH^0=F[D_{0,l}\,;\,l\geqslant 1]$. 
By \eqref{E:rel1bis2} the commutator with $D_{0,l}$
preserves $\SH^>$ and the operators 
$\ad(D_{0,l})$ commute with each other. This
extends uniquely to an action of the algebra $\SH^0$ on $\SH^>$ satisfying
\begin{equation}
D_{0,l} \bullet u=[D_{0,l},u], \qquad u \in \SH^>,\qquad l \geqslant 0.
\end{equation}
Recall that $\Lambda$
carries a comultiplication given by 
\begin{equation}
\Delta(p_l)=p_l \otimes 1 + 1 \otimes p_l,\qquad l\geqslant 1.
\end{equation}
We'll use Sweedler's notation $\Delta(x)=\sum x_1 \otimes x_2$.
We identify $\SH^0$ and $\Lambda$ via 
$D_{0,l} \mapsto p_l$. We hence have an action 
\begin{equation}
\bullet :\Lambda \otimes \SH^> \to \SH^>, 
\end{equation}
which we call the action by \textit{Wilson operators}.
For a field extension $F \subset K$ let 
$\bullet$ denote again the corresponding action of
$\Lambda_K$ on $\SH^>_K$.
The following lemma is left to the reader.

\begin{lem}\label{lem:WilsonSH}
(a) The action of $\Lambda$ on $\SH^>$ preserves each graded piece of $\SH^>$,

(b) the action of $\Lambda$ on the degree $n$ part of $\SH^>$ 
factors through $\Lambda_n$,

(c) the Wilson operators are compatible with the coproduct, namely
$$x \bullet (uv)=\sum (x_1 \bullet u) (x_2 \bullet v),
\qquad x \in \Lambda,\qquad u,v \in \SH^>.$$
\end{lem}

\vspace{.15in}

\subsection{The Heisenberg and Virasoro subalgebras}
\label{sec:heis}
For $l\geqslant 1$ we define the following elements 
\begin{equation}
\label{bl}
\gathered
b_{l}=(-x)^{-l}D_{-l,0},\qquad b_{-l}=y^{-l}D_{l,0},\qquad b_0=E_1/\kappa,\\
H_l=(-x)^{-l}D_{-l,1}/l+(1-l)\cb_0\xi b_{l}/2,\qquad
H_{-l}=y^{-l}D_{l,1}/l+(1-l)\cb_0\xi b_{-l}/2,\\
H_{0}=[H_1,H_{-1}]/2.
\endgathered
\end{equation}
These elements will be important to define a Virasoro subalgebra 
in a completion of $\SH^{(r)}_K$.
In Appendix \ref{app:D} we prove the following.

\begin{prop} \label{prop:heis}
For $k,l\in\Z$ we have
\begin{equation}\label{heis}
[b_{l},b_{-k}]=l\,\delta_{l,k}\,\cb_0/\kappa,
\end{equation}
\begin{equation}\label{heis2}
[H_{-1},b_{l}]=-l\,b_{l-1},
\qquad
[H_{1}, b_{l}]=- l b_{l+1}.
\end{equation}
\end{prop}

\vspace{.05in}

\noindent Let $\mathscr{H}$ be the Heisenberg subalgebra of $\SH^{\cb}$ generated by $\{b_l\,;\,l \in \ZZ\}$
and $\cb_0$.
 
\vspace{.1in}

%\begin{rem}
%We have $H_{1}=\kappa^{-1}D_{-1,1}$, $H_{-1}=D_{1,1}$
%and $H_0=\kappa^{-1}E_2/2$.
%\end{rem}

\begin{rem}
A direct computation yields, see Section \ref{A:comb},
$$E_0=\cb_0,\qquad 
E_1=-\cb_1+\cb_0(\cb_0-1)\xi/2,
\qquad E_2=\cb_2+\cb_1(1-\cb_0)\xi+
\cb_0(\cb_0-1)(\cb_0-2)\xi^2/6+2\kappa D_{0,1}.$$
For $l\geqslant 2$ we have also
\begin{equation}
\label{leadingterm}
E_l =l(l-1)\kappa D_{0,l-1}\ \mod\ \SH^{\cb,0}[\leqslant\! l-2].
\end{equation}
Recall that $\SH^{\cb,0}[\leqslant\! l-2]$ is the space of elements of 
$\SH^{\cb,0}$ of order at most $l-2$.
\end{rem}

\iffalse%%%%%%%%%%%%%%
\begin{rem} We will construct in Section~\ref{sec:U/SH} a 
representation $\rho$ of $\SH^{\cb}$ 
on $\Lambda$ extending both the 
representation $\rho^+$ in Proposition \ref{prop:rho+} and the standard 
Fock space of $\mathscr{H}$.
\end{rem}
\fi%%%%%%%%%%%%%%%%%

\vspace{.2in}

\section{Equivariant cohomology of the Hilbert scheme}

\vspace{.1in}

In this section, we recall briefly the structure of the 
Hilbert scheme of points on the complex plane $\CC^2$, 
and we define a convolution algebra acting on 
its (equivariant, Borel-Moore) homology groups.
This is essentially a homology version of the 
K-theoretic construction given in \cite{SV2}, to which we
refer the reader for a more detailed treatment. 
All the geometric properties of the Hilbert scheme which
we use below may be found in \cite{ES}, \cite{V}.

\vspace{.15in}

\subsection{Equivariant cohomology and Borel-Moore homology} 
\label{sec:homology}
Let $G$ be a complex, connected, linear algebraic group 
and let $X$ be a $G$-variety, that is an algebraic variety
equipped with a rational $G$ action. 
By a variety we always mean a complex quasi-projective variety.
Let $H_G^i(X)$ and $H^G_i(X)$ be the equivariant 
cohomology group and the equivariant Borel-Moore homology group of $X$, 
with $\C$ coefficients. 
%We refer to \cite{Lusztigcuspidals} and \cite{Fultonequiv} for details. 
We write
\begin{equation}
H_G(X)=\bigoplus_i H^i_G(X), \qquad H^G(X)=\bigoplus_i H^G_i(X).
\end{equation}
Both of these spaces are graded modules over the graded ring 
$R_G=H_G(\bullet)$. Recall that $H^G_i(X)=H^{-i}_G(X,\Dc)$ where $\Dc$ is the
$G$-equivariant dualizing complex, see \cite[def.~3.5.1]{BL} or \cite[sec.~5.8]{GKM}.
Recall that
\begin{equation}
H_i^G(X)=H_{i+2\dim E-2\dim G}(X\times_GE),
\end{equation}
where $E\to E/G$ is a principal $G$-bundle such that $H^j(E)=0$ for $j=1,2,\dots,i$.
The cup product endows 
$H_G(X)$ with the structure of a graded commutative $R_G$-algebra. We denote by 
$[Y] \in H^G(X)$ the fundamental class of a $G$-stable subvariety $Y \subset X$.
If $Y$ is pure of dimension $d$ then the class $[Y]$ has the degree $2d$.
Let us now assume that $X$ is smooth and connected. 
Then the map $\alpha \mapsto\alpha\cdot[X]$, where $\cdot$ denotes the cap
product, defines a Poincar\'e 
duality isomorphism 
\begin{equation}\label{poincareduality}
H^i_G(X) \to H_{2\dim\;X-i}^G(X). 
\end{equation}
This allows us to define a 
product  on
$H^G(X)$, dual to the cup product on $H_G(X)$. 
If $E$ is a $G$-equivariant vector bundle over $X$ then we 
write
$$\cc^i(E) \in H^{2i}_G(X), \qquad
\cc_{i}(E)= \cc^i(E) \cdot [X] \in H^G_{2\dim\;X-2i}(X)$$
for the equivariant Chern classes of $E$. 
We write $\eu(E)=\cc_r(E)$ where $r$ is the rank of $E$.  
We call
$\eu(E)$ the {\it Euler class} of $E$.
We have 
$$\cc_1(E \oplus E')=\cc_1(E) + \cc_1(E'), \qquad 
\eu(E \oplus E')=\eu(E)\, \eu(E').$$
Fix a morphism $f:X\to Y$ of complex $G$-varieties. If $f$ is a proper map
there is a direct image homomorphism 
\begin{equation}\label{directimage}f_*:H^G_i(X)\to H^G_i(Y).
\end{equation}
If $f$ is a fibration or if $X$, $Y$ are smooth complex $G$-varieties
there is an inverse image 
homomorphism (given, in the second case, 
by the Poincar\'e duality isomorphism and the pull-back in
equivariant cohomology)
\begin{equation}\label{pullback}f^*:H^G_i(Y)\to H^G_{i+2d}(X),\qquad
d=\dim X-\dim Y.\end{equation}
Note that if $Y$ is smooth and $Z\subset Y$ is closed then
$H^G(Z)=H_G(Y,Y\setminus Z)$. So, if $X$, $Y$ are both smooth and $Z\subset Y$ 
is closed then the pull-back in equivariant cohomology gives a map
\begin{equation}
H^G(Z)=H_G(Y,Y\setminus Z)\to
H_G(X,X\setminus f^{-1}(Z))=H^G(f^{-1}(Z)).
\end{equation}
This maps fits into the commutative square
\begin{equation}
\begin{split}
\xymatrix{
H^G(Z)\ar[r]\ar[d]&H^G(f^{-1}(Z))\ar[d]\cr
H^G(Y)\ar[r]&H^G(X).
}
\end{split}
\end{equation}
Further, given a Cartesian square of smooth complex $G$-varieties
\begin{equation}\label{2.1:diagram1}
\begin{split}
\xymatrix{Y'\ar[d]_-{i'}&X'\ar[l]_-{g}\ar[d]^-{i}\\
Y&X\ar[l]_-{f},}
\end{split}
\end{equation}
where $i$, $i'$ are proper, 
we have the base change identity
$f^*i'_*=i_*g^*$.
If $Y'$, $X'$ are no longer smooth but $i$, $i'$ are closed embedding
then the inverse morphism $g^*$ is still well-defined and the base change
identity above holds.

\vspace{.1in}

\begin{ex}
\label{ex:2.1}
Assume that $T$ is a torus and that $X$ is a point. Then
$R_T=S(\mathfrak{t}^*)$, where $\mathfrak{t}$ is the Lie algebra of $T$. 
Let $E$ be a finite dimensional representation of $T$. Write it as a sum of characters 
$E=\chi_1 \oplus \cdots \oplus \chi_r$ with $\chi_\a : T \to \C^\times$. 
Then we have 
$\cc_l(E)=e_l(d\chi_1, \ldots, d\chi_r)$ and $d\chi_\a \in \mathfrak{t}^*$
is the differential of $\chi_\a$.
In particular, 
\begin{equation}
\cc_1(E)=\sum_\a d\chi_\a \in R^2_T, \qquad \eu(E)=
\prod_\a d\chi_\a \in R^{2r}_T.
\end{equation}
Note that, since the Euler class is multiplicative,  we may consider the element
$\eu(E)$ in $K_T,$ the fraction field of $R_T$, for an arbitrary virtual $T$-module $E$. 
For any characters $\chi$,  $\chi'$ of $T$, we may 
abbreviate 
$\chi^*=\chi^{-1}$ and $\chi\otimes\chi'=\chi\,\chi'.$
If $T=(\C^\times)^2$ we get 
\begin{equation}
R_T=\CC[x,y],
\end{equation} where 
$x=\cc_1(q)=dq$, $y=\cc_1(t)=dt$ and
$q$, $t$ are the characters of $T$ given by
\begin{equation} q(z_1, z_2)=z_1^{-1}, \qquad t(z_1, z_2)=z_2^{-1}.
\end{equation}
\end{ex}

\vspace{.2in}

\subsection{Correspondences}
We can now define the {\it convolution product} in equivariant homology.
Let $X_1, X_2, X_3$ be smooth connected algebraic $G$-varieties. Let us denote by
$\pi_{ij}~: X_1 \times X_2 \times X_3 \to X_i \times X_j$ the 
projection along the factor not named. 
If $X_1, X_2, X_3$ are proper, there 
is a map
\begin{equation}
\star~: H^G(X_1 \times X_2) \otimes H^G(X_2 \times X_3) \to H^G(X_1 \times X_3),\quad 
\alpha \otimes \beta \mapsto 
\pi_{13!}\big(\pi_{12}^*(\alpha) \cdot \pi_{23}^*(\beta)\big).
\end{equation}
If $X_1=X_2=X_3=X$ then the map $\star$ equips $H^G(X \times X)$ with the structure of an 
associative $R_G$-algebra. If  $X_1=X_2=X$ and $X_3=\bullet$ then we obtain an action 
of the $R_G$-algebra $H^G(X \times X)$ on the $R_G$-module $H^G(X)$.
If $X_1, X_2, X_3$ are not proper but there is a smooth closed $G$-subvariety $Z\subset X_1\times X_2$ 
such that
the projection $Z\to X_1$ is proper then any element $z\in H^G(Z)$ defines an  $R_G$-linear
operator 
\begin{equation}\label{correspondence}
H^G(X_2)\to H^G(X_1),\qquad \alpha\mapsto z\star\alpha=\pi_{1!}\big(z\cdot\pi_{2}^*(\alpha)\big).
\end{equation}
If the projection $Z\to X_1$ is not proper but $i:Z^c\to Z$ is the inclusion of a smooth closed $G$-subvariety 
such that $\pi_1\circ i$ is proper, then any element $z^c\in H^G(Z^c)$ defines an $R_G$-linear
operator 
\begin{equation}H^G(X_2)\to H^G(X_1),\qquad \alpha\mapsto z^c\star\alpha=
\pi^c_{1!}\big(z^c\cdot\pi_{2}^{c,*}(\alpha)\big),
\qquad
\pi_a^c=\pi_a\circ i,
\end{equation}
and the projection formula implies that $z^c\star\alpha=i_*(z^c)\star\alpha$.

\vspace{.1in}

\begin{rem}
The $R_G$-modules $H^G(X_1)$, $H^G(X_2)$ are graded 
by the \emph{homological} and \emph{cohomological} degrees, for which $H^G_i(X_a)$ has the 
degree $i$ and $2\dim X_a-i$ respectively.
Let $\deg$ denote the homological degree
and $\cdeg$ the cohomological degree. 
Then, if $z\in H^G_i(Z)$ then the convolution by $z$ is an homogeneous operator
for the (co-)homological degrees, and we have
\begin{equation}
\deg(z\star\bullet)=i-2\dim X_2,\qquad
\cdeg(z\star\bullet)=2\dim X_1-i.
\end{equation}
\end{rem}

\vspace{.2in}

\subsection{The Hilbert scheme}
Let $\Hin$ denote the Hilbert 
scheme parametrizing length $n$ subschemes of $\CC^2$.
By Fogarty's theorem it is a smooth irreducible variety of dimension $2n$.
By associating to a closed point of $\Hin$ its ideal sheaf we obtain 
a bijection (at the level of points)
$$\Hin(\CC)=\{I \subset \CC[X,Y]\;;\; 
I\;\text{is\;an\;ideal\;of\;codimension\;} n\}.$$
Let us denote by $S=\CC[X,Y]$ the ring of regular functions on
$\CC^2$. The tangent space $T_I \Hin$ at a closed point $I \in
\Hin(\CC)$ is canonically isomorphic to the vector space
$\Hom_S(I, S/I)$.

\vspace{.15in}

\subsection{The torus action on $\Hin$}
\label{sec:torushilb}
Consider the torus $T=(\C^\times)^2$.
The torus $T$ acts on $\B^2$ via
$(z_1, z_2) \cdot (u,v)=(z_1u, z_2v)$.
There is an induced action on $S$ given by
$(z_1,z_2) \cdot P(X,Y)=P(z_1^{-1}X, z_2^{-1}Y)$
and one on $\Hin$ such that
\begin{equation}
(z_1, z_2) \cdot I=\{P(z_1^{-1}X, z_2^{-1}Y)\; ;\; P(X,Y) \in I\},
\quad\forall I\in\Hin(\C).
\end{equation}
This action has a finite number of isolated fixed points, indexed by the set of partitions of the integer $n$.
To such a partition $\lambda \vdash n$ corresponds the fixed point 
$I_\lambda$ where
\begin{equation}
I_\lambda=\bigoplus_{s\not\in \lambda} \C X^{x(s)}Y^{y(s)}.
\end{equation}
When $I=I_\lambda$ is a $T$-fixed
point, there is an induced $T$-action on $T_I \Hin$.
In order to describe this action, we fix a few notations concerning $T$.
Consider the characters $q,t$  as in Example \ref{ex:2.1}.
For $V$ a $T$-module let $[V]$ be its class in the Grothendieck group of $T$. We
abbreviate  $T_\lambda=[T_{I_\lambda}\Hin]$. It is given by
\begin{equation}\label{E:21}
T_\lambda=\sum_{s \in \lambda}(t^{l(s)}q^{-a(s)-1}+t^{-l(s)-1}q^{a(s)}).
\end{equation}
We set $\eu_{\lambda}=\eu(T^*_{\lambda})$.

\vspace{.15in}

\subsection{The Hecke correspondence $\Hi_{n,n+1}$} 
\label{sec:nestedhilb}
Let $k \geqslant 0$.
The nested Hilbert scheme $\Hilb_{n,n+k}$ is the reduced closed
subscheme of $\Hin \times \Hi_{n+k}$ parametrizing pairs of ideals
$(I,J)$ where $J \subset I$. One defines the nested Hilbert scheme
$\Hi_{n+k,n}$ in a similar fashion. Of course $\Hilb_{n,n}$ is simply
the diagonal of $\Hin \times \Hin$. The schemes $\Hilb_{n,n+k}$ are
 smooth if
$k=0$ or $k=1$, see \cite{Cheah}. The tangent space at a point
$(I,J) \in \Hilb_{n,n+k}$ is the kernel of the obvious map
\begin{equation}\label{E:24}
\psi: \Hom_S(I,S/I) \oplus \Hom_S(J,S/J) \to \Hom_S(J,S/I).
\end{equation}
When $k=1$ the map $\psi$ is surjective.
The diagonal $T$-action on $\Hin \times \Hi_{n+k}$ preserves
$\Hilb_{n,n+k}$. The fixed points contained in $\Hilb_{n,n+k}$ are those
pairs $I_{\mu,\lambda}=(I_{\mu},I_{\lambda})$ for which 
$\mu \subset \lambda$. 
%We abbreviate
%$T_{\mu,\lambda} =[T_{I_{\mu,\lambda}}\Hi_{n,n+1}]$. We may use
%(\ref{E:24}) to write a formula for $T_{\mu,\lambda}$
%\begin{equation}\label{E:25}
%\begin{split}
%T_{{\mu},{\lambda}}&= [\Hom_S(I_\mu,S/I_\mu)] +
%[\Hom_S(I_\lambda,S/I_\lambda)] -
%[\Hom_S(I_\lambda,S/I_\mu)]\\
%&=T_{{\mu}}+T_{{\lambda}}-N_{\mu,\lambda},
%\end{split}
%\end{equation}
The character of the fiber at $I_{\mu,\lambda}$ of the normal bundle
to $\Hi_{n,n+1}$ in $\Hin\times\Hilb_{n+1}$ is
\begin{equation}\label{E:26}
\begin{split}
N_{{\mu},{\lambda}}=\sum_{s \in \mu} \big(
t^{l_{\mu}(s)}q^{-a_{\lambda}(s)-1}+t^{-l_{\lambda}(s)-1}q^{a_{\mu}(s)}\big).
\end{split}
\end{equation}
Of course, similar formulas hold for the nested Hilbert scheme
$\Hilb_{n+1,n}$.

\vspace{.15in}

\subsection{The tautological bundles}
\label{sec:taut1}
Let $\Theta_n \subset \Hin \times \B^2$ be the universal family and
let $p: \Hin \times\B^2 \to \Hin$ be the projection. The
\textit{tautological bundle} of $\Hin$ is the locally free sheaf
$\tau_{n}=p_*(\mathcal{O}_{\Theta_n})$. The fiber of
$\tau_n$ at a point $I \in \Hin(\C)$ is $S/I$. The
character of the $T$-action on its fiber at the fixed point
$I_{\lambda}$ is
\begin{equation}\label{E:23}
\tau_\lambda=\sum_{s
\in \lambda} t^{y(s)}q^{x(s)}.
\end{equation}
Next, let $\pi_1$, $\pi_2$ be the projections of $\Hin \times
\text{Hilb}_{n+1}$ to $\Hin$ and $\text{Hilb}_{n+1}$ respectively.
Over $\Hilb_{n,n+1}$ there is a surjective map $\pi_2^*(
\tau_{n+1}) \to \pi_1^*( \tau_{n})$. Over
the point $(I,J)$ it specializes to the map $S/J \to
S/I$. The kernel sheaf is  a line bundle,
which we call the \textit{tautological bundle} of $\Hilb_{n,n+1}$ and
which we denote by $\tau_{n,n+1}$. Over a $T$-fixed
point $I_{\mu,\lambda}$ its character is
\begin{equation}\label{E:27}
{\tau}_{\mu,\lambda}=t^{y(s)}q^{x(s)}
\end{equation}
where $s=\lambda \backslash \mu$ is the unique box of $\lambda$ not
contained in $\mu$. 
Finally, let $\pi_1$, $\pi_2$ be the projections of $\Hin \times \Hin$
to $\Hin$. Over $\Hilb_{n,n}$ we have the vector bundle
$\tau_{n,n} =\pi_2^*(\tau_{n})=\pi_1^*(
\tau_{n})$. We call it the \textit{tautological bundle}
of $\Hi_{n,n}$. Over a $T$-fixed point $I_{\lambda,\lambda}$ its
character is
${\tau}_{\lambda,\lambda}={\tau}_{\lambda}.$

\vspace{.15in}

\subsection{The algebra $\widetilde\E^{(1)}_K$ and the 
$\widetilde\E^{(1)}_K$-module $\widetilde\Lb^{(1)}_K$}
Recall that 
\begin{equation}
R_T=\C[x,y],\qquad x=dq,\qquad y=dt. 
\end{equation}
Consider the fraction field
\begin{equation}
K_T=\Frac(R_T)=\C(x,y).
\end{equation}
If no confusion is possible we abbreviate 
\begin{equation}\label{RK}\HT=R_T,\qquad K=K_T.\end{equation}
The direct image by the inclusion
$\Hin^T \subset \Hin$ 
yields a canonical isomorphism
\begin{equation}\bigoplus_{\lambda \vdash n} 
K [I_{\lambda}] = H^T(\Hin) \otimes_{\HT} K.
\end{equation}
Similarly, there is an isomorphism
\begin{equation}
\bigoplus_{\substack{\lambda \vdash n\\ \mu \vdash m}} K[I_{\lambda,\mu}] 
=H^T(\Hin \times \text{Hilb}_m) \otimes_{\HT} K,\qquad
I_{\lambda,\mu}=(I_{\lambda},I_{\mu}).
\end{equation} 
So, we may  define a  
$K$-algebra structure on 
\begin{equation}\label{E:defek}
\widetilde\E^{(1)}_K=\bigoplus_{k \in \Z} \prod_{n}H^T(\text{Hilb}_{n+k} \times \text{Hilb}_n) \otimes_{\HT} K,
\end{equation}
together with an action on the $K$-vector space
\begin{equation}\label{E:deflk}
\widetilde\Lb^{(1)}_K=\bigoplus_n\widetilde\Lb^{(1)}_{n,K}
=\bigoplus_{n} H^T(\Hin) \otimes_{\HT} K.
\end{equation}
In (\ref{E:defek}), the product ranges over all values of $n \geqslant 0$ such that $n+k \geqslant 0$. 
The integer $k$ provides
a $\Z$-grading on $\widetilde\E^{(1)}_K$, and the $\N$-grading on 
$\widetilde\Lb^{(1)}_K$ turns it into a faithful graded $\widetilde\E^{(1)}_K$-module.

\vspace{.15in}

\subsection{The algebra $\widetilde\U^{(1)}_K$} 
\label{sec:U1}
Let $i: \Hilb_{n+1,n} \to \text{Hilb}_{n+1} \times \Hin$ 
be the closed embedding. For notational convenience, 
the pushforward $i_*\cc_1(\tau_{n+1,n})$ of the Chern class of
the line bundle $\tau_{n+1,n}$ on $\Hilb_{n+1,n}$ will simply be denoted by 
$\cc_1(\tau_{n+1,n})$.
We will use similar notation for the tautological 
bundles $\tau_{n,n+1}$ on $\Hilb_{n,n+1}$ and 
$\tau_{n,n}$ on $\Hilb_{n,n}$.
For $l \geqslant 0$ we consider the following elements in $\widetilde\E^{(1)}_K$
\begin{equation}
f_{1,l}=\prod_{n \geqslant 0} \cc_1(\tau_{n+1,n})^l,
\qquad
f_{-1,l}=\prod_{n \geqslant 0} \cc_1(\tau_{n,n+1})^l,
\qquad
e_{0,l}=\prod_{n \geqslant 0} \cc_l(\tau_{n,n}).
\end{equation}
We used the convention that 
$\cc_0(\tau_{n,n})=n[\Hilb_{n,n}]$. 
Let $\widetilde\U^{(1)}_K$ be the 
$K$-subalgebra of $\widetilde\E^{(1)}_K$ generated by 
$$\{f_{-1,l}, e_{0,l}, f_{1,l}\; ;\; l \geqslant 0\}. $$
Observe that since the $e_{0,l}$'s are supported on the diagonal 
of the Hilbert scheme, 
their convolution product is given by the cup
product in the equivariant cohomology groups of the Hilbert scheme. 
Therefore, the subalgebra of $\widetilde\U^{(1)}_K$ generated by 
$\{e_{0,l}\;;\;l \geqslant 0\}$ is
commutative. Let us introduce another set of elements 
$\{f_{0,l}\; ;\; l \geqslant 0\}$ 
defined through the following formula
\begin{equation}
\label{e/f}
\sum_{l \geqslant 1}  f_{0,l}s^{l-1}=-\partial_s \log(e(s)), \qquad 
e(s)=1 + \sum_{k \geqslant 1} (-1)^k e_{0,k}s^k,\qquad
f_{0,0}=e_{0,0}.
\end{equation}
Under restriction, the canonical
representation of $\widetilde\E^{(1)}_K$ on $\widetilde\Lb^{(1)}_K$ 
yields a faithful representation of
$\widetilde\U^{(1)}_K$ on $\widetilde\Lb^{(1)}_K$.
We call it the \emph{canonical representation}
of $\widetilde\U^{(1)}_K$ on $\widetilde\Lb^{(1)}_K$.

\vspace{.1in}

\begin{rem}
\label{rem:f0l}
Given a splitting  into a sum of line bundles
$\tau_{n,n}=\phi_1 \oplus \cdots \oplus\phi_n$, we get
$$f_{0,l}=\prod_{n \geqslant 0}p_l(f_1,\dots,f_n),\qquad f_i=\cc_1(\phi_i),\qquad l\geqslant 0.$$
\end{rem}

\vspace{.15in}

\subsection{From $\widetilde\U^{(1)}_K$ to $\widetilde\SH^{(1)}_K$}
\label{sec:U/SH}
Consider the inclusion
\begin{equation}\label{FK}
F\to K,\qquad \kappa \mapsto -y/x.
\end{equation} 
Let $\widetilde\SH^{(1)}_{K}$ be the specialization of $\SH^{\cb}\otimes K$ 
at $\cb=(1, 0, \ldots)$.
It can be viewed as 
a specialization of the $K_1$-algebra $\SH^{(1)}_{K}$ at $\eps_1=0$. 
We set 
\begin{equation}
h_{0,l+1}=x^{-l}f_{0,l}, \qquad h_{1,l}=x^{1-l}yf_{1,l}, 
\qquad h_{-1,l}=x^{-l}f_{-1,l},\qquad l\geqslant 0.
\end{equation}
We can now state our first result, compare 
\cite[thm.~3.1]{SV2}. 
The proof is given in Section \ref{sec:prooflevel1}.

\begin{theo}\label{thm:SH/U1} There is a $K$-algebra isomorphism
$\Psi:\widetilde\SH^{(1)}_{K} \to\widetilde\U^{(1)}_K$ such that
$D_{\xb} \mapsto h_{\xb}$ for $\xb\in\mathscr E.$
\end{theo}

\noindent
We identify $\widetilde\Lb^{(1)}_K $ with
$\Lambda_K$  
by the $K$-linear map 
\begin{equation}\label{Phi}
\Lambda_K\to\widetilde\Lb^{(1)}_K, \quad  
J_\lambda \mapsto [I_\lambda].
\end{equation}
Via $\Psi,$ the representation
of $\widetilde\U^{(1)}_K$ on $\widetilde\Lb^{(1)}_K $ gives
a faithful representation $\widetilde\rho^{(1)}$ 
of $\widetilde\SH^{(1)}_K$ on $\Lambda_K $.

\begin{prop}\label{P:newrep} We have

(a) $\widetilde\rho^{(1)}(b_{-l})=$ multiplication by $ p_l$ and 
$\widetilde\rho^{(1)}(b_{l})=l\kappa^{-1}\partial_{p_l}$ for $l \geqslant 1,$

(b) $\widetilde\rho^{(1)}(D_{0,1})=\sum_i X_i \partial_{X_i}$
and $\widetilde\rho^{(1)}(D_{0,2})=\kappa\square.$
\end{prop}

\begin{proof} The representation $\widetilde\rho^{(1)}$ 
extends the representation $\rho^+$ 
in Proposition \ref{prop:rho+}, see the proof of
Proposition \ref{P:proof1} for details.
Thus, for $l \geqslant 1$, the operators 
$\widetilde\rho^{(1)}(b_{-l})$, $\widetilde\rho^{(1)}(D_{0,1})$ and 
$\widetilde\rho^{(1)}(D_{0,2})$ are as in the proposition above 
by Remark \ref{rem:LBinfty}.
Next, we have $\widetilde\rho^{(1)}(b_{l}) \cdot 1=0$
and the map 
\begin{equation}
K[b_{-l}\,;\,l\geqslant 1]\to \Lambda_K,\quad u \mapsto 
\widetilde\rho^{(1)}(u) \cdot 1
\end{equation}
is an isomorphism. Further, by Proposition \ref{prop:heis}
the elements $b_{l}$, $l\in \Z$, generate a Heisenberg algebra
of central charge $\kappa^{-1}$.
This forces $\widetilde\rho^{(1)}(b_{l})$ to
be given by the formula above. 
\end{proof}

\vspace{.1in}

The representation $\widetilde\rho^{(1)}$ 
extends both the representation $\rho^+$
of $\SH^{+}$ in Proposition \ref{prop:rho+}
and the standard Fock space of the Heisenberg algebra.
We'll call it the Fock space of $\widetilde\SH^{(1)}_K$.

\vspace{.2in}

\section{Equivariant cohomology of the moduli space of torsion free sheaves}
\label{sec:3}
The Hilbert scheme of $\mathbb{A}^2$ is isomorphic to the moduli space of
framed torsion free rank one coherent sheaves on $\mathbb{P}^2$. 
We now generalize the considerations above to higher ranks.

\vspace{.1in}

\subsection{The moduli space of torsion free sheaves}
Fix integers $r>0$, $n\geqslant 0$. Let $M_{r,n}$ be the moduli
space of framed torsion-free sheaves on $\mathbb P^2$ with rank $r$
and second Chern class $n$. More precisely,
$\C$-points of $M_{r,n}$ are isomorphism classes of pairs
$(\mathcal E,\Phi)$ where $\mathcal E$ is a torsion-free sheaf which
is locally free in a neighborhood of $\ell_\infty$ and
$\Phi:\mathcal E|_{\ell_\infty}\to\Oc_{\ell_\infty}^r$ is a framing
at infinity. Here $\ell_\infty=\{[x:y:0]\in\mathbb P^2\}$ is the
line at infinity. Recall that $M_{r,n}$ is a smooth variety of
dimension $2rn$ which admits the following alternative description.
Let $E$ be a $n$-dimensional vector space. 
We have $M_{r,n}=M_{r,E}^s/GL_E$ where
$M_{r,E}^s=N_{r,E}^s\cap M_{r,E}$, with
\begin{equation}\label{notation1}\gathered 
N_{r,E}^s=\{(a,b,\varphi,v)\in N_{r,E}\,;\,
(a,b,\varphi,v)\ \text{is\ stable}\},\\ 
M_{r,E}=\{(a,b,\varphi,v)\in N_{r,E}\,;\,[a,b]+v\circ\varphi=0\},\cr
N_{r,E}=\gen_E^2\times\Hom(E,\C^r)\times\Hom(\C^r,E).\endgathered
\end{equation}
The $GL_E$-action is given by
$g(a,b,\varphi,v)=(gag^{-1},gbg^{-1},\varphi g^{-1},gv)$.
The tuple $(a,b,\varphi,v)$ is {\it stable} iff there is no proper subspace
$E_1\subsetneq E$ which is preserved by $a,b$ and contains $v(\C^r)$.
From now on we may abbreviate $G=GL_E$ and $\gen=\gen_E$.

\vspace{.2in}

\subsection{The torus action on $M_{r,n}$}
Put  $D=(\C^\times)^{r}$ and $T=(\C^\times)^2$.
We abbreviate $\widetilde D=D\times T$.
Set also $x=\cc_1(q)$, 
$y=\cc_1(t)$ and $e_\a=\cc_1(\chi_{\a})$ for  $\a\in[1,r]$.
We have
\begin{equation}
R_r=R_{\widetilde D}=\C[x,y,e_1,\dots,e_r],\qquad
K_r=K_{\widetilde D}=\C(x,y,e_1,\dots,e_r).
\end{equation}
The characters $q=\chi_x$, $t=\chi_y$ and $\chi_\a=\chi_{e_\a}$ of 
$\widetilde D$ are given by
\begin{equation}\label{3.3}
q(h,z_1, z_2)=z_1^{-1}, \quad t(h,z_1, z_2)=z_2^{-1}, 
\quad \chi_\a(h,z_1, z_2)=h_\a^{-1},\quad
h=(h_1,h_2,\dots h_r).
\end{equation}
We set also $v=(qt)^{-1}$.
We equip the variety $N_{r,E}$ 
with the $\widetilde D$-action given by
\begin{equation}
(h,z_1,z_2)\cdot(a,b,\varphi,v)=(z_1a,z_2b, z_1z_2h\varphi,vh^{-1}).
\end{equation}
%In other words, given a pair $(\mathcal E,\Phi)$ as above, the
%element $h$ acts on the framing at infinity in the obvious way while
%$(z_1,z_2)$ acts on $\mathbb P^2$ by the formula
%\begin{equation}
%(z_1,z_2)\cdot[x:y:w]=[z_1x:z_2y:w].
%\end{equation}
This action has a finite number of isolated fixed points
which are indexed by the set of $r$-partitions of $n$. 
To the $r$-partition
$\lambda$ 
corresponds a fixed point $I_{\lambda}$ such that the character 
$T_\lambda$ of the $\widetilde D$-module 
$T_{I_\lambda}M_{r,n}$
is given by \cite[thm.~2.11]{NY}
\begin{equation}\label{HR:1}
T_\lambda=\sum_{\a,\b=1}^r\sum_{s\in\lambda^{(\a)}} \chi_{\a}
\chi_{\b}^{-1}t^{l_{\lambda^{(\b)}}(s)}q^{-a_{\lambda^{(\a)}}(s)-1}+
\sum_{\a,\b=1}^r\sum_{s\in\lambda^{(\b)}} \chi_{\a}
\chi_{\b}^{-1}t^{-l_{\lambda^{(\a)}}(s)-1}q^{a_{\lambda^{(\b)}}(s)}.
\end{equation}

\vspace{.2in}

\subsection{The Hecke correspondence $M_{r,n,n+1}$}
\label{sec:hecker}
Now, we assume that $\dim(E)=n+1$.
The \emph{Hecke correspondence}
is the geometric quotient 
$M_{r,n,n+1}=Z_{r,E}^s/G,$
where $Z_{r,E}^s$ is the variety of all tuples $(a,b,\varphi,v, E_1)$
where $(a,b,\varphi,v)\in M_{r,E}^s$ and $E_1\subset \Ker\,\varphi$ 
is a line 
preserved by $a,b$. 
We define also
\begin{equation}
M_{r,n,n+1}^c=\{(a,b,\varphi,v,E_1)\in Z_{r,E}^s\,;\,a|_{E_1}=b|_{E_1}=0\}/G.
\end{equation}
Write $E_2=E/E_1$ and consider the induced linear maps
\begin{equation}\bar v=\pi\circ v,\qquad
\bar a,\bar b\in\gen_{E_2},\qquad\bar\varphi\in\Hom(E_2,\C^r).
\end{equation}
Let $\pi_1$, $\pi_2$ be the projections of $M_{r,n}\times M_{r,n+1}$
to $M_{r,n}$, $M_{r,n+1}.$ The following is well-known.

\begin{prop}
(a) The variety
$Z_{r,E}^s$ is a $G$-torsor over $M_{r,n,n+1}$.

(b) The variety $M_{r,n,n+1}$ is a smooth variety of
dimension $2rn+r+1$.  

(c) The closed subvariety $M_{r,n,n+1}^c$ is also smooth.

(d) The map
$(a,b,\varphi,v)\mapsto (\bar a, \bar b,\bar\varphi,\bar v),\,
(a,b,\varphi,v)$ is a closed immersion $M_{r,n,n+1}\subset
M_{r,n}\times M_{r,n+1}$. 
The restriction of $\pi_2$ to  $M_{r,n,n+1}$ is proper. 
The restriction of $\pi_1$ to  $M_{r,n,n+1}^c$ is proper. 
\end{prop}

The pair $I_{\mu,\lambda}=(I_\mu,I_\lambda)$ belongs to $M_{r,n,n+1}$ if and only if $\mu\subset\lambda$ 
and the $r$-partitions
$\mu$, $\lambda$ have weight $n$, $n+1$ respectively. 
Let 
$N_{\mu,\lambda}$ be the character of the fiber at $I_{\mu,\lambda}$
of the normal bundle 
(in $M_{r,n}\times M_{r,n+1}$) of $M_{r,n,n+1}.$ 
We set also $N_{\lambda,\mu}=N_{\mu,\lambda}$. 
Finally, we define
\begin{equation}
\eu_{\lambda}=\eu(T_{\lambda}^*),\qquad\eu_{\lambda,\mu}=\eu_\lambda\eu_\mu.
\end{equation}

\vspace{.2in}

\subsection{The tautological bundles}
\label{sec:tautr}
The \textit{tautological bundle} of $M_{r,n}$ 
is the $\widetilde D$-equivariant
bundle $\tau_{n}=M_{r,E}^s\times_G E.$
The character of the $\widetilde D$-module $\tau_n|_{I_\lambda}$ is
given by \cite[thm.~2.11]{NY}, \cite[Lemma~6]{VV1}
\begin{equation}\label{HR:2}
\tau_\lambda=\sum_\a
\sum_{s \in \lambda^{(\a)}} \chi_{\a}^{-1}
t^{y(s)}q^{x(s)}.
\end{equation}
The characters $T_\lambda$ and $\tau_\lambda$ are related by the following equation
\begin{equation}\label{HR:2.5}
T_\lambda=-(1-q^{-1})(1-t^{-1})\tau_\lambda\otimes\tau_\lambda^*+
\tau_\lambda\otimes W^*+v\tau_\lambda^*\otimes W,
\end{equation}
where $W=\chi_1^{-1}+\cdots+\chi_r^{-1}$ is the tautological 
representation of $D$. For $\mu\subset\lambda$ we have also 
\begin{equation}\label{HR:3}
\begin{split}
\aligned
N_{{\mu},{\lambda}} &=-(1-q^{-1})(1-t^{-1})\tau_\mu\otimes\tau_\lambda^*+
\tau_\mu\otimes W^*+v\tau_\lambda^*\otimes W-v,\\
N_{{\mu},{\lambda}}&=\sum_{\a,\b}\sum_{s\in\mu^{(\a)}} 
\chi_{\a} \chi_{\b}^{-1}t^{l_{\mu^{(\b)}}(s)}q^{-a_{\lambda^{(\a)}}(s)-1}
+\sum_{\a,\b}\sum_{s\in\lambda^{(\b)}} 
\chi_{\a} \chi_{\b}^{-1}t^{-l_{\lambda^{(\a)}}(s)-1}q^{a_{\mu^{(\b)}}(s)}-v.
\endaligned
\end{split}
\end{equation}

Over $M_{r,n,n+1}$ there is a 
surjective map $\pi_2^*( \tau_{n+1}) \to \pi_1^*(
\tau_{n})$. The kernel sheaf is a
line bundle called the \textit{tautological bundle} of $M_{r,n,n+1}$
which we denote by $\tau_{n,n+1}$. Over  $I_{\mu,\lambda}$ its character is
\begin{equation}\label{HR:5}
\tau_{\mu,\lambda}=
\chi_{\a}^{-1}
t^{y(s)}q^{x(s)},\quad \mu\subset\lambda.
\end{equation}
Here $s=\lambda^{(\a)} \backslash \mu^{(\a)}$ is the unique box of
$\lambda$ not contained in $\mu$. We define the Hecke correspondence
$M_{r,n+1,n}$ and the tautological bundle $\tau_{n+1,n}$ over it in
the obvious way, so that we get $\tau_{\lambda,\mu}=\tau_{\mu,\lambda}$.

\vspace{.2in}

\vskip3mm

\subsection{The algebra $\E^{(r)}_{K}$ and the 
$\E^{(r)}_{K}$-module $\Lb^{(r)}_{K}$}
\label{sec:3.5}
Consider the graded $\HT_r$-modules
\begin{equation}
\gathered
\Lb^{(r)}_{n}=H^{\widetilde D}(M_{r,n}),\qquad
\Lb^{(r)}=\bigoplus_{n \geqslant 0}\Lb^{(r)}_{n},
\\
\E^{(r)}_n=\prod_k H^{\widetilde D}(M_{r,n+k} \times M_{r,k})
,\qquad
\E^{(r)}=\bigoplus_{n \in\Z}\E^{(r)}_{n},
\endgathered
\end{equation}
where the product ranges over all integers $k\geqslant 0$ with $n+k\geqslant 0$.
They are known to be torsion free. 
We abbreviate
\begin{equation}
\gathered
H^{\widetilde D}(M_{r,n}\times M_{r,m})_K=
H^{\widetilde D}(M_{r,n}\times M_{r,m})\otimes_{\HT_r}K_r,\\
\E^{(r)}_{n,K}=\prod_k H^{\widetilde D}(M_{r,n+k} \times M_{r,k})_{K},\qquad
\E^{(r)}_K=\bigoplus_{n \in\Z}\E^{(r)}_{n,K}\\
\Lb^{(r)}_{n,K}=\Lb^{(r)}_{n}\otimes_{\HT_r}K_r,
\qquad
\Lb^{(r)}_K=\bigoplus_{n \geqslant 0}\Lb^{(r)}_{n,K}.
\endgathered
\end{equation}
The variety $M_{r,n}$ is not proper but it has a finite number of fixed
points by the $\widetilde D$-action. The direct image by the obvious inclusion
$M_{r,n}^{\widetilde D}\to M_{r,n}$ provides us with canonical isomorphisms
\begin{equation}\label{LrK}
\Lb^{(r)}_{n,K}=\bigoplus_{\lambda}K_r[I_{\lambda}],
\qquad\E^{(r)}_{n,K}=\prod_{k}\bigoplus_{\lambda,\mu}
K_r[I_{\lambda,\mu}],
\end{equation}
where $\lambda$, $\mu$ run over the set of 
$r$-partitions of $n+k$, $k$ respectively.
This allows us to define, by convolution, an associative multiplication
on $\E^{(r)}_{K}$ and an action of 
$\E^{(r)}_{K}$ on $\Lb^{(r)}_{K}$.

\vspace{.2in}

\subsection{The algebras $\U^{(r)}_{K}$ and $\SH^{(r)}_{K}$}
\label{sec:Ur}
Consider the inclusion $F\subset K$ in \eqref{FK}.
For $l\geqslant 0$ we define the following elements in $\E^{(r)}_{K}$
\begin{equation}
\gathered
f_{1,l}=\prod_{n\geqslant 0} \cc_1({\tau}_{n+1,n})^l,\quad
f_{-1,l}=\prod_{n\geqslant 0} \cc_1({\tau}_{n,n+1})^l,\quad
e_{0,l}=\prod_{n\geqslant 0} \cc_l(\tau_{n,n}).
\endgathered
\end{equation}
We define also the element $f_{0,l}$ through the relations
\eqref{e/f}.
We abbreviate
\begin{equation}
h_{0,l+1}=x^{-l}f_{0,l}, \qquad h_{1,l}=x^{1-l}yf_{1,l}, 
\qquad h_{-1,l}=(-1)^{r-1}x^{-l}f_{-1,l}.
\end{equation}
From \eqref{HR:2} and the formulas above
we get the following identity, compare \eqref{E:36},
\begin{equation}\label{3.6:rem}
f_{0,l}([I_\lambda])=\sum_\a\sum_{s\in\lambda^{(\a)}}
c_\a(s)^l\,[I_\lambda],
\qquad
c_\a(s)=x(s)\,x+y(s)\,y-e_\a.
\end{equation}
Recall the field $K_r=K(\eps_1,\dots,\eps_r)$ from Definition \ref{def:spe}.
We fix
\begin{equation}
\eps_\a=e_\a/x,\qquad \a\in[1,r].\end{equation}
We consider the $K_r$-subalgebras of $\E^{(r)}_{K}$ given by
\begin{itemize}
\item $\U^{(r)}_{K}$ is generated by 
$\{f_{-1,l}, e_{0,l}, f_{1,l}\; ;\; l \geqslant 0\},$
\item $\U^{(r),+}_K$ by $\{e_{0,l}, f_{1,l}\; ;\;l \geqslant 0\}$
and $\U^{(r),-}_K$ by $\{f_{-1,l}, e_{0,l}\; ;\; l \geqslant 0\}$,
\item $\U^{(r),<}_K$ by $\{f_{-1,l}\;;\;l \geqslant 0\}$
and $\U^{(r),>}_K$ by $\{f_{1,l}\;;\;l \geqslant 0\}$,
\item $\U^{(r),0}_K$ by $\{f_{0,l}\;;\;l \geqslant 0\}$,
\end{itemize}
\vspace{.05in}
The following is proved in Section \ref{sec:proofthmSH/Ur}.

\vspace{.05in}

\begin{theo}
\label{thm:SH/Ur} 
The assignment 
$D_{\xb} \mapsto h_{\xb}$ 
for $\xb\in\mathscr E$ extends to a $K_r$-algebra isomorphism
$\Psi~:\SH^{(r)}_{K}\to\U^{(r)}_{K}$ which takes
$\SH^{(r),>}_{K}$, $\SH^{(r),0}_{K}$, $\SH^{(r),<}_{K}$
into $\U^{(r),>}_{K}$, $\U^{(r),0}_{K}$, $\U^{(r),<}_{K}$. 
\end{theo}

\vspace{.1in}

\noindent
The map $\Psi$ in Theorem \ref{thm:SH/Ur} gives a representation
\begin{equation}\label{rhor}
\rho^{(r)}~:\SH^{(r)}_{K} \longrightarrow \End(\Lb^{(r)}_{K}).
\end{equation}
This representation is faithful by Proposition \ref{prop:faithfulr}.

\vspace{.1in}

\begin{rem}\label{rem:3.6A}
We have $\rho^{(1)}(D_{0,2})=x^{-1}f_{0,1}$. Therefore, comparing 
Proposition \ref{P:newrep} with \eqref{3.6:rem} we get the following formula
$\rho^{(1)}(D_{0,2})+\varepsilon_1\rho^{(1)}(D_{0,1})=\kappa\square.$
\end{rem}

\vspace{.1in}

\begin{rem}\label{rem:int}
Since $\Lb^{(r)}$ is torsion 
free as a $R_r$-module, we can view it as a $R_r$-submodule
of $\Lb^{(r)}_K$. Since the projection 
$\pi_1:M_{r,n+1,n}\to M_{r,n+1}$ is proper, we have
$f_{1,l}(\Lb^{(r)})\subset \Lb^{(r)}$. Since 
$$xy\,c_1(\tau_{n,n+1})\in \operatorname{Im}
\big\{H_*^{\widetilde D}(M_{r,n,n+1}^c)\to 
H_*^{\widetilde D}(M_{r,n,n+1})\big\},$$
we have also $xy\,f_{-1,l}(\Lb^{(r)})\subset \Lb^{(r)}.$
Finally, we have $f_{0,l}(\Lb^{(r)})\subset \Lb^{(r)}.$
Therefore, the operators 
\begin{equation}
\rho^{(r)}\big(x^{l-1}y^{-1}D_{1,l}\big),\qquad
\rho^{(r)}\big(x^{l+1}yD_{-1,l}\big),\qquad
\rho^{(r)}\big(x^lD_{0,l+1}\big),\qquad
l\geqslant 0,
\end{equation}
preserve the lattice $\Lb^{(r)}$.
More generally, using \eqref{E:defdrl}, 
we get that the operators 
\begin{equation}
x^{d-1}y^{-l}\rho^{(r)}(D_{l,d}),\qquad
x^{d-1+2l}y^{l}\rho^{(r)}(D_{-l,d}),\qquad
l\geqslant 0
\end{equation}
\iffalse%%%%%%%%%%%%%%%%%%%%%
Using this and formulas 
\eqref{1.8:Dl0}, \eqref{bl} and \eqref{ bl}
we get that the operators 
\begin{equation}
\rho^{(r)}(x^{-1}y^{-l}b_{-l}),\qquad
\rho^{(r)}(x^{l-1}y^{2l}b_{l}),\qquad
\rho^{(r)}(x^ly^{2l}H_l),\qquad
\rho^{(r)}(y^{-l}H_{-l}),\qquad l\geqslant 0,
\end{equation}
\fi%%%%%%%%%%%%%%%%%%%%%%%%%
preserve also the lattice $\Lb^{(r)}$.
\end{rem}

\vspace{.1in}

\begin{rem}
\label{rem:homdeg}
The $R_r$-module $\Lb^{(r)}$ is bi-graded : it is first graded
by the $c_2$, for which the degree $n$ piece is $\Lb^{(r)}_{n,K},$ and then
by the (co-)homological degree, for which $H^{\widetilde D}_i(M_{r,n})$ has the 
degree $4rn-2i$ or $2i$ respectively. 
The operator
$\rho^{(r)}(D_{\xb})$ is homogeneous for the (co-)homological degrees.
For $\xb\in\mathscr E$ with $\xb=(\epsilon,d)$ we have
\begin{equation}\label{homdeg}
%\deg\big(\rho^{(r)}(D_{\epsilon,d})\big)=2\epsilon(r-1),\qquad
\cdeg\big(\rho^{(r)}(D_\xb)\big)=2\epsilon(r+1).
\end{equation}
More generally, using \eqref{E:defdrl}, the formula \eqref{homdeg} is again true for any
$\epsilon\in\Z$.
%Using this and formulas 
%\eqref{1.8:Dl0}, \eqref{bl} and \eqref{ bl}, for each $l\in\Z$
%we get also
%\begin{equation}
%\gathered
%\deg(\rho^{(r)}(H_l))=\deg(\rho^{(r)}(b_l))=-2l(r-1),\\
%\cdeg(\rho^{(r)}(H_l))=\cdeg(\rho^{(r)}(b_l))=-2l(r+1).
%\endgathered
%\end{equation}
\end{rem}

\vspace{.15in}

\subsection{The pairing on $\Lb^{(r)}_K$}
\label{sec:pairing}

The cup product equips the $K_r$-vector space $\Lb^{(r)}_K$ with a
$K_r$-bilinear form
$(\bullet,\bullet)$ such that for each $r$-partitions $\lambda$, $\mu$ we have
\begin{equation}\label{pairing}
([I_\lambda],[I_\mu])=\delta_{\lambda,\mu}\eu_\lambda.
\end{equation}
Let $f^*$ denote the adjoint of a $K_r$-linear operator $f$ on $\Lb^{(r)}_K$
with respect to this pairing.
Using this anti-involution, we can prove the following.

\begin{prop}
\label{prop:involutionU}
The assignment
$h_{1,l} \mapsto h_{-1,l}$ and $h_{0,l} \mapsto h_{0,l}$ 
for $l \geqslant 0$ extends to an algebra anti-involution
$\U^{(r),+}_K\to\U^{(r),-}_K$ which takes $\U^{(r),>}_K$ onto $\U^{(r),<}_K$. 
\end{prop}

\begin{proof}
By \eqref{formdeg}, for any $r$-partitions $\lambda$, $\pi$ such that
$\lambda\subset\pi$ and $|\lambda|=|\pi|-1$, we have
\begin{equation}
\big([I_\pi],f_{1,l}[I_\lambda]\big)=
\cc_1\big(\tau_{\lambda,\pi})^l\eu(N_{\lambda,\pi}^*\big)=
\big(f_{-1,l}[I_\pi],[I_\lambda]\big).
\end{equation}
Thus, we get the following 
\begin{equation}\label{unitarity}
f_{1,l}^*=f_{-1,l},\qquad 
h_{1,l}^*=(-1)^{r-1}xy\,h_{-1,l}.
\end{equation}
Clearly, we have also
\begin{equation}\label{unitaritybis}
h_{0,l}^*=h_{0,l}.
\end{equation}
The proposition follows.
\end{proof}

\begin{prop}\label{prop:unitarity} 
For $(l,d)\in\N^2_0$ we have
$\rho^{(r)}(D_{l,d})^*=(-1)^{(r-1)l}x^ly^l\rho^{(r)}(D_{-l,d}).$
\end{prop}

\begin{proof}
For $(l,d)\in\mathscr E$ this  is \eqref{unitarity}, \eqref{unitaritybis}.
The claim follows by applying \eqref{1.8:Dl0}, \eqref{E:defdrl}. 
\end{proof}

\vspace{.1in}

\begin{rem}
The cup-product in cohomology gives 
a $K$-bilinear form 
$(\bullet,\bullet)$ 
on $\widetilde\Lb^{(1)}_{K}$ 
such that
\begin{equation}
\aligned
([I_\lambda],[I_\mu])
%&=\delta_{\lambda,\mu}\eu_\lambda\\
&=\delta_{\lambda,\mu}(-1)^{|\lambda|}
\prod_{s\in\lambda}\big(xa(s)-y(l(s)+1)\big)
\big(x(a(s)+1)-yl(s)\big).
\endaligned
\end{equation}
Under the map \eqref{Phi} this pairing is taken to 
$
\bigoplus_{n\geqslant 0}(-y^2)^n (\bullet,\bullet)_{1/\kappa},
$
where $(\bullet,\bullet)_{1/\kappa}$ is Macdonald inner product \cite[chap.~VI, sec.~10]{Mac}.
%Since $(\bullet,\bullet)$ is a Hopf pairing, and since we have 
%\cite[chap.~VI, p.~377]{Mac} 
%\begin{equation}
%(p_l,p_k)=\delta_{l,k}(-y^2)^ll/\kappa,
%\end{equation}
%the adjoint for this pairing is 
%\begin{equation}
%\widetilde\rho^{(1)}(b_{-l})^*=
%(-y^2)^{l}\widetilde\rho^{(1)}(b_{l}).
%\end{equation}
\end{rem}

\vspace{.15in}

\subsection{Wilson operators on $\U^{(r),>}_{K}$ and $\Lb^{(r)}_{K}$}
The product of the $\HT_r$-algebra homomorphisms
\begin{equation}
R_{\widetilde G}\to H_{\widetilde D}(M_{r,n}),\qquad
e_l^{(n)}\mapsto \cc_l(\tau_n),\qquad n\geqslant 0
\end{equation}
gives a
$R_r$-algebra homomorphism
\begin{equation}
\Lambda_{R_r}\to\prod_n H_{\widetilde D}(M_{r,n}),\qquad
p\mapsto p(\tau)=(p(\tau_n)).
\end{equation}
Composing it with the cup product, we get a 
$\Lambda_{K_r}$-module structure $\bullet$ on $\Lb^{(r)}_{K}$ 
which preserves the direct summand $\Lb^{(r)}_{n,K}$ for each $n$.
The $\Lambda_{K_r}$-action on $\Lb^{(r)}_{n,K}$ factors through a
$\Lambda_{n,K_r}$-action via the map $\pi_n$. 
We define an action of $\Lambda_{K_r}$
on $\End(\Lb^{(r)}_{K})$ by setting
\begin{equation}
p_l \bullet u=[p_l(\tau), u],  \qquad u \in \End(\Lb^{(r)}_{K}),\qquad
l \geqslant 1.
\end{equation}
This action preserves each graded component of $\End(\Lb^{(r)}_{K})$.
Note that the $K_r$-subalgebra $\U^{(r),>}_{K}$ 
of $\End(\Lb^{(r)}_{K})$ carries an induced $\N$-grading, 
with $f_{1,l}$ being of degree one for all $l$.

\vspace{.1in}

\begin{ex}\label{EX:wilsonU} The restriction of the Wilson operator $p_l$ to 
$\Lb^{(r)}_{n,K}$ is the cup product with 
\begin{equation}
p_l(\tau_n)=p_l(\cc_1(\rho_1), \ldots, \cc_1(\rho_n)),
\end{equation}
if $\tau_n=\rho_1 + \cdots + \rho_n$ is a sum of invertible $\widetilde D$-equivariant bundles.
Thus
$p_l \bullet f_{1,k}$ is represented by the correspondence
$$\prod_{n \geqslant 0}\cc_1(\tau_{n,n+1})^k (p_l(\tau_{n+1})-p_l(\tau_n))
=\prod_{n \geqslant 0} \cc_1(\tau_{n,n+1})^k p_l(\tau_{n,n+1})
=\prod_{n \geqslant 0} \cc_1(\tau_{n,n+1})^{l+k}$$
from which we get 
\begin{equation}
p_l \bullet f_{1,k}=f_{1,l+k}.
\end{equation}
For $r$-partitions $\lambda$, $\mu$ with $\mu\subset\lambda$ and for any 
$p\in\Lambda_{K_r}$ we write also
\begin{equation}
\tau_{\mu,\lambda}=\tau_{\lambda}-\tau_{\mu}, 
\qquad
p(\tau_{\mu,\lambda})=p(\cc_1(\rho_1), \ldots, \cc_1(\rho_n)),
\end{equation}
if $\tau_{\mu,\lambda}=\rho_1 + \cdots + \rho_n$ is a sum of $\widetilde D$-characters.
\end{ex}

\vspace{.1in}

The following lemma is left to the reader.

\begin{lem}\label{L:WilsonUL}
(a) The action of $\Lambda_{K_r}$ on $\End(\Lb^{(r)}_{K})$ 
preserves $\U_{K}^{(r),>}$,

(b) the action of $\Lambda_{K_r}$ on the degree $n$ part
of $\U^{(r)}_{K}$ factors
through $\Lambda_{n,K_r}$,

(c) 
for $a \in \Lambda_{K_r},$ $ u,u' \in \U^{(r),>}_{K}$ and 
$v \in \Lb^{(r)}_{K}$ we have
$$a \bullet  u(v)=\sum (a_1 \bullet u)(a_2 \bullet v), \qquad 
a \bullet uu'=\sum (a_1 \bullet u)(a_2 \bullet u'),$$

(d) the $K_r$-algebra isomorphism 
$\Psi: \SH^{(r),>}_{K} \to \U^{(r),>}_{K}$ 
intertwines the $\Lambda_{K_r}$-actions.
\end{lem}

\vspace{.1in}

For an element $u \in \U^{(r),>}_{K}$ and for $r$-partitions 
$\lambda,$ $\mu$ let 
$\langle \lambda \;;\; u \;;\; \mu \rangle$ be the coefficient of 
$[I_{\lambda}]$ in $u( [I_{\mu}])$. This coefficient is zero
unless $\mu \subset \lambda$. For $p \in \Lambda_{K_r}$ we have
\begin{equation}
\langle \lambda \;;\; p \bullet u \;;\; \mu \rangle=p(\tau_{\mu,\lambda}) \langle \lambda \;;\; u \;;\; \mu \rangle.
\end{equation}
We will say an element $p \in \text{Frac}(\Lambda_{n,K_r})$ is \emph{regular} 
if it is regular at $\tau_{\mu,\lambda}$
for any $\lambda, \mu$ with $|\lambda \backslash \mu|=n$. 
If $p$ is regular then its action on $\U^{(r),>}_{K}$ is well-defined. Indeed, it is well-defined on any operator 
$\gamma \in \End(\Lb^{(r)}_K)$ satisfying
$$\langle \lambda \;;\; \gamma \;;\; \mu \rangle \neq 0 \Rightarrow \mu \subset \lambda.$$

We now provide an explicit description of the action 
of some element of
$\U^{(r),>}_{K}$ on $\Lb^{(r)}_{K}$ in terms of Wilson operators. 
For this we define a surjective $K_r$-linear map
\begin{equation}
\iota : K_r[z_1, \ldots, z_n] \to \U^{(r),>}_{K}, 
\qquad z_1^{l_1} \cdots z_n^{l_n} \mapsto f_{1,l_1} \cdots f_{1,l_n}
\end{equation}
and a twisted symmetrization map
\begin{equation}
\begin{split}
\varpi_n : \begin{cases}K_r[z_1, \ldots, z_n] &\to 
K_r(z_1, \ldots, z_n)^{\mathfrak{S}_n}=\text{Frac}(\Lambda_{n,K_r})\\
P(z_1, \ldots, z_n) &\mapsto 
\text{SYM}_n\big(g(z_1, \ldots, z_n) P(z_1, \ldots, z_n)\big)
\end{cases}
\end{split}
\end{equation}
where $\SYM_n$
is the standard symmetrization map 
\begin{equation}
\SYM_n:\quad
K_r[z_1, \ldots, z_n] \to K_r[z_1, \ldots, z_n]^{\mathfrak{S}_n},\qquad
P \mapsto \sum_{\sigma} \sigma \cdot P,\end{equation} and where 
\begin{equation}
g(z_1, \ldots, z_n)=\prod_{i<j} g(z_i-z_j),
\qquad
g(z)=\frac{(z+x)(z+y)}{z(z+x+y)}.
\end{equation}
For $n\geqslant 1$ consider the element 
$\gamma_n$ in $\E^{(r)}_{n,K}$ given by
\begin{equation}
\gamma_n=\prod_{\mu \subset \lambda} a_{\mu,\lambda} 
\eu_{\mu}^{-1} [I_{\lambda\mu}],
\end{equation}
\begin{equation}\label{E:defclassfond}
a_{\mu,\lambda}=\eu\big( (1-q)(1-t) ( \tau^*_{\mu,\lambda}
\otimes \tau_\lambda) - \tau^*_{\mu,\lambda} \otimes W -nv^{-1}\big),
\end{equation}
where the product ranges over all $r$-partitions $\lambda,$ $\mu$ 
such that $\mu\subset\lambda$, $|\lambda|=|\mu|+n$.
This element
gives rise to an operator of degree $n$ in 
$\End(\Lb^{(r)}_{K})$. Let $\gamma_n$ denote also this operator. It does not belong to 
$\U^{(r),>}_{K}$ unless $n=1$ (then it is the product of the
fundamental classes of the correspondences $M_{r,k,k+1}$ for $k\geqslant 0$).

\begin{lem}\label{L:shuffleWilson} For 
$P \in K_r[z_1, \ldots, z_n]$ the element $\varpi_n(P)$ is regular and 
$\iota (P)=\varpi_n(P) \bullet \gamma_n$
in $\End(\Lb^{(r)}_{K})$.
\end{lem}

\begin{proof} See Appendix \ref{app:D2}.
\end{proof}

\vspace{.1in}

\begin{prop}\label{prop:torsionfreeWilson}
The action of $\Lambda_{n,K_r}$ 
on the degree $n$ part of $\U^{(r),>}_{K}$ is torsion free.
\end{prop}

\begin{proof} By Lemma \ref{L:shuffleWilson}, it is enough to show that the map
\begin{equation}\Lambda_{n,K_r} \to \End(\Lb^{(r)}_{K}),\qquad
p \mapsto p \bullet \gamma_n\end{equation}
is injective. Now, an element 
$p \in \Lambda_{n,K_r}$ annihilates $\gamma_n$ if and only if 
\begin{equation}
\mu \subset \lambda,\qquad |\lambda \backslash \mu|=n,\qquad
a_{\mu,\lambda} \neq 0\quad\Rightarrow\quad p(\tau_{\mu,\lambda})=0.
\end{equation}
We claim
that in fact $a_{\mu,\lambda} \neq 0$ for any pair satisfying 
$\mu \subset \lambda$ and $|\lambda \backslash \mu|=n$. 
This indeed implies that 
any $p$ which annihilates $\gamma_n$ must be zero because the collection 
of possible values of $(\cc_1(\rho_1), \ldots, \cc_1(\rho_n))$ is Zariski 
dense in $K_r^n$. 
To prove the claim we must check that the trivial representation 
does not appear in 
\begin{equation}\label{E:mult.triv}
(1-q)(1-t) (\tau^*_{\mu,\lambda} \otimes \tau_\lambda)-\tau^*_{\mu,\lambda} \otimes W.
\end{equation}
Recall that 
\begin{equation}\label{E:mult-triv2}
\tau_{\mu,\lambda}=\sum_{a=1}^r \sum_{s \in \lambda^{(a)} \backslash \mu^{(a)}} \chi_a^{-1} t^{y(s)}q^{x(s)}.
\end{equation}
The multiplicity of the trivial representation in \eqref{E:mult.triv} is a sum of 
contributions from each box $s \in \lambda \backslash \mu$. It is easy to check
using \eqref{E:mult-triv2} that this contribution is precisely zero for each box.
We are done.
\end{proof}

\vspace{.2in}

\section{Equivariant cohomology of the commuting variety}

In this section we introduce an algebra $\mathbf{SC}$ in the equivariant cohomology of the 
commuting variety. Then we provide a description of $\SC$ in terms of shuffle algebras. 
In Section 6 we will construct an 
action of $\mathbf{SC}$ on 
$\Lb^{(r)}$ and we'll compare
$\SC$ with $\SH^>$. 

\vspace{.15in}

\subsection{Correspondences in equivariant Borel-Moore homology}
\label{sec:correspondence}
Let $G$ be a complex linear algebraic group. 
Let $P\subset G$ a parabolic subgroup and $M\subset P$ a Levi
subgroup. Fix a $M$-variety $Y$.  The group $P$ acts on $Y$ through
the obvious group homomorphism $P\to M$.
Let $X=G\times_PY$ be the induced $G$-variety.
Now assume that $Y$ is smooth. For any smooth subvariety $O\subset Y$
let $T^*_OY$ be the conormal bundle to $O$. It is
well-known that the induced $M$-action on $T^*Y$ is Hamiltonian and
that the zero set of the moment map is the closed $M$-subvariety
$$T^*_MY=\bigsqcup_{O}T^*_OY,$$
where $O$ runs over the set of $M$-orbits. Further we have \cite{SV2}
\begin{equation}
\label{induction}
T^*X=T^*_P(G\times Y)/P,\qquad T^*_GX=G\times_PT^*_MY. 
\end{equation}
So the
induction yields a canonical isomorphism 
\begin{equation}H^M(T^*_MY)=H^G(T^*_GX).\end{equation}
We'll call \emph{fibration} a smooth morphism which is locally trivial in
the analytic topology. Let $X'$ be a smooth $G$-variety and $V$ be a
smooth $M$-variety. Fix $M$-equivariant
homomorphisms
\begin{equation}\label{4.1:diagram1}
\xymatrix{Y&V\ar[l]_p\ar[r]^q&X'}
\end{equation}
with $p$ a fibration and $q$ a
closed embedding. Set $W=G\times_PV$ and consider the following maps
\begin{equation}\gathered
\xymatrix{X&W\ar[l]_f\ar[r]^g& X'},\\
f: (g,v)\, \mod\, P\mapsto (g,p(v))\, \mod\, P,\\
g: (g,v)\, \mod\, P\mapsto gq(v).
\endgathered
\end{equation}
Note that $V$, $W$, $X$, $X'$ are smooth.
Further, the map $f$ is a $G$-equivariant fibration, the map $g$ is a
$G$-equivariant proper morphism, and the map $f\times g$ is a closed
embedding $W\subset X\times X'$. See \cite{SV2} for details.
We'll identify $W$ with its image in $X\times X'$. The $G$-variety
\begin{equation}
Z=T^*_W(X\times X')\end{equation}
is again smooth and
the obvious projections yield $G$-equivariant maps
\begin{equation}\label{4.1:diagram2}
\xymatrix{T^*X&Z\ar[l]_-\phi\ar[r]^-\psi&T^*X'}.
\end{equation} 
We define the $G$-variety
\begin{equation}
Z_G=Z\cap(T^*_GX\times T^*_GX').
\end{equation}
The following is immediate.

\begin{lem}
(a) The map $\psi$ is proper, the varieties $T^*X$, $Z$ and $T^*X'$ are smooth.

(b) We have $\phi^{-1}(T^*_GX)=Z_G$
and $\psi(Z_G)\subset T^*_GX'$.
\end{lem}

\vskip3mm

\noindent We'll abbreviate
$\phi_G=\phi|_{Z_G}$ and $\psi_G=\psi|_{Z_G}$.  We have the following diagram
of singular varieties
\begin{equation}\label{4.1:diagram3}
\xymatrix{T^*_GX&Z_G\ar[l]_-{\phi_G}\ar[r]^-{\psi_G}&T^*_GX'}.
\end{equation}
Since the map $\psi_G$ is 
proper the direct image yields maps
\begin{equation}\gathered
\psi_{G,*}:H^G(Z_G)\to H^G(T^*_GX').
\endgathered
\end{equation}
Since $Z$, $T^*X$ are smooth and $\phi^{-1}(T^*_GX)=Z_G$,
the pull-back by $\phi$ yields a map
\begin{equation}\gathered
\phi_G^*:H^G(T^*_GX)\to H^G(Z_G).
\endgathered
\end{equation}
Composing $\psi_{G,*}$ and $\phi_G^*$ we get a map
\begin{equation}\label{CO:00}\psi_{G,*}\circ\phi_G^*:H^G(T^*_GX)\to H^G(T^*_GX').
\end{equation}
By \eqref{induction} the induction yields also an isomorphism
\begin{equation}H^M(T^*_MY)=H^G(T^*_GX).
\end{equation}
Composing it by $\psi_{G,*}\circ \phi_G^*$
we obtain a map
\begin{equation}\label{CO:0}H^M(T^*_MY)\to H^G(T^*_GX').\end{equation}

\vskip3mm

\subsection{The commuting variety}
We'll apply the general construction above to the commuting variety. 
First, we fix some notation. Let $E$ be a
finite dimensional $\CC$-vector space. 
Write
\begin{equation}
\gen_E=\End(E),\qquad
C_E=\{(a,b)\in \gen_E\times\gen_E\,;\,[a,b]=0\}.
\end{equation}
We may abbreviate $G=GL_E$, $\gen=\gen_E$
and $C=C_\gen=C_E$.
Put $\widetilde G=T\times G$ with $T=(\CC^\times)^{2}.$ 
The group $\widetilde G$ acts on $C$ :
the subgroup $G$ acts diagonally by the adjoint action on $\gen$, while $T$
acts by $(e,f)\cdot(a,b)=(ea,fb).$
We set $\Cb'_{E}=H^{\widetilde G}(C).$ 
Let $K_{\widetilde G}$ be the fraction field of $R_{\widetilde G}$. 
Let $\mathscr V_n$ be the groupoid formed by all $n$-dimensional
vector spaces with their isomorphisms and set 
$\mathscr V=\bigoplus_{n\geqslant 0}\mathscr V_n$.
An isomorphism $E\to E'$ yields an $\HT$-module isomorphism 
$\Cb'_{E}\to \Cb'_{E'}.$ Let
$\Cb'$ be the colimit of the system  $(\Cb'_E)$ where $E$ varies in 
$\mathscr V$. It is a $\N$-graded vector space.
The piece $\Cb'_n$ of degree $n$ is 
the colimit over the groupoid $\mathscr V_n$.

\vspace{.2in}

\subsection{The cohomological Hall algebra}
\label{sec:cohHall}
Fix a flag of finite dimensional vector spaces
\begin{equation}\label{flag}
\xymatrix{0\ar[r]& E_1\ar[r]& E\ar[r]& E_2\ar[r]& 0}.
\end{equation}
Set $G=GL_E$, $M=GL_{E_1}\times GL_{E_2}$ and $P=\{g\in G;g(E_1)=E_1\}$. 
Let $\gen$, $\men$ and $\pen$ be the corresponding Lie algebras. Put
$Y=\men$, $V=\pen$, and $X'=\gen$. The $G$-action on
$X'$ and the $M$-action on $Y$ are the adjoint ones. Put 
\begin{equation}
C_\men=(\men\times\men)\cap C,\quad
C_\pen=(\pen\times\pen)\cap C_\gen,\quad
\widetilde C_\men=\{(d,a,b)\in\pen\times\men\times\men;\,
d_\men=[a,b]\},
\end{equation}
where $p:\pen\to\men$, 
$a\mapsto a_\men$
is the canonical projection. 
We apply the general construction in Section 
\ref{sec:correspondence}
to the diagram \eqref{4.1:diagram1} equal to
\begin{equation}\xymatrix{\men&\pen\ar[l]_p\ar[r]^q&\gen},
\end{equation}
where $q$ is the obvious inclusion. 
The $P$-actions on $\pen\times\pen$ and on
$\pen\times\men\times\men$ are the obvious ones.
Further we identify $\gen^*=\gen$
and $\men^*=\men$ via the trace.

\begin{lem}\label{lem:Comm}
(a) There are isomorphisms of $G$-varieties
$$\begin{gathered}
T^*X=G\times_P\widetilde C_\men,\quad
Z=G\times_P(\pen\times\pen),\quad
T^*X'=\gen\times\gen.
\end{gathered}$$ 

(b) For 
$a,b\in\pen$ we have
$$\begin{gathered}
\phi((g,a,b)\ \mod\ P)=(g,[a,b],a_\men,b_\men)\ \mod\
P,\quad \psi((g,a,b)\ \mod\ P)=(gag^{-1},gbg^{-1}).
\end{gathered}$$

(c) There are isomorphisms of $G$-varieties
$$\begin{aligned}
T^*_GX=G\times_PC_\men,\quad
Z_G=G\times_PC_\pen,\quad
T^*_GX'=C_\gen.
\end{aligned}$$

(d)
The maps $\phi$, $\psi,$
$\phi_G$, $\psi_G$ in the following diagram are the obvious ones
$$\xymatrix{G\times_PC_\men&G\times_PC_\pen\ar[l]_-{\phi_G}\ar[r]^-{\psi_G}&
C_\gen},$$
$$\xymatrix{G\times_P\widetilde C_\men&G\times_P(\pen\times\pen)
\ar[l]_-{\phi}\ar[r]^-{\psi}&
\gen\times\gen}.$$
\end{lem}

\vspace{.2in}

\noindent 
We define as in
(\ref{CO:0}) a $\HT$-linear map 
\begin{equation}
H^{\widetilde M}(C_\men)\to H^{\widetilde G}(C_\gen).
\end{equation}
By the Kunneth formula, it can be
viewed as a map
\begin{equation}\label{CO:3}\Cb'_{E_1}\otimes_{\HT} 
\Cb'_{E_2}\to \Cb'_{E}.\end{equation}
The following is proved as in \cite[prop.~7.5]{SV2}.

\vspace{.1in}

\begin{prop} The map (\ref{CO:3}) equips $\Cb'$
with the structure of a $\HT$-algebra with 1.
\end{prop}

We call the $\N$-graded $\HT$-algebra $\Cb'$ 
the \emph{cohomological Hall algebra}.
Let $\SC'$ be the $\HT$-subalgebra of $\Cb'$ 
generated by $\Cb'_1$. 
We'll abbreviate
$\SC'_n=\Cb'_n\cap\SC'$
and $G=GL_n$.
The direct image by the obvious inclusion
$C_\gen\subset\gen\times\gen$, which is a proper map,
yields a $R_{\widetilde G}$-module homomorphism
\begin{equation}\label{CO:14}\Cb'_n\to H^{\widetilde G}(\gen\times\gen).
\end{equation} We conjecture that
(\ref{CO:14}) is an injective map. 
Since the kernel of (\ref{CO:14}) is the torsion 
submodule $\Cb^\tor_n$ of $\Cb'_n$ by the localization theorem, 
this conjecture is equivalent to
the following one. 

\begin{conj}\label{C:torsion} The
$R_{\widetilde G}$-module $\Cb'_n$ is torsion-free.
\end{conj}

Let $\Cb_n$, $\SC_n$ be the image of $\Cb'_n$, $\SC'_n$
by \eqref{CO:14} and set
\begin{equation}
\Cb=\bigoplus_{n\geqslant 0}\Cb_n,\qquad
\SC=\bigoplus_{n\geqslant 0}\SC_n.
\end{equation}
We call $\SC$ the \emph{spherical subalgebra} of $\Cb.$

\begin{prop}\label{prop:SC}
The map (\ref{CO:14}) yields surjective 
$\HT$-algebra homomorphisms 
$\Cb' \to \Cb$ and $\SC' \to \SC$.
\end{prop}

\begin{proof} For $E\in\mathscr V$
let $\Cb_E$ be the quotient of $\Cb'_E$ by its torsion 
$R_{\widetilde GL_E}$-submodule $\Cb_E^\tor$.
Given $E_1,E_2,E$ as in \eqref{flag}, we must check that the map 
(\ref{CO:3}) fits into a commutative square
\begin{equation}\begin{split}\xymatrix{
\Cb'_{E_1}\otimes_{\HT}\Cb'_{E_2}\ar[r]\ar[d]&\Cb'_E\ar[d]\cr
\Cb_{E_1}\otimes_{\HT}\Cb_{E_2}\ar[r]&\Cb_E.}
\end{split}
\end{equation}
Recall that $\Cb_E$ is identified with the image by the obvious map
\begin{equation}
H^{\widetilde G}(C_\gen)\to H^{\widetilde G}(\gen\times\gen).
\end{equation}
Similarly, since $\widetilde C_\men$ is isomorphic to 
$\uen\times\men\times\men$ 
as a $\widetilde M$-module, where $\uen$ is the nilpotent radical of $\pen$,
we can identify $\Cb_{E_1}\otimes_{\HT}\Cb_{E_2}$ with the image of 
the direct image by the obvious inclusion
\begin{equation}
H^{\widetilde  M}(C_\men)\to H^{\widetilde M}(\widetilde C_\men).
\end{equation}
So the proposition follows from the commutativity of the diagram 
\begin{equation}
\begin{split}
\xymatrix{
H^{\widetilde G}(G\times_PC_\men)\ar[r]^-{\phi_G^*}\ar[d]&
H^{\widetilde G}(G\times_PC_\pen)\ar[r]^-{\psi_{G,*}}\ar[d]
&H^{\widetilde G}(C_\gen)\ar[d]\cr
H^{\widetilde G}(G\times_P\widetilde C_\men)\ar[r]^-{\phi^*}&
H^{\widetilde G}(G\times_P(\pen\times\pen))\ar[r]^-{\psi_{*}}
&H^{\widetilde G}(\gen\times\gen).}
\end{split}
\end{equation}
\end{proof}

\vspace{.1in}

For any commutative ring extension $R\subset L$ we abbreviate
\begin{equation}
\Cb'_{L}=\Cb'\otimes_{\HT}L,\qquad
\SC'_{L}=\SC'\otimes_{\HT}L,\qquad 
\SC_{L}=\SC\otimes_{\HT}L,\qquad etc.
\end{equation}

\vspace{.2in}

\subsection{The shuffle algebra}
Fix $E\in\mathscr V_n$. 
Let $G=GL_n$ and let $D\subset G$ be a maximal torus. 
The Poincar\'e duality and
the inverse image by the
obvious inclusion $\{0\}\to\gen\times\gen$ yield an isomorphism
$H^{\widetilde G}(\gen\times\gen)=R_{\widetilde G}.$
Composing it with \eqref{CO:14} we get a $R_{\widetilde G}$-linear map
\begin{equation}
\gamma_G:\Cb'_n\to R_{\widetilde G}.
\end{equation}
Taking the tensor power over $\HT$, we define a $R_{\widetilde D}$-linear map
\begin{equation}
\gamma_D=(\gamma_{\CC^\times})^{\otimes n}:
(\Cb'_1)^{\otimes n}\to R_{\widetilde D}.
\end{equation}
Recall that there are obvious isomorphisms
\begin{equation}
\gathered
R_{\widetilde D}=\HT[z_1,z_2,\dots,z_n],\qquad 
R_{\widetilde G}=\HT[z_1,z_2,\dots,z_n]^{\Sen_n}.
\endgathered
\end{equation}
Let
$K_{\widetilde D}$ and
$K_{\widetilde G}$ 
be the fraction fields.
We have the usual symmetrization operator
\begin{equation}
\SYM_n: K_{\widetilde D}\to K_{\widetilde G}.
\end{equation}

\vspace{.1in}

\begin{prop}\label{P:KHall} 
We have the commutative diagram
$$\xymatrix{
(\Cb'_1)^{\otimes
n}\ar[r]^{\mu_n}\ar[d]_{\gamma_D}&\Cb'_n\ar[d]^{\gamma_G}\cr
R_{\widetilde D}\ar[r]^{\nu_n}&R_{\widetilde G},}
$$ where $\mu_n$ is the multiplication in $\Cb'$ and
$\nu_n$ is given by
\begin{equation}\label{E:KHallform}
\gathered
\nu_n(P(z_1, \ldots, z_n))=
\SYM_n\bigl(k(z_1,z_2,\dots z_n)P(z_1,z_2,\dots z_n)\bigr),\\
k(z)=z^{-1}(x+y+z)(x-z)(y-z),\\
k(z_1,z_2,\dots z_n)=\prod_{i<j}k(z_i-z_j).
\endgathered
\end{equation}
\end{prop}

\begin{proof} 
Let $\den$ be the Lie algebra of $D$.
Since $C_\den$ is a vector space, 
the $R_{\widetilde D}$-module $H^{\widetilde D}(C_\den)$ is spanned by 
the set 
\begin{equation}
\{z^m\cdot[C_\den]\;;\;m\in\N^n\},\qquad
z^m=z_1^{m_1}z_2^{m_2}\cdots z_n^{m_n},\qquad
m=(m_1,m_2,\dots m_n)\in\N^n. 
\end{equation}
Here $[C_\den]$ is the fundamental class and 
$\cdot$ is the $R_{\widetilde D}$-module structure on
$H^{\widetilde D}(C_\den)$.  Note that
\begin{equation}\label{CO:24}\gamma_D(z^m\cdot[C_\den])=z^m.\end{equation} 
Let $B\subset G$ be a Borel
subgroup containing $T$. Let $\ben=\text{Lie}(B)$ and let
$\nen$ be its nilpotent radical.
We have
\begin{equation}
T^*_GX=G\times_BC_\den,\quad
T^*X=G\times_B(\nen\times C_\den),\quad
C_\den=\den\times\den,\quad
Z=G\times_B(\ben\times\ben).
\end{equation} 
Let $\Ind$ denote the induction 
\begin{equation}
H^{\widetilde D}(\bullet)=
H^{\widetilde B}(\bullet)\to H^{\widetilde G}(G\times_B\bullet).
\end{equation}
Consider the elements in $H^{\widetilde G}(T^*_GX)$ given by
\begin{equation}
\alpha_m=\Ind(z^m\cdot[C_\den]),\qquad m\geqslant 0.
\end{equation}
For a future use, we consider also the following
commutative diagram
\begin{equation}
\begin{split}
\xymatrix{
T^*_GX\ar[d]_j&&T^*_GX'=C_\gen\ar[d]\cr
T^*X&Z\ar[r]^-\psi\ar[l]_-\phi&T^*X'=\gen\times\gen\cr
&G/B\ar[u]_i\ar[r]^\pi&\{0\}.\ar[u]_{h}}
\end{split}
\end{equation}
The vertical maps are the obvious inclusions.
The multiplication \eqref{CO:3} gives 
\begin{equation}\label{CO:15}\nu_n(z^m)=
h^*\psi_*\phi^*j_*(\alpha_m).
\end{equation}
Now, we compute the right hand side of (\ref{CO:15}). 
We have
\begin{equation}
j_*(\alpha_m)=\Ind\bigl(z^m\eu(v\nen^*)\cdot[\nen\times
C_\den]\bigr).
\end{equation}
Therefore we have also
\begin{equation}
\phi^*
j_*(\alpha_m)=
\Ind\bigl(z^m\eu(v\nen^*)\cdot[\ben\times\ben]\bigr).
\end{equation}
Tensoring by $K_{\widetilde G}$, the maps $i_*$, $i^*$ become invertible by
the localization theorem. We have
\begin{equation}
\aligned
\nu_n(z^m)&=h^*\psi_*i_*\Ind\bigl(z^m \eu(v\nen^*)
\eu(q^{-1}\ben^*+t^{-1}\ben^*)^{-1}\cdot[G/B]\bigr),\cr
&=h^*h_*\pi_*\Ind\bigl(z^m \eu(v\nen^*)
\eu(q^{-1}\ben^*+t^{-1}\ben^*)^{-1}\cdot[G/B]\bigr),\cr
&=\eu(q^{-1}\gen^*+t^{-1}\gen^*)\cdot
\pi_*\Ind\bigl(z^m \eu(v\nen^*)
\eu(q^{-1}\ben^*+t^{-1}\ben^*)^{-1}\cdot[G/B]\bigr),\cr &=
\pi_*\Ind\bigl(z^m \eu(v\nen^*+q^{-1}\nen+
t^{-1}\nen)\cdot[G/B]\bigr).
\endaligned
\end{equation}
Thus the integration over the set $(G/B)^{\widetilde D}$ 
yields the formula
\begin{equation}
\nu_n(z^m)=\SYM_n(k(z_1,z_2,...z_n)\, z^m).
\end{equation}
\end{proof}

\vspace{.1in}

Now, we equip the $R$-module
\begin{equation}
\Shb=\bigoplus_{n \geqslant 0} \Shb_n,\qquad
\Shb_n=\HT[z_1, \ldots, z_n]^{\Sen_n}
\end{equation}
with the shuffle multiplication given by
\begin{equation}
(P\cdot Q)(z_1, \ldots, z_{m+n})=
\frac{1}{n! m!}\, \SYM_{n+m} \big( \prod_{i,j} k(z_i-z_j) 
P(z_1, \ldots, z_n) Q(z_{n+1}, \ldots, z_{n+m})\big).
\end{equation}
The product runs over all $i,j$ with 
$1 \leqslant i \leqslant n < j \leqslant n+m$. 
For $\dim\,E=1$ and $l\geqslant 0$ we abbreviate 
\begin{equation}\theta_l=z^l\cdot[C_E].\end{equation}
The following direct consequence of Proposition \ref{P:KHall}
is the first main result of this chapter.

\begin{theo}\label{thm:shuffleC} There is a unique $R$-algebra embedding 
$\mathbf{SC} \subset \Shb$ such that
$\theta_l \mapsto (z_1)^l$, where $z_1$ is viewed as 
an element in $\Shb_1$.
\end{theo}

\vspace{.1in}

We state one useful consequence.

\vspace{.1in}

\begin{cor}\label{Cor:automshuffle} For $u \in \C$ the assignment 
$\theta_l \mapsto \sum_{i=0}^l \binom{l}{i} u^{l-i} \theta_i$ extends to an 
algebra automorphism $\tau_u \in \Aut(\mathbf{SC})$. Moreover, 
we have $\tau_u \circ \tau_v=\tau_{u+v}$ for $u, v \in \C$.
\end{cor}

\begin{proof} Under the embedding $\mathbf{SC} \subset \Shb$ the
map $\tau_u$ is the restriction of the automorphism induced by the substitution
$z_i \mapsto z_i+u$. Observe that 
$k(z_1, \ldots, z_n)$ is invariant under this substitution.
\end{proof}

\vspace{.15in}

\subsection{Wilson operators on $\Cb$ and $\mathbf{SC}$} 
The canonical $R_{\widetilde{GL}_n}$-module structure on $\Cb'_n$ gives a 
graded $\Lambda$-algebra structure on $\Cb'$ and $\Cb$, which we will
denote by $\bullet$.

\vspace{.1in}

\begin{lem}\label{L:WilsonC} 
(a) The action of $\Lambda$ on $\Cb,$ $ \Cb'$ preserves the spherical 
subalgebras $\mathbf{SC}',$ $ \mathbf{SC}$.

(b) The $\Lambda$-action on $\Cb'_n,$ $\Cb_n$ 
factors through $\Lambda_{n}$.

(c) 
For $p \in \Lambda$ and $u,v \in \Cb'$ (or $\Cb$) we have
$p \bullet (uv)=\sum (p_1 \bullet u)(p_2 \bullet v).$
\end{lem}

\begin{proof} Statement $(b)$ is clear. 
Observe that $p_l \bullet \theta_k=\theta_{l+k}$, hence $(c)$
implies $(a)$. Finally $(c)$ is a consequence of the commutativity 
of the following diagram
\begin{equation}
\begin{split}
\xymatrix{\Lambda_K \ar[d]_-{\Delta} \ar[r]^-{\pi_{n+m}} & \Lambda_{n+m,K} 
\ar@{=}[r] &   R_{\widetilde{GL}_{n+m}} \ar[d]\\
\Lambda_K \otimes_K \Lambda_K \ar[r]^-{\pi_n \otimes \pi_m} & 
\Lambda_{n,K} \otimes_K \Lambda_{m,K} 
\ar@{=}[r]  &R_{\widetilde{M}}}
\end{split}
\end{equation}
where $M\subset GL_{n+m}$ is the standard parabolic with Levi 
$GL_n \times GL_m$, and where the rightmost arrow is 
the restriction map.
\end{proof}

\vspace{.2in}

\section{Proof of Theorem~\ref{thm:SH/U1}}
\label{sec:prooflevel1}

\subsection{Part 1 : the positive and negative halves} 
Our first task is to construct an isomorphism 
$\widetilde\U^{(1),+}_K \to\SH^{+}_{K}$. 
For this, we will use the canonical representation 
of $\U_K^{(1),+}$ on $\widetilde\Lb^{(1)}_K$. It is the restriction of the 
canonical representation of $\U_K^{(1)}$ on $\Lb^{(1)}_K$
considered in Section \ref{sec:U1}.

\vspace{.1in}

\begin{prop}
\label{P:proof1}
(a) The map 
$D_{\xb} \mapsto h_{\xb}$ for
$\xb\in\mathscr E^+$ yields an
algebra isomorphism $\Psi^+:\widetilde\SH^{(1),+}_{K}\to\widetilde\U^{(1),+}_K$ 
which takes $\widetilde\SH^{(1),>}_{K}$ onto $\widetilde\U^{(1),>}_K$. 

(b) The map \eqref{Phi} intertwines $\rho^+$ 
%in Proposition \ref{prop:rho+}
with the canonical representation of $\widetilde\U^{(1),+}_K$ on $\widetilde\Lb^{(1)}_K$.
\end{prop}

\begin{proof} 
First, we compare the action of  $D_{\xb}$ on $\Lambda_{K}$
with the action of $h_{\xb}$ on $\widetilde\Lb^{(1)}_K$, 
under the isomorphism \eqref{Phi}. 
These actions are described by the formulas \eqref{E:Horace}
and (\ref{E:Pieri2}) for $D_{\xb}$,
and by the formulas \eqref{E:36} and \eqref{E:34} 
for $h_{\xb}$. 
These formulas coincide, because $\psi_{\lambda\setminus\mu}=L_{\mu,\lambda}$.
Since $\rho^+$ 
is a faithful representation, 
see Proposition~\ref{prop:rho+},
this yields the isomorphism $\Psi^+$ above.
\end{proof}

\vspace{.1in}

\begin{rem}
By Propositions \ref{prop:generators/involution2} 
and  \ref{prop:involutionU},
the $K$-algebras $\widetilde\SH^{(1),-}_K$ and 
$\widetilde\U^{(1),-}_K$ are isomorphic to the 
opposite $K$-algebras of $\widetilde\SH^{(1),+}_K$ and $\widetilde\U^{(1),+}_K$ respectively.
Thus, by Proposition~\ref{P:proof1}, the assignment 
$D_{\xb} \mapsto h_{\xb}$
for $\xb\in\mathscr E^-$ extends to an algebra isomorphism 
$\Psi^-:\widetilde\SH^{(1),-}_{K} \to\widetilde\U^{(1),-}_K$. 
\end{rem}

\vspace{.15in}

\subsection{Part 2 : glueing the positive and negative halves}
\label{sec:part3}
We must prove that the two algebra isomorphisms $\Psi^+,$ $\Psi^-$
glue together to an algebra homomorphism 
\begin{equation}\label{Psi} 
\Psi:\widetilde\SH_{K}^{(1)}\to\widetilde \U^{(1)}_K.
\end{equation} 
It suffices to check  \eqref{E:rel3bis2}.
The proof of this relation follows from 
Appendix \ref{app:B} by setting $r=1$ and $e_\a=0$ there.
To finish the proof of Theorem~\ref{thm:SH/U1}, it remains to show that the map 
$\Psi$ is an isomorphism. 
Since it is clearly surjective, we only have to check that it is 
injective. Our argument is based on the existence of 
triangular decompositions for $\widetilde\SH^{(1)}_K$ and $\widetilde\U^{(1)}_K$. 
First, let us quote the following proposition
whose proof is given in Appendix \ref{app:C}.

\vspace{.1in}

\begin{prop}\label{P:proof2} The multiplication gives an isomorphism 
$$m:\widetilde\U^{(1),>}_K \otimes_K \widetilde\U^{(1),0}_K \otimes_K 
\widetilde\U^{(1),<}_K \to\widetilde\U^{(1)}_K.$$
\end{prop}

\vspace{.1in}

Let $\Psi^>,$ $ \Psi^0,$ $ \Psi^<$ be the restrictions of 
$\Psi^+$, $\Psi^-$  
to $\widetilde\SH^{(1),>}_K,$ $\widetilde\SH^{(1),0}_K$ and $\widetilde\SH^{(1),<}_K$.
We have the following commutative diagram
\begin{equation}\begin{split}
\label{C:proof1}\xymatrix{ \widetilde\SH^{(1),>}_{K} \otimes_K
\widetilde\SH^{(1),0}_{K} \otimes_K\widetilde \SH^{(1),<}_{K}
\ar[rr]^-{\Psi^>\otimes\Psi^0\otimes\Psi^<} \ar[d]_-{m} && 
\widetilde\U^{(1),>}_K \otimes_K \widetilde\U^{(1),0}_K \otimes_K \widetilde\U^{(1),<}_K \ar[d]^-{m}\\
\widetilde\SH_{K}^{(1)}\ar[rr]^-{\Psi}& &\widetilde \U^{(1)}_K.}
\end{split}
\end{equation}
Further, we have the following isomorphisms
\begin{equation}
\widetilde\SH^{(1),0}_K=K[D_{0,l}\,;\,l\geqslant 1], \qquad
\widetilde\U^{(1),0}_K=K[h_{0,l}\,;\,l\geqslant 1].
\end{equation}
Thus, by Proposition~\ref{P:proof1} and Proposition~\ref{P:proof2}, 
the top arrow and the right one are isomorphisms. 
The left arrow is surjective by Proposition \ref{1.8:prop1}.
Thus the left arrow and the bottom one are both isomorphisms.
Theorem~\ref{thm:SH/U1} is proved.

\vspace{.2in}

\section{Proof of Theorem \ref{thm:SH/Ur} }
\label{sec:proofthmSH/Ur}

\subsection{Part 1 : the positive and negative halves} 
Given $E_1, E\in\mathscr V$ with
$E_1\subset E$, we write 
\begin{equation}
\gathered
E_2=E/E_1, \qquad
M=GL_{E_1}\times GL_{E_2},\qquad P=\{g\in G\,;\,g(E_1)=E_1\},\\
X'=\gen\times \Hom(\C^r,E),\quad Y=\men\times \Hom(\C^r,E_2),\qquad
V=\pen\times \Hom(\C^r,E).
\endgathered
\end{equation}
Here $\pen$, $\men$ are the Lie algebras of $P$, $M$.
Consider the obvious maps
\begin{equation}
\pi:E\to E_2,\qquad
p:V\to Y,\qquad
q:V\to X'.
\end{equation}
For $x\in \pen$ let $x_\men$ be its projection in $\men$, modulo the nilpotent radical $\uen$ of $\pen$.
Let $X$, $W$, $Z$, $Z_G$, $\phi$, $\psi$
be as in Section \ref{sec:correspondence}
and $M_{r,E}$, $N_{r,E}$ be as in \eqref{notation1}.
Define 
\begin{equation}\gathered
N_\men=(\gen_{E_1})^2\times N_{r,E_2}=\men^2\times\Hom(E,\CC^r)\times\Hom(\CC^r,E),\\
M_\men=C_{E_1}\times M_{r,E_2}=\{(a,b,\varphi,v)\in
N_\men\,;\,0=[a,b]+v\circ\varphi\},\\
N_\pen=\pen^2\times \Hom(E_2,\C^r)\times \Hom(\C^r,E), \\
M_\pen=N_\pen\cap M_{r,E}=\{(a,b,\varphi,v)\in
N_\pen\,;\,0=[a,b]+v\circ\varphi\},\\
\widetilde N_\men=\{(c,a,b,\varphi,v)\in\pen\times
N_\men\,;\,c_\men=[a,b]+v\circ\varphi\}\simeq \uen\times N_\men.\\
\endgathered
\end{equation}
We have the following technical lemma \cite[lem.~8.2]{SV2}.

\vspace{.1in}

\begin{lem}
\label{lem:toto89}
(a) We have canonical isomorphisms of $G$-varieties
$$\begin{aligned}
T^*X=G\times_P\widetilde N_\men,\quad Z=G\times_PN_\pen,
\quad T^*X'=N_{r,E},
\end{aligned}$$
$$\begin{aligned}
T^*_GX=G\times_P M_\men,\quad 
Z_G=G\times_PM_\pen,
\quad T^*_GX'=M_{r,E}.
\end{aligned}$$

(b) The maps $\phi:Z\to T^*X$ and $\psi:Z\to T^*X'$ in 
\eqref{4.1:diagram2} are given,
for $(a,b,\varphi,v)\in N_\pen$, by
$$\begin{gathered}
\phi((g,a,b,\varphi,v)\ \mod\ P)=
(g,[a,b]+v\circ\varphi,a_\men,b_\men,\varphi,\pi\circ v)\ \mod\
P,\cr \psi((g,a,b,\varphi,v)\ \mod\
P)=(gag^{-1},gbg^{-1},(\varphi\circ\pi)g^{-1},gv).
\end{gathered}$$

(c) The inclusion $T^*_GX\subset T^*X$ is induced by the inclusion
$M_\men\to\widetilde N_\men$, $(a,b,\varphi,v)\mapsto(0,a,b,\varphi,v)$.
The inclusion $Z_G\subset Z$ is induced by the obvious inclusion
$M_\pen\subset N_\pen$.
The inclusion $T^*_GX'\subset T^*X'$ is the obvious one.
\end{lem}

\vspace{.1in}

Using this lemma, we can now prove the following.

\vspace{.1in}

\begin{prop} 
There is a representation $\eta'$ of $\Cb'$ on $\Lb^{(r)}$
such that $\eta'(\theta_l)=f_{1,l}$ for $l\in\N$. 
\end{prop}

\begin{proof}
To define $\eta'$ we consider the closed embeddings 
\begin{equation}
N_\pen\subset N_{r,E},\qquad M_\pen\subset M_{r,E},\qquad
(a,b,\varphi,v)\mapsto(a,b,\varphi\circ\pi,v).
\end{equation} 
Then, we set 
\begin{equation}
N_\pen^s=N^s_{r,E}\cap N_\pen,\qquad M^s_\pen=M^s_{r,E}\cap M_\pen,
\qquad
Z^s=G\times_PN^s_\pen,\qquad
Z^s_G=G\times_PM^s_\pen.
\end{equation}
Note that $N^s_\pen$, $M^s_\pen$, $Z^s$, $Z^s_G$ are open in
$N_\pen$, $M_\pen$, $Z$ and $Z_G$. 
Next, the proper map $\psi:Z\to T^*X'$ restricts to a
proper map $\psi_s:Z^s\to N^s_{r,E}$, 
because $Z^s=Z\cap\psi^{-1}(N^s_{r,E})$.
Finally, we have $\psi_s(Z^s_G)\subset M^s_{r,E}$, because $Z^s_G=Z_G\cap Z^s$.
Thus, taking the direct image by $\psi_s$, we get the commutative diagram 
\begin{equation}
\label{step1}
\begin{split}
\xymatrix{
H^{\widetilde G}(Z^s_G)\ar[r]^-{\psi_{s,*}}\ar[d]
&H^{\widetilde G}(M^s_{r,E})\ar[d]\\
H^{\widetilde G}(Z^s)\ar[r]^-{\psi_{s,*}}
&H^{\widetilde G}(N^s_{r,E}).}
\end{split}
\end{equation}
Now, set $N^s_\men=(\gen_{E_1})^2\times N_{r,E_2}^s$, 
$M_\men^s=C_{E_1}\times M_{r,E_2}^s$,
$\widetilde N^s_\men=\widetilde N_\men\cap(\pen\times N_\men^s)$
and
$$T^*X^s=G\times_P\widetilde N_\men^s,\quad
T^*_GX^s=G\times_PM^s_\men.$$
For $(a,b,\varphi,v)\in N_\pen^s$ we have 
$(a_\men,b_\men,\varphi,\pi\circ v)\in N_\men^s$.
Thus the map $\phi:Z\to T^*X$ restricts to a map
$\phi_s:Z^s\to T^*X^s$. The varieties $Z^s$ and $T^*X^s$ are both smooth and
we have 
$\phi^{-1}_s(T^*_GX^s)=Z^s\cap Z_G=Z^s_G$. Hence 
the pull-back by $\phi_s$ gives the commutative diagram 
\begin{equation}
\label{step2}
\begin{split}
\xymatrix{
H^{\widetilde G}(T^*_GX^s)\ar[r]^-{\phi_s^*}\ar[d]&
H^{\widetilde G}(Z^s_G)\ar[d]
\\
H^{\widetilde G}(T^*X^s)\ar[r]^-{\phi_s^*}&
H^{\widetilde G}(Z^s).
}
\end{split}
\end{equation}
Set $n_1=\dim\, E_1$, $n_2=\dim\, E_2$ and $n=n_1+n_2$. 
Since $M^s_{r,E_2}$ is a $GL_{E_2}$-torsor over $M_{r,n_2}$, 
by descent we have an isomorphism
\begin{equation}
\Lb^{(r)}_{n_2}=H^{\widetilde GL_{E_2}}(M^s_{r,E_2}).
\end{equation}
Thus, the induction and the Kunneth formula yield an isomorphism
\begin{equation}\label{isom22}\Ind:\Cb'_{n_1}\otimes_{\HT}\Lb^{(r)}_{n_2}=
H^{\widetilde M}(C_{E_1}\times M^s_{r,E_2})=
H^{\widetilde G}(T^*_GX^s).\end{equation}
We have also
$\Lb^{(r)}_{n}=H^{\widetilde G}(M^s_{r,E}).$
Thus, composing \eqref{isom22}
with \eqref{step1} and \eqref{step2}, 
we get a map
\begin{equation}
\label{step3}
\Cb'_{n_1}\otimes_{\HT}\Lb^{(r)}_{n_2}\to\Lb^{(r)}_{n}.
\end{equation}
The same argument as in 
the proof of \cite[prop.~7.9]{SV2} implies that
\eqref{step3} defines a $R$-linear representation of $\Cb'$ on $\Lb^{(r)}$.
Details are left to the reader.
Let $\eta'$ denote this representation.

Now, we compute the image of the element $\theta_l\in\Cb'_1$ by the map $\eta'$.
To do so, we change slightly the notation.
Assume that $E_1\in\mathscr V_1$ and $E\in\mathscr V_{n+1}$.
Fix $x\in\Lb^{(r)}_{n}$
and let $y$ be the image of $\theta_l\otimes x$ by
the map 
\begin{equation}\Cb'_{1}\otimes_{\HT}\Lb^{(r)}_{n}\to\Lb^{(r)}_{n+1}
\end{equation}
given in \eqref{step3}.
We must check that $y=\cc_1(\tau_{n+1,n})^l\cdot x$.
By definition of \eqref{step3}
we have 
\begin{equation}
y=\psi_{s,*}\phi_s^*\Ind(\theta_l\otimes x).
\end{equation}
The variety $Z^s_G$ is the set of all pairs
$\big((a,b,\varphi,v), E_1\big)$
where $(a,b,\varphi,v)\in M^s_{r,E}$, $a,b\in\pen$ and $\varphi(E_1)=0$. 
It is a smooth $G$-torsor over $M_{r,n+1,n}$. 
Hence, by descent we have an isomorphism
\begin{equation}\label{descent}
H^{\widetilde G}(Z^s_G)\to H^T(M_{r,n+1,n}).
\end{equation}
So we have the commutative diagram
\begin{equation}
\label{step4}
\begin{split}
\xymatrix{
H^{\widetilde G}(Z^s_G)\ar[r]^-{\psi_{s,*}}\ar[d]
&H^{\widetilde G}(M^s_{r,E})\ar[d]\\
H^{T}(M_{r,n+1,n})\ar[r]^-{\pi_{1,*}}
&H^{T}(M_{r,n+1})}
\end{split}
\end{equation}
where both vertical maps are given by descent.
Therefore, it is enough to observe that the isomorphism \eqref{descent}
takes
$\phi_{s}^*\Ind(\theta_l\otimes x)$ to
$\cc_1(\tau_{n+1,n})^l\,\pi_2^*(x)$.
\end{proof}

\vspace{.1in}

%\begin{ex} For $n_1=n$ and $n_2=0$ we have
%$$T^*_GX^s=C_\gen,\qquad Z^s_G=\big(C_\gen\times\Hom(\CC^r,E)\big)^s,\qquad
%T^*X^s=\gen^2,\qquad Z^s=\big(\gen^2\times\Hom(\CC^r,E)\big)^s.$$
%%Thus we have $\eta'(x)([I_\emptyset])=0$ for each $x\in\Cb^\tor$.
%\end{ex}
%
%\vspace{.1in}

Since the representation of $\U^{(r)}_K$ on $\Lb^{(r)}_K$ is faithful, 
the map $\eta'$ gives a surjective $K_r$-algebra homomorphism 
\begin{equation}
\eta':\SC'_{K_r}\to\U^{(r),>}_K, \qquad
\eta'(\theta_l)=f_{1,l}, \qquad
l\in\N. 
\end{equation}
We can now prove the following.

\begin{theo} \label{thm:isomC/U}
The map $\eta'$ factors to a 
$K_r$-algebra isomorphism 
\begin{equation}\eta:\SC_{K_r}\to\U_{K}^{(r),>}\end{equation}
taking $\theta_l$ to $f_{1,l}$,
which commutes with the action of $\Lambda_{K_r}$.
\end{theo}

\begin{proof}
First, we claim that $\eta'$ commutes with Wilson operators.
It is enough to check it on generators by 
Lemma \ref{L:WilsonUL}$(c)$ and Lemma \ref{L:WilsonC}$(c)$. 
Example \ref{EX:wilsonU} gives
\begin{equation}
\eta'(p_l \bullet \theta_k)=\eta'(\theta_{l+k})=f_{1,l+k}
=p_l \bullet f_{1,k}=p_l \bullet \eta'(\theta_k),\qquad
l \geqslant 1,\qquad k \geqslant 0, 
\end{equation}
proving the claim. 
Next, by Proposition \ref{prop:torsionfreeWilson} 
the action of $\Lambda_{n,K_r}$ on 
$\U^{(r),>}_{K}[n]$ is torsion free. 
Hence the map $\eta'$ factors to a 
surjective $K_r$-algebra homomorphism 
\begin{equation}\eta:\SC_{K_r}\to\U_{K}^{(r),>}\end{equation}
taking $\theta_l$ to $f_{1,l}$.
It remains to show that $\eta$ 
is injective. Let $x \in \Sb\Cb_{n,K_r}$ and assume that $\eta(x)=0$. 
If $x \neq 0$ then, by the localization theorem, for any $y \in\Sb\Cb_{n,K_r}$ 
there exists $p,p' \in \Lambda_{n,K_r}$ such that $p \bullet x=p'\bullet y$. 
But, then, we have
\begin{equation}
p' \bullet \eta(y)=\eta(p' \bullet y)=\eta(p \bullet x)=
p \bullet \eta(x)=0.
\end{equation} 
It follows that $\eta(y)$ is torsion, hence $\eta(y)=0$ 
by Proposition \ref{prop:torsionfreeWilson}. This contradicts the surjectivity 
of $\eta$. We deduce that $x=0$, i.e., that $\eta$ is injective.
\end{proof}

\vspace{.1in}

Proposition \ref{P:proof1} and Theorem \ref{thm:isomC/U} (for $r=1$)
yield the following.

\vspace{.05in}

\begin{cor}
\label{cor:SC/SH}
There is a $K$-algebra isomorphism $\SC_K\to\SH^{>}_K,$
$\theta_l\mapsto x^{l}D_{1,l}$.
\end{cor}

\vspace{.05in}

\begin{rem}
\label{opisom}
Proposition \ref{prop:involutionU} and Theorem \ref{thm:isomC/U}
give a $K_r$-algebra homomorphism
\begin{equation}
\eta^\op:(\SC_{K_r})^{\op}\to\U_K^{(r),<},\qquad \theta_l\mapsto f_{-1,l}.
\end{equation}
Proposition \ref{prop:generators/involution2} 
and Corollary \ref{cor:SC/SH}
give a $K$-algebra isomorphism 
\begin{equation}
(\SC_K)^\op\to\SH^{<}_K,\qquad
\theta_l\mapsto x^{l}D_{-1,l}.
\end{equation}
We define $\U^{(r),>}$ and $\U^{(r),<}$ to be 
the images of $\SC_{R_r}$, $(\SC_{R_r})^\op$
by the maps $\eta$ and $\eta^\op$.
We have $R_r$-algebra isomorphisms
\begin{equation}
\SC_{R_r}\to\U^{(r),>},\qquad (\SC_{R_r})^{\op}\to\U^{(r),<}.
\end{equation}
\end{rem}

\vskip3mm

\subsection{Part 2 : glueing the positive and negative halves}  
Theorem \ref{thm:isomC/U},  Corollary \ref{cor:SC/SH} and Remark
\ref{opisom} give $K_r$-algebra isomorphisms 
\begin{equation}
\gathered
\Psi^>:\SH^{(r),>}_K\to\U_K^{(r),>},\qquad
\Psi^<:\SH^{(r),<}_K\to\U_K^{(r),<}
\endgathered
\end{equation}
such that $\Psi^>(D_{1, l})=h_{1, l}$ and $\Psi^<(D_{-1, l})=h_{-1, l}.$
Next, Appendix \ref{app:B} gives the following.

\begin{prop}
\label{prop:glue/r}
The class $[h_{-1,k},h_{1,l}]$ is supported on the diagonal of 
$M_{r,n}\times M_{r,n}$ and it coincides, as an element of $\U^{(r)}_K$, with
the operator $E_{k+l}$ on $\Lb^{(r)}_K$ given by
$$1+\xi\sum_{l\geqslant 0}E_{l}\,s^{l+1}
=\exp\Bigl(\sum_{l\geqslant 0}(-1)^{l+1}p_l(\eps_\a)\phi_l(s)\Bigr)\,
\exp\Bigl(\sum_{l\geqslant 0}h_{0,l+1}\,\varphi_l(s)\Bigr).$$
\end{prop}

We can now prove Theorem \ref{thm:SH/Ur}.
First, note that we have
a $K_r$-algebra homomorphism
\begin{equation}
\Psi~:\SH^{(r)}_K\to\U^{(r)}_K,\qquad
D_{\xb} \mapsto h_{\xb},\qquad\xb\in\mathscr E.
\end{equation}
Indeed, relation \eqref{E:rel3bis2} 
follows from Proposition \ref{prop:glue/r} and
\eqref{E:rel1bis2}, \eqref{E:rel2bis2} from Remark \ref{rem:f0l}. 
Thus, we are reduced to checking the following proposition,
whose proof is given in Section \ref{app:B.3}.

\begin{prop}
\label{prop:faithfulr}
The representation $\rho^{(r)}$ is faithful.
\end{prop}

\vspace{.2in}

\section{The comultiplication}
So far, we have defined an algebra $\SH^\cb$ and we have constructed a representation
$\rho^{(r)}$ of $\SH^\cb$ in $\Lb^{(r)}_K$. In order to compare $\SH^\cb$  with W-algebras, it is
important to equip it with a Hopf algebra structure. We do not know how to describe the
(topological) coproduct on $\SH^\cb$ in an elementary algebraic way. Our argument uses
our previous work \cite{SV2}. First, we prove that $\SH^\cb$ can be regarded as a degeneration of the 
elliptic Hall algebras which was studied there. This is Theorem \ref{7.7:thm1}. Next, using this result, 
we prove that the coproduct of the elliptic 
Hall algebra degenerates and induces a coproduct on
$\SH^\cb$. This is Theorem \ref{7.7:thm2}. 

\vspace{.1in}

\subsection{The DAHA}
We'll abbreviate $\Am=\C[q^{\pm 1/2}, t^{\pm 1/2}]$,
$\Km=\CC(q^{1/2},t^{1/2})$ and
$v^{1/2}=(qt)^{-1/2}$. Fix an integer $n> 1$.
The \emph{double affine Hecke algebra} (=DAHA) 
of $GL_n$ is the associative $\Km$-algebra
${\Hm}_n$ generated by
\begin{equation}
\label{7.1}
X_1^{\pm 1},\dots,X_n^{\pm 1}, 
Y_1^{\pm 1},\dots,Y_n^{\pm 1}, 
T_1,\dots,T_{n-1}
\end{equation}
subject to the following relations \cite[sec.~1.4.3]{C2}
\begin{equation}\label{E:1qt}
T_iX_iT_i=X_{i+1},\quad
T_i^{-1}Y_iT_i^{-1}=Y_{i+1},
\end{equation}
\begin{equation}\label{E:2qt}
T_iX_j=X_jT_i,\quad
T_iY_j=Y_jT_i,\quad j\neq i,i+1,
\end{equation}
\begin{equation}\label{E:3qt}
(T_i+t^{1/2})(T_i-t^{-1/2})=0,\quad
T_iT_{i+1}T_i=T_{i+1}T_iT_{i+1},
\end{equation}
\begin{equation}\label{E:4qt}
T_iT_j=T_jT_i,\quad j\neq i-1,i,i+1,
\end{equation}
\begin{equation}\label{E:5qt}
P X_i=X_{i+1}P,\quad P X_n=q^{-1}X_1P,
\quad P=Y_1^{-1}T_1\cdots T_{n-1}, 
\quad i\neq n,
\end{equation}
Let $\Hm_n^+$ be the $\Km$-subalgebra generated by
\begin{equation}
\label{7.7}
X_1,\dots,X_n, 
Y_1^{\pm 1},\dots,Y_n^{\pm 1}, 
T_1,\dots,T_{n-1},
\end{equation}
and let $\Sm$ be the complete idempotent. 
We set
\begin{equation}
\Sm\Hm_n=\Sm\,  \Hm_n  \Sm,\qquad \Sm\Hm_n^+=\Sm\,  \Hm_n^+  \Sm.
\end{equation}
For $\xb\in\Z^2_0$ we define an element
$P_\xb^{(n)}$ of $\Sm\Hm_n$ as in \cite[sec.~2.2]{SV1}.
For $l\geqslant 1$ we have
\begin{equation}
\label{7.1:form}
\gathered
P^{(n)}_{l,0}=q^l\Sm p_l(X_1,\dots,X_n)\Sm,\quad
P^{(n)}_{-l,0}=\Sm p_l(X_1^{-1},\dots,X_n^{-1})\Sm,\\
P^{(n)}_{0,l}=\Sm p_l(Y_1,\dots,Y_n)\Sm ,\quad
P^{(n)}_{0,-l}=q^l\Sm p_l(Y_1^{-1},\dots,Y_n^{-1})\Sm .
\endgathered
\end{equation}
There is a unique $\Km$-algebra automorphism 
\cite[sec.~3.1]{SV1}, \cite[sec.~3.2.2]{C2},
\begin{equation}\label{sigman}
\sigma:\Sm\Hm_n\to\Sm\Hm_n,\qquad
P_\xb^{(n)}\mapsto P^{(n)}_{\sigma(\xb)},\qquad
\sigma(i,j)=(j,-i).
\end{equation}
Let $\Hm_{n,\Am}$ be the $\Am$-subalgebra of $\Hm_n$ generated by
\eqref{7.1} and set 
$\Sm\Hm_{n,\Am}=\Sm\,\Hm_{n,\Am}\,\Sm$.
Note that 
\begin{equation}
\Hm_n=\Hm_{n,\Am}\otimes_\Am\Km,
\qquad\Sm\Hm_n=\Sm\Hm_{n,\Am}\otimes_\Am\Km,
\end{equation}
We have an $\Am$-basis of $\Hm_{n,\Am}$ given by \cite{C2}
\begin{equation}\label{PBWA}
\{ X^{{\alpha}}Y^{{\beta}}T_w\;;\; 
{\alpha} \in \Z^n, {\beta} \in \Z^n, w \in \mathfrak{S}_n\}.
\end{equation}
Consider the $\Km$-vector spaces
\begin{equation}
\mathbb{W}_n=\Wb_{n,\Km},\qquad
\mathbb{V}_n=\V_{n,\Km}.
\end{equation}
The $\Km$-algebra $\Hm_n$ is equipped with a faithful representation 
\cite[sect. 4.1]{SV1}
\begin{equation}
\varphi_n: \Hm_n \to \text{End}(\mathbb{W}_n)
\end{equation}
called the \emph{polynomial representation}.
The subalgebras $\Sm\Hm_n$ and $\Sm\Hm_n^+$ 
act faithfully on the subspaces
$\mathbb{W}^{\mathfrak{S}_n}_n$ and $\Lambda_{n,\mathbb{K}}.$
For a partition $\lambda$ with at most $n$ parts let $J_\lambda(X;q,t^{-1})$ be the 
integral form of the Macdonald polynomial $P_\lambda(X;q,t^{-1})$, see \cite[chap.~VI, (8.3)]{Mac}.
We abbreviate
\begin{equation}
J^{(n)}_\lambda(q,t^{-1})=J_\lambda(X_1,\dots,X_n;q,t^{-1}).
\end{equation}
This yields the following $\Km$-basis of $\Lambda_{n,\mathbb{K}}$
\begin{equation}
\{J^{(n)}_\lambda(q,t^{-1})\,;\,l(\lambda)\leqslant n\}.
\end{equation}
Finally, for $\xb\in\Z^2_0$ we define new elements
$u_\xb^{(n)}$, $\theta(u_\xb^{(n)})$ of $\Sm\Hm_n$ as follows.
First, we set
\begin{equation}
a_l=(q^l-1)(t^l-1),\qquad l\geqslant 1.
\end{equation}
Next, for $\xb=(i,j)$ and $l=\text{gcd}(\xb)$, we set
\begin{equation}\label{7.2:theta}
\gathered
P_\xb^{(n)}=(q^l-1)\,u^{(n)}_\xb,\qquad
\theta(u^{(n)}_\xb)=
t^{j(n-1)/2}u^{(n)}_\xb-\delta_{i,0}(t^{jn}-1)/a_j.
\endgathered
\end{equation}
We have the following formula \cite[cor.~1.5]{SV2}
\begin{equation}\label{7.2:eigenvalue}
\theta(u^{(n)}_{0,l})\cdot J^{(n)}_\lambda(q,t^{-1})=\sum_{s\in\lambda}q^{lx(s)}t^{ly(s)}\,
J^{(n)}_\lambda(q,t^{-1}),\qquad l\geqslant 1.
\end{equation}
By \cite{SV1}, \cite[sec.~1.3]{SV2} we have also
\begin{equation}
\label{7.20}
\gathered
{[u_{0,l}^{(n)}, u^{(n)}_{\pm 1,k}]}=\pm sgn(l)\, u^{(n)}_{\pm 1,k+l},
\endgathered
\end{equation}
where 
\begin{equation}\label{sgn}
sgn(l)=1,\qquad sgn(-l-1)=-1,\qquad l\geqslant 0.
\end{equation} 
To unburden the notation, 
let $\Sm\Hm_n$ denote also the smash product
\begin{equation}
\Km[u_{0,0}]\otimes_\Km\Sm\Hm_n,
\end{equation}
where $u_{0,0}^{(n)}$ is a new formal variable and the commutator with $u_{0,0}^{(n)}$ 
is the $\Km$-derivation 
\begin{equation}
[u_{0,0}^{(n)},u_{i,j}^{(n)}]=iu_{i,j}^{(n)}.
\end{equation}
The element $u^{(n)}_{0,0}$ acts on $\mathbb V_n$ as the grading operator.
We'll set $\theta(u^{(n)}_{0,0})=u^{(n)}_{0,0}$.

\vspace{.15in}

\subsection{The degeneration of $\Hm_n$}
Our aim is to construct a degeneration from $\Hm_n$ to 
$\mathbf{H}_n$. 
The degenerations of $\Hm_n$ have been 
extensively studied, see e.g., \cite{VV2}. Here we only need a very particular 
one introduced for the first time by Cherednik. 
We set 
\begin{equation}\scrK=F((h)),\qquad \scrA=F[[h]]. 
\end{equation}
We refer to \cite{Ka} for a reminder on 
topological $\scrA$-modules
(for the $h$-adic topology). 
Let $\widetilde\otimes$ denote the \emph{topological 
tensor product} of $\scrA$-modules.
An $\scrA$-module is \emph{topologically free} if it is isomorphic to $V[[h]]$ for
an $F$-vector space $V$.
Let $F\langle X\rangle$ be the free $F$-algebra on $X$.
For a future use, recall 
that a complete separated $\scrA$-algebra $\mathbf A$ is 
\emph{topologically generated} by a subset $X$
if the obvious continuous map $F\langle X\rangle [[h]]\to\mathbf A$ 
is surjective.

First, consider the algebras embedding $\Am\subset\scrA$ given by
\begin{equation}\label{completion1}
q^{1/2}\mapsto\exp(h/2),\qquad 
t^{1/2}\mapsto\exp(-\kappa h/2).
\end{equation}
Let $I$ be the ideal of $\Am$ given by $I=(h)\cap\Am$.
Let $\Bm_{n}\subset\Hm_n$ be the $\Am$-subalgebra generated by
$\Hm_{n,\Am}$ and the elements $(Y_i-1)/(q-1)$ with $i\in[1,n]$.
We set
\begin{equation}
\gathered
\Hc_{n,\scrA}=\varprojlim_k(\Bm_{n}/I^k\Bm_{n}).
\endgathered
\end{equation}
By \eqref{PBWA}
the $\Am$-algebra $\Bm_{n}$ is topologically linearly spanned by the elements of the form
$X^{{\alpha}}f(Y)T_w$ where $\alpha\in\Z^n$, $w\in\Sen_n$, and
$f(Y)\in \Am[Y_i^{\pm 1},(Y_i-q^kt^l)/(q-1)]$ for $i\in[1,n]$ and $k,l\in\Z$.
Consider the element $y_i$ in $\Hc_{n,\scrA}$ given by
\begin{equation}
y_i=\sum_{l \geqslant 1}(-1)^{l-1}(Y_i-1)^l/lh.
\end{equation}
We have faithful representations, see Section 1.3,
\begin{equation}
\varphi_n:\Hm_n\to\text{End}(\mathbb{W}_n),\qquad
\rho_n: \mathbf{H}_n \to \text{End}(\mathbf{W}_{n}).
\end{equation}
From \eqref{completion1} we get inclusions
\begin{equation}
\mathbb{W}_n\subset \mathbf{W}_{n}((h)),\qquad
\operatorname{End}(\mathbb{W}_n)\subset \operatorname{End}(\mathbf{W}_{n})((h)).
\end{equation}
We abbreviate 
\begin{equation}
\Oc(h^l)=h^l\Hc_{n,\scrA},\qquad\Hc_{n}=\Hc_{n,\scrA}/h\Hc_{n,\scrA},\qquad l\in\N.
\end{equation}

\vspace{.1in}

\begin{lem}
\label{lem:7.1}
The $\scrA$-module $\Hc_{n,\scrA}$ is topologically free.
As a topological $\scrA$-algebra it is generated by the set
$\{T_j, X_i^{\pm 1},y_i\,;\,i\in[1,n],\,j\in[1,n)\}.$
We have 
\begin{equation}
X_i^{\pm 1} \in \Oc(1), 
\quad T_j \in \Oc(1), 
\quad Y_i = 1 + \Oc(h),
\quad y_i=(Y_i-1)/h+\Oc(h).
\end{equation}
The map $\varphi_n$ yields a continuous embedding 
$\varphi_n:\Hc_{n,\scrA}\to\operatorname{End}(\mathbf{W}_n)[[h]].$ 
\end{lem}

\begin{proof}
The $\scrA$-module $\Hc_{n,\scrA}$ is topologically free
because it is separated, complete and torsion free. 
The other statements are easy and are left to the reader.
\end{proof}

\vspace{.1in}

Finally, we set
\begin{equation}
\Hc_{n,\scrK}=\Hc_{n,\scrA}\otimes_\scrA\scrK.
\end{equation}
By base change, the map $\varphi_n$ yields a continuous embedding 
\begin{equation}
\varphi_n:\Hc_{n,\scrK}\to\operatorname{End}(\mathbf{W}_n)((h)).
\end{equation}
The following is standard. The proof is left to the reader.

\vspace{.1in}

\begin{prop}\label{P:degen1}
(a) 
We have $\Hc_{n,\scrA}=\{x\in\Hc_{n,\scrK}\,;\,\varphi_n(x)\in\operatorname{End}(\mathbf{W}_n)[[h]]\}$.

(b)
The map $\varphi_n$ factors to an injection
$\varphi'_n : 
\Hc_{n}\to\operatorname{End}(\mathbf{W}_{n}).$

(c) There is a unique $F$-algebra isomorphism 
$\phi_n:\Hc_{n}\to\H_n$ such that
\begin{equation*}\label{7.31:isomorphism}\qquad
\phi_n(X_j^{\pm 1})=X_{n+1-j}^{\pm 1},\qquad 
\phi_n(y_j)=y_{n+1-j}- (n-1)\kappa /2,\qquad \phi_n(T_i)=s_{n-i}.
\end{equation*}

(d) We have $\varphi'_n=\rho'_n \circ \phi_n$, where
$\rho'_n=w_0 \rho_n w_0$ and $w_0\in\Aut(\mathbf{W}_{n})$ is
given by $X_i \mapsto X_{n+1-i}$. 
\end{prop}

\vspace{.15in}

\subsection{The degeneration of $\Sm\Hm_n$}
We now turn our attention to the spherical subalgebras. 
Set
\begin{equation}
\Sc\Hc_{n,\scrA}=
\mathbb{S} \cdot \Hc_{n,\scrA} \cdot \mathbb{S},
\qquad
\Sc\Hc_{n}=\Sc\Hc_{n,\scrA}/h\Sc\Hc_{n,\scrA}.
\end{equation}
The map $\varphi'_n$ factors to an injective map
\begin{equation}
\varphi'_n : \Sc\Hc_{n} 
\to \text{End}(\V_{n}).
\end{equation}
For $l\geqslant 1$ 
we consider the following elements
\begin{equation}
\label{7.4:Q1ln}
\gathered
Q_{0,l}^{(n)}=
h^{1-l}\sum_{k=0}^{l-1}\Bigl(\begin{matrix} l-1\\k\end{matrix}\Bigr)
(-1)^k\theta(u_{0,l-1-k}^{(n)}),\\
Q_{l,0}^{(n)}=(-1)^l\kappa^l P_{l,0}^{(n)},\qquad
Q_{-l,0}^{(n)}=P_{-l,0}^{(n)},\qquad
Q_{1,l}^{(n)}=[Q_{0,l+1}^{(n)},Q_{1,0}^{(n)}],\qquad
\endgathered
\end{equation}

\vspace{.1in}

\begin{prop}\label{P:degen2}
(a)  For $l\geqslant 1$ the elements $Q_{\pm l,0}^{(n)}$
and $Q_{0,l}^{(n)}$ belong to $\Sc\Hc_{n,\scrA}$.

(b) The map $\phi_n$ restricts to an
$F$-algebra isomorphism
$\Sc\Hc_{n}\to\SH_n$
such that we have
$P_{\pm l,0}^{(n)}\mapsto D_{\pm l,0}^{(n)}$
and $Q_{0,l}^{(n)}\mapsto D_{0,l}^{(n)}$
for $l\geqslant 1.$

(c) The algebra $\Sc\Hc_{n,\scrA}$ is topologically generated by 
$P^{(n)}_{\pm 1,0}$ and $Q^{(n)}_{0,2}$.

%(d) We have $\rho_n \circ \phi_n=\varphi_n$
%on $\Sc\Hc_{n,\scrA}/h\,\Sc\Hc_{n,\scrA}.$
\end{prop}

\begin{proof} We first prove $(a)$. We have
$P_{\pm l, 0}^{(n)} \in \mathcal{S}\Hc_{n,\scrA}$. 
We consider the inclusions 
\begin{equation}
\V_{n},\, \mathbb V_n\subset
\V_{n,\scrK}
\end{equation}
associated with the obvious inclusion
$F\subset\scrK$ and with the embedding
$\Km\subset\scrK$ in \eqref{completion1}. 
We have, by \cite[chap.~VI, (10.23)]{Mac},
\begin{equation}\label{E:loijk}
(1-t^{-1})^{-|\lambda|}J_\lambda^{(n)}(q,t^{-1})= J_\lambda^{(n)}
\,\mod\,h\V_{n,\scrA}.               
\end{equation}  
By \eqref{7.2:eigenvalue}, for  
$l(\lambda)\leqslant n$, we have
\begin{equation}\label{E:hyut}
\begin{split}
Q_{0,l}^{(n)}\cdot J_\lambda^{(n)}(q,t^{-1})
&=h^{1-l}\sum_{s\in\lambda}
\sum_{k=0}^{l-1}
\Bigl(\begin{matrix}l-1\\k\end{matrix}\Bigr)
(-1)^k
q^{(l-1-k)c(s)}\,J_\lambda^{(n)}(q,t^{-1}),\\
&=\sum_{s\in\lambda}
\Big(\frac{q^{c(s)}-1}{h}\Big)^{l-1}\,J_\lambda^{(n)}(q,t^{-1}).
\end{split}
\end{equation}
By Proposition \ref{P:degen1} the $\scrA$-algebra 
$\Sc\Hc_{n,\scrA}$ is the subalgebra of 
\begin{equation}
\Sc\Hc_{n,\scrK}=\Sc\Hc_{n,\scrA}\otimes_\scrA\scrK
\end{equation}
which preserves the subspace $\V_{n,\scrA}$ of
$\mathbf{V}_{n,\scrK}$. 
It entails that $Q_{0,l}^{(n)} \in \Sc\Hc_{n,\scrA}$ as wanted. 
We now deal with $(b)$ and $(c)$. 
Note that $$h \Hc_{n,\scrA} \cap \Sc\Hc_{n,\scrA}= h \Sc\Hc_{n,\scrA}.$$
Therefore the natural map gives an injection
$
\Sc\Hc_{n} \to \Hc_{n}.
$
Since
$\Sc\Hc_{n,\scrA}=\Sm \cdot \Hc_{n,\scrA} \cdot \Sm$, the map $\phi_n$
in Proposition \ref{P:degen1} restricts to an injection
\begin{equation}\label{phi2}\phi_n: \Sc\Hc_{n} \to \mathbf{SH}_n.
\end{equation}
The equality $\phi_n(Q_{0,l}^{(n)})= D_{0,l}^{(n)}$
is a consequence of \eqref{E:Horace}, \eqref{E:loijk} and \eqref{E:hyut}. 
The equality $\phi_n(P_{\pm l,0}^{(n)})= D_{\pm l,0}^{(n)}$ follows from
\eqref{1.5:form} and \eqref{7.1:form}. The map $\phi_n$ in
\eqref{phi2} is surjective because, by Lemma \ref{lem:uno}, the $F$-algebra $\SH_n$ is generated by
$\{D^{(n)}_{\pm l,0},\,D^{(n)}_{0,l}\,;\,l\geqslant 1\}$.
Claim $(c)$ is a consequence of Nakayama's lemma together with the fact that 
$\mathbf{SH}_n$ is generated by
$D_{0,2}^{(n)}$ and $D^{(n)}_{\pm 1, 0}$.
\end{proof}

\vspace{.1in}

Let $\Sc\Hc^>_{n,\scrA},$ $\Sc\Hc^<_{n,\scrA}$ and $\Sc\Hc^{0}_{n,\scrA}$
be the closed $\scrA$-subalgebras of $\Sc\Hc_{n,\scrA}$
topologically generated respectively by the sets
$\{Q_{1,l}^{(n)}\,;\,l\geqslant 0\}$, $\{Q_{-1,l}^{(n)}\,;\,l\geqslant 0\}$ and
$\{Q_{0,l}^{(n)}\,;\,l\geqslant 0\}$.
We abbreviate
\begin{equation}
\Sc\Hc^>_n=\Sc\Hc_{n,\scrA}^>/h\,\Sc\Hc_{n,\scrA}^>,\qquad
\Sc\Hc^<_n=\Sc\Hc_{n,\scrA}^<\big/h\,\Sc\Hc_{n,\scrA}^<,\qquad
\Sc\Hc^{0}_n=\Sc\Hc_{n,\scrA}^{0}\big/h\,\Sc\Hc_{n,\scrA}^{0}.
\end{equation}
Using Proposition \ref{P:degen2}, we get the following.

\begin{cor}\label{cor:degen2}
The map $\phi_n$ gives $F$-algebra isomorphisms
$$\Sc\Hc_{n}^>\to\SH_n^>, \qquad
\Sc\Hc_{n}^<\to\SH_n^<,\qquad \Sc\Hc_n^{0} \to \SH_n^0$$
such that
$Q_{\pm l,0}^{(n)}\mapsto D_{\pm l,0}^{(n)}$, $Q_{\pm 1, l} \mapsto D_{\pm 1, l}$
and
$Q_{0,l}^{(n)}\mapsto D_{0,l}^{(n)}$
for $l\geqslant 1.$
\end{cor}

\begin{proof} 
Let $\Sc\Hc^>_{n,\scrK}$ be the closed $\scrK$-subalgebra of $\Sc\Hc_{n,\scrK}$
generated by $\{Q_{1,l}^{(n)}\,;\,l \geqslant 0\}$. We have 
\begin{equation}
\Sc\Hc_{n,\scrA}^> \subset \Sc\Hc_{n,\scrK}^> \cap \Sc\Hc_{n,\scrA},
\end{equation} 
and the map $\phi_n$ yields an isomorphism
\begin{equation}
\Sc\Hc_{n,\scrK}^> \cap \Sc\Hc_{n,\scrA} /h (\Sc\Hc_{n,\scrK}^> \cap \Sc\Hc_{n,\scrA}) \to \SH_n^>.
\end{equation} 
Since the induced map
$\Sc\Hc^>_{n,\scrA} \to \SH^>_n$ is surjective, we deduce that 
\begin{equation}
\Sc\Hc_{n,\scrA}^> = \Sc\Hc^>_{n,\scrK} \cap \Sc\Hc_{n,\scrA}.
\end{equation} 
In particular, we have also 
\begin{equation}
\Sc\Hc_{n,\scrA}^> \cap h \Sc\Hc_{n,\scrA}=\Sc\Hc^>_{n,\scrK} \cap h \Sc\Hc_{n,\scrA}=h \Sc\Hc_{n,\scrA}^>.
\end{equation} 
This shows the existence of an $F$-algebra isomorphism 
\begin{equation}
\Sc\Hc_n^> \to \SH_n^>,\qquad
P_{l,0}^{(n)}\mapsto D_{l,0}^{(n)},
\qquad
[Q_{0,l+1}^{(n)},P_{1,0}^{(n)}]\mapsto D_{1,l}^{(n)}.
\end{equation} 
Since $\Sc\Hc_{n,\scrA}^>$ is $\mathbb{N}$-graded, there
exists an automorphism of $\Sc\Hc_{n,\scrA}^>$ sending $P_{l,0}^{(n)}$ to $Q_{l,0}^{(n)}$. This proves the corollary for
$\Sc\Hc^>_n$. The other cases are similar.
\end{proof}

\vspace{.15in}

\subsection{The algebra $\Sm\Hm^\cm$}
\label{sec:SHgamma}
Consider the $\Km$-algebra
$\widehat\Ecb$ in \cite[sec~1]{SV2} associated with the parameters
\begin{equation}
\sigma^{1/2}=q^{-1/2},\qquad\bar\sigma^{1/2}=t^{-1/2}.
\end{equation}
It is generated by elements $u_{\xb},$ $\kappa_\xb$ with
$\xb\in\Z^2_0,$ 
satisfying the relations in \cite[sec.~1.1]{SV2}.
For $\gcd(\xb)=1$ and $l \geqslant 1$, we set
\begin{equation}
\label{7.2:eq4}
\gathered
\alpha_l=(1-q^l)(1-t^l)(1-v^{-l})/l,\\
P_{l\xb}=(q^l-1)\,u_{l\xb},\qquad
\sum_{l\geqslant 0} \theta_{l\xb}\, s^{l}=
\exp\big(\sum_{l\geqslant 1} \alpha_{l}\, u_{l\xb}\,s^l\big).
\endgathered
\end{equation}
Since $\widehat\Ecb$ is an \textit{extended Hall algebra}
in Ringel's sense, see e.g., \cite[sec.~1.6]{Kap}, 
it admits a \emph{topological coproduct} $\Delta$, which is given by the following 
formula, compare \cite[sec. 7]{BS},
\begin{equation}
\gathered
\Delta(\kappa_\xb)=\kappa_\xb\otimes\kappa_\xb,\\
\Delta(u_{0,l})=
u_{0,l}\otimes 1+\kappa_{0,l}\otimes u_{0,l},\quad l\neq 0,\\
\Delta(u_{1,l})=
u_{1,l}\otimes 1+\sum_{k\geqslant 0}
\theta_{0,k}\kappa_{1,l-k}\otimes u_{1,l-k},\quad l\in\Z.
\endgathered
\end{equation}
The expression 
``topological coproduct" means that $\Delta$ 
maps into some completion of the tensor square of
$\widehat\Ecb$, see \cite[sec. 2]{BS} for details.
By \cite[sec. 5]{BS}, 
there is a unique $\Km$-algebra automorphism 
\begin{equation}\label{sigma}
\sigma:\widehat\Ecb\to\widehat\Ecb,\quad 
\kappa_\xb\mapsto \kappa_{\sigma(\xb)},\quad 
u_\xb\mapsto u_{\sigma(\xb)}.
\end{equation}
Compare \eqref{sigman}.
We define a new topological coproduct on $\widehat\Ecb$ by the formula
\begin{equation}{}^\sigma\!\Delta=(\sigma^{-1}\otimes\sigma^{-1})\circ\Delta\circ\sigma.
\end{equation}

\vspace{.05in}

Now, we fix a family of formal parameters $\cm_l$ with $l\in\Z$ and we set
\begin{equation}
\Km^\cm=\Km[ \cm_{l}\,;\,l\in\Z][\cm_0^{- 1}],\qquad
\Am^\cm=\Am[\cm_{l}\,;\,l\in\Z][\cm_0^{-1}].
\end{equation}
Let
$\Ecb^\cm$ be the specialization of
$\widehat\Ecb\otimes_\Km\Km^\cm$ at
$\kappa_{1,0}=\cm_0$ and $\kappa_{0,1}=1$. 
Let $u_{0,0}$ be a new formal variable, and consider the smash product
\begin{equation}
\Sm\Hm^\cm=\Km[u_{0,0}]\otimes_\Km\Ecb^\cm,
\end{equation}
where the commutator with $u_{0,0}$ 
is the $\Km^\cm$-derivation on
$\Ecb^\cm$ such that
\begin{equation}
[u_{0,0},u_{i,j}]=iu_{i,j},\qquad (i,j)\in\Z^2_0.
\end{equation}
The $\Km^\cm$-algebra $\Sm\Hm^\cm$ is $\Z^2$-graded,
with 
$\deg(u_\xb)=\xb$ and
$\deg(\kappa_\xb)=0.$
We have a
topological coproduct ${}^\sigma\!\Delta$ on $\Sm\Hm^\cm$ given by
the following formulas 
\begin{equation}
\label{formcoprod}
\gathered
{}^\sigma\!\Delta(\cm_0)=\cm_0\otimes\cm_0,\\
{}^\sigma\!\Delta(\cm_l)=\delta(\cm_l)\qquad\text{if}\ l\neq 0,\\
{}^\sigma\!\Delta(u_{l,0})=u_{l,0}\otimes 1+
(\cm_0)^l\otimes u_{l,0},\\
{}^\sigma\!\Delta(u_{l,1})=u_{l,1}\otimes 1+(\cm_0)^l\otimes u_{l,1}+
\sum_{k\geqslant 1}\theta_{-k,0}(\cm_0)^{k+l}\otimes u_{k+l,1}.
\endgathered
\end{equation}
Let $\Sm\Hm^>$, $\Sm\Hm^{\cm,0}$ and $\Sm\Hm^<$ 
be the $\Km$-subalgebras generated respectively by
\begin{equation}
\{u_{1,l}\,;\,l\in\Z\},\qquad
\Km^\cm\cup\{u_{0,l}\,;\,l\in\Z\},\qquad
\{u_{-1,l}\,;\,l\in\Z\}.
\end{equation}
%Let $\Sm\Hm^+$ and $\Sm\Hm^-$ 
%be the $\Km$-subalgebras generated by
%\begin{equation}
%\{u_{1,l},\,u_{0,l}\,;\,l\in\Z\},\qquad
%\{u_{-1,l},\,u_{0,l}\,;\,l\in\Z\}.
%\end{equation}
The following holds.

\begin{lem}
%(a) The $\Km$-algebra
%$\Sm\Hm^\pm$ is generated by the set
%$\{u_{\pm l,0},\,u_{0,l}\,;\,l\geqslant 0\}.$
%
(a) The multiplication yields an isomorphism 
$\Sm\Hm^>\otimes_\Km\Sm\Hm^{\cm,0}\otimes_\Km\Sm\Hm^<
\to\Sm\Hm^\cm.$

(b) We have
$\Sm\Hm^{\cm,0}
=\Km^\cm[u_{0,l}\,;\,l\in\Z].$
\end{lem}

\begin{proof}
%Part $(a)$ follows from \cite[cor.~5.2]{BS} 
Part $(a)$ follows from \cite[sec.~1.1]{SV2}, which is proved using the formulas \cite[sec.~1.2]{SV2}
\begin{equation}
\label{7.2:eq}
\gathered
{[u_{0,l}, u_{\pm 1,k}]}=\pm sgn(l)\, u_{\pm 1,k+l},\\
{[u_{-1,k},u_{1,l}]}=\begin{cases}
sgn(k+l)\,\cm_0^{sgn(k+l)}\,\theta_{0,k+l}/\alpha_1&\text{if}\ k+l\neq 0,\\
(\cm_0-\cm_0^{-1})/\alpha_1&\text{else},
\end{cases}
\endgathered
\end{equation}
where $sgn(l)$ is as in \eqref{7.20}. Part $(b)$ is \cite[Sec. 4]{BS}.

\end{proof}

\vspace{.1in}

Next, we consider the $\Am^\cm$-subalgebra $\Sm\Hm^\cm_\Am$ generated by
the elements $u_\xb$ with $\xb\in\Z^2$. We have
\begin{equation}
\Sm\Hm^\cm=\Sm\Hm_\Am^\cm\otimes_{\Am^\cm}\Km^\cm.
\end{equation}
Finally, let $\Sm\Hm_n^>$, $\Sm\Hm_n^<$ 
be the subalgebras of $\Sm\Hm_n$ generated by
\begin{equation}
\{u_{1,l}^{(n)}\,;\,l\in\Z\},\qquad
\{u_{-1,l}^{(n)}\,;\,l\in\Z\}.
\end{equation}
By \cite[thm. 3.1]{SV1}, \cite[sec.~1.4]{SV2} there is a unique surjective
algebra homomorphism 
\begin{equation}\label{Psin}
\Psi_n:\Sm\Hm^> \to \Sm\Hm_n^>,\qquad
u_{\xb} \mapsto \theta(u_{\xb}^{(n)}).
\end{equation}
The map $\Psi=\prod_n \Psi_n$ is an embedding 
of $\Sm\Hm^>$ into $\prod_n \Sm\Hm_n^>$
by \cite[thm. 4.6]{SV1}.

\vspace{.15in}

\subsection{The degeneration of $\Sm\Hm^\cm$}
\label{sec:coproduct}
For $\xb=(i,j)\in\Z^2_0$ and $l\geqslant 1$ we define 
\begin{equation}
\label{Q0l}
\gathered
u^\cm_{0,0}=u_{0,0},\qquad
u^\cm_\xb=u_\xb+\delta_{i,0}\,\cm_j/a_j,
\\
Q_{0,l}=
h^{1-l}\sum_{k=0}^{l-1}\Bigl(\begin{matrix} l-1\\k\end{matrix}\Bigr)
(-1)^ku_{0,l-1-k}^\cm,\\
Q_{l,0}=(-1)^l\kappa^l P_{l,0},\qquad
Q_{-l,0}=P_{-l,0},\\ 
Q_{1,l}=[Q_{0,l+1},Q_{1,0}],\qquad
Q_{-1,l}=-[Q_{0,l+1},Q_{-1,0}].
\endgathered
\end{equation}
We have an inclusion of $F$-algebras
$\Am^\cm\subset\scrA^\cb$, where $\scrA^\cb=F^\cb[[h]],$ which is given by
\begin{equation}
\label{7.7:deg}
\gathered
q^{1/2}\mapsto\exp(h/2),\qquad
t^{1/2}\mapsto\exp(-\kappa h/2),\qquad
\cm_0\mapsto\exp(\xi h \cb_0/2),\\
\cm_{\pm l}\mapsto\pm\sum_{k\geqslant 0}(\mp lh)^k\cb_k/k!,
\qquad l>0.
\endgathered
\end{equation}
Consider the ideal $I=(h)\cap\Am^\cm$ in $\Am^\cm$.
Let $\Bm^\cm\subset\Sm\Hm^\cm$ 
be the $\Am^\cm$-subalgebra 
generated by
$\{Q_{\pm l, 0}, Q_{0,l}\,;\,l \geqslant 1\}.$
We define an $\scrA^\cb$-algebra by setting
\begin{equation}
\Sc\Hc_{\scrA}^\cb=\varprojlim_k\big(\Bm^\cm/I^k\,\Bm^\cm\bigr).
\end{equation}
Let $\Sc\Hc^>_{\scrA}$ and $\Sc\Hc^<_{\scrA}$ 
be the closed $\scrA^\cb$-subalgebras of $\Sc\Hc_{\scrA}^\cb$ generated by 
the sets
$\{Q_{1,l} \,;\,l\geqslant 0\}$ and
$\{Q_{-1,l}\,;\,l\geqslant 0\}$.
We write 
\begin{equation}
\Sc\Hc^>=\Sc\Hc_{\scrA}^>/h\,\Sc\Hc_{\scrA}^>,\qquad
\Sc\Hc^<=\Sc\Hc_{\scrA}^</h\,\Sc\Hc_{\scrA}^<,\qquad
\Sc\Hc^\cb=\Sc\Hc_{\scrA}^\cb/h\,\Sc\Hc_{\scrA}^\cb.
\end{equation}

\vspace{.1in}

\begin{prop} \label{7.4:prop}
(a) The $\scrA$-modules $\Sc\Hc^>_{\scrA}$ and
$\Sc\Hc^<_{\scrA}$ are topologically free.

(b) There are $F$-algebra isomorphisms
$\phi:\Sc\Hc^>\to\SH^>$ and $\phi:\Sc\Hc^<\to\SH^<$
such that we have
$\phi(Q_{\pm l,0})=D_{\pm l,0}$ and
$\phi(Q_{\pm 1,l})=D_{\pm 1,l}$ for
$l\geqslant 1.$
\end{prop}

\begin{proof}
Part $(a)$ is obvious, because 
$\Sc\Hc^>_{\scrA}$ and $\Sc\Hc^<_{\scrA}$ are 
separated, complete and torsion free. Now we prove $(b)$.
First, consider the map $\Psi_n$.
For $l\neq 0$ we have $\theta(u_{l,0}^{(n)})=u_{l,0}^{(n)}$ by \eqref{7.2:theta}.
Thus, by \eqref{7.2:theta}, \eqref{7.4:Q1ln}, \eqref{7.2:eq4}, \eqref{Psin} and \eqref{Q0l} we have also
\begin{equation}\label{7.62}
\Psi_n(Q_{l,0})=Q_{l,0}^{(n)}.
\end{equation}
Next, for $l\geqslant 1$, the formulas
\eqref{7.2:eq4}, \eqref{7.2:eq}, \eqref{Psin} and \eqref{Q0l} give
\begin{equation}
\Psi_n(Q_{1,l})=\kappa(1-q)h^{-l}\sum_{k=0}^l\Bigl(\begin{matrix} l\\k\end{matrix}\Bigr)
(-1)^k\theta(u^{(n)}_{1,l-k}), 
\end{equation}
and by \eqref{7.2:theta},  \eqref{7.20} and \eqref{7.4:Q1ln} we have also
\begin{equation}\label{7.64}
Q_{1,l}^{(n)}=\kappa(1-q)h^{-l}\sum_{k=0}^l\Bigl(\begin{matrix} l\\k\end{matrix}\Bigr)
(-1)^k\theta(u^{(n)}_{1,l-k}).
\end{equation}
Therefore,  by \eqref{7.62}, \eqref{7.64} the map $\Psi_n$ 
gives rise to a continuous $\scrA$-algebra homomorphism 
\begin{equation}\label{Psinbis}
\Psi_n :\Sc\Hc_{\scrA}^>\to\Sc\Hc_{n,\scrA}^>,
\qquad
\Psi_n(Q_{l,0})=Q_{l,0}^{(n)},
\qquad \Psi_n(Q_{1,l})=Q_{1,l}^{(n)},
\qquad l\geqslant 1.
\end{equation}
The map $\Psi$ is a closed embedding 
$\Sc\Hc_{\scrA}^> \to \prod_n \Sc\Hc_{n,\scrA}^>$.
By Proposition \ref{P:<0>} 
and Corollary \ref{cor:degen2}, 
composing $\Psi$ and $\prod_n\phi_n$ we get a map
\begin{equation}\gathered
\phi': \Sc\Hc^> \to \prod_n \mathbf{SH}_n^>,\qquad
\phi'(Q_{l,0})=(D^{(n)}_{l,0}), 
\qquad \phi'(Q_{1,l})=(D^{(n)}_{1,l}).
%\qquad \phi'(Q_{1,\lambda})=(D^{(n)}_{1,|\lambda|-l}).
\endgathered
\end{equation}
By definition of $\SH^>$, there is an embedding of $F$-algebras 
\begin{equation}
i: \mathbf{SH}^> \to \prod_n \mathbf{SH}_n^>,\qquad
i(D_{l,0})=(D^{(n)}_{l,0}), \qquad 
i(D_{0,l})=(D^{(n)}_{0,l}).
\end{equation}
Thus, we have a surjective 
$F$-algebra homomorphism $\phi$
which is given by
\begin{equation}
\phi=i^{-1}\circ\phi':\Sc\Hc^>\to\mathbf{SH}^>.
\end{equation}
We must prove that it is injective. 
We consider the partial order on $\Z^2$ given by
\begin{equation}
(r,d)\leqslant (r',d') \iff r \leqslant r'\ \text{and}\ d \leqslant d'.
\end{equation}
The $\Z^2$-grading on $\Sm\Hm^\cm$ yields a filtration on $\Sc\Hc_\scrA^>$
such that the piece
$\Sc\Hc_{\scrA}^>[\leqslant\!\xb]$
consists of the elements whose $\Z^2$-degree is $\leqslant\xb.$
The $\scrA$-module $\Sc\Hc^>_{\scrA}[\leqslant\! \xb]$
has a finite rank and we have
\begin{equation}
\gathered
\Sc\Hc^>_{\scrA}[\leqslant \! \xb] \cap 
h\, \Sc\Hc^>_{\scrA}=h\,\Sc\Hc^>_{\scrA}[\leqslant \! \xb],\\
\Sc\Hc_{\scrA}^>[\leqslant\!\xb]
\big/h\,\Sc\Hc_{\scrA}^>[\leqslant\!\xb]
\subset\Sc\Hc^>=
\bigcup_\xb
\Sc\Hc_{\scrA}^>[\leqslant\!\xb]\big/h\,\Sc\Hc_{\scrA}^>[\leqslant\!\xb].
\endgathered
\end{equation}
We define
$\Sc\Hc^>_{n,\scrA}[\leqslant\!\xb]$ in an identical fashion.
From Corollary \ref{cor:degen2}, we get
\begin{equation}
\label{inclusion32}
\Sc\Hc_{n,\scrA}^>[\leqslant\!\xb]
\big/h\,\Sc\Hc_{n,\scrA}^>[\leqslant\!\xb]
\subset\SH_n^>.
\end{equation}
Next, given $\xb$, for $n$ large enough the map 
$\Psi_n$ in \eqref{Psinbis} yields an isomorphism
\begin{equation}
\Sc\Hc_{\scrA}^>[\leqslant\!\xb]\to
\Sc\Hc_{n,\scrA}^>[\leqslant\!\xb].
\end{equation}
Thus it factors to an isomorphism 
\begin{equation}
\label{isom32}
\Sc\Hc_{\scrA}^>[\leqslant\!\xb]
\big/h\,\Sc\Hc_{\scrA}^>[\leqslant\!\xb]
\to\Sc\Hc_{n,\scrA}^>[\leqslant\!\xb]
\big/h\,\Sc\Hc_{n,\scrA}^>[\leqslant\!\xb].
\end{equation}
Composing \eqref{isom32} with \eqref{inclusion32} we obtain an inclusion,
for $n$ large enough,
\begin{equation}
\Sc\Hc_{\scrA}^>[\leqslant\!\xb]
\big/h\,\Sc\Hc_{\scrA}^>[\leqslant\!\xb]
\subset\SH_n^>.
\end{equation}
We conclude that $\phi'$ is injective.
Hence $\phi$ is also injective.
\end{proof}

\vspace{.1in}

Let $\Sc\Hc^{0,\cb}_{\scrA}$ be the closed 
$\scrA^\cb$-subalgebra of $\Sc\Hc^\cb_{\scrA}$ topologically generated by 
$\{Q_{0,l}\,;\,l\geqslant 1\}$. By Lemma \ref{7.4:prop}$(b)$
we have an isomorphism 
\begin{equation}
\label{isom33}
\Sc\Hc^{0,\cb}_{\scrA}=F^\cb[Q_{0,l+1}\,;\,l\geqslant 0][[h]].
\end{equation}
As above, we abbreviate
$
\Sc\Hc^{0,\cb}=
\Sc\Hc^{0,\cb}_{\scrA}/h\Sc\Hc^{0,\cb}_{\scrA}.
$
We can now prove the following theorem.

\vspace{.1in}

\begin{theo} \label{7.7:thm1}
(a)
There is an $F$-algebra isomorphism
$\phi:\Sc\Hc^\cb\to\SH^\cb$
such that
$$\phi(Q_{0,l})=D_{0,l}, \qquad \phi(Q_{\pm l,0})=D_{\pm l,0},\qquad
\phi(Q_{\pm 1,l})=D_{\pm 1,l},
\quad l\geqslant 1.$$

(b) The algebra $\Sc\Hc_{\scrA}^\cb$ is topologically generated 
by $Q_{-1,0},$ $Q_{1,0}$ and $Q_{0,2}$.
\end{theo}

\vspace{.05in}

\begin{proof}
By Proposition \ref{prop:generators/involution2}
the $F$-algebra $\SH^{\cb}$ is generated by the elements 
$\cb_l$, $D_{\pm 1,0}$ and $D_{0,2}$. 
Thus, part $(b)$ follows from $(a)$.
Now, we prove $(a)$.
We'll identify the ring 
\begin{equation}
\Am_r=\Z[q^{\pm 1}, t^{\pm 1},\chi_1^{\pm 1},\dots,\chi_r^{\pm 1}]
\end{equation}
with the Grothendieck ring of the group $\widetilde D$ as in \eqref{3.3}. 
Let $\Lm^{(r)}_K$ be the \emph{localized Grothendieck group} of the category of
$\widetilde D$-equivariant coherent sheaves 
on $\bigsqcup_{n\geqslant 0}M_{r,n}.$ 
The word localized means that the ring of scalars is extended from 
the ring $\Am_r$ to the field
\begin{equation}
\Km_r=\Km(\chi_1,\dots,\chi_r).
\end{equation}
The set of fixed points $\{I_\lambda\}$ of $M_{r,n}$ for the 
$\widetilde D$-action gives bases in $\mathbf L^{(r)}_K$ 
and $\mathbb L^{(r)}_K$. 
Set 
\begin{equation}
\scrK_r=K_r((h)),\qquad\scrA_r=K_r[[h]].
\end{equation}
We have an  embedding $\Km_r\subset \scrK_r$ given by the following formulas,
compare \eqref{7.7:deg},
\begin{equation}
q=\exp(h),\qquad t=\exp(-\kappa h),\qquad \chi_\a=
\exp(\eps_\a h),\qquad
\kappa=-y/x,\qquad \eps_\a=e_\a/x.
\end{equation}
Identifying the bases above, we get inclusions of
$\Lb^{(r)}_K=\bigoplus_\lambda K_r\,[I_\lambda]$
and
$\mathbb L^{(r)}_K=\bigoplus_\lambda \Km_r\,[I_\lambda]$
into the $\scrK_r$-vector space
\begin{equation}
\Lc^{(r)}=\bigoplus_\lambda \scrK_r\,[I_\lambda].
\end{equation}

Now, a representation of $\Ecb\otimes_\Km\Km_r$ in $\Lm^{(r)}_K$ 
is constructed in \cite[sec.~8]{SV2}.
It can be upgraded to a representation of 
$\mathbb{S}\mathbb{H}^{\cm} \otimes_\Km \scrK_r$
on $\mathcal{L}^{(r)}$ in which $u_{0,0}$ acts as the grading operator.
We have
\begin{equation}\label{formpasdeg}
\gathered
\cm_0=v^{-r/2},\qquad 
\cm_{\pm l}=\pm p_{l}(\chi_\a^{\mp 1}),\qquad
u_{0,\pm l}^\cm=sgn(\pm l)\,\f_{0,\pm l},\\
u_{1,l}=v^{-1}(q-1)^rx^{1-r}\f_{1, l-r}, \qquad 
u_{-1,l}=(-1)^{r-1}\det(W)\,\cm_0^{-1}(q-1)^{-r}x^{r-1}\f_{-1, l},\\
\fb_{1,l}[I_\lambda]=\sum_{\lambda\subset\pi}\tau_{\lambda,\pi}^l
\Lambda(N_{\lambda,\pi}^*-T_\pi^*)\,[I_\pi],\qquad
\fb_{-1,l}[I_\lambda]=\sum_{\sigma\subset\lambda}\tau_{\sigma,\lambda}^l
\Lambda(N_{\sigma,\lambda}^*-T_\sigma^*)\,[I_\sigma],\\
\fb_{0,l}[I_\lambda]=
\sum_{\a,s}\chi_\a^{-l} q^{lx(s)}t^{ly(s)}[I_\lambda].
\endgathered
\end{equation}
Here $l\geqslant 0$ and $\Lambda$ is the Koszul complex and 
$\det(W)=(\chi_1\chi_2\dots\chi_r)^{-1}$. 
On the other hand, the representation $\rho^{(r)}$ is given by the following
formulas, see Section \ref{sec:Ur} and Appendix \ref{app:B}, 
\begin{equation}
\label{formdeg}
\gathered
\cb_l=p_l(\eps_\a),\\
D_{1,l}=x^{1-l}yf_{1, l}, \qquad 
D_{-1,l}=(-1)^{r-1}x^{-l}f_{-1, l},\qquad 
D_{0,l+1}= x^{-l}f_{0,l},\\
f_{1,l}[I_\lambda]=\sum_{\lambda\subset\pi}\text c_1(\tau_{\lambda,\pi})^l
\eu(N_{\lambda,\pi}^*-T_\pi^*)\,[I_\pi],\\
f_{-1,l}[I_\lambda]=
\sum_{\sigma\subset\lambda}\text c_1(\tau_{\sigma,\lambda})^l
\eu(N_{\sigma,\lambda}^*-T_\sigma^*)\,[I_\sigma],\\
f_{0,l}[I_\lambda]=\sum_{\a,s}
\cc_1(\chi_\a^{-1}q^{x(s)}t^{y(s)})^{l}\, [I_\lambda].
\endgathered
\end{equation}

The above formulas allow us to compare the action of 
$Q_{\pm 1,l},$ $Q_{0,l}$ and of
$D_{\pm 1,l}$, $D_{0,l}$. Write 
\begin{equation}
\Oc(h^l)=\bigoplus_\lambda h^l\scrA_r\,[I_\lambda],\qquad l\in\Z.
\end{equation}
Using \eqref{Q0l}, \eqref{formpasdeg} and \eqref{formdeg}, we get
\begin{equation}
\label{7.7:Q0}
Q_{0,l}[I_\lambda]=D_{0,l}[I_\lambda]+\Oc(h),
\qquad l\geqslant 1.
\end{equation}
Next, for $\sigma\subset\lambda\subset\pi$ 
such that $|\lambda|=|\sigma|+1=|\pi|-1$ we have
\begin{equation}
\dim(N_{\lambda,\pi}-T_\pi)=-r-1,\qquad \dim(N_{\sigma,\lambda}-T_\sigma)=r-1.
\end{equation}
Therefore, we have the following estimates in $\scrK_r$ 
\begin{equation}\label{estimate}
\gathered
\Lambda(N_{\lambda,\pi}^*-T_\pi^*)\equiv (x/h)^{1+r}\eu(N_{\lambda,\pi}^*-T_\pi^*),
\qquad
\Lambda(N_{\sigma,\lambda}^*-T_\sigma^*)\equiv
(x/h)^{1-r}\eu(N_{\sigma,\lambda}^*-T_\sigma^*)
\endgathered
\end{equation}
modulo lower terms for the $h$-adic topology.
Finally,  \eqref{Q0l}, \eqref{formpasdeg}, \eqref{formdeg} 
and \eqref{estimate} give
\begin{equation}
\label{7.7:Q+/-}
\gathered
Q_{1,0}[I_\lambda]=D_{1,0}[I_\lambda]+\Oc(h),
\qquad
Q_{-1,0}[I_\lambda]=D_{-1,0}[I_\lambda]+\Oc(h).
\endgathered
\end{equation}
By \eqref{E:rel1bis2}, \eqref{E:rel2bis2} and  \eqref{Q0l} we have
\begin{equation}
\gathered
D_{1,l}=[D_{0,l+1},D_{1,0}],\qquad
D_{-1,l}=-[D_{0,l+1},D_{-1,0}],\\
Q_{1,l}=[Q_{0,l+1},Q_{1,0}],\qquad
Q_{-1,l}=-[Q_{0,l+1},Q_{-1,0}],
\endgathered
\end{equation}
Thus \eqref{7.7:Q0} and \eqref{7.7:Q+/-} imply that
\begin{equation}
\label{7.7:Q+/-2}
\gathered
Q_{1,l}[I_\lambda]=D_{1,l}[I_\lambda]+\Oc(h),
\qquad
Q_{-1,l}[I_\lambda]=D_{-1,l}[I_\lambda]+\Oc(h).
\endgathered
\end{equation}
Now, the algebra homomorphism $F^\cb\to K_r$ in Definition \ref{def:spe} yields an 
algebra homomorphism $\scrA^\cb\to\scrA_r$. 
Consider the algebras
\begin{equation}
\Sc\Hc_{\scrA}^{(r)}=\Sc\Hc^\cb_{\scrA} \otimes_{\scrA^\cb} \scrA_r,\qquad
\Sc\Hc^{(r)}=\Sc\Hc^{(r)}_{\scrA} / h \Sc\Hc^{(r)}_{\scrA}.
\end{equation}
Note that the composed map $\mathbb A^{\mathbb c}\to\scrA^\cb\to K_r$ is given by
\begin{equation}
\cm_0=v^{-r/2},\qquad\cm_{\pm l}=\pm p_{l}(\chi_\a^{\mp 1}).
\end{equation}
We define $\Sc\Hc^{(r),>}$, $\Sc\Hc^{(r),<}$ and
$\Sc\Hc^{(r),0}$ in the same way,
using $\Sc\Hc^>_\scrA$,
$\Sc\Hc^<_\scrA$ and
$\Sc\Hc^{0,\cb}_\scrA$.
Formulas \eqref{7.7:Q0} and \eqref{7.7:Q+/-2} imply that the $\Am^\cm$-subalgebra 
$\mathbb B^{\mathbb c}\subset\Sm\Hm^\cm$ preserves the lattice $\Oc(1)$.
This yields a representation of $\Sc\Hc_\scrA^\cb$ on
$\Lc^{(r)}$ which preserves also
$\Oc(1)$ and which factors to a representation of
$\Sc\Hc^{(r)}$ on  $\Oc(1)/\Oc(h)=\Lb^{(r)}_K$.
Since  $\rho^{(r)}$ is faithful, 
this yields also an  algebra homomorphism
\begin{equation}\label{map666}\Sc\Hc^{(r)}\to\SH^{(r)}_K.\end{equation}
It is surjective, because 
$\SH^{(r)}_K$ is generated by the elements $D_{0,l+1}$, $D_{-1,l}$ and
$D_{1,l}$ with $l\geqslant 0$.

Now, the $\scrA$-algebra embeddings of
$\Sc\Hc_{\scrA}^>,$ $ \Sc\Hc_{\scrA}^>$ and 
$\Sc\Hc^{0,\cb}_\scrA$ into $\Sc\Hc_{\scrA}^\cb$ give obvious maps 
\begin{equation}
\Sc\Hc^{(r),>},\, 
\Sc\Hc^{(r),0},\,
\Sc\Hc^{(r),<}\to
\Sc\Hc^{(r)}.
\end{equation}
Composing them with \eqref{map666} we get $K_r$-algebra homomorphisms
\begin{equation}
\label{7.7:degen1}
\Sc\Hc^{(r),>}\to\SH^{(r),>}_K,\qquad
\Sc\Hc^{(r),<}\to\SH^{(r),<}_K,\qquad
\Sc\Hc^{(r),0}\to\SH^{(r),0}_K,
\end{equation}
which give the commutative square
\begin{equation}\label{diagram}
\begin{split}
\xymatrix{
\Sc\Hc^{(r),>}\,\otimes_{K_r}\,
\Sc\Hc^{(r),0}\,\otimes_{K_r}\,
\Sc\Hc^{(r),<}\ar[r]^-m\ar[d]&
\Sc\Hc^{(r)}\ar[d]\\
\SH^{(r),>}_K\,\otimes_{K_r}\,
\SH^{(r),0}_K\,\otimes_{K_r}\,
\SH^{(r),<}_K\ar[r]^-m&\SH^{(r)}_K.
}
\end{split}
\end{equation}
Here $m$ is the multiplication map.

Now, by Proposition \ref{7.4:prop} there are $K_r$-algebra isomorphisms
\begin{equation}
\label{7.7:degen2}
\Sc\Hc^{(r),>}\to\SH^{(r),>}_K,\qquad
\Sc\Hc^{(r),<}\to\SH^{(r),<}_K,\qquad
Q_{\pm l,0}\mapsto D_{\pm l,0},\qquad
Q_{\pm 1,l}\mapsto D_{\pm 1,l}.
\end{equation}
Further, by \eqref{1.68} and \eqref{isom33}, we have a $K_r$-algebra isomorphism
\begin{equation}
\label{7.7:degen3}
\Sc\Hc^{(r),0}
\to \SH^{(r),0}_K,\qquad 
Q_{0,l+1}\mapsto D_{0,l+1}.
\end{equation}
Thus the left vertical map in \eqref{diagram} is invertible.
The bottom horizontal map is invertible by Proposition \ref{1.8:prop1}. 
Thus the upper map $m$ is injective. 
Therefore, to prove that the right map is invertible it is enough to check the following.

\begin{lem}
\label{lem:surjectivity22}
The multiplication gives a surjective map 
$$m:\Sc\Hc^{(r),>}\,\otimes_{K_r}\, \Sc\Hc^{(r),0}\,\otimes_{K_r}\, 
\Sc\Hc^{(r),<}\to\Sc\Hc^{(r)}.$$ 
\end{lem}

\begin{proof}
It is enough to prove that
\begin{equation}
[Q_{0,l}, Q_{1, k}] \in \Sc\Hc^{(r),>}, \qquad 
[Q_{-1,l},Q_{0,k}] \in \Sc\Hc^{(r),<}, \qquad 
[Q_{-1,l}, Q_{1,k}] \in \Sc\Hc^{(r),0}.
\end{equation}
The first two relations follow from a simple computation, since \eqref{7.2:eq} implies that
\begin{equation}
[Q_{0,l}, Q_{1, k}]=Q_{1,l+k-1},\qquad
[Q_{-1,l},Q_{0,k}]=Q_{-1,l+k-1}.
\end{equation} 
For the third one, we must check that 
$[Q_{-1,l},Q_{1,k}]$ 
belongs to
$\Sc\Hc_{\scrA}^{0,\cm}.$
First, by \eqref{7.2:eq} we have 
\begin{equation}
[Q_{-1,l},Q_{1,k}] \in 
\Sc\Hc_{\scrA}^{0,\cb}\otimes_\scrA\scrK.
\end{equation}
Next, one can check that 
$[Q_{-1,l},Q_{1,k}]$ lies indeed in $\Sc\Hc^{0,\cm}_{\scrA}$
by looking at its image by $\rho^{(r)}$. The details are left to the reader.

\end{proof}

We have proved that the assignment 
$Q_{0,l} \mapsto D_{0,l},$ $Q_{\pm l, 0} \mapsto D_{\pm l, 0}$ extends to
an isomorphism of $K_r$-algebras $\Sc\Hc^{(r)} = \SH^{(r)}_K$
for any $r$. The theorem follows.
\end{proof}

\vspace{.1in}

\vspace{.15in}

\subsection{The coproduct of $\SH^{\cb}$}
\label{sec:coproduct2}
The $F$-algebra $\SH^\cb$ carries 
a $\Z$-grading 
$\SH^\cb=\bigoplus_{s\in\Z}\SH^\cb[s].$
We consider the topological tensor product
$\SH^\cb\,\widehat\otimes\,\SH^\cb$ over $F$ defined by
\begin{equation}
\gathered
\SH^\cb\,\widehat\otimes\,\SH^\cb=\bigoplus_{s\in\Z}
\varprojlim_N\Big(\bigoplus_{t\in\Z}
\big(\SH^\cb[s-t]\otimes\SH^\cb[t]\big)\Big)\Big/\mathscr I_N[s],\\
\mathscr I_N[s]=
\bigoplus_{t\geqslant N}\big(\SH^\cb[s-t]\otimes\SH^\cb[t]\big).
\endgathered
\end{equation}
We can now prove the following.

\begin{theo}
\label{7.7:thm2} 
(a) The map ${}^\sigma\!\Delta$ factors to a $F$-algebra homomorphism
$\pmb\Delta: \SH^\cb\to\SH^\cb\,\widehat\otimes\,\SH^\cb$ which
is uniquely determined by the following formulas

\begin{itemize}
\item $\pmb\Delta(\cb_l)= \delta(\cb_l)$ for $l\geqslant 0$,

\item $\pmb\Delta(D_{l,0})=\delta(D_{l,0})$ for $l\neq 0$,

\item $\pmb\Delta(D_{0,1})=\delta(D_{0,1})$,

\item $\pmb\Delta(D_{0,2})=\delta(D_{0,2})
+\xi\sum_{l\geqslant 1}l\kappa^{1-l}D_{-l,0}\otimes D_{l,0},$

\item
$\pmb\Delta(D_{1,1})=\delta(D_{1,1})+\xi\cb_0\otimes D_{1,0}$ and
$\pmb\Delta(D_{-1,1})=\delta(D_{-1,1})+\xi D_{-1,0}\otimes \cb_0$.
\end{itemize}

(b) The algebra homomorphism $\varepsilon:\SH^\cb\to F$ 
in Remark \ref{rem1.42} is a counit for
$\pmb\Delta$.
\end{theo}

\vspace{.05in}

For $l\in\Z$ we abbreviate
$\Oc(h^l)=h^l\Sc\Hc_{\scrA}.$
First, let us quote the following formulas.

\vspace{.05in}

\begin{lem} \label{7.7:lem-form} The following hold

(a) $\alpha_{l}=\kappa\xi l^2h^3+\Oc(h^4)$,

(b) 
$\theta_{l,0}=\alpha_lu_{l,0} + \Oc(h^3)= 
\kappa \xi | l | h^2P_{l,0}+\Oc(h^3)$ for $l\neq 0$,

(c) $P_{l,1} =P_{l,0}+\Oc(h)$ for $l \neq 0$.

%(d) $P_{0,1}=h^{-1}\cb_0/\kappa+(\cb_0/2-\cb_1\kappa)+
%h (Q_{0,1} + \cb_0\kappa/12-\cb_1/2+\cb_2/2\kappa) + \Oc(h^2)$,

%(e) $P_{0,2}=2h Q_{0,1} + h^2\big( 2Q_{0,2} + 2\xi Q_{0,1} + \Oc(h^3),$

%(d) $u^{\cm}_{0,l}=Q_{0,1} + \Oc(h)$ for $l \geqslant 1$.
\end{lem}

\begin{proof}
Part $(a)$ follows from \eqref{7.2:eq4}.
Note that $P_{l,0}\in\Oc(1)$ by definition of $\Sc\Hc_\scrA^\cb$.
Thus, $(b)$ follows from \eqref{7.2:eq4}, which gives the following 
formulas for $l\geqslant 1$
\begin{equation}
P_{\pm l,0}=(q^l-1)\,u_{\pm l,0},\qquad
\sum_{l\geqslant 0} \theta_{l,0}\, s^{l}=
\exp\big(\sum_{l\geqslant 1} \alpha_{l}\, u_{l,0}\,s^l\big).
\end{equation}
Finally, for $l\geqslant 1$, using
\eqref{sigma}, \eqref{7.2:eq} we get
\begin{equation}
u_{\pm l,1}=\pm[u_{0,1},u_{\pm l,0}].
\end{equation}
From \eqref{Q0l} we get also
\begin{equation}
u_{0,1}=(q-1)Q_{0,2}+u_{0,0}-\cm_1/(q-1)(t-1).
\end{equation}
Thus, part $(c)$ follows from the following computation 
\begin{equation}
\aligned
P_{\pm l,1}&=(q-1)u_{\pm l,1}\\
&=\pm (q-1)[u_{0,1},P_{\pm l,0}]/(q^l-1)\\
&=\pm (q-1)^2[Q_{0,2},P_{\pm l,0}]/(q^l-1)+
l(q-1)P_{\pm l,0}/(q^l-1)\\
&=P_{\pm l,0}+\Oc(h).
\endaligned
\end{equation}
\end{proof}

We can now turn to the proof of the theorem.

\begin{proof} 
We must prove that ${}^\sigma\!\Delta$ preserves the lattice
$\Sc\Hc_\scrA$ and we must compute the image of the elements
$Q_{l,0}$, $Q_{0,1}$ and $Q_{0,2}$.
By \eqref{formcoprod} we have
${}^\sigma\!\Delta(P_{l,0})=\delta(P_{l,0})$  for all $l\in\Z$.
Thus, we have also $\pmb\Delta (D_{l,0})=\delta(D_{l,0})$.
Next, using \eqref{formcoprod} and \eqref{Q0l}, we get
$$Q_{0,1}=u_{0,0},\qquad {}^{\sigma}\!\Delta(Q_{0,1})=\delta(Q_{0,1}) .$$  
This implies that $\pmb\Delta(D_{0,1})=\delta(D_{0,1})$.
Finally, using \eqref{formcoprod} and \eqref{Q0l} again, we get
\begin{equation}
\gathered
Q_{0,2}=(q-1)^{-1}(u_{0,1}^\cm-u_{0,0}),\\
{}^{\sigma}\!\Delta(Q_{0,2})=\delta(Q_{0,2})+(q-1)^{-1}\sum_{k\geqslant 1}\theta_{-k,0}\otimes u_{k,1}.
\endgathered
\end{equation}
Thus, by Lemma~\ref{7.7:lem-form} we have
\begin{equation}\label{F:kip1}
%{}^{\sigma}\!\Delta(Q_{0,1})=\xi h^{-1}\cb_0\otimes\cb_0/2\kappa+\Oc(1).
{}^{\sigma}\!\Delta(Q_{0,2})=\delta(Q_{0,2}) + 
\kappa h\xi\sum_{k \geqslant 1} kP_{-k,0} \otimes P_{k,0} + \Oc(h^2).
\end{equation}
This implies that 
\begin{equation}
\pmb\Delta(D_{0,2})=\delta(D_{0,2})
+\xi\sum_{l\geqslant 1}l\kappa^{1-l}D_{-l,0}\otimes D_{l,0}.
\end{equation}

\end{proof}

For a future use, let us mention the following fact.
For $l \geqslant 0$ we put
\begin{equation}
\SH^-[\leqslant \!-l]=\bigoplus_{s \geqslant l} \SH^-[-s], \qquad 
\SH^+[\geqslant\! l]=\bigoplus_{s \geqslant l} \SH^+[s]
\end{equation}
where the grading is the rank mentioned above.

\begin{lem}\label{lem:coproduct-D0n} 
For $l \geqslant 1$ we have
$\pmb\Delta(D_{0,l}) = \delta(D_{0,l})$ modulo
$\SH^-[\leqslant\! -1]\, \widehat{\otimes}\, \SH^+[\geqslant\! 1].$
\end{lem}

\begin{proof} 
A simple computation shows that, modulo 
$\SH^-[\leqslant\! -1]\, \widehat{\otimes}\, \SH^+[\geqslant\! 2]$, we have
\begin{equation}
\begin{split}
\pmb\Delta(D_{1,l})&= \pmb\Delta(\ad(D_{0,2})^l (D_{1,0}))\\
&=\ad \big(\delta(D_{0,2}) + 
\xi \sum_{l \geqslant 1} l\kappa^{1-l}D_{-l,0} 
\otimes D_{l,0}\big)^l \big(\delta(D_{0,1})\big)\\
&=\delta(D_{1,l}) +\xi \sum_{k=1}^l 
\ad\big(\delta(D_{0,2})\big)^{k-1} \circ 
\ad(D_{-1,0} \otimes D_{1,0}) \circ 
\ad \big(\delta(D_{0,2})\big)^{l-k} \big(\delta(D_{1,0})\big)\\
&=\delta(D_{1,l}) + \xi \sum_{k=1}^l E_{l-k} \otimes D_{1,k-1}.
\end{split}
\end{equation}
Applying the commutator with $\pmb\Delta(D_{-1,0})$, 
we get, modulo 
$\SH^-[\leqslant\! -1]\, \widehat{\otimes}\, \SH^+[\geqslant\! 1]$,
\begin{equation}\label{E:Delta-En}
\pmb\Delta(E_l)=\delta(E_l) + \xi \sum_{k=1}^l E_{l-k} \otimes E_{k-1}.
\end{equation}
It follows in particular that 
\begin{equation}
\pmb\Delta(D_{0,l}) \in \SH^{\cb,0} \,\otimes\, \SH^{\cb,0} + 
\SH^-[\leqslant\! -1] \,\widehat{\otimes}\, \SH^+[\geqslant\! 1].
\end{equation}
Using 
\eqref{leadingterm},
modulo the ideal 
$\SH^-[\leqslant\! -1] \,\widehat{\otimes}\, \SH^+[\geqslant\! 1]$, 
we deduce from
\eqref{E:Delta-En} the desired estimate on $\pmb\Delta(D_{0,l})$.
\end{proof}

\vspace{.2in}

\section{Relation to $W_k(\gen\len_r)$}

\vspace{.15in}

\subsection{The vertex algebra $W_k(\gen\len_r)$}
\label{sec:vertex}
Fix an integer $r>0$, a field $\kb$ containing $\C$ and an element $k\in \kb$. Let
$W_k(\sen\len_r)_\kb$ be the \emph{$W$-algebra over $\kb$ at level $k$ associated 
with $\mathfrak{s}\mathfrak{l}_r$}.
We may abbreviate
$W_k(\sen\len_r)=W_k(\sen\len_r)_\kb.$
Recall that $W_k(\sen\len_r)$ is a $\Z$-graded vertex 
algebra with quasi-primary vectors 
$\PW_2,\PW_3\dots,\PW_r$ of 
\emph{conformal weight} $2,3,\dots r$. 
The corresponding fields are 
\begin{equation}
\PW_i(z)=\sum_{l\in\Z}\PW_{i,l}\,z^{-l-i},\qquad 
\PW_{i,l}\in\End(W_k(\sen\len_r)).
\end{equation}
Let $|0\rangle$ be the vacuum of $W_k(\gen\len_r)$.
Recall that $|0\rangle$ has the degree zero, and that 
$\PW_{i,l}$ is an operator of degree $-l$.
We abbreviate $\PW_{i,(l)}=\PW_{i,l-i+1},$ so that we have
$\PW_i=\PW_{i,(-1)}|0\rangle.$ 
Then $W_k(\gen\len_r)$ is spanned, as a $\kb$-vector space, by the elements
\begin{equation}\label{8.1:mon}
\PW_{i_1,(-l_1)}\PW_{i_2,(-l_2)}\dots\PW_{i_t,(-l_t)}|0\rangle,\qquad
l_i\geqslant 1,\qquad t\geqslant 0.
\end{equation}
The vertex algebra $W_k(\sen\len_r)$ admits a \emph{strict filtration},
in the sense of \cite[sec.~3.4, 3.8]{Ara}, such that the subspace
$W_k(\sen\len_r)[\leqslant\!d]$ is spanned by the elements \eqref{8.1:mon} with 
$i_1, i_2, \dots, i_t\geqslant 2$ and
\begin{equation}\label{8.1:filt}
i_1+i_2+\dots+i_t\leqslant d+t.
\end{equation}
We'll call it the \emph{order filtration}. 
This filtration differs from the \emph{standard filtration}
on any conformal vertex algebra \cite[sec.~3.5, rem.~4.11.3]{Ara}.
The associated graded $\overline W_k(\sen\len_r)$ 
is a commutative vertex algebra. 
Let $\overline W_{i,l}$ denote the symbol of $\PW_{i,l}$
in $\End(\overline W_k(\sen\len_r)).$
The vectors $\PW_2,\dots,\PW_r$
\emph{generate a PBW-basis} of
$W_k(\sen\len_r)$, see \cite[sec.~3.6, prop.~4.12.1]{Ara}.
This means that the map
\begin{equation}
\kb[w_{i,(-l)}\,;\,i\in[2,r],\,l\geqslant 1]\to
\overline W_k(\sen\len_r),\qquad
f(w_{i,(-l)})\mapsto f(\overline W_{i,(-l)})|0\rangle
\end{equation}
is invertible. 
Let $W_k(\gen\len_r)$ be the \emph{$W$-algebra over $\kb$ at level $k$ 
associated with $\frak g\frak l_r$}.
It is the tensor product of $W_k(\sen\len_r)$ 
with the vertex algebra associated with
a free bosonic field of conformal weight 1
\begin{equation}
\PW_1(z)=\sum_{l\in\Z}\PW_{1,l}\,z^{-l-1}.
\end{equation}
The results above generalize immediately to
$W_k(\gen\len_r)$.
In particular $W_k(\gen\len_r)$ 
admits a strict filtration such that the subspace
$W_k(\gen\len_r)[\leqslant\!d]$ is spanned by the elements \eqref{8.1:mon} with 
$i_1, i_2, \dots, i_t\geqslant 1$ as in \eqref{8.1:filt}.
Finally, recall that $\PW_2$ is a conformal vector of central charge 
\begin{equation}\label{Ck}
C_k=(r-1)-r(r^2-1)(k+r-1)^2/(k+r).
\end{equation}
In other words, the Fourier modes of the field 
$\PW_2(z)$ 
satisfy the relations
\begin{equation}
[\PW_{2,l},\PW_{2,k}]=(l-k)\PW_{2,l+k}+(l^3-l)\,\delta_{l,-k}\,C_k/12.
\end{equation}

\vspace{.15in}

\subsection{The current algebra of $W_k(\gen\len_r)$}
\label{sec:8.2}
Let $\Uen(W_k(\gen\len_r))$ be the \emph{current algebra} of $W_k(\gen\len_r)$,
see \cite[sec.~3.11]{Ara}.
It is a \emph{degreewise complete topological $\kb$-algebra}.
This means that it is a $\Z$-graded algebra
\begin{equation}
\Uen(W_k(\gen\len_r))=\bigoplus_{s\in\Z}\Uen(W_k(\gen\len_r))[s]
\end{equation}
which is equipped 
with a \emph{degreewise linear topology}
such that the multiplication
\begin{equation}
\Uen(W_k(\gen\len_r))[s]\times\Uen(W_k(\gen\len_r))[s']
\to\Uen(W_k(\gen\len_r))[s+s'] 
\end{equation}
is continuous, and that each piece 
$\Uen(W_k(\gen\len_r))[s]$  is complete.
We call the degree with respect to this grading the
\emph{conformal degree}, and we call
this degreewise linear topology the
\emph{standard degreewise topology}. 
See \cite[sec.~1]{MNT} and \cite[sec.~A.2]{Ara} for the terminology.
Next, the algebra $\Uen(W_k(\gen\len_r))$  
is equipped with a degreewise dense family of elements
\cite[sec.~3.9, prop.~3.11.1]{Ara}
\begin{equation}
\{v_{\{n\}}\,;\,v\in W(\gen\len_r),\,n\in\Z\}.
\end{equation} 
We'll abbreviate
$\PW_{i,l}=(\PW_i)_{\{l+i-1\}}.$ 
Thus $\PW_{i,l}$ may be viewed both as a linear operator on
$W_k(\gen\len_r)$ and as an element of $\Uen(W_k(\gen\len_r))$
of conformal degree $-l$.
We hope that this will not create any confusion.
Note that the elements
\begin{equation}
\PW_{i_1,l_1}\PW_{i_2,l_2}\dots\PW_{i_t,l_t}
\end{equation}
with 
$i_1,i_2,\dots,i_t\geqslant 1$ and $l_1+l_2+\dots+l_t=s$
span a dense subset of
$\Uen(W_k(\gen\len_r))[s]$.
Now, the order filtration on $W_k(\gen\len_r)$ induces a filtration on
$\Uen(W_k(\gen\len_r)),$ called again the \emph{order filtration}.
The element $\PW_{i,l}$ has the order $i-1$.
Let $\overline W_{i,l}$ denote its symbol in the piece \cite[thm.~3.13.3]{Ara}
\begin{equation}
\overline\Uen( W_k(\gen\len_r))[i]=
\Uen( W_k(\gen\len_r))[\leqslant\! i]/\Uen( W_k(\gen\len_r))[<\!i].
\end{equation}
The conformal weight yields a $\Z$-grading on
$\overline\Uen( W_k(\gen\len_r))[i]$ such that
$\overline W_{i,l}$ has the (conformal) degree $-l$.
Note that $\overline\Uen(W_k(\gen\len_r))$ is also a
degreewise complete topological  $\kb$-algebra. It is isomorphic to
the \emph{standard degreewise completion} of the algebra
\begin{equation}
\kb[w_{i,l}\,;\,i\in[1,r],\,l\in\Z]
\end{equation} as a degreewise topological  $\kb$-vector space.
Here $w_{i,l}$ is given the degree $-l$.

\vspace{.15in}

\subsection{The $W_k(\gen\len_r)$-modules}

\begin{df}
We define a $W_k(\gen\len_r)$-module to be a
$\Uen(W_k(\gen\len_r))$-module. 
A $W_k(\gen\len_r)$-module is \emph{admissible} if it is a 
$\Z$-graded $\Uen(W_k(\gen\len_r))$-module $M=\bigoplus_{s\in\Z}M[s]$ such that
$M[s]=0$ for $s\gg 0$. 
\end{df}
If $M$ is an admissible $W_k(\gen\len_r)$-module the action
\begin{equation}
\Uen(W_k(\gen\len_r))[s]\times M[s']\to M[s+s']
\end{equation}
is continuous with respect to the topology on
$\Uen(W_k(\gen\len_r))[s]$ and the discrete topology on $M$.

Now, let $\hen$ be the Cartan subalgebra of $\gen\len_r$.
For $\beta\in\hen$, the \emph{Verma module
with the highest weight}
$\beta$ is an admissible module  $M_\beta$ with basis elements
\begin{equation}\label{8.1:verma}
\PW_{i_1,-l_1}\PW_{i_2,-l_2}\dots\PW_{i_t,-l_t}|\beta\rangle,\qquad
l_i\geqslant 1,\qquad t\geqslant 0.
\end{equation}
Here $|\beta\rangle$ is the \emph{highest weight vector},
see \cite[sec.~5.1]{Ara}.
We have the following relations
\begin{equation}\label{8.1:verma2}
\PW_{i,0}|\beta\rangle=e_i(\beta)|\beta\rangle,
\qquad
\PW_{i,l}|\beta\rangle=0,
\qquad l\geqslant 1,
\end{equation}
where $e_i(\beta)$ is the evaluation of the $i$-th elementary
symmetric function at $\beta$.

\vspace{.1in}

\begin{rem}
The order filtration on $W_k(\gen\len_r)$ 
induces a filtration on  $M_\beta$ such that 
$M_\beta[\leqslant\!d]$ is spanned by the elements 
\begin{equation}\label{8.1:filtr}
\PW_{i_1,-l_1}\PW_{i_2,-l_2}\dots\PW_{i_t,-l_t}|\beta\rangle,\qquad
l_i\geqslant 1,\qquad t\geqslant 0,
\end{equation}
with $i_1,i_2,\dots,i_t$ satisfying 
\eqref{8.1:filt}. 
By \cite[prop.~5.1.1]{Ara}, the associated graded is a 
$\overline\Uen(W_k(\gen\len_r))$-module 
$\overline M_\beta$. The conformal weight yields a $\Z$-grading on $\overline M_\beta$.
As a graded vector space $\overline M_\beta$ is isomorphic to 
\begin{equation}
\kb[w_{i,-l}\,;\,i\in[1,r],\,l\geqslant 1],
\end{equation}
where $w_{i,-l}$ is given the degree $l$.
%with $\overline W_{i,-l}$ acting by multiplication
%by $w_{i,-l}$ if $l\geqslant 1$.
\end{rem}

\vspace{.15in}

\subsection{The quantum Miura transform for $W_k(\gen\len_r)$}
\label{sec:miura}
Let $b_1,b_2,\dots,b_r$ be a basis of $\hen$
and let $b^{(1)}, b^{(2)},\dots,b^{(r)}$ be the dual basis.
Let $\langle \bullet,\bullet\rangle$ denote both the canonical pairing
$\hen^*\times\hen\to\kb$ and the pairing $\hen^*\times\hen^*\to\kb$
such that $(b^{(i)})$ is orthonormal. 
Fix $\kappa\in \kb^\times$ and 
fix $r$ commuting boson fields
$b^{(1)}(z),b^{(2)}(z),\dots, b^{(r)}(z)$ 
of level $\kappa^{-1}$. Thus, we have
\begin{equation}
[b_l^{(i)},b_{-h}^{(j)}]=l\delta_{i,j}\delta_{l,h}/\kappa,\qquad
b^{(i)}(z)=\sum_{l\in\Z}b^{(i)}_lz^{-l-1}.
\end{equation}
Let $\mathscr H^{(r)}$ be the Heisenberg algebra
generated by the elements $b^{(i)}_l$ with $i\in[1,r]$ and $l\in\Z$. 
For $\beta\in\hen$ let $\pi_\beta$ be the $\mathscr H^{(r)}$-module
generated by the vector $|\beta\rangle$ with the relations 
\begin{equation}
b_{l}^{(i)}|\beta\rangle=\delta_{l,0}\,\langle b^{(i)},\beta\rangle\,
|\beta\rangle,\qquad
l\geqslant 0.
\end{equation}
To avoid confusions we may write $\pi_\beta=\pi_{\beta,\kb}$.
Consider the fields
\begin{equation}\label{oper}
b(z)=\sum_ib^{(i)}(z)\,b_i,\qquad
h(z)=\sum_i\langle h,b_i\rangle\,b^{(i)}(z),\qquad h\in\hen^*.
\end{equation}
We call $\pi_0$ the \emph{Fock space}. 
It has the structure of a conformal vertex algebra such that
\begin{equation}
Y(b^{(i)}_{-1}|0\rangle,z)=b^{(i)}(z).
\end{equation}
The Virasoro field has the central charge 
$r-12\langle h,h\rangle/\kappa$ and is given by
\begin{equation}
\frac{\kappa}{2}\sum_i\pmb :\!b^{(i)}(z)^2\!\pmb :+\,\partial_z h(z).
\end{equation}
Here $:\;:$ denotes the \textit{normal ordering}
(from right to left).
The module
$\pi_\beta$ has the structure of a module over $\pi_0$.

Now, let $h^{(1)}, h^{(2)},\dots, h^{(r)}$ 
be the weights of the first fundamental 
representation of $\sen\len_r$. 
Let also $\alpha_i$, $\omega_i$ be the simple roots and 
the fundamental weights of $\sen\len_r$,
and $\rho$ be the sum of the fundamental weights.
Given $Q\in\kb$ we 
define the fields $W_1(z),W_2(z),\dots,W_r(z)$ in
$\End(\pi_0)[[z^{-1},z]]$ by the following formula
\begin{equation}
\label{miura}
-\kappa\,\pmb :\!\prod_{i=1}^r\big(Q\,\partial_z+h^{(i)}(z)\big)
\!\pmb :\,=
\sum_{d=0}^r\,W_d(z)\,(Q\,\partial_z)^{r-d}.
\end{equation}
Note that
$$\sum_{i=1}^rh^{(i)}=0,\qquad
%\langle h^{(i)},h^{(j)}\rangle =\delta_{i,j}-1/r,\qquad
-\sum_{i\neq j}h^{(i)}\otimes h^{(j)}
%=\sum_{i=1}^rh^{(i)}\otimes h^{(i)}
=\sum_{i=1}^{r-1}\alpha_i\otimes\omega_i
=\sum_{i=1}^{r}b^{(i)}\otimes b^{(i)}-\frac{1}{r}J\otimes J,
\qquad J=\sum_ib^{(i)}.$$
Therefore, we have
\begin{equation}\label{W2}
\begin{aligned}
W_0(z)&=1,\\
W_1(z)&=0,\\
W_2(z)&=
-\kappa\sum_{i<j}\pmb :\!h^{(i)}(z)\,h^{(j)}(z)\!\pmb :+\,
\kappa Q\,\partial_z\rho(z)\\
&=\frac{\kappa}{2}\sum_{i=1}^{r-1}\pmb :\!\alpha_i(z)\omega_i(z)\!\pmb : 
+ \,\kappa Q\,\partial_z\rho(z)\\
&=\frac{\kappa}{2}\sum_{i=1}^r\pmb :\!b^{(i)}(z)^2\!\pmb : 
-\,\frac{\kappa}{2r}\pmb :\!J(z)^2\!\pmb :+ \,\kappa Q\,\partial_z\rho(z).
\end{aligned}
\end{equation}
For $r\geqslant 2$ the field $W_2(z)$ 
is a Virasoro field
of central charge 
\cite[prop.~4.10]{Kac}
\begin{equation}\label{CQ}
C_Q%=(r-1)-12\,\kappa\, Q^2\,\langle\rho,\rho\rangle
=(r-1)-r(r^2-1)\kappa\,Q^2.
\end{equation}
Although this notation is not compatible with \eqref{W2}, we'll write 
\begin{equation}
W_1(z)=J(z)=\sum_{i=1}^r b^{(i)}(z). 
\end{equation}
Comparing \eqref{Ck} and \eqref{CQ} we get $C_k=C_Q$ if
\begin{equation}\label{k/kappa} 
Q=-\xi/\kappa, \qquad \kappa=k+r.\end{equation}
\textit{We'll always assume that} (\ref{k/kappa}) \textit{holds}.
%Since $\kappa$ is generic, 
The fields $W_1(z),\dots,W_r(z)$ 
generate a vertex subalgebra of $\pi_0$ which is isomorphic to
$W_k(\gen\len_r)$, see \cite[sec.~5.4.11]{BZF}.
An explicit expression of the field
$W_d(z)$ in $\End(\pi_0)[[z^{-1},z]]$ yields to complicated formulas. The following is enough for our
purpose.

\begin{prop} \label{prop:symbolW}
For $d\neq1$, modulo lower terms in the order filtration of 
$\Uen(W_k(\gen\len_1))^{\widehat\otimes r}$, 
$$W_d(z)=-\kappa \sum_{s=0}^d(-r)^{s-d}
\Big(\begin{smallmatrix}r-s\\r-d\end{smallmatrix}\Big)
\sum_{i_1<i_2<\cdots<i_s}
\pmb:J(z)^{d-s}b^{(i_1)}(z)b^{(i_2)}(z)\cdots b^{(i_s)}(z)\pmb:.$$
\end{prop}

\begin{proof}
Obvious because, modulo lower terms, we have
\begin{equation}
\aligned
W_d(z)
&\equiv-\kappa\sum_{i_1<i_2<\cdots<i_d}
\pmb :\!h^{(i_1)}(z)\,h^{(i_2)}(z)\dots h^{(i_d)}(z)\!\pmb :.
\endaligned
\end{equation}
\end{proof}

Since $\pi_\beta$ is a module over the vertex algebra $\pi_0$, 
it is also a module over
$W_k(\gen\len_r)$.
Let $\scrU(W_k(\gen\len_r))$ denote the image of
$\Uen(W_k(\gen\len_r))$ in $\End(\pi_\beta)$. 
This image may depend on the choice of $\beta$. 
We hope this will not create any confusion.
We have the following, see, e.g., \cite{Bi}.

\begin{prop}
\label{prop:pibeta} 
The representation of
$W_k(\gen\len_r)$ on $\pi_\beta$ is such that
$$\gathered
W_{d,0}|\beta\rangle=w_d(\beta)|\beta\rangle,
\qquad
W_{d,l}|\beta\rangle=0,
\qquad l\geqslant 1,\\
w_1(\beta)=\sum_{i=1}^r\langle b^{(i)},\beta\rangle,\qquad
w_d(\beta)=-\kappa\!\sum_{i_1<i_2<\cdots<i_d}\prod_{t=1}^d\Big(
\langle h^{(i_t)},\beta\rangle+(d-t)\,\xi/\kappa \Big),\qquad d\geqslant 2.
\endgathered
$$
\end{prop}

\vspace{.15in}

\subsection{The free field representation of $\SH^{(r)}_K$}
A \emph{composition} $\nu$ of $r$ is a tuple
$(\nu_1,\nu_2,\dots,\nu_d)$ of positive integers summing to $r$.
For each composition, we set
\begin{equation}
\label{8.30}
\gathered
\SH^\nu_K=\widehat{\bigotimes_{1\leqslant i\leqslant d}}\SH^{(\nu_i)}_{K_r},
\qquad
\Lb^\nu_K=\bigotimes_{1\leqslant i\leqslant d}\Lb^{(\nu_i)}_{K_r},\\
\SH^{(\nu_i)}_{K_r}=\SH^{(\nu_i)}_K\otimes_{K_{\nu_i}}\!K_r,
\qquad
\Lb^{(\nu_i)}_{K_r}=\Lb^{(\nu_i)}_K\otimes_{K_{\nu_i}}\!K_r.
\endgathered
\end{equation}
Here, the symbol $\bigotimes$ denotes the tensor product over $K_r$ and
$\widehat\bigotimes$ is the topological tensor product over $K_r$ as in 
Section \ref{sec:coproduct2}. For instance, for $d=2$, we have
\begin{equation}\gathered
\SH^\nu_K=\bigoplus_{s\in\Z}\SH^\nu_K[s],\qquad
\SH^\nu_K[s]=\varprojlim_N\Big(\bigoplus_{s_1+s_2=s}\bigotimes_{i=1,2}
\SH^{(\nu_i)}_{K_r}[s_i]\Big)\Big/\mathscr I_N[s],\\
\mathscr I_N[s]=\bigoplus_{s_2\geqslant N}\bigotimes_{i=1,2}
\SH^{(\nu_i)}_{K_r}[s_i].
\endgathered
\end{equation}
Taking only the terms in 
$\SH^{(\nu_i)}_{K_r}[s_i]$ or
$\SH^{(\nu_i)}_{K_r}[\leqslant\!l_i]$ in the definition of $\SH^\nu_K$, 
we get the subspaces
\begin{equation}
\SH^\nu_K[s_1,\dots,s_d],\qquad
\SH^\nu_K[\leqslant\!l_1,\dots,\leqslant\!l_d].
\end{equation}
For a future use, let us quote the following easy fact.

\begin{prop}
\label{8.5:prop1}
The map $\pmb\Delta^{d-1}$ factors 
to an algebra embedding
$\pmb\Delta^\nu:\SH^{(r)}_K\to\SH^\nu_K$
and
$$\gathered
\pmb\Delta^\nu\big(\SH^{(r)}_K[s]\big)\subset
\bigoplus_{s_1,\dots,s_d}\SH^\nu_K[s_1,\dots,s_d],\qquad
\pmb\Delta^\nu\big(\SH^{(r)}_K[\leqslant\!l]\big)\subset
\bigoplus_{l_1,\dots,l_d}\SH^\nu_K[\leqslant\!l_1,\dots,\leqslant\!l_d].
\endgathered$$
Here the sums run over all tuples summing
to $s$ and $l$ respectively.
\end{prop}

\begin{proof} 
Since the coproduct $\pmb\Delta$ admits a counit 
$\varepsilon$, the map $\pmb\Delta^{d-1}$
is an injection
\begin{equation}
\SH^\cb\to(\SH^\cb)^{\widehat\otimes d}.
\end{equation}
It factors to a $K_r$-algebra homomorphism
$\pmb\Delta^\nu:\SH^{(r)}_K\to\SH^\nu_K$ which is again injective.
\end{proof}

\vspace{.1in}

\begin{df} We define a representation $\rho^\nu$ of $\SH^{(r)}_K$ on
$\Lb^\nu_K$ by composing $\pmb\Delta^\nu$ 
with the representation of
$\SH^\nu_K$ on $\Lb^\nu_K$ in  \eqref{rhor}.
\end{df}
%This representation generalizes the Fock space of $\SH^{(1)}_K$.

\begin{cor}
\label{cor:rhonu}
The representation $\rho^\nu$ is faithful.
\end{cor}

\begin{proof} Use Proposition \ref{8.5:prop1} and 
Theorem \ref{thm:SH/Ur}.

\end{proof}

\vspace{.1in}

\begin{rem}
We will mostly be interested in the case $\nu=(1^r)$, where
we abbreviate $(1^r)=(1,1,\dots,1)$.
In this case,  we have
\begin{equation}
\Lb^{(1^r)}_K=(\Lb^{(1)}_{K_r})^{\otimes r}=
\Lb^{(1)}_K\otimes_K \Lb^{(1)}_{K}\otimes_K\cdots\otimes_K\Lb^{(1)}_{K},
\end{equation}
and the $K_r$-vector space structure  is given by 
\begin{equation}
\eps_i=1\otimes 1\otimes\cdots\otimes 1\otimes\eps_1\otimes 
1\otimes\cdots\otimes 1,\qquad i\in[1,r],
\end{equation}
where $\eps_1$ is at the $i$-th spot.
\end{rem}

\vspace{.15in}

\subsection{The degreewise completion of $\SH^{(r)}_K$}
We refer to \cite[sec.~1.1-1.4]{MNT} for the terminology concerning 
degreewise topological algebras.
The $K_r$-algebra $\SH^{(r)}_K$ carries 
a $\Z$-grading and an $\N$-filtration inherited from $\SH^{\cb}_K$, see Section~\ref{sec:SHc}.

\begin{df}\label{df:degreewise}
The \emph{standard degreewise topology} of $\SH^{(r)}_K$ is the 
\emph{degreewise topology}
defined by the sequence
\begin{equation}
\mathscr J_N=
\bigoplus_{s\in\Z}\mathscr J_N[s],\qquad
\mathscr J_N[s]=\sum_{t\geqslant N}\SH^{(r)}_K[t-s]\,\SH^{(r)}_K[-t].
\end{equation}
The \emph{standard degreewise completion}
of $\SH^{(r)}_K$ is the $\Z$-graded algebra given by
\begin{equation}
\gathered
\Uen(\SH^{(r)}_K)=\bigoplus_{s\in\Z}\Uen(\SH^{(r)}_K)[s],
\qquad
\Uen(\SH^{(r)}_K)[s]=\varprojlim_N\SH^{(r)}_K[s]/\mathscr J_N[s].
\endgathered
\end{equation}
The standard degreewise topologies on
$\Uen(\SH^{(r)}_K)$ is the projective limit degreewise topology.
\end{df}

The standard degreewise topologies on
$\SH^{(r)}_K$ and $\Uen(\SH^{(r)}_K)$ are linear.
They equip $\Uen(\SH^{(r)}_K)$ with the structure 
of a degreewise complete topological algebra
and the canonical map 
$\SH^{(r)}_K\to\Uen(\SH^{(r)}_K)$ 
is a morphism of degreewise topological algebras
with a \emph{degreewise dense} image.

\begin{df} A module $M$ over
$\SH^{(r)}_K$ or $\Uen(\SH^{(r)}_K)$ is \emph{admissible} if $M=\bigoplus_{s\in\Z}M[s]$
is $\Z$-graded and $M[s]=0$ for $s$ large enough. 
\end{df}

By an embedding of degreewise topological algebras we mean an 
injective morphism of degreewise topological algebras.
The following is an immediate consequence of Corollary
\ref{cor:rhonu}.

\begin{prop}\label{8.4:prop}
(a) The map $\rho^{(1^r)}$ is a faithful admissible representation 
of $\SH^{(r)}_K$ on $\Lb_K^{(1^r)}$ which extends to
an admissible representation of $\Uen(\SH^{(r)}_K)$.

(g) The canonical map $\SH^{(r)}_K\to\Uen(\SH^{(r)}_K)$ is an embedding 
of degreewise topological algebras.
\end{prop}

\vspace{.05in}

\begin{rem}
If $M$ is admissible then the actions
\begin{equation}
\SH^{(r)}_K[s]\times M[s']\to M[s+s'],
\qquad \Uen(\SH^{(r)}_K)[s]\times M[s']\to M[s+s']
\end{equation}
are continuous with respect to the standard topology on
$\SH^{(r)}_K[s]$, $\Uen(\SH^{(r)}_K)[s]$
and the discrete topology on $M[s']$, $M[s+s']$.
\end{rem}

\vspace{.05in}

\begin{rem}
\label{rem:8.11}
The order filtration on $\SH^{(r)}_K$ induces a filtration on
$\Uen(\SH^{(r)}_K)$ called again the order filtration. 
By Proposition \ref{prop:filtration-SHc} it is determined by 
putting $D_{r,d}$ in degree $d$ for any $r,d$.
\end{rem}

\vspace{.15in}

\subsection{From $\SH^{(1)}_K$ to $W_k(\gen\len_1)$}
\label{sec:gl1}
In this section we set $\kb=K_1$ and $\kappa=k+1$. 
Recall that $W_k(\gen\len_1)$ 
is the vertex algebra 
$\pi_{0}$ associated with the Heisenberg algebra $\mathscr H^{(1)}$.
We'll abbreviate 
\begin{equation}
W_1(z)=b(z)=\sum_{l \in \Z} b_l z^{-l-1}.
\end{equation}
Thus $W_k(\gen\len_1)$ is spanned, as a vector space, by elements
\begin{equation}\label{E:gl11}
 b_{-l_1} \cdots b_{-l_t} |0\rangle,\qquad
l_i\geqslant 1,\qquad t\geqslant 0.
\end{equation}

\begin{df}
The subspace $W_k(\gen\len_1)[\preccurlyeq \!d]$ of 
\emph{standard order}
at most $d$ is the span of the elements in \eqref{E:gl11} with $t \leqslant d$. 
\end{df}

The \emph{standard filtration} on $W_k(\gen\len_1)$ should not be confused with the order filtration.
The associated graded of $W_k(\gen\len_1)$ 
with respect to the standard filtration is a commutative vertex 
algebra. The current algebra
$\Uen(W_k(\gen\len_1))$ has a standard filtration as well, 
for which the elements $b_l$ are of standard order $1$.
Now, we  consider the $\mathscr H^{(1)}$-module
\begin{equation}\pi^{(1)}=
\pi_{\beta},\qquad \beta=-\eps_1/\kappa.
\end{equation}
Recall that $\scrU(W_k(\gen\len_1))$ is the image of
$\Uen(W_k(\gen\len_1))$ in $\End(\pi^{(1)})$. 
By Proposition \ref{prop:heis} there is a unique isomorphism $K_1$-vector space 
$\Lb^{(1)}_K \to \pi^{(1)}$ such that
$[I_{\emptyset}]\mapsto |\beta \rangle$ 
which intertwines the operator $\rho^{(1)}(b_{-l})=\rho^{(1)}(y^{-l}D_{l,0})$ on $\Lb^{(1)}_K$ with the
operator $b_{-l}$ on $\pi^{(1)}$.
Following \eqref{Phi}, we identify $\pi^{(1)}$ 
with $\Lambda_{K_1}$ in the usual way.
This yields an isomorphism 
\begin{equation}\label{Phibis}
\Lb^{(1)}_K =\pi^{(1)}=\Lambda_{K_1}.
\end{equation}
Our next result describes the action of 
the element $H_l$ introduced in \eqref{bl}.

\vspace{.1in}

\begin{prop}
\label{prop:LBinfty} 
We have the following relation in $\End(\pi^{(1)})$ 
$$\rho^{(1)}(H_{l})=
\frac{\kappa}{2}\sum_{h\in\Z}\rho^{(1)}(\pmb:\!b_{l-h}b_h\pmb:),
\qquad l\in\Z.$$
\end{prop}

\begin{proof} 
To unburden the notation we omit the symbol $\rho^{(1)}$ everywhere.
We must prove that
\begin{equation}
H_{0}=\kappa\sum_{l\geqslant 1}b_{-l}b_l+\kappa b_0^2/2,
\qquad
H_{k}=\kappa\sum_{l\in\Z}b_{k-l}b_l/2,\qquad k\neq 0.
\end{equation}
Recall that $b_{-l}$ acts on $\Lambda_{K_1}$
by multiplication by $p_l$ and that
$b_{l}$ acts by the operator $\kappa^{-1}l\partial_{p_l}$.
Next, the computation in the proof of \cite[thm.~3.1]{Stanley} implies that
\begin{equation}
\square(p_\lambda)=\kappa^{-1}\xi n(\lambda')p_\lambda+
\frac{1}{2\kappa}\sum_{r\neq s}\lambda_r\lambda_s
p_{\lambda_r}^{-1}p_{\lambda_s}^{-1}
p_{\lambda_r+\lambda_s}
p_\lambda+
\frac{1}{2}\sum_r\sum_{j=1}^{\lambda_r-1}
\lambda_rp_{\lambda_r}^{-1}p_jp_{\lambda_r-j}p_\lambda.
\end{equation}
So we have the following formula 
\begin{equation}\label{LB}
\square=\xi
\sum_{l\geqslant 1}(l-1)b_{-l}b_l/2+ 
\kappa\sum_{l, k \geqslant 1}
\big(b_{-l-k}b_l b_k + b_{-l}b_{-k} b_{l+k}\big)/2.
\end{equation}
Now, Remark \ref{rem:3.6A} yields
\begin{equation}\label{LBinfty}
D_{0,2}+\eps_1D_{0,1}=\kappa\,\square.
\end{equation}
Further, a direct computation (left to the reader)
using \eqref{bl}, \eqref{LB} and \eqref{LBinfty} gives
\begin{equation}
[H_{k},b_l]=-l b_{l+k},\qquad 
[H_{-k},b_l]=-l b_{l-k},\qquad l\in\Z,\qquad k\geqslant 1.
\end{equation}
This implies the formula for $H_{k}$ and $k\neq 0$.
Next, a direct computation using \eqref{bl} yields
\begin{equation}
H_{0}=D_{0,1}+\kappa b_0^2/2.
\end{equation}
Further, by Lemma \ref{lem:D5} we have
$[D_{0,1},D_{l,0}]=lD_{l,0}$ and by
\eqref{3.6:rem} we have $D_{0,1}([I_\emptyset])=0$.
This yields the following formula for $D_{0,1}$, which implies the formula for $H_0$,
\begin{equation}\label{LB'}
D_{0,1}=\kappa\sum_{l\geqslant 1}b_{-l}b_l.
\end{equation}

\end{proof}

\vspace{.1in}

Equations \eqref{LB}, \eqref{LBinfty} and \eqref{LB'} give 
the expression for the action of $D_{0,1}$ and $D_{0,2}$ on 
$\pi^{(1)}$. 
Since $\SH^{(1)}_K$ is generated by $\{D_{0,2},\,b_l\;;\; l \in \Z\},$ 
the proof above also gives the following.

\vspace{.1in}

\begin{prop}\label{prop:8.6:rank1} There is an embedding 
\begin{equation}\label{8.49}
\Theta~:\SH^{(1)}_K\to\scrU(W_k(\gen\len_1)),\quad b_l\mapsto b_l,
\end{equation}
which intertwines the representations of $\SH^{(1)}_K$ 
and $\scrU(W_k(\gen\len_1))$ on $\pi^{(1)}$.
\end{prop}

\vspace{.1in}

Thanks to Proposition \ref{prop:8.6:rank1}, we may speak of the standard order of
an element of $\SH^{(1)}_K$.

\begin{prop}\label{prop:leading_term_D_0n} 
For $d\geqslant 1$ we have
\begin{equation}\label{symbol}
\rho^{(1)}(D_{0,d})\equiv \frac{\kappa^d}{d(d+1)}
\sum_{l_0,\dots,l_d}\rho^{(1)}(:\!b_{l_0}b_{l_1}\dots b_{l_d}\!:).
\end{equation}
The sum runs over all tuples of integers with sum $0$.
The symbol $\equiv$ means that the equality holds 
modulo the action of terms of standard order $\preccurlyeq d-1$.
\end{prop}

\begin{proof}
To unburden the notation we omit the 
symbol $\rho^{(1)}$ everywhere.
Further, for any integers $m_1, \ldots, m_d$ we abbreviate
\begin{equation}
b_{m_1,\dots,m_d}=\ad(b_{m_1})\circ\ad(b_{m_2})\circ\cdots\circ\ad(b_{m_d}).
\end{equation}
Recall that $:\!b_{m_1}\cdots b_{m_d}\!:$ is the monomial obtained from
$b_{m_1}\cdots b_{m_d}$ by moving all $b_{m_i}$, $m_i<0$, to the left of all 
$b_{m_j}$ with $m_j\geqslant 0$.
First, we prove that for any  $m_1, \ldots, m_d$ we have 
\begin{equation}\label{proofleadingterm}
b_{m_1,\dots,m_d}(D_{0,d})=
(d-1)! \,(m_1m_2\dots m_d)\,b_m,\qquad
m=m_1+\cdots+m_d.
\end{equation}
We proceed by induction on $d$. 
Note that  \eqref{LB}-\eqref{LB'} imply that
\begin{equation}
D_{0,1}\equiv\frac{\kappa}{2}\sum_{l_0, l_1}:\!b_{l_0}b_{l_1}\!:\,,\qquad
D_{0,2}\equiv\frac{\kappa^2}{6}\sum_{l_0, l_1, l_2}:\!b_{l_0}b_{l_1}b_{l_2}\!:\,,
\end{equation}
where the $l_i$'s are integers which sum to 0.
This implies the claim for $d=1,2$.
Assume that \eqref{proofleadingterm} 
is proved for $d$. Applying $\ad(D_{1,1})$ 
to \eqref{proofleadingterm}, the formula \eqref{heis2} gives
\begin{equation*}
\begin{split}
(d-1)!(m_1m_2\cdots m_d) m b_{m -1}
&=b_{m_1,\dots,m_d}(D_{1,d})
+ \sum_i m_i 
b_{m_1,\dots,m_i-1,\dots,m_d}
(D_{0,d})\\
&=b_{m_1,\dots,m_d} (D_{1,d})
+ (d-1)!\sum_i (m_i-1) (m_1m_2\cdots m_d)  b_{m -1}.
\end{split}
\end{equation*}
This implies the formula
\begin{equation}\label{E:prooflead2}
b_{m_1,\dots,m_d}(D_{1,d})=
d! (m_1m_2\cdots m_d)b_{m-1}.
\end{equation}
Similarly, we have
\begin{equation}\label{E:prooflead3}
b_{m_1,\dots,m_d}(D_{-1,d})=
\kappa d! (m_1m_2\cdots m_d)b_{m+1}.
\end{equation}
Next, we compute 
\begin{equation}\label{E:prooflead4}
\begin{split}
b_{m_1,\dots,m_{d+1}} (E_{d+2})&= 
b_{m_1,\dots,m_{d+1}} ([D_{-1,2}, D_{1,d}])\\
&=\sum_{i<j} \big[b_{m_i,m_j}(D_{-1,2}), 
b_{m_1,\dots,\widehat m_i,\dots,\widehat m_j,\dots ,m_{d+1}}(D_{1,d})\big]+\\
&\qquad+\sum_i\big[b_{m_i}(D_{-1,2}),
b_{m_1,\dots,\widehat m_i,\dots,m_{d+1}}(D_{1,d})\big],
\end{split}
\end{equation}
where the symbol $\widehat m_i$ means that the index $m_i$ is omitted.
Write $m^+=m+m_{d+1}$.
The first sum on the right hand side of \eqref{E:prooflead4} is equal to
\begin{equation*}
\begin{aligned}
\sum_{i<j} 2\kappa l_il_j
b_{m_i+m_j+1,m_1,\dots,\widehat m_i,\dots,\widehat m_j,\dots ,m_{d+1}}(D_{1,d})
&=
2\kappa d!
(m_1\cdots m_{d+1})
\sum_{i<j}(m_i+m_j+1)
b_{m^+}\\
&=2\kappa d!(m_1\cdots m_{d+1})\big(dm^++d(d+1)/2\big)
b_{m^+}
\end{aligned}
\end{equation*}
while the second sum evaluates to
\begin{equation*}
\begin{aligned}
-d!\sum_i
(m_1\cdots\widehat m_i\cdots m_{d+1})
b_{m^+-m_i-1,m_i}(D_{-1,2})
&=
-2\kappa d!(m_1\cdots m_{d+1})
\sum_i(m^+-m_i-1)\,b_{m^+}\\
&=
2\kappa d!(m_1\cdots m_{d+1})\big(-dm^++(d+1)\big)b_{m^+}.
\end{aligned}
\end{equation*}
We obtain
\begin{equation}
b_{m_1,\dots,m_{d+1}}(E_{d+2})=
\kappa(d+2)!(m_1\cdots m_{d+1})b_{m^+}.
\end{equation}
By \eqref{leadingterm} we have 
$E_{d+2}=\kappa(d+2)(d+1) D_{0,d+1} +u$ where $u$ is a polynomial in 
$D_{0,1}, \ldots, D_{0,d}$ of order $\leqslant d$. 
Thus, we have
\begin{equation}
b_{m_1,\dots,m_{d+1}}(E_{d+2})=
\kappa(d+2)(d+1)b_{m_1,\dots,m_{d+1}} (D_{0,d+1}).
\end{equation}
From this we finally
deduce relation \eqref{proofleadingterm} for $d+1$. We are done.

Now, relation \eqref{symbol} follows from
\eqref{proofleadingterm}.
Indeed, given integers $l_0,l_1,\dots,l_d$ we have
\begin{equation}
\ad(b_m)(b_{l_0}b_{l_1}\dots b_{l_d})=
\frac{m}{\kappa}\sum_{l_i=-m}
b_{l_0}\dots b_{l_{i-1}}b_{l_{i+1}}\dots b_{l_d},
\end{equation}
where the sum is over all $i$'s with $l_i=-m$. 
Thus, if $l_0,l_1, \ldots, l_d$ 
sum to $0$ we have
\begin{equation}
\gathered
b_{m_1,\dots,m_d}
(:\!b_{l_0}b_{l_1}\dots b_{l_d}\!:)=c\, b_m,
\endgathered
\end{equation}
for some constant $c$ which is zero unless
$l_0, l_1,\dots l_d$ are equal to
$m,-m_1, -m_2,\dots, -m_d$, up to a permutation,
and which, in this case, is equal to
$(m_1\cdots m_d)/\kappa^d$ times the number 
$c_{l_0,\dots,l_d}$ of permutations $\sigma$ of
$\{0,1,\dots,d\}$ such that $l_{\sigma(0)}=m$ and
$l_{\sigma(s)}=-m_s$ for $s=1,2,\dots,d$.
In other words, if  $l_0,l_1, \ldots, l_d$
are equal to $m,-m_1, -m_2,\dots, -m_d$ up to a permutation, then we have
\begin{equation}
b_{m_1,\dots,m_d}
\Big(D_{0,d}-\frac{\kappa^d(d-1)!}{c_{l_0,\dots,l_d}}\, :\!b_{l_0}b_{l_1}\dots b_{l_d}\!:\Big)=0.
\end{equation}
Therefore, for any integers $m_1,m_2,\dots,m_d$ we have
\begin{equation}
b_{m_1,\dots,m_d}
\Big(D_{0,d}-\frac{\kappa^d}{d(d+1)}\sum_{l_0,\dots,l_d}:\!b_{l_0}b_{l_1}\dots 
b_{l_d}\!:\Big)=0.
\end{equation}
The sum runs over all tuples of integers 
summing to $0$.
To conclude, we use the following lemma.

\begin{lem} Let $u \in \Uen(W_k(\gen\len_1))$ be annihilated by
$b_{m_1,\dots,m_d}$ for any integers $m_1, \ldots, m_d.$
Then $u$ is of standard order $\preccurlyeq d-1$.
\end{lem}

\begin{proof} We may express $u$ as an infinite sum 
\begin{equation}
u=\sum_{s\geqslant 0}\sum_{l_1,\ldots,l_s} a_{l_1, \ldots, l_s} :\!b_{l_1} 
\cdots b_{l_s}\!:.
\end{equation}
Now observe that
\begin{equation}
\gathered
s<t\ \Rightarrow\  b_{m_1,\dots,m_t}(:\!b_{l_1} \cdots b_{l_s}\!: )=0,\qquad
b_{m_1, \cdots ,m_s}(:\!b_{l_1} \cdots b_{l_s}\!:)
=c_{m_1, \ldots, m_s},
\endgathered
\end{equation}
where $c_{m_1, \ldots, m_s}\neq 0$ if and only if
$l_1,\dots l_s$ are equal, up to a permutation, to $-m_1,\dots,-m_s$.
The lemma follows easily.
\end{proof}
\end{proof}

\vspace{.15in}

\subsection{From $\SH^{(2)}_K$ to $W_k(\gen\len_2)$}
We are interested in higher rank analogues of the inclusion \eqref{8.49}.
In this section we deal with the case $r=2$. 
We set $\kb=K_2$ and $\kappa=k+2$. 
We write
\begin{equation}\label{beta_r=2}
\pi^{(1^2)}=\pi_{\beta},\qquad
\langle b^{(i)},\beta\rangle=-\eps_i/\kappa+(i-1)\xi/\kappa, \qquad i=1,2.
\end{equation}
Recall that $\scrU(W_k(\gen\len_2))$ is the image of
$\Uen(W_k(\gen\len_2))$ in $\End(\pi^{(1^2)})$. 
The isomorphism  \eqref{Phibis}
yields an isomorphism $\Lb_K^{(1^2)}=\Lambda^{\otimes 2}_{K_2}$.
Composing it with the isomorphism 
$\Lambda^{\otimes 2}_{K_2}=\pi^{(1^2)}$ such that
$1\otimes 1\mapsto|\beta\rangle$ which intertwines the operators
$b_{-l} \otimes 1$, $1 \otimes b_{-l}$ on $\Lambda^{\otimes 2}_{K_2}$ with the 
operators $b^{(1)}_{-l}$, $b^{(2)}_{-l}$  on $\pi^{(1^2)}$, we get an isomorphism
\begin{equation}\label{8.66}
\Lb_K^{(1^2)}=\pi^{(1^2)}=\Lambda^{\otimes 2}_{K_2}
\end{equation}
which identifies  $[I_\emptyset]^{\otimes 2}$, $|\beta\rangle$ and $1^{\otimes 2}$.
Using \eqref{8.66} together with Propositions \ref{prop:pibeta} and 
\ref{8.4:prop} we get inclusions of
$\scrU(W_k(\gen\len_2))$ and $\SH^{(2)}_K$ into $\End(\pi^{(1^2)}).$

\begin{prop}\label{8.6:prop}
The representation $\rho^{(1^2)}$ 
yields an embedding of degreewise topological $K_2$-algebras
$\Theta~:\SH^{(2)}_K\to\scrU(W_k(\gen\len_2))$.
\end{prop}

\begin{proof} 
It is enough to check that  $\rho^{(1^2)}(b_l)$ and $\rho^{(1^2)}(D_{0,2})$ belong to
$\scrU(W_k(\gen\len_2))$. For $b_l$, this follows from the easily checked relation 
\begin{equation}
\rho^{(1^2)}(b(z))=J(z),\qquad
b(z)=\sum_{l \in \Z} b_l z^{-l-1} \in \SH^{(2)}_K[[z,z^{-1}]].
\end{equation} 
For $D_{0,2}$, this is a consequence of the lemma below.

\begin{lem} 
There is a constant $c$ such that
\begin{equation*}
\begin{split}
\rho^{(1^2)}(D_{0,2})=&
\frac{\kappa}{2}\sum_{l\in\Z}\pmb :\!W_{1,-l}W_{2,l}\!\pmb :+
\frac{\kappa^2}{24}\sum_{k,l\in\Z}\pmb :\!W_{1,-k-l}W_kW_l\!\pmb :+\\
&+\frac{\kappa\xi}{4}\sum_{l\in \Z}(|l|-1)\pmb:W_{1,-l}W_{1,l}\pmb: 
+\xi W_{2,0} +c.
\end{split}
\end{equation*}
\end{lem}

\begin{proof}
First, note that \eqref{LB}, \eqref{LBinfty}, \eqref{LB'} and Theorem 
\ref{7.7:thm2} imply that
\begin{equation}\label{8.6:form1}
\aligned
\rho^{(1^2)}(D_{0,2})
&=
\frac{\kappa^2}{2}\sum_{k,l \geqslant 1}
\Big(
b_{-l-k}^{(1)}b_l^{(1)} b_k^{(1)}+b_{-l}^{(1)}b_{-k}^{(1)} b_{l+k}^{(1)}+
b_{-l-k}^{(2)}b_l^{(2)} b_k^{(2)}+b_{-l}^{(2)}b_{-k}^{(2)} b_{l+k}^{(2)}
\Big)+\\
&+\frac{\kappa\xi}{2}
\sum_{l\geqslant 1}(l-1)\Big(
b_{-l}^{(1)}b_l^{(1)}+ b_{-l}^{(2)}b_l^{(2)}\Big)
-\kappa\sum_{l\geqslant 1}
\Big(\eps_1b_{-l}^{(1)}b_l^{(1)}+\eps_2b_{-l}^{(2)}b_l^{(2)}\Big)+\\
&+\kappa\xi\sum_{l\geqslant 1}lb_{-l}^{(2)}b_l^{(1)}.
\endaligned
\end{equation}
Using $\eps_1=-\kappa b_0^{(1)}$ and $\eps_2=\xi-\kappa b_0^{(2)}$,
we can rewrite \eqref{8.6:form1} in the following way
\begin{equation}
\aligned
\rho^{(1^2)}(D_{0,2})
&=
\frac{\kappa^2}{6}\sum_{k,l\in\Z}
\pmb:b_{-l-k}^{(1)}b_l^{(1)} b_k^{(1)}+b_{-l-k}^{(2)}b_l^{(2)} b_k^{(2)}\pmb:
+\\
&+
\frac{\kappa\xi}{4}
\sum_{l\in\Z}(|l|-2)\pmb:
b_{-l}^{(1)}b_l^{(1)}+ b_{-l}^{(2)}b_l^{(2)}\pmb:
+\frac{\kappa\xi}{4}
\sum_{l\in\Z}\pmb:
b_{-l}^{(1)}b_l^{(1)}- b_{-l}^{(2)}b_l^{(2)}\pmb:+
\\
&+\kappa\xi\sum_{l\geqslant 1}lb_{-l}^{(2)}b_l^{(1)}
+ c_1
\endaligned
\end{equation}
for some constant $c_1$.
%where $c_1=p_3(\vec\varepsilon)/6\kappa + 
%p_2(\vec\varepsilon)\xi/4\kappa - \vec\varepsilon_2\xi^2/\kappa +7 \xi^3/12\kappa$.
Next, recall that 
\begin{equation}
\aligned
W_1(z)&=b^{(1)}(z)+b^{(2)}(z),\\
W_2(z)&=\frac{\kappa}{2}\pmb :\!b^{(1)}(z)^2+b^{(2)}(z)^2\!\pmb :
-\,\frac{\kappa}{4}\pmb :\!W_1(z)^2\!\pmb :-\,\xi\,\partial_z\rho(z).
\endaligned
\end{equation}
This implies that
\begin{equation}
\aligned
W_{2,l}&=
-\frac{\kappa}{2}\sum_{k\in\Z}
\pmb:b_{l-k}^{(1)}b_k^{(2)}\pmb:
+\frac{\kappa}{4}\sum_{k\in\Z}
\pmb:b_{l-k}^{(1)}b_k^{(1)}+b_{l-k}^{(2)}b_k^{(2)}\pmb:
+\frac{\xi}{2}(l+1)\big(b_{l}^{(1)}-b_{l}^{(2)}\big).
\endaligned
\end{equation}
Further, we have the following formulas
\begin{equation}
\aligned
\sum_{k,l\in\Z}\pmb :\!W_{1,-k-l}W_{1,k}W_{1,l}\!\pmb :
&=3\sum_{k,l\in\Z}\pmb:
b_{-k-l}^{(1)}b_k^{(1)}b_{l}^{(2)}+
b_{-k-l}^{(2)}b_k^{(2)}b_{l}^{(1)}\pmb:+\\
&+
\sum_{k,l\in\Z}\pmb:
b_{-k-l}^{(1)}b_k^{(1)}b_{l}^{(1)}+
b_{-k-l}^{(2)}b_k^{(2)}b_{l}^{(2)}\pmb:
\endaligned
\end{equation}
\begin{equation}
\sum_{l\in\Z}|l|\pmb:\!W_{1,-l}W_{1,l}\pmb:=
2\sum_{l\in\Z}|l| \pmb:b_{-l}^{(1)}b_l^{(2)}\pmb:+
\sum_{l\in\Z}|l| \pmb:b_{-l}^{(1)}b_l^{(1)}+b_{-l}^{(2)}b_l^{(2)}\pmb:
\end{equation}
\begin{equation}
\aligned
\sum_{l\in\Z}\pmb :\!W_{1,-l}W_{2,l}\!\pmb :
&=
-\frac{\kappa}{4}\sum_{k,l\in\Z} 
\pmb:b_{-k-l}^{(1)}b_{k}^{(1)}b_{l}^{(2)}+
b_{-k-l}^{(2)}b_{k}^{(2)}b_{l}^{(1)}\pmb:+\\
&+\frac{\kappa}{4}\sum_{k,l\in\Z} 
\pmb:b_{-k-l}^{(1)}b_{k}^{(1)}b_{l}^{(1)}+
b_{-k-l}^{(2)}b_{k}^{(2)}b_{l}^{(2)}
\pmb:-\\
&+\frac{\xi}{2}\sum_{l\in\Z}\pmb:
b_{-l}^{(1)}b_l^{(1)}-b_{-l}^{(2)}b_l^{(2)}\pmb:
-\xi\sum_{l\in\Z}l\pmb:b_{-l}^{(1)}b_l^{(2)}\pmb:\\
\endaligned
\end{equation}
Therefore, we get
\begin{equation}
\aligned
\rho^{(1^2)}(D_{0,2})&
-\frac{\kappa}{2}\sum_{l\in\Z}\pmb :\!W_{1,-l}W_{2,l}\!\pmb :
-\frac{\kappa^2}{24}\sum_{k,l\in\Z}\pmb :\!W_{1,-k-l}W_{1,k}W_{1,l}\!\pmb :
-\frac{\kappa\xi}{4}\sum_{l\in\Z}|l|\pmb:W_{1,-l}W_{1,l}\pmb:
\\
&=-\frac{\kappa\xi}{2}
\sum_{l\in\Z}\pmb:
b_{-l}^{(1)}b_l^{(1)}+ b_{-l}^{(2)}b_l^{(2)}\pmb:
+c_1.
\endaligned
\end{equation}
Next, observe that
\begin{equation}
\gathered
W_{2,0}=
-\frac{\kappa}{2}\sum_{l\in\Z}
\pmb:b_{-l}^{(1)}b_l^{(2)}\pmb:+
\frac{\kappa}{4}\sum_{l \in\Z}
\pmb:b_{-l}^{(1)}b_l^{(1)}+b_{-l}^{(2)}b_l^{(2)}\pmb:+
\frac{\xi}{2}\big(b_{0}^{(1)}-b_{0}^{(2)}\big),\\
\sum_{l\in\Z}\pmb:\!W_{1,-l}W_{1,l}\pmb:=
2\sum_{l\in\Z}\pmb:b_{-l}^{(1)}b_l^{(2)}\pmb:+
\sum_{l\in\Z}\pmb:b_{-l}^{(1)}b_l^{(1)}+b_{-l}^{(2)}b_l^{(2)}\pmb:.
\endgathered
\end{equation}
Therefore, we have
\begin{equation}
\frac{\kappa\xi}{2}\sum_{l \in\Z}\pmb:b_{-l}^{(1)}b_l^{(1)}+b_{-l}^{(2)}b_l^{(2)}\pmb:=
\frac{\kappa\xi}{4}\sum_{l\in\Z}\pmb:\!W_{1,-l}W_{1,l}\pmb:+\xi W_{2,0}-
\frac{\xi^2}{2}\big(b_{0}^{(1)}-b_{0}^{(2)}\big).
\end{equation}
The lemma follows, the constant $c$ being given by
\begin{equation}
\aligned
c&=p_3(\vec\varepsilon)/6\kappa + p_2(\vec\varepsilon)\xi/4\kappa - 
p_1(\vec\varepsilon)\xi^2/2\kappa + \xi^3/12\kappa.
%&=c_1+ \xi^2\big(b_{0}^{(1)}-b_{0}^{(2)}\big)/2.
\endaligned
\end{equation}
\end{proof}
\end{proof}

\vspace{.15in}

\subsection{From $\SH^{(r)}_K$ to $W_k(\gen\len_r)$}
Now $r$ is arbitrary. 
We set $\kb=K_r$ and $\kappa=k+r$. 
We write 
\begin{equation}\label{beta}
\pi^{(1^r)}=\pi_{\beta},\qquad
\langle b^{(i)},\beta\rangle=-\eps_i/\kappa+(i-1)\xi/\kappa,\qquad
i\in[1,r].
\end{equation}
Recall that $\scrU(W_k(\gen\len_r))$ is the image of
$\Uen(W_k(\gen\len_r))$ in $\End(\pi^{(1^r)})$. 
We construct as in \eqref{8.66} a $K_r$-linear isomorphism
\begin{equation}\label{piL}
\Lb_K^{(1^r)}=\pi^{(1^r)}=\Lambda^{\otimes r}_{K_r}
\end{equation}
which identifies $[I_\emptyset]^{\otimes r}$,
$|\beta\rangle$ and $1^{\otimes r}$ and
which intertwines the operator
$b^{(i)}_{-l}$ on $\pi^{(1^r)}$ with 
\begin{equation}
\label{8.81}
1\otimes \cdots\otimes 1\otimes b_{-l}\otimes 1\otimes\cdots\otimes 1
\qquad(b_{-l}\text{\ is\ at\ the\ $i$-th\ spot})
\end{equation}
on $\Lambda^{\otimes r}_{K_r}$ and with the operator on $\Lb_K^{(1^r)}$ given by
\begin{equation}
1\otimes \cdots\otimes 1\otimes \rho^{(1)}(y^{-l}D_{l,0})\otimes 1\otimes\cdots\otimes 1.
\end{equation}
Propositions \ref{prop:pibeta}, \ref{8.4:prop} then provide inclusions of
$\scrU(W_k(\gen\len_r))$ and $\SH^{(r)}_K$ into $\End(\pi^{(1^r)}).$
We equip $\SH^{(r)}_K$, $\Uen(\SH^{(r)}_K)$ and $\scrU(W_k(\gen\len_r))$ 
with the standard degreewise topologies.

\vspace{.1in}

\begin{theo}
\label{8.7:thm}
The representation $\rho^{(1^r)}$ 
yields an embedding of degreewise topological $K_r$-algebras
$\Theta~:\SH^{(r)}_K\to\scrU(W_k(\gen\len_r))$
with a degreewise dense image.
The morphism $\Theta$ is compatible with the order filtrations.
\end{theo}

The theorem is a direct consequence of Lemmas
\ref{8.6:lem1}, \ref{8.6:lem2} below. 
Note that the map $\Theta$ is homogeneous of degree zero
relatively to the rank degree on $ \SH^{(r)}_K$ and the conformal degree
on $\scrU(W_k(\gen\len_r))$.

\begin{lem}
\label{8.6:lem1}
The representation $\rho^{(1^r)}$ 
yields an embedding of degreewise topological $K_r$-algebras
$\Theta~:\SH^{(r)}_K\to\scrU(W_k(\gen\len_r)).$
\end{lem}

\begin{proof}
We may assume that $r\geqslant 2$.
Consider the composition of $r$ given by
\begin{equation}
\omega_i=(1,\dots,1,2,1,\dots, 1),\qquad i\in[1,r),
\end{equation}
where $2$ is at the $i$-th spot. 
Let $\phi^{\omega_i}:\SH^{(r)}_K\to\SH^{\omega_i}_K$
be the $K_r$-algebra homomorphism given by the iterated coproduct,
and let $\rho^{\omega_i}$ be the representation of  
$\SH^{\omega_i}_K$ on $\pi^{(1^r)}$ given by
\begin{equation}
\rho^{\omega_i}=\rho^{(1)}\otimes\cdots\otimes \rho^{(1)}\otimes
\rho^{(1^2)}\otimes \rho^{(1)}\otimes\cdots\otimes \rho^{(1)}.
\end{equation}
The coassociativity of the coproduct
implies that the representation $\rho^{(r)}$ is the pull-back of the representation $\rho^{\omega_i}$
by the algebra homomorphism $\phi^{\omega_i}$.
By Propositions \ref{prop:8.6:rank1}, \ref{8.6:prop} the representations 
$\rho^{(1)}$ and $\rho^{(1^2)}$ give inclusions
\begin{equation}
\SH^{(1)}_K\subset\scrU(W_{\kappa-1}(\gen\len_1)),\qquad
\SH^{(2)}_K\subset\scrU(W_{\kappa-2}(\gen\len_2)).
\end{equation}
Therefore, for each $i$ the representation $\rho^{(1^r)}$ gives an inclusion
of $\SH^{(r)}_K$ into the subalgebra of $\End(\pi^{(1^r)})$ given by
\begin{equation}
\scrU^{\omega_i}=\scrU(W_{\kappa-1}(\gen\len_1))^{\widehat\otimes(i-1)}
\;\widehat\otimes\;\scrU(W_{\kappa-2}(\gen\len_2))\;\widehat\otimes\;
\scrU(W_{\kappa-1}(\gen\len_1))^{\widehat\otimes(r-i-1)}.
\end{equation}
Thus, the lemma is a consequence of the following classical result due to Feigin and Frenkel
\footnote{A. Okounkov told us that he used a similar argument in his proof with D. Maulik}.

\begin{theo}
We have the equality $\bigcap_{i=1}^{r-1}\scrU^{\omega_i}=\scrU(W_k(\gen\len_r))$
in $\End(\pi^{(r)})$.
\end{theo}

This is a direct corollary of the characterization of $\scrU(W_k(\gen\len_r))$ as the intersection of screening 
operators
associated with each simple root of $\mathfrak{sl}_n$, see \cite[thm.~4.6.9]{FF}. 
The above formulation appears in \cite[sec.~15.4.15]{BZF}.
\end{proof}

\vspace{.1in}

\begin{lem}
\label{8.6:lem2}
The inclusion $\Theta:\SH^{(r)}_K\to\scrU(W_k(\gen\len_r))$ gives a surjective 
morphism of degreewise topological $K_r$-algebras
$\Uen(\SH^{(r)}_K)\to\scrU(W_k(\gen\len_r))$ which
is compatible with the order filtrations.
\end{lem}

\begin{proof}
By the universal property of completions,
for each integer $s$, the inclusion
\begin{equation}
\SH^{(r)}_K[s]\to\scrU(W_k(\gen\len_r))[s],
\end{equation}
which is a continuous map,
extends uniquely to a continuous map
\begin{equation}\Uen(\SH^{(r)}_K)[s]\to\scrU(W_k(\gen\len_r))[s].
\end{equation}
Taking the sum over all $s$ we get a map
\begin{equation}
\Theta:\Uen(\SH^{(r)}_K)\to\scrU(W_k(\gen\len_r)).
\end{equation}
It is a morphism of degreewise topological $K_r$-algebras.
We must prove that it is surjective. 
We have already seen that $\Theta(b(z))=W_1(z)$.
We now consider the fields $W_d(z)$ with $d>1$. 
The free field representation of $W_k(\gen\len_r)$ yields an embedding
\begin{equation}
\label{8.46}
\scrU(W_k(\gen\len_r)) \subset 
\scrU(W_k(\gen\len_1))^{\widehat{\otimes} r}.
\end{equation}
By Propositions \ref{8.4:prop}, \ref{prop:8.6:rank1}, the  representation 
$\rho^{(1^r)}$ 
yields an embedding
\begin{equation}
\Uen(\SH^{(r)}_K) \subset \scrU(W_k(\gen\len_1))^{\widehat{\otimes} r}.
\end{equation}
The standard filtration on
$\scrU(W_k(\gen\len_1))$ introduced in Section~\ref{sec:gl1} induces the standard filtrations
\begin{equation}\label{inducedfiltration}
\Uen(\SH^{(r)}_K)=\bigcup_d \Uen(\SH^{(r)}_K)[\preccurlyeq\! d], \qquad  
\scrU(W_k(\gen\len_r))=\bigcup_d  \scrU(W_k(\gen\len_r))[\preccurlyeq\! d].
\end{equation}
The map $\Theta$ is compatible with these filtrations. 
Proposition~\ref{prop:symbolW} yields the following.

\begin{slem}\label{slem:1} For $d \neq 1$,
under the inclusion \eqref{8.46} we have
$$W_d(z)=-\kappa \sum_{s=0}^d(-r)^{s-d}
\Big(\begin{smallmatrix}r-s\\r-d\end{smallmatrix}\Big)
\sum_{i_1<i_2<\cdots<i_s}
\pmb:J(z)^{d-s}b^{(i_1)}(z)b^{(i_2)}(z)\cdots b^{(i_s)}(z)\pmb:$$
modulo terms of standard order $\preccurlyeq d-1$
in $\Uen(W_k(\gen\len_1))^{\widehat{\otimes} r}[[z,z^{-1}]]$.
\end{slem}

\vspace{.1in}

Recall the elements $Y_{r,d}$ defined in \eqref{yrd}.

\begin{slem}\label{slem:2}
For $l,$ $d$ with $d \geqslant 0$ there is a constant $c(l,d) \neq 0$ such that
\begin{equation}\label{E:sublemma2}
\rho^{(1^r)}(Y_{l,d})\equiv
c(l,d) \sum_{i=1}^r \sum_{ l_0, \ldots, l_d} :\! b^{(i)}_{l_0} \cdots b^{(i)}_{l_d} \!:\,.
\end{equation}
The sum runs over all tuples of integers with sum $-l$.
The symbol $\equiv$ means that the equality holds modulo 
terms of standard order $\preccurlyeq d$
in $\Uen(W_k(\gen\len_1))^{\widehat{\otimes} r}[[z,z^{-1}]]$.
\end{slem}

\begin{proof} First, we prove the following estimate
\begin{equation}\label{eq:estimate_1}
\rho^{(1^r)}(Y_{l,d}) \equiv \delta^{r-1}(\rho^{(1)}(Y_{l,d})).
\end{equation}
Equation \eqref{eq:estimate_1} is clear
from the definition of the coproduct on $\SH^{\cb}$ for $d=0,1$ or 
for $l=0,$ $d=2$. Next, the operator
$\ad(\rho^{(1^r)}(D_{0,2}))$ 
increases the standard order by at most one, see e.g., 
formula \eqref{8.6:form1} in the case $r=2$.
Hence using relations
\begin{equation}
\ad(D_{0,2})^d(D_{1,0})=D_{1,d}, \qquad 
\ad(D_{0,2})^d (D_{-1,0})=(-1)^d D_{-1,d}
\end{equation}
we deduce \eqref{eq:estimate_1} for $l=\pm 1$. 
Likewise, the operator $\ad(\rho^{(1^r)}(D_{\pm 1,1}))$ preserves
the standard filtration and the operator $\ad(\rho^{(1^r)}(D_{1,0}))$
decreases the standard filtration by one. This implies that 
\eqref{eq:estimate_1} holds for $l\neq 0$. Thus we have also
\begin{equation}
\rho^{(1^r)}(E_{d}) \equiv \delta^{r-1}(\rho^{(1)}(E_{d})).
\end{equation} 
This implies that \eqref{eq:estimate_1} also holds for
$D_{0,d}$ for any $d$. 

Next, combining \eqref{eq:estimate_1} and
Proposition~\ref{prop:leading_term_D_0n} yields \eqref{E:sublemma2} for $l=0$ with
\begin{equation}\label{c0d}
c(0,d)=\frac{\kappa^d}{d(d+1)}.
\end{equation}
Finally, acting by 
\begin{equation}
\gathered
\ad(\rho^{(1^r)}(D_{\pm 1,1}))
\equiv\ad(\delta^{r-1}(\rho^{(1)}(D_{\pm 1,1}))),\\
\ad(\rho^{(1^r)}(D_{\pm 1,0}))\equiv\ad(\delta^{r-1}(\rho^{(1)}(D_{\pm 1,0})))
\endgathered
\end{equation}
now yields \eqref{E:sublemma2} for all values of $l,$ $d$. We are done.
Note that, since \eqref{heis2} implies that
\begin{equation}
[b_l,D_{1,1}]=ylb_{l-1},\qquad
[b_l,D_{-1,1}]=y^{-1}\kappa lb_{l+1},
\end{equation}
we get
\begin{equation}\label{cpm1d}
c(1,d)=y\kappa^d/(d+1),\qquad c(-1,d)=-y^{-1}\kappa^{d+1}/(d+1).
\end{equation}
\end{proof}

By Lemma~\ref{8.6:lem1}, we have
$\Im(\Theta) \subseteq \scrU(W_k(\gen\len_r))$.
Using Claims~\ref{slem:1} and \ref{slem:2} we see that the associated 
graded of $\Im(\Theta)$ and $\scrU(W_k(\gen\len_r))$
with respect to the standard filtration are equal. 
This implies that 
\begin{equation}
\Im(\Theta)=\scrU(W_k(\gen\len_r)).
\end{equation}
To finish, we prove the compatibility of $\Theta$ with the order filtration $\scrU(W_k(\gen\len_r))[\leqslant\!d]$ 
defined in Section \ref{sec:8.2}. Recall the filtration $\scrU(W_k(\gen\len_r))[\preccurlyeq\! d]$ 
defined in \eqref{inducedfiltration}.
By Claims~\ref{slem:1} and \ref{slem:2}, there exists
for any $l,d$ an explicit element 
\begin{equation}
u_{l,d} \in \scrU(W_k(\gen\len_r))[\leqslant\! d]
\end{equation}
such that 
\begin{equation}
\Theta(D_{l,d})-u_{l,d} \in \scrU(W_k(\gen\len_r))[\preccurlyeq\! d].
\end{equation}
But from the definition of the filtrations, we have
\begin{equation}
\scrU(W_k(\gen\len_r))[\preccurlyeq\! d] \subseteq 
\scrU(W_k(\gen\len_r))[<\!d].
\end{equation}
By Remark \ref{rem:8.11}, 
the order filtration on $\Uen(\SH^{(r)}_K)$ is determined by 
putting $D_{r,d}$ (or equivalently $Y_{r,d}$) in degree $d$ for any $(r,d)$.
Thus Lemma~\ref{8.6:lem2} is proved.

\end{proof}

Theorem \ref{8.7:thm} has the following consequence.

\begin{cor}\label{8.7:cor}
The pull-back by the morphism
$\Theta:\SH^{(r)}_K\to\scrU(W_k(\gen\len_r))$ 
is an equivalence from the category of
admissible $\scrU(W_k(\gen\len_r))$-modules to the category
of admissible $\SH^{(r)}_K$-modules.
This equivalence takes $\pi^{(1^r)}$ to $\rho^{(1^r)}$.
\end{cor}

\begin{proof}
Since the image of $\SH^{(r)}_K$ in 
$\scrU(W_k(\gen\len_r))$ is degreewise dense,  
this functor is fully faithful. 
Thus, it is enough to check that it is essentially surjective.
To do that, let $M$ be an admissible $\SH^{(r)}_K$-module.
View $\SH^{(r)}_K$ as a degreewise dense degreewise topological subalgebra
of $\scrU(W_k(\gen\len_r))$. Then, for any $s,s'$ the action map
\begin{equation}
\SH^{(r)}_K[s]\times M[s']\to M[s+s']
\end{equation}
extends uniquely to a continuous map
\begin{equation}
\scrU(W_k(\gen\len_r))[s]\times M[s']\to M[s+s'].
\end{equation}
This yields an admissible $\scrU(W_k(\gen\len_r))$-module structure on $M$.
The corollary follows.
\end{proof}

\vspace{.1in}

\begin{rem}
It is likely that the map $\Theta$
is an isomorphism of degreewise topological $K_r$-algebras
$\Uen(\SH^{(r)}_K)\to\scrU(W_k(\gen\len_r))$. We'll not need this.
\end{rem}

\vspace{.15in}

\subsection{The Virasoro field} 
We set $\kb=K_r$ and $\kappa=k+r$. 
Now, we describe the preimage under the map $\Theta$  in 
Theorem \ref{8.7:thm}
of the Virasoro field $W_2(z)$.
We keep all the conventions of the previous section.
We have introduced in \eqref{bl} some elements $b_l$, $H_l$. 
Consider the fields in $\SH^{(r)}_K[[z,z^{-1}]]$ given by
\begin{equation}
H(z)=\sum_{l\in\Z}H_l\,z^{-l-2}, \qquad b(z)=\sum_{l \in \Z} b_l z^{-l-1}.
\end{equation} 
Recall the field 
$\rho(z)$ in $\End(\pi^{(1^r)})[[z^{-1}, z]]$ given by
\begin{equation}
\rho(z)=\sum_{i=1}^r(r/2-i+1/2)\,b^{(i)}(z).
\end{equation}

\vspace{.1in}

\begin{prop}\label{prop8.2:vir}
We have the following equalities
$$\rho^{(1^r)}(b(z))=J(z),\qquad 
\rho^{(1^r)}(H(z))=\frac{\kappa}{2}\sum_i\pmb: b^{(i)}(z)^2\pmb:
-\xi\,\partial_z\rho(z).$$
\end{prop}

\begin{proof}
The first claim is obvious. Note, indeed, that we have
\begin{equation}
\rho^{(1^r)}(b_0)=-p_1(\eps_1,\dots,\eps_r)/\kappa+r(r-1)\xi/2\kappa=\sum_{i=1}^r\langle b^{(i)},\beta\rangle.
\end{equation}
Let us concentrate on the second one.
For $k\geqslant 1$ we set
$$\aligned
H'_{k}=H_k+(r-1)(k-1)\xi b_{k}/2,\qquad
H'_{-k}=H_{-k}+(r-1)(k-1)\xi b_{-k}/2.
\endaligned$$
We must prove the following formulas
\begin{equation}
\gathered
\rho^{(1^r)}(H_0)=\kappa\sum_i\sum_{l\geqslant 1} b_{-l}^{(i)} b_l^{(i)}+
\kappa\sum_i(b_0^{(i)})^2/2+\xi\rho_0,\\
\rho^{(1^r)}(H'_{-k})=
\kappa\sum_i\sum_{l}b_{-k-l}^{(i)}b_l^{(i)}/2-(k-1)\xi\rho_{-k}+
(r-1)(k-1)\xi J_{-k}/2,\\
\rho^{(1^r)}(H'_{k})=
\kappa\sum_i\sum_{l}b_{k-l}^{(i)}b_l^{(i)}/2+(k+1)\xi\rho_{k}
+(r-1)(k-1)\xi J_k/2.
\endgathered
\end{equation}
Write
\begin{equation}
H_k^{(i)}=1\otimes 1\otimes\cdots\otimes 1\otimes \rho^{(1)}(H_k)\otimes 
1\otimes\cdots\otimes 1,\qquad i\in[1,r],
\end{equation}
where $H_k$ is at the $i$-th spot.
We have
\begin{equation}
\aligned
H'_{k}
=(-x)^kD_{-k,1}/k+(1-k)\xi b_{k}/2,\qquad
H'_{-k}
=y^{-k}D_{k,1}/k+(1-k)\xi b_{-k}/2.
\endaligned
\end{equation}
Thus, Theorem \ref{7.7:thm2} yields
\begin{equation}
\gathered
\rho^{(1^r)}(H'_{-k})=\sum_i H_{-k}^{(i)}+k\xi \sum_i(i-1)b_{-k}^{(i)},\\
\rho^{(1^r)}(H'_{k})=\sum_i H_{k}^{(i)}+k\xi \sum_i(r-i)b_{k}^{(i)}.
\endgathered
\end{equation}
Proposition \ref{prop:LBinfty} now yields
\begin{equation}
\gathered
\aligned
\rho^{(1^r)}(H'_{-k})
&=\kappa\sum_i\sum_{l\neq 0,-k}b_{-k-l}^{(i)}b_l^{(i)}/2+
\sum_i\big(-\eps_i+k(i-1)\xi\big)b_{-k}^{(i)},\\
&=\kappa\sum_i\sum_{l}b_{-k-l}^{(i)}b_l^{(i)}/2+
(k-1)\xi\sum_i(i-1) b_{-k}^{(i)},\\
&=\kappa\sum_i\sum_{l}b_{-k-l}^{(i)}b_l^{(i)}/2-
(k-1)\xi\rho_{-k}+(r-1)(k-1)\xi J_{-k}/2,\\
\endaligned
\\
\aligned
\rho^{(1^r)}(H'_{k})
&=\kappa\sum_i\sum_{l\neq 0,k}b_{k-l}^{(i)}b_l^{(i)}/2+
\sum_i\big(-\eps_i+k(r-i)\xi\big)b_{k}^{(i)},\\
&=\kappa\sum_i\sum_{l}b_{k-l}^{(i)}b_l^{(i)}/2+
\xi\sum_i\big((k-1)(r-1)+(k+1)(r-2i+1)\bigr) b_{k}^{(i)}/2,\\
&=\kappa\sum_i\sum_{l}b_{k-l}^{(i)}b_l^{(i)}/2+(k+1)\xi\rho_{k}+
(r-1)(k-1)\xi J_k/2.
\endaligned
\endgathered
\end{equation}
We have
$[H_k^{(i)},b_l^{(i)}]=-lb_{k+l}^{(i)}.$ Therefore, we get
\begin{equation}
\aligned
\rho^{(1^r)}(H_0)
&=\sum_i H_{0}^{(i)}+\xi\sum_i(i-1)[H_1^{(i)},b_{-1}^{(i)}]/2-
\xi\sum_i(r-i)[H_{-1}^{(i)},b_{1}^{(i)}]/2+\\
&\qquad+\xi^2\sum_i(r-i)(i-1)[b_1^{(i)},b_{-1}^{(i)}]/2\\
%&=\sum_i H_{0}^{(i)}+\xi\sum_i(i-1)b_{0}^{(i)}/2+
%\xi\sum_i(r-i)b_{0}^{(i)}/2+\xi^2\sum_i(r-i)(i-1)/2\kappa\\
&=\sum_i H_{0}^{(i)}-(r-1)\xi\sum_i\eps_i/2\kappa+
\xi^2\sum_i(r-i)(i-1)/2\kappa\\
&=\kappa\sum_i\sum_{l\geqslant 1}b_{-l}^{(i)}b_l^{(i)}+
\sum_i(\eps_i-(i-1)\xi)(\eps_i-(r-i)\xi)/2\kappa\\
&=\kappa\sum_i\sum_{l\geqslant 1}b_{-l}^{(i)}b_l^{(i)}-
\sum_ib_0^{(i)}(\eps_i-(r-i)\xi)/2\\
&=\kappa\sum_i\sum_{l\geqslant 1}b_{-l}^{(i)}b_l^{(i)}+
\kappa\sum_i(b_0^{(i)})^2/2+
\xi\sum_ib_0^{(i)}(r-2i+1)/2\\
&=\kappa\sum_i\sum_{l\geqslant 1}b_{-l}^{(i)}b_l^{(i)}+
\kappa\sum_i(b_0^{(i)})^2/2+\xi\rho_0.
\endaligned
\end{equation}

\end{proof}

\vspace{.1in}

The fields $W_1(z)$ and $W_2(z)$ give two fields
in $\End(\pi^{(1^r)})[[z^{-1},z]]$. Let us denote them again by
$W_1(z)$ and $W_2(z)$. 
Consider the field $L(z)$ in $\Uen(\SH^{(r)}_K)[[z^{-1},z]]$ given by
\begin{equation}
\label{Virasoro}
L(z)=H(z)-\frac{\kappa}{2r}\pmb:\,b(z)^2\!\pmb:.
\end{equation}
Proposition \ref{prop8.2:vir} and \eqref{W2} imply that
\begin{equation}
\label{8.6:8.20}
\rho^{(1^r)}(b(z))=W_1(z),\qquad
\rho^{(1^r)}(L(z))=W_2(z).
\end{equation}
% As for $H(z)$, we write
% \begin{equation}
% L(z)=\sum_{l\in\Z}L_l\,z^{-l-2}.
% \end{equation}
Therefore, by definition of the map $\Theta$, we have the following.

\begin{cor}\label{cor:8.26} We have
$\Theta(b(z))=W_1(z)$ and
$\Theta(L(z))=W_2(z).$
\end{cor}

\vspace{.15in}

\subsection{The representation $\pi^{(r)}$ of 
$W_k(\gen\len_r)$ on $\Lb^{(r)}_K$}
We set $\kb=K_r$ and $\kappa=k+r$. 
Let $\beta$ be as in \eqref{beta}.
The representation $\rho^{(r)}$ of $\SH^{(r)}_K$ on $\Lb^{(r)}_K$ is admissible.

\begin{df}
Let $\pi^{(r)}$ be the unique admissible representation of $W_k(\gen\len_r)$ which is taken to
$\rho^{(r)}$ by the equivalence of categories in Corollary \ref{8.7:cor}.
\end{df}

By Corollary \ref{cor:8.26} we have
\begin{equation}
\label{8.9:8.32}
\rho^{(r)}(b(z))=\pi^{(r)}(W_1(z)),\qquad
\rho^{(r)}(L(z))=\pi^{(r)}(W_2(z)).
\end{equation}
Write $|0\rangle$ for the element $[I_\emptyset]$ of $\Lb^{(r)}_K$.
Write $|\beta\rangle$ for the $r$-th tensor power of
the element $[I_\emptyset]$ in $\Lb^{(1)}_K$.
We view $|\beta\rangle$ as an element of $\Lb^{(1^r)}_K$.
The following is one of the main results of this paper.

\begin{theo}
\label{Theo:Main}
The representation $\pi^{(r)}$ of $W_k(\gen\len_r)$ on
$\Lb^{(r)}_K$ is isomorphic to the Verma module
whose highest weight is given by the following rules
$$\gathered
\pi^{(r)}(W_{d,0})|0\rangle=w_d|0\rangle,
\qquad
\pi^{(r)}(W_{d,l})|0\rangle=0,
\qquad l\geqslant 1,\\
w_1=\sum_{i=1}^r\langle b^{(i)},\beta\rangle,\qquad
w_d=-\kappa\!\sum_{i_1<i_2<\cdots<i_d}\prod_{t=1}^d\Big(
\langle h^{(i_t)},\beta\rangle+(d-t)\xi/\kappa\, \Big),\qquad d\geqslant 2.
\endgathered$$
This Verma module is irreducible.
Further, for  $l\geqslant 0$ and $d\in [2,r]$ we have
\begin{equation}
\label{unitarityW}
\gathered
\pi^{(r)}(W_{1,-l})^*=
(-1)^{rl}\pi^{(r)}(W_{1,l}),\qquad
\pi^{(r)}(W_{d,-l})^*=
(-1)^{rl+d}\pi^{(r)}(W_{d,l}).
\endgathered
\end{equation}
\end{theo}

\begin{proof}
The homomorphism
$\Theta:\SH^{(r)}_K\to\scrU(W_k(\gen\len_r))$
is compatible with the $\ZZ$-gradings.
Therefore $\Lb^{(r)}_K$ is 
a $\N$-graded $\scrU(W_k(\gen\len_r))$-module.
Thus $|0\rangle$ is a highest weight vector of $\Lb^{(r)}_K$, because it has the degree 0.
Next, we must prove that $|0\rangle$ is a generator of $\Lb^{(r)}_K$ over
$\scrU(W_k(\gen\len_r))$. Since $\Lb^{(r)}_K$ is admissible and $\SH^{(r)}_K$ is degreewise dense in 
$\scrU(W_k(\gen\len_r)),$ it is enough to prove the following.

\begin{lem}\label{lem:cyclicr} We have
$\Lb^{(r)}_K=\rho^{(r)}\big(\SH^{(r)}_K\big)|0\rangle.$
\end{lem}

\begin{proof}
We must check that $[I_\lambda]$ 
belongs to the right hand side for each $\lambda$.
We proceed by induction on the weight $|\lambda|$ of the $r$-partition $\lambda$.
Assume that $|\lambda|=n$ and that $[I_\mu]$ belongs to 
$\rho^{(r)}\big(\SH^{(r)}_K\big)|0\rangle$ whenever $|\mu|<n$.
The formulas from Section \ref{app:B} imply that there is a $r$-partition $\mu$ of $n-1$ 
such that the coefficient of $[I_\lambda]$
in $\rho^{(r)}(D_{1,l})([I_\mu])$ is non zero (in $K_r$) for some $l\in\N$.
Next, we have
\begin{equation}
\rho^{(r)}(D_{0,l+1})([I_\lambda])=\sum_a\sum_{s\in\lambda^{(a)}}(c_a(s)/x)^l\,[I_\lambda],
\qquad l\geqslant 0.
\end{equation}
We can regard $[I_\lambda]$ as the set
\begin{equation}
\{c_a(s)/x\,;\,a=1,\dots,r,\,s\in\lambda^{(a)}\}.
\end{equation}
Then, the action of $D_{0,l+1}$ on $[I_\lambda]$ 
is simply the evaluation of the $l$-th power sum
polynomial on the $K_r$-point $[I_\lambda]$ of $(K_r)^n/\Sen_n.$ Since all these points are distincts,
by Hilbert's Nullstellensatz, for each $\lambda$ there is a polynomial $f$ in the $D_{0,l+1}$'s 
such that $f([I_\lambda])=1$ and $f([I_\sigma])=0$ for any $r$-partition $\sigma$
of $n$ different from $\lambda$.
This finishes the proof.

\end{proof}

\vspace{.1in}

Next, the graded dimension of $\Lb^{(r)}_K$ 
is given by the number of $r$-partitions.
Therefore,
the previous arguments imply that $\Lb^{(r)}_K$ is a Verma module with highest weight vector $|0\rangle$.
Now, let us compute the weight of 
$|0\rangle$.
We claim that it is the same as the weight of the element
$|\beta\rangle$ in $\Lb^{(1^r)}_K$.
The later has been computed in Proposition
\ref{prop:pibeta} because $\Lb^{(1^r)}_K$ is isomorphic to
$\pi^{(1^r)}$ as a $W_k(\gen\len_r)$-module by Corollary \ref{8.7:cor}.
So the claim implies the first part of the theorem.
To prove the claim observe first that we have

\begin{lem}
\label{lem:weight2}
(a) We have
$\rho^{(r)}(D_\xb)|0\rangle=0$ for $\xb\in\mathscr E^-$.

(b) We have 
$\rho^{(1^r)}(D_\xb)|0\rangle=0$ for $\xb\in\mathscr E^-$.
\end{lem}

\begin{proof}
Part $(a)$ follows from \eqref{3.6:rem},
and $(b)$ from $(a)$ and Lemma \ref{lem:coproduct-D0n}.
\end{proof}

\noindent 
Now, for each $d\geqslant 1$, we fix an element $W'_{d,0}$ in 
$\Uen(\SH^{(r)}_K)$ which is taken to $W_{d,0}$ by the map $\Theta$
in Lemma \ref{8.6:lem2}. We must prove that it acts
in the same way on the vacua of $\Lb^{(r)}_K$ and $\Lb^{(1^r)}_K$.
By Proposition \ref{1.8:prop1} the element $W'_{d,0}$ is an infinite sum of 
monomials 
\begin{equation}
D_{k_1,l_1}D_{k_2,l_2}\cdots D_{k_r,l_r},\qquad (k_s,l_s)\in\mathscr E,\qquad
k_1+k_2+\dots+k_r=0,
\end{equation}
where the $D_{0,l}$'s and the $D_{-1,l}$'s are on the right.
Thus the claim follows from Lemma \ref{lem:weight2}.

Now, we must check that $\rho^{(r)}$ is irreducible.
It is enough to check that $\Lb^{(r)}_K$
is irreducible as a $\SH^{(r)}_K$-module.
The bilinear form $(\bullet,\bullet)$ on $\Lb^{(r)}_K$ is nondegenerate,
because the elements $[I_\lambda]$ form an orthogonal basis. 
Further, by Lemmas \ref{lem:cyclicr} and \ref{lem:weight2}, we have
\begin{equation}
\Lb^{(r)}_K=K_r|0\rangle\oplus\rho^{(r)}\big(\SH^{(r),>}_K\big)|0\rangle,
\end{equation}
Thus, by Proposition \ref{prop:unitarity},
any element in $\Lb^{(r)}_K$ which is killed by 
$\rho^{(r)}(\SH^{(r),<}_K)$ is proportional to $|0\rangle$.
This implies that $\Lb^{(r)}_K$ does not contain any proper
$\SH^{(r)}_K$-submodule.

Finally, we must prove \eqref{unitarityW}. By Proposition 
\ref{prop:generators/involution2} there is a unique anti-involution
\begin{equation}
\omega:\SH^{(r)}_K\to\SH^{(r)}_K,\quad
D_{l,d}\mapsto (-1)^{(r-1)l}x^ly^lD_{-l,d},\quad d,l\geqslant 0.
\end{equation}
Further, by Proposition \ref{prop:unitarity} we have
\begin{equation}\label{8.124}
\rho^{(r)}(u)^*=\rho^{(r)}(\omega(u)),\qquad u\in\SH^{(r)}_K.
\end{equation}
Next, recall that
$\Lb^{(1^r)}_K=\big(\Lb^{(1)}_{K_r}\big)^{\otimes r}$
and that $\Lb^{(1)}_K$ is equipped with the pairing in \eqref{pairing}.
Thus we can equip $\Lb^{(1^r)}_K$ with the unique $K_r$-bilinear form such that
\begin{equation}\label{8.128}
(u_1\otimes\cdots\otimes u_r,v_r\otimes\cdots\otimes v_1)=(u_1,v_1)\cdots(u_r,v_r),\qquad
u_i,v_i\in\Lb^{(1)}_{K_r}.
\end{equation}
Let $f^*$ denote the adjoint of a $K_r$-linear operator $f$ on $\Lb^{(1^r)}_K$
with respect to this pairing. Note that we used the same symbol for the adjoint with respect to the pairing on
$\Lb^{(r)}_K$ in Section \ref{sec:pairing}.
We claim that
\begin{equation}\label{8.125}
\rho^{(1^r)}(u)^*=\rho^{(1^r)}(\varpi(u)),\qquad u\in\SH^{(r)}_K,
\end{equation}
where $\varpi$ is the anti-involution
\begin{equation}
\varpi:\SH^{(r)}_K\to\SH^{(r)}_K,\quad
D_{l,d}\mapsto x^ly^lD_{-l,d},\quad d,l\geqslant 0.
\end{equation}
Indeed, it is enough to prove \eqref{8.125} for $u=D_{l,0},$ $D_{0,2}$. Then, it follows from the formulas
\begin{equation}
\gathered
\rho^{(1^r)}(D_{l,0})=\sum_{i=1}^r\rho^{(1)}(D_{l,0})^{(i)},\\
\rho^{(1^r)}(D_{0,2})=\sum_{i=1}^r\rho^{(1)}(D_{0,2})^{(i)}+
\xi\sum_{l\geqslant 1}\sum_{i<j}l\kappa^{1-l}\rho^{(1)}(D_{-l,0})^{(i)}\rho^{(1)}(D_{l,0})^{(j)}\endgathered
\end{equation}
which are proved in Theorem \ref{7.7:thm2}, and from the formulas
\begin{equation}
\gathered
\rho^{(1)}(D_{l,0})^*=x^ly^l\rho^{(1)}(D_{-l,0}),\qquad
\rho^{(1)}(D_{0,2})^*=\rho^{(1)}(D_{0,2}),
\endgathered
\end{equation}
which follows from \eqref{8.124}.
On the other hand, there is a unique anti-involution 
\begin{equation}
\gathered
\varpi:\Uen(W_k(\gen\len_r))\to\Uen(W_k(\gen\len_r)),\\
W_{d,-l}\mapsto(-1)^{l+d}W_{d,l},\quad 
W_{1,-l}\mapsto (-1)^lW_{1,l},\quad d\geqslant 2,\quad l\geqslant 0.
\endgathered
\end{equation}
By \eqref{piL} we have $\Lb_K^{(1^r)}=\pi^{(1^r)}$.
Let $\pi^{(1^r)}$ denote also the map $\Uen(W_k(\gen\len_r))\to\End(\pi^{(1^r)})$.
An easy computation using \eqref{miura} yields
\begin{equation}\label{8.129}
\pi^{(1^r)}(u)^*=\pi^{(1^r)}(\varpi(u)),\qquad u\in\Uen(W_k(\gen\len_r)).
\end{equation}
Finally, by Corollary  \ref{8.7:cor} we have
\begin{equation}
\rho^{(1^r)}=\pi^{(1^r)}\circ\Theta,\qquad \rho^{(r)}=\pi^{(r)}\circ\Theta,
\end{equation}
and $\Theta$ is compatible with the rank grading on $\SH^{(r)}_K$ and the conformal grading
on $\Uen(W_k(\gen\len_r))$.
Therefore, comparing \eqref{8.124}, \eqref{8.125} and \eqref{8.129}, we get
\eqref{unitarityW}.

\end{proof}

%\vspace{.1in}

%\begin{ex}
%For $r=2$ we get $\beta=(-\eps_1,\xi-\eps_2)/\kappa$ and
%\begin{equation}
%\gathered
%w_1=-\frac{(x+y)-(e_1+e_2)}{y},\quad
%w_2=\frac{(x+y)^2-(e_1-e_2)^2}{4xy},\quad
%C_Q=13+6\Big(\frac{x}{y}+\frac{y}{x}\Big).
%\endgathered
%\end{equation}
%These are the expected values in the AGT conjecture, see, e.g., 
%\cite[sec.~6.3]{BFRF}.
%\end{ex}

\vspace{.2in}

\section{The Gaiotto state}

\vspace{.15in}

\subsection{The definition of the element $G$}

Let $[M_{r,n}]$ denote the fundamental class of $M_{r,n}$. 
It is characterized, up to a scalar,
by the fact that it lies in $\Lb^{(r)}_n$ and has
the cohomological degree zero. 
Further, we have the following formula 
\begin{equation}
[M_{r,n}]=\sum_{\lambda} \eu_{\lambda}^{-1} [I_{\lambda}],
\end{equation}
where the sum runs over all $r$-partitions of size $n$. 
We define an element in $\widehat\Lb^{(r)}_K=\prod_{n \geqslant 0} \Lb^{(r)}_n$ by
\begin{equation}
G=\sum_{n \geqslant 0} [M_{r,n}].
\end{equation}

\begin{prop}\label{prop:Whittaker-SH}
The element $G$ satisfies the following properties 
\begin{equation}\label{E:whittaker1}
\rho^{(r)}(D_{-l,d})(G)=0, \qquad 
l \geqslant 1, \qquad d\in[0,r-2],
\end{equation}
\begin{equation}\label{E:whittaker2}
\rho^{(r)}(D_{-1,r-1})(G)=x^{-r}y^{-1} G, 
\qquad 
\rho^{(r)}(D_{-l,r-1})(G)=0, \qquad l \geqslant 2,
\end{equation}
\begin{equation}\label{E:whittaker3}
\rho^{(r)}(D_{-1,r})(G)=-x^{-r}y^{-1}\big(\sum_i \eps_i\big) G.
\end{equation}
\end{prop}
\begin{proof} See Appendix~G.
\end{proof}

\vspace{.1in}

\begin{rem} It is not true that $G$ is an eigenvector for the operators 
$\rho^{(r)}(D_{-1,l})$ with $l >r$. 
\end{rem}

\vspace{.15in}

\subsection{The Whittaker condition for $G$}
Now, we give a characterization of $G$ using only the representation
$\pi^{(r)}$ of $W_k(\gen\len_r)$. 
Let $\chi$ be a character of the
subalgebra of $\Uen(W_k(\gen\len_r))$ generated by $W_{d,l}$ for 
$l\geqslant 1$ and $d\in [1, r]$. 

\begin{df}
An element $v$ of $\widehat\Lb^{(r)}_{K}$
is a \textit{Whittaker vector} for $W_k(\gen\len_r)$ associated with $\chi$ if
\begin{equation}
\pi^{(r)}(W_{d,l})v=\chi(W_{d,l}) v, 
\qquad d\in [1,r],\qquad l\geqslant 1. 
\end{equation}
\end{df}

\begin{prop}
\label{prop:8.12}
The element $G$ is a Whittaker vector for $W_k(\gen\len_r)$ 
associated with the character 
\begin{equation}\label{whit}
\chi(W_{r,1})=y^{1-r}x^{-1}, \qquad \chi(W_{d,l})=0 \;\;\; \text{if}\;\; 
d \neq r\;\;\text{or}\;\; d=r,\, l \neq 1.
\end{equation}
It is characterized, up to a scalar, by this property.
\end{prop}

\begin{proof}
%The order filtration on $\Uen(\SH^{(r)}_K)$ is given by 
%putting $D_{r,d}$ in degree $d$, see Remark \ref{rem:8.11}. 
%Thus \eqref{8.56}
We will work in the representation $\Lb_K^{(r)}$ and omit to write the symbol $\rho^{(r)}$ to unburden the notations.
Equation (\ref{E:whittaker1}) implies that
\begin{equation}
\Uen(\SH^{(r)}_K)[\leqslant\!r-2]\cdot G=0.
\end{equation}
By Lemma \ref{8.6:lem2}
the map $\Theta$ gives a surjective 
morphism of degreewise topological $K_r$-algebras
$
\Theta:\Uen(\SH^{(r)}_K)\to\scrU(W_k(\gen\len_r))
$
which is compatible with the order filtrations.
This implies that 
\begin{equation}
\Uen(W_k(\gen\len_r))[\leqslant\!r-2] \cdot G=0.
\end{equation}
Since $W_{d,l}$ has the order $d-1$ by \eqref{8.1:filt}, this implies that
\begin{equation}
W_{d,l}\cdot G=0,\qquad
\qquad d<r.
\end{equation}
Let us now assume that $d=r$. It will be convenient to use the elements 
$Y_{l,n}$ from Section~1.9. We have
\begin{equation}
\SH_K^{-}[l, \leqslant \!n]=\SH_K^-[l, <\! n] \oplus K_r Y_{-l, n}
\end{equation}
with
\begin{equation}
Y_{-l,n}=\begin{cases} [D_{-1,0}, D_{1-l,n+1}] & \text{if}\; l-1=n,\\ 
[D_{-1,1}, D_{1-l,n}] & \text{if}\; l-1 \neq n,\end{cases} 
\end{equation}
Assume first that $r=2$. Then 
\begin{equation}
Y_{-2,1} \cdot G=[D_{-1,0}, D_{-1,2}] \cdot G=0.
\end{equation}
More generally, we have
\begin{equation}
Y_{-l-1,1} \cdot G=[D_{-1,1}, Y_{-l,1}] \cdot G=0
\end{equation}
for any $l \geqslant 2$. Next, let us assume
that $r >2$. Then
\begin{equation}
\gathered
Y_{-2,r} \cdot G=[D_{-1,1}, D_{-1,r}] \cdot G=0,\\
Y_{-2,r-1} \cdot G=[D_{-1,1}, D_{-1,r-1}] \cdot G=0
\endgathered
\end{equation}
and, acting by $\ad(D_{-1,1})$,
\begin{equation}
Y_{-l,r} \cdot G=Y_{-n,r-1} \cdot G=0, \qquad 2 \leqslant l \leqslant r, 
\;\;2 \leqslant n \leqslant r-1.
\end{equation}
Therefore we have
\begin{equation}
Y_{-r,r-1} \cdot G=[D_{-1,0}, Y_{1-r,r}] \cdot G=0
\end{equation}
from which we deduce, by acting by $\ad(D_{-1,1})$ again, 
that $Y_{-l,r-1} \cdot G=0$ for $l > r$.
We have thus proved that $\Uen(\SH^{(r)})[-l, \leqslant r-1]\cdot G=0$ 
for $l >1$, and hence that
\begin{equation}
\Uen(W_k(\gen\len_r))[l,\leqslant\!r-1] \cdot G=0, \qquad l >1.
\end{equation}
In particular, we have $W_{r,l} \cdot G=0$ for $l >1$. 
To prove that $G$ is a Whittaker vector, it now remains to
compute $W_{r,1} \cdot G$. We will do this by expressing $W_{r,1}$ 
in terms of the elements $D_{l,n}$ up
to terms of order $<r-1$. We will use the representation
$\rho^{(1^r)}$ of $\SH^{(r)}$. Let us first introduce some notation. If 
$f=f(z_1, \ldots, z_r)=
\sum_{\underline{i}} a_{\underline{i}} z_1^{i_1} \cdots z_r^{i_r}$ 
is a polynomial
then we write 
\begin{equation}
:\!f (\underline{z}) \!:=\sum_{\underline{i}} 
a_{\underline{i}} :\! b^{(1)}(z)^{i_1} \cdots b^{(r)}(z)^{i_r}\!:.
\end{equation}
Further,
if $u(z)=\sum_i u_i z^{-i-d}$ is a field of conformal dimension $d$ then we write $\big(u(z)\big)_i=u_i$. 
By Claim \ref{slem:1} we have,
up to terms of order $<r-1$ in the order filtration on $\Uen(W_k(\gen\len_1))^{\hat{\otimes} r}$
\begin{equation}\label{E:Whit4}
W_r(z)=-\kappa \sum_{s=0}^r (-r)^{s-r} 
:\! p_1(\underline{z})^{r-s} e_s(\underline{z}) \!:
\end{equation}
while by Claim \ref{slem:2} and \eqref{cpm1d} we have, again up to order $<r-1$,
\begin{equation}\label{E:Whit5}
\rho^{(1^r)}(D_{-1,d})=-\frac{y^{-1}\kappa^{d+1}}{d+1} \big( :\! p_{d+1}(\underline{z})\!:\big)_{-1}, \qquad 
\rho^{(1^r)}(D_{0,d})=\frac{\kappa^{d+1}}{d(d+1)} \big( :\! p_{d+1}(\underline{z})\!:\big)_{0}.
\end{equation}
Combining \eqref{E:Whit4} and \eqref{E:Whit5} and using the identity 
\begin{equation}
(p_r, p_1^{r-s}e_s)=\delta_{r,s} (-1)^{r-1}/r
\end{equation}
from the theory of 
symmetric functions we deduce that, up to terms of order $<r-1$,
\begin{equation}
\pi^{(1^r)}(W_{r,1})=(-1)^{r-1}y\kappa^{1-r} \rho^{(1^r)}(D_{-1,r-1}) + u
\end{equation}
where $u$ is a linear combination of monomials 
$\rho^{(1^r)} (D_{0,d_1} \cdots D_{0,d_s} D_{-1,d})$ with $d <r-1$.
Acting on $G$ and using 
Proposition \ref{prop:Whittaker-SH} we obtain
\begin{equation}
W_{r,1} \cdot G = y^{1-r}x^{-1} G.
\end{equation}
To finish the proof of Proposition \ref{prop:8.12}, we now show that there is, 
up to a scalar, at most one Whittaker vector
of $W_k(\gen\len_r)$ in $\Lb^{(r)}_K$ associated with the character $\chi$. 
So, assume that $v=\sum_{n\geqslant 0}v_n$ is a Whittaker vector, with
$v_n\in\Lb_{n,K}^{(r)}$ for all $n$. Assume also that we have proved that
for some $n_0\geqslant 1$ we have
$v_n=[M_{r,n}]$ for all $n<n_0$. Then the equation \eqref{whit} for $G$ and $v$
gives the following identities in $\Lb^{(r)}_{n_0-l,K}$ for any $l\geqslant 1$
\begin{equation}
\pi^{(r)}(W_{d,l})(v_{n_0}-[M_{r,n_0}])=0,\qquad
d\in [1,r].
\end{equation}
Since $\Lb^{(r)}_K$ is irreducible as a $W_k(\gen\len_r)$-module,
this implies that $v_{n_0}=[M_{r,n_0}]$.
The proposition follows easily.

\end{proof}

\vspace{.1in}

\vspace{.3in}

{\centerline{\textbf{Acknowledgments}}}

\vspace{.1in}

We would like to thank T. Arakawa, A. Braverman, D. Maulik, H. Nakajima,
A. Okounkov, V. Pasquier and S. Yanigada for useful discussions and correspondences. 
O.S. would especially like to thank H. Nakajima
for his suggestion to construct a coproduct on
the algebra $\SH^{\cb}$, leading to the free field realization. 
Part of this work was done while the second author was visiting Columbia university. 
E.V. is grateful to A. Okounkov for his kind invitation.

\newpage

\appendix

\section{Some useful formulas}
\label{A:comb}
In this section we gathered a few formulas concerning the functions
$\lg_l$, $\varphi_l$ and $\phi_l$ which are used throughout the paper.
Recall that, for $l\geqslant 0$, we have
\begin{equation}
\label{formule/lg}
\lg_0(s)=-\log(s),\qquad\lg_l(s)=(s^{-l}-1)/l, \qquad l\neq 0,
\end{equation}
\begin{equation}
\label{formule/phi}
\varphi_l(s)=s^l\lg_l(1-s)+s^l\lg_l(1+\kappa s)+s^l\lg_l(1+\xi s)-
s^l\lg_l(1+s)-s^l\lg_l(1-\kappa s)-s^l\lg_l(1-\xi s),
\end{equation}
\begin{equation}
\label{formule/phi2}
\phi_l(s)=s^l\lg_l(1+\xi s).
\end{equation}
In particular, we have
\begin{equation}\label{D7}
\gathered
\varphi_l(s)=(l+2)(l+1)\kappa\xi s^{l+3}+\Oc(s^{l+4}),\\
\phi_l(s)=
-\xi s^{l+1}+(l+1)\xi^2s^{l+2}/2-(l+2)(l+1)\xi^3s^{l+3}/6+\Oc(s^{l+4}).
\endgathered
\end{equation}
Note also that for each $a,b$ we have
\begin{equation}
\label{formule/log}
\log(1+s(a+b))=
\sum_{l\geqslant 1}(-1)^{l+1}(a+b)^ls^l/l=
\sum_{l\geqslant 0}(-1)^{l+1}a^l\,s^l\,\lg_l(1+bs).
\end{equation}

\begin{rem}\label{rem:combA}
Note that for each $l\in\N$ there is a non-zero constant 
$a\in F$ such that $\phi_{l+2}-a\varphi_l$ is a formal series in
$\varphi_{l+1}, \varphi_{l+2},\dots$.
\end{rem}

\vspace{.2in}

\section{Proof of Proposition \ref{P:<0>}}
\label{app:A}

\vspace{.1in}

\subsection{The reduction}
We begin with the proof of relation (\ref{E:rel3}).
We will use the polynomial representation $\rho_n$ of $\SH_n$ in
$\V_n$ in order to compute the expression \eqref{E:kop1} below. 
However, because the theory of Jack 
polynomials is only well-behaved for symmetric polynomials 
(as opposed to symmetric Laurent polynomials), 
we will need to somehow restrict ourselves to the subspace 
$\Lambda_n$. 
For this we will use the inner automorphism 
$$\sigma=\text{Ad}(e_n^{(n)}) \in \Aut(\SH_n),\qquad
e_n^{(n)}=X_1X_2\cdots X_n.$$
Note that $\V_n=\Lambda_n[(e_n^{(n)})^{-1}]$

\vspace{.1in}

\begin{lem}\label{lem:kop2} 
Let $U \subset \SH_n$ be a finite dimensional subspace which is 
stable under $\sigma$ and let
$u\in U$. 
If $u((e_n^{(n)})^k\Lambda_n)=\{0\}$ for some integer $k$ then $u=0$.
\end{lem}

\begin{proof} For $k \in \Z$ let $Z_k \subset \SH_n$ 
be the annihilator of $(e_n^{(n)})^k\Lambda_n$. 
We have $Z_k \subset Z_{k+1}$ and $\sigma(Z_k)=Z_{k+1}$.
Further, since $\rho_n$ is faithful we have also
$\bigcap_{k} Z_k=\{0\}$. 
Thus, since $U$ is finite dimensional, there exists 
$l \in \Z$ such that $U \cap Z_l=\{0\}$. 
But $\sigma(U \cap Z_k)=U \cap Z_{k+1}$ for all $k$.
Thus, we have $U \cap Z_k=\{0\}$ for all $k$. \end{proof}

\vspace{.1in}

For $k \geqslant 0$ let $A(k)$ be the subspace of elements of 
$F[D_{0,1}^{(n)}, \ldots, D_{0,k}^{(n)}]$ 
of degree $k$. Here $D_{0,l}^{(n)}$ is in degree $l$.
Consider the following finite-dimensional subspace of $\SH_n$
$$B(k)=\big[ D_{-1,0}^{(n)},[H_k, D_{1,0}^{(n)}]\big] + A(k).$$
From \eqref{E:rel'1} and \eqref{E:rel'2} it follows that 
for $l, k \geqslant 0$ we have
\begin{equation}\label{E:kop1}\begin{split}
[D_{-1,k}^{(n)}, D_{1,l}^{(n)}]
&=[[D_{-1,0}^{(n)},D_{0,k}^{(n)}],D_{1,l}^{(n)}]\\
&=[D_{-1,0}^{(n)},[D_{0,k}^{(n)},D_{1,l}^{(n)}]]\\
&=[D_{-1,0}^{(n)},D_{1,l+k}^{(n)}]\\
&=[D_{-1,0}^{(n)},[D_{0,l+k}^{(n)},D_{1,0}^{(n)}]].
\end{split}\end{equation}
Thus, both sides of \eqref{E:rel3} belong to $B(k+l)$. 
One checks that 
$$\sigma(y_l)=y_l-1,\qquad
\sigma(D_{\pm 1, 0}^{(n)})=D_{\pm 1,0}^{(n)},$$ 
from which we see that the subspace $B(k+l)$ is stable under $\sigma$. 
By Lemma~\ref{lem:kop2} we see that it is 
enough to check
(\ref{E:rel3}) in $(e_n^{(n)})^k \Lambda_n$ for some $k\in\Z$. 
This is what we will do in the next paragraphs.

\vspace{.15in}

\subsection{The Pieri formula for $e_{-1}$} 
We state here a Pieri formula for the multiplication of Jack 
polynomials by the elementary symmetric Laurent polynomial 
$$e_{-1}^{(n)}=X_1^{-1} + \cdots + X_n^{-1}.$$  
Since the product $e_{-1}^{(n)} \cdot J_{\lambda}^{(n)}$ may not be 
a polynomial, we need to restrict the range of application. 
For $\lambda=(\lambda_1, \ldots, \lambda_n)$ we write
$$\lambda-(1^n)=(\lambda_1-1, \ldots, \lambda_n-1),\qquad (1^n)=(1,1,\dots,1).$$
We'll use the following result \cite[sec.~5]{Stanley}.

\begin{lem}\label{lem:kop25} Let $\lambda$ be a partition of 
length $n$. We have
$$J_{\lambda}^{(n)}=c_{\lambda}(\kappa)\, e_n^{(n)} \,
J_{\lambda-(1^n)}^{(n)},\qquad
c_{\lambda}(\kappa)=\prod_{i=1}^n h^{\lambda}(i,1).$$
\end{lem}
 
\noindent
Thus, we have
$e_n^{(n)} \Lambda_n=
\bigoplus_{\lambda} F J_{\lambda}^{(n)},$
where the sum runs over the partitions of length $n$.

\vspace{.1in}

\begin{prop}\label{E:Pierin} Let $\lambda$ be a partition of length $n$. We have
$$e_{-1}^{(n)} J_{\lambda}^{(n)}=
\sum_{\mu \subset \lambda} \phi_{\lambda\setminus\mu}\, J_{\mu}^{(n)}$$
where the sum ranges over all $\mu \subset \lambda$ with 
$|\mu|=|\lambda|-1$ and where
$$\phi_{\lambda\setminus\mu}=
\frac{1}{\kappa} \frac{{h}^{\lambda}(1,j)}{h^{\mu}(1,j)}
\prod_{s \in C_{\lambda\backslash \mu}} \frac{h^{\lambda}(s)}{h^{\mu}(s)} 
\prod_{s \in R_{\lambda \backslash \mu}} \frac{h_{\lambda}(s)}{h_{\mu}(s)},\quad
j=y(\lambda \backslash \mu)+1.$$
\end{prop}

\begin{proof}We have
\begin{equation}\label{E:pierideux}
\begin{split}
e_{-1}^{(n)} J_{\lambda}^{(n)}&=
c_{\lambda}(\kappa)\, e_{-1}^{(n)}\, e_n^{(n)}\, J_{\lambda-(1^n)}^{(n)}
=c_{\lambda}(\kappa)\, e_{n-1}^{(n)}\, J_{\lambda-(1^n)}^{(n)}.
\end{split}
\end{equation}
This allows us to use the Pieri formulas for the multiplication by 
$e_{n-1}^{(n)}$ 
given in \cite[thm.~6.1]{Stanley},
using the duality \cite[thm.~3.3]{Stanley}. We leave the details to the reader. 
\end{proof}

\vspace{.15in}

\subsection{Proof of Proposition \ref{P:<0>}}
For a linear operator $f$ on $\Lambda_n$ we define 
$\langle \mu \; ;\; f \; ;\; \lambda \rangle$ by 
$$f(J_{\lambda}^{(n)})=\sum_{\mu} \langle \mu \; ;\; f \; ;\; \lambda \rangle 
\,J_{\mu}^{(n)}.$$
Using the explicit expressions of the Pieri rules for $e_{\pm 1}^{(n)}$ 
it is easy to check that 
\begin{equation}\label{E:kop3}
\langle \mu \; ;\; [D_{-1,0}^{(n)},[D_{0,l}^{(n)}, D_{1,0}^{(n)}]]\; ;\; 
\lambda \rangle=0 \end{equation}
for any $\mu \neq \lambda$ with $l(\lambda)=n$,
compare \cite[app.~ A]{SV2}. In the remainder of this paragraph, 
we compute precisely the coefficient arising in \eqref{E:kop3}.
We will use the following notation introduced by Garsia and Tesler.
Label the removable boxes of $\lambda$ by $B_1, B_2,
\ldots, B_r$ from left to right, and the addable boxes  $A_0,
\ldots, A_{r}$ also from left to right. Set 
$I=\{0, \ldots, r\}$, $J=\{1, \ldots, r\}$ and 
\begin{equation}
\label{ux}
a_i=c(A_i), \quad  b_j= c(B_j),\quad i\in I, \quad j\in J.
\end{equation}
Observe that we have
$$x(A_0)=y(A_r)=0,\quad
x(A_j)=x(B_j)+1,\quad y(A_{j-1})=y(B_{j})+1,\quad j\in J.$$ 

\begin{ex} Here is an example
with $\lambda=(4,2^2,1^2)$

\centerline{
\begin{picture}(250,150)
\thicklines
\put(0,20){\line(0,1){100}}
\put(0,20){\line(1,0){80}} 
\put(0,40){\line(1,0){80}}
\put(0,60){\line(1,0){40}}
\put(0,80){\line(1,0){40}}
\put(0,120){\line(1,0){20}}
\put(0,100){\line(1,0){20}}
\put(20,20){\line(0,1){100}}
\put(40,20){\line(0,1){60}}
\put(60,20){\line(0,1){20}}
\put(80,20){\line(0,1){20}}
\put(5,108){{${B_1}$}}
\put(25,68){{${B_2}$}}
\put(65,28){{${B_3}$}}
\put(85,28){{${A_3}$}}
\put(5,128){{${A_0}$}}
\put(25,88){{${A_1}$}}
\put(45,48){{${A_2}$}}
\thinlines
\multiput(0,120)(0,4){5}{\line(0,1){2}}
\multiput(20,120)(0,4){5}{\line(0,1){2}}
\multiput(0,140)(4,0){5}{\line(1,0){2}}
\multiput(20,100)(4,0){5}{\line(1,0){2}}
\multiput(40,80)(0,4){5}{\line(0,1){2}}
\multiput(60,40)(0,4){5}{\line(0,1){2}}
\multiput(40,60)(4,0){5}{\line(1,0){2}}
\multiput(100,20)(0,4){5}{\line(0,1){2}}
\multiput(80,20)(4,0){5}{\line(1,0){2}}
\multiput(80,40)(4,0){5}{\line(1,0){2}}
\put(160,80){{$b_3=3$}} \put(160,95){{$b_2=1-2\kappa$}}
\put(160,65){{$a_0=-5\kappa$}} \put(160,110){{$b_1=-4\kappa$}}
\put(160,50){{$a_1=1-3\kappa$}} \put(160,20){{$a_3=4$}}
\put(160,35){{$a_2=2-\kappa$}} 
\end{picture}}
\vspace{.05in}
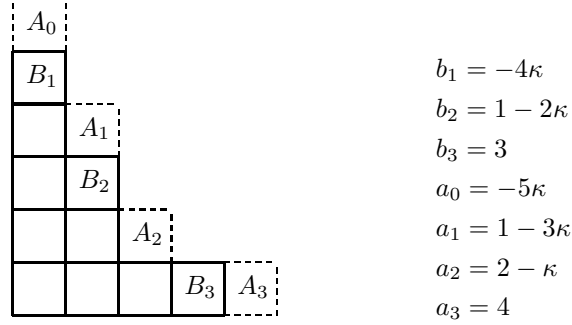
\centerline{\textbf{Figure 2.} Garsia and Tesler's variables.}
\end{ex}

\vspace{.15in}

Let us begin by rewriting the expressions appearing in the Pieri rules in terms of Garsia and Tesler's notation. 
Let $\lambda$ be a fixed partition and let 
$B_j, A_i, a_i, b_j$ be associated with $\lambda$ as above. 
A direct computation yields

\vspace{.1in}

\begin{lem}\label{lem:formuliux} For $i\in I$, $j\in J$ we have
$$\prod_{s \in C_{A_i}} \frac{h_{\lambda}(s)}{h_{\sigma}(s)} \prod_{s \in R_{A_i}} 
\frac{h^{\lambda}(s)}{h^{\sigma}(s)}=\prod_{j\in J}(a_i-\xi-b_j) 
\prod_{\substack{k\in I\setminus\{ i\}}} \frac{1}{a_i-a_k},\quad
\sigma=\lambda + A_i,$$
$$\prod_{s \in C_{B_j}}\frac{h^{\lambda}(s)}{h^{\nu}(s)} \prod_{s \in R_{B_j}} \frac{h_{\lambda}(s)}{h_{\nu}(s)}
=-\frac{1}{\kappa}\prod_{i\in I}(a_i-\xi-b_j) \prod_{\substack{k\in J\setminus\{j\}}} \frac{1}{b_k-b_j},\quad
\nu=\lambda-B_j.$$
\end{lem}

\vspace{.1in}

Set $I^\times=I\setminus\{0\}$. The above lemma yields the following.

\begin{cor} If $\lambda$ has length $n$ then, for $l\geqslant 0$, we have
\begin{equation}\label{E:kop4}
\begin{split}
\langle  \lambda \; ;\; [D_{-1,0}^{(n)},[D_{0,l}^{(n)},D_{1,0}^{(n)}]]\; ;\; \lambda \rangle &=
\sum_{i \in I^\times} a_i^l  
\prod_{k\in I^\times\setminus\{i\}} \frac{a_i-a_k+\xi}{a_i-a_k}
\prod_{j\in J} \frac{a_i-b_j-\xi}{a_i-b_j} \\
& \qquad  -\sum_{j\in J} b_j^l \prod_{i\in I^\times} \frac{b_j-a_i+\xi}{b_j-a_i} 
\prod_{k \in J\setminus\{j\}} \frac{b_j-b_k-\xi}{b_j-b_k}
\end{split}
\end{equation}
 \end{cor}

Note that in (\ref{E:kop4}) the variable $a_0$ never appears.
In fact, we have $a_0=-n\kappa$ since $l(\lambda)=n$.
Let us now form the generating series 
$$X^{(n)}(t)=\sum_{l\geqslant 0}\langle 
 \lambda \; ;\; [D_{-1,0}^{(n)},[D_{0,l}^{(n)},D_{1,0}^{(n)}]]\; ;\; \lambda \rangle\, t^l.$$
By (\ref{E:kop4}) we have
\begin{equation*}
\begin{split}
X^{(n)}(t)=&
\sum_{i \in I^\times}\frac{1}{1-a_it}
\prod_{k\in I^\times\setminus\{i\}} \frac{a_i-a_k+\xi}{a_i-a_k}\prod_{j\in J} \frac{a_i-b_j-\xi}{a_i-b_j}
 \\
& \qquad  -\sum_{j\in J} \frac{1}{1-b_jt} \prod_{i\in I^\times} \frac{b_j-a_i+\xi}{b_j-a_i} 
\prod_{k \in J\setminus\{j\}} 
\frac{b_j-b_k-\xi}{b_j-b_k}.
\end{split}
\end{equation*}

\vspace{.1in}

\begin{lem}\label{lem:kop5} Given two disjoint sets of commutative formal variables
$\{a_i\;;\;i\in I^\times\}$ and $\{b_j\;;\;j\in J\}$ we have
\begin{equation*}
\begin{split}
& \sum_{i\in I^\times} \frac{t\xi}{1-a_it}\, 
\prod_{k \in I^\times\setminus\{i\}} \frac{a_i-a_k+\xi}{a_i-a_k}\prod_{j\in J} \frac{a_i-b_j-\xi}{a_i-b_j}-
\sum_{j\in J} \frac{t\xi}{1-b_jt}\,\prod_{i\in I^\times}\frac{b_j-a_i+\xi}{b_j-a_i}
\prod_{k\in J\setminus\{j\}} \frac{b_j-b_k-\xi}{b_j-b_k}=\\
&=\prod_{i\in I^\times} \frac{1-t(a_i-\xi)}{1-ta_i}\prod_{j\in J}\frac{1-t(b_j+\xi)}{1-tb_j} -1.
\end{split}
\end{equation*}
\end{lem}

\begin{proof} Both sides of the equality are rational functions in $t$ of degree $0$, with at most simple poles. 
One checks that the poles and residues are the same. This implies the equality, 
up to a possible constant. But both sides vanish at $t=0$. So this constant is zero.
\end{proof}

\vspace{.1in}

The above lemma implies the equality
\begin{equation*}
\begin{split}
1+t\xi X^{(n)}(t)&=\prod_{i\in I^\times}\frac{1-t(a_i-\xi)}{1-ta_i}\prod_{j\in J}\frac{1-t(b_j+\xi)}{1-tb_j}\\
&=\exp\bigg(\sum_{l \geqslant 1} \big(p_l(a_i)^\times-p_l(a_i-\xi)^\times
+p_l(b_j)-p_l(b_j+\xi)\big){t^l}/l\bigg),
\end{split}
\end{equation*}
where
$$p_l(a_i)^\times=\sum_{i \in I^\times} a_i^l,\quad
p_l(a_i-\xi)^\times=\sum_{i \in I^\times}(a_i-\xi)^l,\quad 
p_l(b_j)=\sum_{j\in J}b_j^l,\quad\text{etc.}$$
The last step is to identify the expression above
with the 
eigenvalue of an element in $\SH_n^0$ on the Jack polynomial 
$J_{\lambda}(X_1, \ldots, X_n)$.
From \eqref{E:Horace} we get
$$\langle \mu\; ;\;D_{0,l}^{(n)}\; ;\; \mu \rangle=\sum_{s \in \mu}c(s)^{l-1}.$$
We'll use the following notation
$$\gathered
p_l(a_i)=\sum_{i\in I} a_i^l,\\
p_l(a_i-\xi)=\sum_{i\in I}(a_i-\xi)^l,\\
\sigma_l(x)=(x+1)^l-(x-1)^l+(x+\xi-1)^l-(x-\xi+1)^l+(x-\xi)^l-(x+\xi)^l.
\endgathered$$

\vspace{.1in}

\begin{lem}\label{lem:kop6} 
We have
\begin{equation}\label{E:kop6}
\gathered
p_l(a_i)-p_l(a_i-\xi)+p_l(b_j)-p_l(b_j+\xi)=
(-1)^{l+1}\xi^l+\sum_{s \in \lambda} \sigma_l\big(c(s)\big).
\endgathered
\end{equation}
\end{lem}

\begin{proof} 
The proof is by induction on $|\lambda|$. If $|\lambda|=0$
then $r=0$ and
$a_0=0$. Assume that (\ref{E:kop6}) holds for all partitions of
 size at most $m-1$ and let $\lambda$ be a partition
of size $m$. Let $\mu \subset \lambda$ be a subpartition of $\lambda$ of size $m-1$, and set $s=\lambda \backslash \mu$. Let $r,b_j, a_i$ and $r', b'_j,a'_i$ be associated with $\lambda$ and $\mu$ respectively. Note 
that we may have $r'=r$, $r'= r-1$ or $r'=r+1$. One checks that
$$
p_l(a_i)-p_l(a_i-\xi)+p_l(b_j)-p_l(b_j+\xi)=p_l(a'_i)-p_l(a'_i-\xi)+p_l(b'_j)-p_l(b'_j+\xi)
+ \sigma_l(c(s)),$$
which closes the induction step. We leave the details to the reader.
\end{proof}

\vspace{.1in}

Using Lemma~\ref{lem:kop6} and the fact that for $l(\lambda)=n$ we have 
$$a_{0}=-n\kappa=-\kappa D_{0,0}^{(n)},$$ we
get that the formal series $1+ \xi t X^{(n)}(t)$ is equal to
\begin{equation*}
\begin{split}
&\exp\Big(\sum_{l\geqslant 1}(-1)^{l+1}\xi^l t^l/l\Big) 
\exp\Big(\sum_{l\geqslant 1}(-1)^l\big((\xi+\kappa D_{0,0}^{(n)})^l-
(\kappa D_{0,0}^{(n)})^l\big)t^l/l\Big)
\exp \Big(\sum_{l\geqslant 1}\sum_{s\in\lambda}\sigma_l(c(s)) 
t^l/l\Big).
\end{split}
\end{equation*}

Now, from \eqref{formule/log} we get
\begin{equation}
\label{formule/log2}
\frac{1+at}{1+at+\xi t}=\exp\Bigl(\sum_{l\geqslant 0}(-1)^{l}a^l\phi_l(t)\Bigr).
\end{equation}
Using this, we may finally write 
$$\gathered
1+ \xi t X^{(n)}(t)=K(\kappa,D^{(n)}_{0,0},t) \; 
\exp \big( \sum_{l \geqslant 0}\sum_{s \in \lambda}
c(s)^{l}\, \varphi_l(t)\big),\\
K(\kappa,\omega,t)=\frac{\big(1+\xi t\big)\big(1+\kappa \omega t\big)}
{1+\xi t+\kappa\omega t}=(1+\xi t)
\exp\Bigl(\sum_{l\geqslant 0}(-1)^{l}\kappa^l\omega^l\phi_l(t)\Bigr).
\endgathered$$
Therefore, by \eqref{E:Horace}, we see that (\ref{E:rel35}) holds when applied to $J_{\lambda}$. 
Since this is true for all $\lambda$ of length $n$, 
the identity \eqref{E:rel35} holds when applied to any 
$v$ in $e_n\mathbf{V}_n^+$ by Lemma~\ref{lem:kop25}. 
But then, by Lemma~\ref{lem:kop2}, (\ref{E:rel35}) holds unconditonally. 
This concludes the proof of Proposition~\ref{P:<0>}.

\vspace{.2in}

\section{Complements on Section \ref{sec:prooflevel1}}
\label{app:B1}

\vspace{.1in}

\subsection{The canonical representation of $\widetilde\U_K^{(1),+}$ on $\widetilde\Lb^{(1)}_K$}
\label{sec:part1level1}
In this section we describe the canonical representation 
of $\widetilde\U_K^{(1),+}$ on $\widetilde\Lb^{(1)}_K$ explicitly. 
The following lemma is well-known.

\vspace{.1in}

\begin{lem} 
(a) The convolution product gives
$[I_{\lambda\mu}] \dot [I_{\nu}]=\delta_{\mu,\nu} \eu_{\nu} [I_{\lambda}].$

(b) For a $T$-equivariant vector bundle $\mathcal{V}$ over $\Hin$ of rank $r$
we have
$$\cc_l(\mathcal{V})=\sum_{\lambda \vdash n} \eu_{\lambda}^{-1} 
\cc_l(\mathcal{V}|_{I_{\lambda}}) [I_{\lambda}],\qquad l\in[1,r].$$
\end{lem}

\vspace{.1in}

From the above lemma we obtain the formulas
\begin{equation}\label{E:31}
\cc_1(\tau_{n,n+1})^l=\sum_{\mu \subset \lambda} 
\cc_1(\tau_{\mu,\lambda})^l 
\eu(N^*_{\mu,\lambda})\eu_{\mu,\lambda}^{-1}  [I_{\mu,\lambda}] 
\end{equation}
\begin{equation}\label{E:32}
\cc_1(\tau_{n+1,n})^l=\sum_{\mu \subset \lambda} 
\cc_1(\tau_{\lambda,\mu})^l 
\eu(N^*_{\lambda,\mu})\eu_{\lambda,\mu}^{-1}  [I_{\lambda,\mu}] 
\end{equation}
\begin{equation}\label{E:33}
\cc_l(\tau_{n,n})=\sum_{\mu\vdash n} 
\cc_l(\tau_{\mu,\mu}) \eu_{\mu}^{-1}  [I_{\mu,\mu}], 
\end{equation}
where the first two sums range over all pairs $\mu,\lambda$ 
with $\mu \subset \lambda$ and 
$\mu \vdash n,$ $ \lambda \vdash n+1$. Combining
the above \eqref{E:31}-\eqref{E:33} with the explicit expressions deduced 
from \eqref{E:21} and \eqref{E:26}
$$\eu_{\lambda}=\prod_{s \in \lambda} \big( l(s) y -(a(s)+1)x\big)\big( -(l(s) +1)y+a(s)x\big)$$
$$\eu(N^*_{\lambda,\mu})=\eu(N^*_{\mu,\lambda})=
\prod_{s \in \mu} \big(l_{\mu}(s) y-(a_{\lambda}(s)+1)x\big)
\big( -(l_{\lambda}(s) +1)y+a_{\mu}(s)x\big)$$
we get the following formulas 
\begin{equation}\label{E:34}
f_{1,l} ([I_{\mu}])=
y^{-1}x^{l-1}\sum_{\lambda \supset \mu} c(\lambda\backslash \mu)^{l} 
L_{\mu,\lambda}(x,y)\, [I_{\lambda}],
\end{equation}
\begin{equation}\label{E:35}
f_{-1,l} ([I_{\lambda}])=x^l\sum_{\mu \subset \lambda} 
c(\lambda\backslash \mu)^{l} L_{\lambda,\mu}(x,y)\, [I_{\mu}],
\end{equation}
\begin{equation}\label{E:36}
f_{0,l} ([I_{\lambda}])=x^l\sum_{s \in \lambda}c(s)^{l}\, [I_{\lambda}].
\end{equation}
Here $c(s)$ is defined in (\ref{E:defc}), we have
$$L_{\mu,\lambda}(x,y)=\prod_{s \in C_{\lambda \backslash \mu}}\frac{l_{\mu}(s)y-(a_{\mu}(s)+1)x}
{(l_{\mu}(s)+1)y-(a_{\mu}(s)+1)x}
\prod_{s \in R_{\lambda \backslash \mu}}\frac{(l_{\mu}(s)+1)y-a_{\mu}(s)x}{(l_{\mu}(s)+1)y-(a_{\mu}(s)+1)x}$$
and the sum in (\ref{E:34}) ranges over all $\lambda$ containing $\mu$ satisfying $|\lambda|=|\mu|+1$.
We set also
$$L_{\lambda,\mu}(x,y)=\prod_{s \in C_{\lambda \backslash \mu}}
\frac{(l_{\lambda}(s)+1)y-a_{\lambda}(s)x}{l_{\lambda}(s)y-a_{\lambda}(s)x}
\prod_{s \in R_{\lambda \backslash \mu}}
\frac{l_{\lambda}(s)y-(a_{\lambda}(s)+1)x}{l_{\lambda}(s)y-a_{\lambda}(s)x}$$
and the sum in (\ref{E:35}) ranges over all $\mu$ which are contained in 
$\lambda$ and satisfy 
$|\mu|=|\lambda|-1$.

\vspace{.15in}

\subsection{The triangular decomposition of $\widetilde\U^{(1)}_K$}
\label{app:C}
We begin with the following lemma.

\vspace{.05in}

\begin{lem}\label{L:automut}
There are one parameter subgroups 
$\tau^{\pm}~: \C \to \Aut(\widetilde\U^{(1),\pm}_{K})$ defined by
$$\tau^{\pm}_u (f_{0,l})=\tau_u (f_{0,l})=\sum_{i=0}^l \binom{l}{i} u^{l-i} f_{0,i},
\qquad
\tau^{\pm}_u(f_{\pm 1, l})=\sum_{i=0}^l \binom{l}{i} u^{l-i} f_{\pm 1, i},
\qquad l \geqslant 0.$$
\end{lem}

\begin{proof} It is enough to deal with $\tau^+$. 
By Theorem~\ref{thm:isomC/U} there is an algebra isomorphism 
$$\eta : \mathbf{SC}_K \to\widetilde\U^{(1),>}_K,\qquad \theta_l\mapsto f_{1,l},\qquad
l \geqslant 0.$$ By Corollary \ref{Cor:automshuffle}, the assignment 
$\theta_l \mapsto \sum_{i=0}^l \binom{l}{i} u^{l-i} \theta_{i}$ extends to an 
automorphism of $\mathbf{SC}_K$. This shows that $\tau^+_u$ is well-defined on 
$\widetilde\U^{(1),>}_K$. Next, since 
$\widetilde\U^{(1),0}_K = K[f_{0,l}\,;\,l\geqslant 1]$, 
the map $\tau^+_u$ is well-defined on $\widetilde\U^{(1),0}_K$ as well. 
To finish the proof, it remains to observe that we have
\begin{equation}
\widetilde\U^{(1),+}_K = \widetilde\U^{(1),0}_K \ltimes \widetilde\U^{(1),>}_K
\end{equation}
 with respect to the adjoint 
action $[f_{0,l}, f_{1,n}]=f_{1,l+n}$, and that
\begin{equation*}
\begin{split}
[\tau^+_u(f_{0,l}),\tau^+_u(f_{1,n})]&=
\sum_{i=0}^l\sum_{j=0}^n\binom{l}{i}\binom{n}{j} u^{l+n-i-j}[f_{0,i},f_{1,j}]\\
&=\sum_{k=0}^{l+n} \binom{l+n}{k} u^{l+n-k} f_{1,k}=\tau^+_u(f_{1,n+l}).
\end{split}
\end{equation*}
\end{proof}

We now turn to the proof of Proposition~\ref{P:proof2}. 
It is adapted from the proof of \cite[prop.~4.8]{SV2}. 
The same argument as for $\SHo$ implies that the multiplication map 
is surjective
$$m:\widetilde \U^{(1),>}_K \otimes_K \widetilde\U^{(1),0}_K \otimes_K 
\widetilde\U^{(1),<}_K \to\widetilde \U^{(1)}_K.$$
We only have to prove its injectivity. We argue by 
contradiction. Let $x=\sum_i P_i \otimes R_i \otimes Q_i$ be a nonzero 
homogeneous element in $\Ker(m)$. We
may assume that the elements $R_i$
are linearly independent polynomials 
in the $f_{0,l}$'s and that $P_i,Q_i\neq 0$. Multiplying by an
element of $\widetilde\U^{(1),>}_K$ or $\widetilde\U^{(1),<}_K$ if necessary, we may also assume
that $x$ is of degree zero. For all partition $\lambda$ we have
\begin{equation}\label{E:PP4}
\sum_i P_i \circ R_i \circ Q_i ([I_{\lambda}])=0.
\end{equation}
We'll apply \eqref{E:PP4} to certain partitions.
Given partitions $\lambda_1, \lambda_2, \ldots, \lambda_k$ and given an 
integer $n \gg |\lambda_1|,\ldots, |\lambda_k|,$ let the symbol
$\lambda_1\circledast\ldots\circledast\lambda_k$ denote the following partition

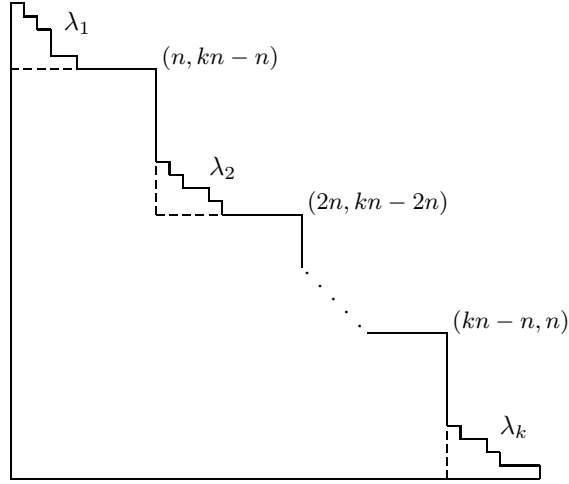
\centerline{
\begin{picture}(200,200)
\put(20,170){{$\lambda_1$}}
\put(75,115){{$\lambda_2$}}
\put(185,17){{$\lambda_k$}}
\put(57,157){{\small{$(n,kn-n)$}}}
\put(112,102){{\small{$(2n,kn-2n)$}}}
\put(167,57){{\small{$(kn-n,n)$}}}
\multiput(0,155)(5,0){5}{\line(1,0){3}}
\put(0,0){\line(0,1){180}}
\put(0,0){\line(1,0){200}}
\put(0,180){\line(1,0){5}}
\put(5,180){\line(0,-1){5}}
\put(5,175){\line(1,0){5}}
\put(10,175){\line(0,-1){5}}
\put(10,170){\line(1,0){5}}
\put(15,170){\line(0,-1){10}}
\put(15,160){\line(1,0){10}}
\put(25,160){\line(0,-1){5}}
\put(25,155){\line(1,0){30}}
\multiput(55,100)(0,5){4}{\line(0,1){3}}
\multiput(55,100)(5,0){5}{\line(1,0){3}}
\put(55,155){\line(0,-1){35}}
\put(55,120){\line(1,0){5}}
\put(60,120){\line(0,-1){5}}
\put(60,115){\line(1,0){5}}
\put(65,115){\line(0,-1){5}}
\put(65,110){\line(1,0){10}}
\put(75,110){\line(0,-1){5}}
\multiput(165,0)(0,5){4}{\line(0,1){3}}
\put(75,105){\line(1,0){5}}
\put(80,105){\line(0,-1){5}}
\put(80,100){\line(1,0){30}}
\multiput(112,78)(5,-5){5}{\circle*{1}}
\put(110,100){\line(0,-1){20}}
\put(135,55){\line(1,0){30}}
\put(165,55){\line(0,-1){35}}
\put(165,20){\line(1,0){5}}
\put(170,20){\line(0,-1){5}}
\put(170,15){\line(1,0){10}}
\put(180,15){\line(0,-1){5}}
\put(180,10){\line(1,0){5}}
\put(185,10){\line(0,-1){5}}
\put(185,5){\line(1,0){15}}
\put(200,5){\line(0,-1){5}}
\end{picture}}
\vspace{.2in}

Note that $\lambda_1\circledast\ldots\circledast\lambda_k$ is well-defined as soon as 
$n > \text{sup}_i(l(\lambda_i),l(\lambda_i'))$.
Put
\begin{equation}\label{defr}
t=\text{sup}_i (\deg(P_i)) =\text{sup}_i(-\deg(Q_i)).\end{equation}
For an operator $f$ on
$\widetilde\Lb^{(1)}_K$ we denote by $\langle \mu ; f ; \lambda\rangle$ the
coefficient of $[I_{\mu}]$ in $f([I_{\lambda}])$.
For $n$ large enough we consider the coefficients $$\langle
\bar\lambda_1\circledast\lambda_2\circledast\bar\lambda_3\, ; P_i R_i Q_i\,;
\lambda_1\circledast\lambda_2\circledast\lambda_3\rangle,\quad \bar\lambda_1 \subset
\lambda_1,\quad \lambda_3 \subset \bar\lambda_3,\quad |\lambda_1 \backslash
\bar\lambda_1|=|\bar\lambda_3\backslash \lambda_3|=t.$$
Since  $Q_i$ is an annihilation operator and $P_i$ is a
creation operator, by \eqref{defr} the only way to
obtain $\bar\lambda_1\circledast\lambda_2\circledast\bar\lambda_3$ from
$\lambda_1\circledast\lambda_2\circledast\lambda_3$ is to use all of $Q_i$ to reduce $\lambda_1$
to $\bar\lambda_1$ and to use all of $P_i$ to increase $\lambda_3$ to
$\bar\lambda_3$. Therefore we have
\begin{equation}\label{E:PP5}
\begin{gathered}
\langle \bar\lambda_1\circledast\lambda_2\circledast\bar\lambda_3\,; P_i R_i Q_i\,; 
\lambda_1\circledast\lambda_2\circledast\lambda_3\rangle=\cr
=\langle\bar\lambda_1\circledast\lambda_2\circledast\bar\lambda_3\,;P_i\,;
\bar\lambda_1\circledast\lambda_2\circledast\lambda_3
\rangle\,\langle\bar\lambda_1\circledast\lambda_2\circledast\lambda_3\,; R_i \,;
\bar\lambda_1\circledast\lambda_2\circledast
\lambda_3\rangle\cr
\langle \bar\lambda_1\circledast\lambda_2\circledast\lambda_3\,;Q_i
\,;\lambda_1\circledast\lambda_2\circledast\lambda_3\rangle.
\end{gathered}
\end{equation}
Note that \eqref{E:PP5} is zero unless $\deg(P_i)=-\deg(Q_i)=t$.

\vspace{.1in}

\begin{lem}\label{L:PK1} There are  non-zero $c,d\in K$ such that,
for $P\in\widetilde \U^{(1),>}_K[t]$ and $Q \in\widetilde \U^{(1),<}_K[-t]$,
\begin{equation}\label{E:tauy}
\langle \bar\lambda_1\circledast\lambda_2\circledast\lambda_3\,; Q\,;
\lambda_1\circledast\lambda_2\circledast\lambda_3\rangle=c \;  \langle \bar\lambda_1;
\tau^-_{2ny}(Q);\lambda_1 \rangle,
\end{equation}
\begin{equation}\label{E:rhoy}
\langle \bar\lambda_1\circledast\lambda_2\circledast\bar\lambda_3\,; P\,;
\bar\lambda_1\circledast\lambda_2\circledast\lambda_3\rangle=d \; \langle \bar\lambda_3;
\tau^+_{2nx}(P) ;\lambda_3 \rangle.
\end{equation}
\end{lem}

\begin{proof} We prove \eqref{E:tauy}. The proof of \eqref{E:rhoy} is identical.
If $Q=f_{-1,k_t} \cdots f_{-1,k_1}$ then
\begin{equation}\label{formB}
\langle\bar\lambda_1\circledast\lambda_2\circledast\lambda_3\,; Q\,;  
\lambda_1\circledast\lambda_2\circledast\lambda_3\rangle=
\sum\prod_{i=1}^t c(s_i)^{k_i}\,
L_{\mu_i\circledast\lambda_2\circledast\lambda_3,\,
\mu_{i+1}\circledast\lambda_2\circledast\lambda_3}(x,y).
\end{equation}
In \eqref{formB}
the sum runs over all sequences
$$\lambda_1 =\mu_1 \supsetneq \mu_2 \cdots
\supsetneq \mu_{t+1}=\bar\lambda_1$$
and we have set $s_i=\mu_i\backslash \mu_{i+1}$.
For partitions $\alpha \supset \beta$ with $|\alpha|=|\beta|+1$ we have
$$L_{\alpha,\beta}(x,y)=\prod_{s \in C_{\alpha \backslash \beta}} \frac{
(l_{\alpha}(s)+1)y-a_{\alpha}(s)x}{l_{\alpha}(s)y-a_{\alpha}(s)x}
\cdot \prod_{s \in R_{\alpha \backslash \beta}} \frac{
l_{\alpha}(s)y-(a_{\alpha}(s)+1)x}{l_{\alpha}(s)x-a_{\alpha}(s)y}.
$$
Next, in \eqref{formB} again, for a box $s$ in $\mu_i$ we have
$x(s)=x_{\lambda_1}(s)$ and $y(s)=y_{\lambda_1}(s)+2n$,
where $x_{\lambda_1}$ and $y_{\lambda_1}$ denote the coordinate values
when we place the origin at the
bottom left corner of $\lambda_1$, i.e., at the point $(0,2n)$,
as opposed to the point $(0,0)$ which is the origin of $\lambda_1\circledast\lambda_2\circledast\lambda_3$.
Similarly, we have
\begin{align*}
R(s)=R_{\lambda_1}(s),\qquad
C(s)=C_{\lambda_1}(s) \sqcup C'(s),\qquad
C'(s)=\{(x(s),0), \ldots, (x(s),2n-1)\}.
\end{align*}
Finally, observe that the armlength $a(s)$ or the
leglength $l(s)$ are
the same whether we consider $s$ as belonging to $\mu_i$
or to $\mu_i\circledast\lambda_2\circledast\lambda_3$.
Now, write $\sigma_i=\mu_i\circledast\lambda_2\circledast\lambda_3$ and
$\sigma=\lambda_1\circledast\lambda_2\circledast\lambda_3$. 
From the above formulae we deduce that
$$c(s_i)=c_{\lambda_1}(s_i) + 2ny,$$
\begin{equation*}
\begin{split}
\prod_{i=1}^t L_{\mu_i\circledast\lambda_2\circledast\lambda_3, \,
\mu_{i+1}\circledast\lambda_2\circledast\lambda_3}(x,y)=
\prod_{i=1}^t L_{\mu_i,\mu_{i+1}}(x,y) \prod_{s \in C'(s_i)}
\frac{(l_{\sigma_i}(s)+1)y-a_{\sigma_i}(s)x}{l_{\sigma_i}(s)y-a_{\sigma_i}(s)x}.
\end{split}
\end{equation*}
Note also that the quantity
$$\prod_{i=1}^t
\prod_{s \in
C'(s_i)}\frac{(l_{\mu_i}(s)+1)y-a_{\mu_i}(s)x}{l_{\mu_i}(s)y-a_{\mu_i}(s)x}
=\prod_{s \in \lambda_1 \backslash\bar \lambda_1}\prod_{u\in C'(s)}
\frac{(y(s)-y(u)+1)y-a_{\sigma}(u)x}{(y(s)-y(u))y-a_{\sigma}(u)x}$$
is independent
of the choice of the chain of subdiagrams $(\mu_i)$, and that
$$\sum\prod_{i=1}^t (c_{\lambda_1}(s_i)+2ny)^{k_i}\,
L_{\mu_i,\mu_{i+1}}(x,y)= \langle
\bar\lambda_1\,; \tau^-_{2ny}(Q)\,; \lambda_1 \rangle.$$ The lemma is proved.
\end{proof}

\vspace{.1in}

Using \eqref{E:PP5} and Lemma \ref{L:PK1}, the linear relation (\ref{E:PP4}) gives
\begin{equation}\label{E:PP6}
\sum_i \langle \bar\lambda_3; \tau^+_{2nx}(P_i); \lambda_3\rangle \,
\langle \bar\lambda_1\circledast\lambda_2\circledast\lambda_3\,; R_i
\,;\bar\lambda_1\circledast\lambda_2\circledast\lambda_3\rangle \, \langle
\bar\lambda_1;\tau^-_{2ny}( Q_i) ; \lambda_1\rangle=0
\end{equation}
for all $\lambda_1,\bar\lambda_1,\lambda_2, \lambda_3, \bar\lambda_3$ as above and
all large enough $n$. Choose $\lambda_3$, $\bar\lambda_3$ and
$\lambda_1$, $\bar\lambda_1$ such that 
$$\langle \bar\lambda_3\,;
\tau^+_{2nx}(P_i)\,; \lambda_3\rangle ,\langle \bar\lambda_1\,;
\tau^-_{2ny}(Q_i)\,; \lambda_1\rangle \neq 0$$ for some
$i$. Fix the integer $n$ and let us vary the partition $\lambda_2$. We abbreviate $\lambda=\emptyset\circledast\emptyset\circledast\emptyset$.
The matrix coefficient $\langle \bar\lambda_1\circledast\lambda_2\circledast\lambda_3\,; \f_{0,l}\,; 
\bar\lambda_1\circledast\lambda_2\circledast\lambda_3\rangle$ is equal to
\begin{equation*}
\langle \bar\lambda_1\,;\tau_{2ny}(\f_{0,l})\,; \bar\lambda_1\rangle +
 \langle \lambda_2\,;\tau_{nx+ny}(\f_{0,l})\,; \lambda_2\rangle +
\langle \lambda_3\,; \tau_{2nx}(\f_{0,l})\,; \lambda_3\rangle + \langle
\lambda\,; \f_{0,l}\,;
\lambda\rangle.
\end{equation*}
Recall that
$R_i=R_i(\f_{0, 1},\f_{0,2}, \ldots)$ is a polynomial in the operators $\f_{0,l}.$ 
Set
\begin{equation*}
R'_i= \langle \bar\lambda_3\,; \tau^+_{2nx}(P_i)\,; \lambda_3\rangle\, \langle
\bar\lambda_1\,;\tau^-_{2ny}(Q_i) \,; \lambda_1\rangle \, R_i(\tau_{nx+ny}(\f_{0, 1})+
\alpha_{1}, \tau_{nx+ny}(\f_{0,2})+\alpha_{ 2}, \ldots)
\end{equation*}
where
$$\alpha_l=
\langle \bar\lambda_1\,; \tau_{2ny}(\f_{0,l})\,; \bar\lambda_1\rangle +
\langle \lambda_3\,; \tau_{2nx}(\f_{0,l})\,; \lambda_3\rangle+
\langle \lambda\,; \f_{0,l}\,; 
\lambda\rangle,\qquad l \geqslant 1.$$ We may
rewrite (\ref{E:PP6}) as
\begin{equation}\label{E:PP8}
\sum_i \langle \lambda_2\,; R'_i\,; \lambda_2 \rangle =0.
\end{equation}
Since this holds for all $\lambda_2$ with $l(\lambda_2), l(\lambda_2') < n$ and $n$ is large enough, we 
deduce that 
$\sum_i R'_i  = 0$. Remember that the $R_i$'s were chosen to be linearly independent. Then the 
$R'_i$'s are also linearly independent and we arrive at a contradiction. This concludes the proof of 
Proposition~\ref{P:proof2}.\qed

\vspace{.2in}

\section{Complements on Sections \ref{sec:3} and \ref{sec:proofthmSH/Ur}}
\label{app:B}

\vspace{.1in}

\subsection{Proof of Proposition \ref{prop:glue/r}}
The proof is adapted from the computations in \cite[sec.~4.5]{VV1}. 
First, we have the following formulas, compare
\eqref{E:31}, \eqref{E:32} and \eqref{E:33},

\begin{equation}\label{Er:31}
\cc_1(\tau_{n-1,n})^l=\sum_{\sigma \subset \lambda} 
\cc_1(\tau_{\sigma,\lambda})^l \,
\eu(N^*_{\sigma,\lambda})\,\eu_{\sigma,\lambda}^{-1}\,  [I_{\sigma,\lambda}],
\end{equation}
\begin{equation}\label{Er:32}
\cc_1(\tau_{n+1,n})^l=\sum_{\lambda\subset \pi} 
\cc_1(\tau_{\pi,\lambda})^l \,
\eu(N^*_{\pi,\lambda})\,\eu_{\pi,\lambda}^{-1}\,  [I_{\pi,\lambda}], 
\end{equation}
\begin{equation}\label{Er:33}
\cc_l(\tau_{n,n})=\sum_{\lambda} 
\cc_l(\tau_{\lambda})\, \eu_{\lambda}^{-1} \, [I_{\lambda,\lambda}].
\end{equation}
Here $\sigma$, $\lambda$ and $\pi$ are $r$-partitions 
of $n-1$, $n$ and $n+1$ respectively. 
Now, assume that $\lambda$, $\mu$ are $r$-partitions 
of $n$ and that $\sigma$, $\pi$ are $r$-partitions 
of $n-1$, $n+1$ respectively, with
$\sigma\subset \lambda,\mu\subset\pi$ and
 $\lambda\neq\mu$. Then, the $r$-partitions
$\sigma$, $\pi$ are completely determined by $\lambda$, $\mu$ and \eqref{HR:2} gives
the following identity
$$\tau_\lambda+\tau_\mu=\tau_\sigma+\tau_\pi.$$
Therefore, using the identities from Sections 
\ref{sec:tautr}, \ref{sec:hecker}, a short computation gives
\begin{equation}
\label{tutu2}
N_{\lambda,\sigma}+N_{\mu,\sigma}-T_\sigma= 
N_{\pi,\lambda}+N_{\pi,\mu}-T_\pi.
\end{equation}
Therefore, using \eqref{Er:31}, \eqref{Er:32} and \eqref{HR:5}, \eqref{tutu2} we get 
\begin{equation}
\label{tutu3}
[h_{-1,k},h_{1,l}]([I_\lambda])=c_{\lambda,k+l}\,[I_\lambda]
\end{equation}
for some constant $c_{\lambda,k+l}$ which remains to be computed. 
To do so, observe first that
$$\gathered
f_{-1,k}f_{1,l}([I_\lambda])=\sum_{\lambda\subset\pi}
\cc_1(\tau_{\lambda,\pi})^{k+l}\,
\eu(N_{\lambda,\pi}^*+N_{\pi,\lambda}^*)\,
\,\eu_{\lambda,\pi}^{-1}\,[I_\lambda],\\
f_{1,l}f_{-1,k}([I_\lambda])=\sum_{\sigma\subset\lambda}
\cc_1(\tau_{\sigma,\lambda})^{k+l}\,
\eu(N_{\lambda,\sigma}^*+N_{\sigma,\lambda}^*)\,
\,\eu_{\lambda,\sigma}^{-1}\,[I_\lambda],
\endgathered$$
modulo $[I_\mu]$'s with $\mu\neq\lambda$. 
Next, set $H_\lambda=(1-q)(1-t)\tau_\lambda-W.$ 
For $\lambda\subset\pi$ we have
\begin{equation}
\label{formH}
\aligned
H_\lambda&=H_\pi-(1-q)(1-t)\tau_{\lambda,\pi},\\
N_{\lambda,\pi}-T_\lambda
&=-v\tau_{\lambda,\pi}^* H_{\lambda}-v,\\
&=-v\tau_{\lambda,\pi}^* H_{\pi}+1-q^{-1}-t^{-1},\\
N_{\pi,\lambda}-T_\pi
&=\tau_{\lambda,\pi}H_\pi^*-v,\\
&=\tau_{\lambda,\pi}H_{\lambda} ^*+1-q^{-1}-t^{-1}.
\endaligned
\end{equation}
Now, we consider the following sums
\begin{equation}\label{AB}
B_\lambda=\sum_{\sigma\subset\lambda}\tau_{\sigma,\lambda},
\quad A_\lambda=\sum_{\lambda\subset\pi}\tau_{\lambda,\pi}.
\end{equation}
The proof of the following lemma is left to the reader, 
compare \cite[lem.~7]{VV1}.

\begin{lem}
\label{lem:toto1}
For each $r$-partition $\lambda$ of $n$ we have the equality of characters
$$H_\lambda=v^{-1}B_\lambda-A_\lambda.$$
\end{lem}

\noindent
Thus, we get
$$\aligned
N_{\lambda,\pi}^*+N_{\pi,\lambda}^*-T_\lambda^*-T_\pi^*
&=v^{-1}(\tau_{\lambda,\pi}A_{\lambda}^*+\tau_{\lambda,\pi}^*B_\lambda-1)-
(\tau_{\lambda,\pi}B_{\lambda}^*+\tau_{\lambda,\pi}^*A_\lambda-1)
-q-t,\\
N_{\lambda,\sigma}^*+N_{\sigma,\lambda}^*-T_\lambda^*-T_\sigma^*
&=v^{-1}(\tau_{\sigma,\lambda}A_{\lambda}^*+\tau_{\sigma,\lambda}^*B_\lambda-1)-
(\tau_{\sigma,\lambda}B_{\lambda}^*+\tau_{\sigma,\lambda}^*A_\lambda-1)-q-t.
\endaligned$$
We get
$$
\aligned
(-1)^rc_{\lambda,k}\,x^k&=
\sum_{\sigma\subset\lambda}\cc_1(\tau_{\sigma,\lambda})^{k}\,
\frac{\eu(v^{-1}\tau_{\sigma,\lambda} A_{\lambda}^*+
v^{-1}\tau_{\sigma,\lambda}^*B_\lambda-v^{-1})}
{\eu(\tau_{\sigma,\lambda}^* A_{\lambda}+\tau_{\sigma,\lambda}B_\lambda^*-1)}\;-\\
&\qquad-\sum_{\lambda\subset\pi}\cc_1(\tau_{\lambda,\pi})^{k}\,
\frac{\eu(v^{-1}\tau_{\lambda,\pi} A_{\lambda}^*+
v^{-1}\tau_{\lambda,\pi}^*B_\lambda-v^{-1})}
{\eu(\tau_{\lambda,\pi}^*A_{\lambda}+\tau_{\lambda,\pi}B_\lambda^*-1)}.
\endaligned
$$
Consider the formal series 
$$C(s)=\sum_{k\geqslant 0} c_{\lambda,k}\,x^ks^k.$$ 
For $u=x+y$, we have
$$\aligned
C(s)&=-\sum_{j\in J} \frac{1}{1-b_js}\, 
\prod_{i\in I} \frac{b_j-a_i+u}{b_j-a_i}
\prod_{k \in J\setminus\{j\}} \frac{b_j-b_k-u}{b_j-b_k}\;+\\
&\qquad+\sum_{i\in I} \frac{1}{1-a_is}\,
\prod_{k\in I\setminus\{i\}} \frac{a_i-a_k+u}{a_i-a_k}
\prod_{j\in J}\frac{a_i-b_j-u}{a_i-b_j},
\endaligned
$$
with
$$\{b_j\;;\;j\in J\}=\{\cc_1(\tau_{\sigma,\lambda})\;;\;\sigma\subset\lambda\},
\qquad
\{a_i\;;\;i\in  I\}=\{\cc_1(\tau_{\lambda,\pi})\;;\;\lambda\subset\pi\}.$$
Thus, by Lemma \ref{lem:kop5}, 
we get the following equality
$$\aligned
1+usC(s)
&=\prod_{i\in I}\frac{1-s(a_i-u)}{1-sa_i}\prod_{j\in J} \frac{1-s(b_j+u)}{1-sb_j}\\
&=\prod_{\sigma\subset\lambda} \frac{1-s\cc_1(v^{-1}\tau_{\sigma,\lambda})}
{1-s\cc_1(v^{-1}\tau_{\sigma,\lambda})+su}\,\Bigg/
\prod_{\lambda\subset\pi}\frac{1-s\cc_1(\tau_{\lambda,\pi})}{1-s\cc_1(\tau_{\lambda,\pi})+su}.
\endaligned$$
Now, fix splitting sums of one-dimensional characters
$$\tau_\lambda^*=\phi_{\lambda,1}+\cdots+\phi_{\lambda,n},\qquad
W^*=\chi_1+\cdots+\chi_r.$$ 
Set $f_{\lambda,i}=\eu(\phi_{\lambda,i})$
and $e_\a=\eu(\chi_\a)$.
Then, by Lemma \ref{lem:toto1}, we have
$$H_\lambda=\sum_{\sigma\subset\lambda}v^{-1}\tau_{\sigma,\lambda}-
\sum_{\lambda\subset\pi}\tau_{\lambda,\pi}=
\sum_i(1-q)(1-t)\phi_{\lambda,i}^*-\sum_\a\chi^*_\a.$$
Therefore, we get
$$
\aligned
1+usC(s)&=
\prod_{i=1}^n\frac{1+s(f_{\lambda,i}+x)}{1+s(f_{\lambda,i}-x)}\,
\frac{1+s(f_{\lambda,i}+y)}{1+s(f_{\lambda,i}-y)}\,
\frac{1+s(f_{\lambda,i}-u)}{1+s(f_{\lambda,i}+u)}
\,\prod_{\a=1}^r\frac{1+s(e_\a+u)}{1+se_\a}.
\endaligned
$$
Recall that $u=x\xi$. 
Using \eqref{formule/log} and \eqref{formule/phi}, we finally get
$$\aligned
1+\xi\sum_{k\geqslant 0}c_{\lambda,k}\,s^{k+1}
&=\prod_{\a=1}^r\frac{1+s\eps_\a+s\xi}{1+s\eps_\a}\,
%\exp\Bigl(\sum_{l\geqslant 0}(-1)^{l+1}p_l(e_\a)\,x^{-l}\phi_l(s)\Bigr)\,
\exp\Bigl(\sum_{l\geqslant 0}(-1)^{l}
p_l(f_{\lambda,i})\,x^{-l}\varphi_l(s)\Bigr).
\endaligned
$$
Now, Remark \ref{rem:f0l} gives 
$$h_{0,l+1}\,[I_\lambda]=(-1)^lx^{-l}p_l(f_{\lambda,i})\,[I_\lambda].$$ 
Thus, we obtain
$$\aligned
\Bigl(1+\xi\sum_{k\geqslant 0}c_{\lambda,k}\,s^{k+1}\Bigr)\,[I_\lambda]
&=\prod_{\a=1}^r\frac{1+s\eps_\a+s\xi}{1+s\eps_\a}\,
\exp\Bigl(\sum_{l\geqslant 0}h_{0,l+1}\,\varphi_l(s)\Bigr)\,[I_\lambda],\\
&=\exp\Bigl(\sum_{l\geqslant 0}(-1)^{l+1}p_l(\vec\eps)\phi_l(s)\Bigr)\,
\exp\Bigl(\sum_{l\geqslant 0}h_{0,l+1}\,\varphi_l(s)\Bigr)\,[I_\lambda]
\endaligned
$$
where $p_l(\vec\eps)=\sum_a \eps_a^l$. 
Comparing this expression with \eqref{tutu3}, we get the proposition.

\vspace{.15in}

\subsection{Proof of Proposition \ref{prop:faithfulr}}
\label{app:B.3}
Now, we prove that the
representation $\rho^{(r)}$ of $\SH^{(r)}_K$ on $\Lb^{(r)}_K$ is faithful.
For an operator $f$ on
$\mathbf{L}^{(r)}_K$ and $r$-partitions $\lambda,$ $\mu$
we denote by $\langle \mu ; f ; \lambda\rangle$ the
coefficient of $[I_{\mu}]$ in $f([I_{\lambda}])$.
Given partitions $\lambda_1, \lambda_2, \ldots, \lambda_k$ and given an 
integer $n \gg |\lambda_1|,\ldots, |\lambda_k|,$ let the symbol
$\lambda_1\circledast\ldots\circledast\lambda_k$ denote the $r$-partition
whose first part is the partition $\lambda_1\circledast\ldots\circledast\lambda_k$
from Section \ref{app:C} and the $r-1$ other partitions are empty.
Given a finite family of elements 
$$P_i\in\SH^{(r),>}_K,\qquad R_i\in\SH^{(r),0}_K,\qquad Q_i\in\SH^{(r),<}_K,$$
we set $x=\sum_iP_iR_iQ_i$.
Assume that $\rho^{(r)}(x)=0$. We may also assume that $P_iQ_iR_i$ is homogeneous of degree 0
(for the rank grading) for each $i$. Then
$$\sum_i\langle \mu\,;\,\rho^{(r)}(P_iR_iQ_i)\,;\,\lambda\rangle=0$$
for each $r$-partitions $\lambda$, $\mu$.
For $n$ large enough we consider the coefficients 
\begin{equation}
\langle
\bar\lambda_1\circledast\lambda_2\circledast\bar\lambda_3\, ; \rho^{(r)}(P_i R_i Q_i)\,;
\lambda_1\circledast\lambda_2\circledast\lambda_3\rangle,\quad \bar\lambda_1 \subset
\lambda_1,\quad \lambda_3 \subset \bar\lambda_3,\quad |\lambda_1 \backslash
\bar\lambda_1|=|\bar\lambda_3\backslash \lambda_3|=t
\end{equation}
with 
$t=\text{sup}_i (\deg(P_i)) =\text{sup}_i(-\deg(Q_i)).$
Since  $Q_i$ is an annihilation operator and $P_i$ is a
creation operator, the coefficient
\begin{equation}
\begin{gathered}
\langle \bar\lambda_1\circledast\lambda_2\circledast\bar\lambda_3\,; \rho^{(r)}(P_i R_i Q_i)\,; 
\lambda_1\circledast\lambda_2\circledast\lambda_3\rangle\end{gathered}
\end{equation}
factorizes as in \eqref{E:PP5}, and it is zero unless $\deg(P_i)=-\deg(Q_i)= t$.
We claim that \eqref{E:tauy}, \eqref{E:rhoy} hold again for some non-zero $c,d\in K_r$.
Then \eqref{E:PP6} hold again
for all $\lambda_1,\bar\lambda_1,\lambda_2, \lambda_3, \bar\lambda_3$ as above and
all large enough $n$. If $x\neq 0$ we may assume that the elements $R_i$
are linearly independent polynomials in the $f_{0,l}$'s and that $P_i,Q_i\neq 0$.
Then the same argument as in Section \ref{app:C} yields a contradiction.
The proof of the claim is the same as the proof of Lemma \ref{L:PK1}. It is left to the reader.
\qed

\vspace{.1in}

The proof above has the following corollary.

\begin{cor}\label{cor:injectivity}
The multiplication map
$\SH^{(r),>}_K\otimes\SH^{(r),0}_K\otimes\SH^{(r),<}_K\to\SH^{(r)}_K$ 
is injective.
\end{cor}

\vspace{.15in}

\subsection{Proof of Lemma \ref{L:shuffleWilson}} 
\label{app:D2}
Fix $r$-multipartitions $\mu,$ $\lambda$ such that
$\mu \subset \lambda$ and $|\lambda \backslash \mu|=n$. 
Let $l_1, \ldots, l_n$ be integers $\geqslant 0$.
We need to prove that 
\begin{equation}
\langle \lambda; f_{1,l_1} \cdots f_{1,l_n} ; \mu \rangle=
\varpi_n(z_1^{l_1} \cdots z_n^{l_n}) (\tau_{\mu,\lambda}) \,
a_{\mu,\lambda}\, \eu_{\mu}.
\end{equation}
Let $s_1, \ldots, s_n$ be the boxes of $\lambda \backslash \mu$. We have
\begin{equation}\langle \lambda; f_{1,l_1} \cdots f_{1,l_n} ; \mu \rangle=
\sum_{\sigma \in \mathfrak{S}_n} \langle \lambda^{\sigma,0}; f_{1,l_1}; 
\lambda^{\sigma,1}\rangle \cdots \langle \lambda^{\sigma,n-1}; f_{1,l_1}; 
\lambda^{\sigma,n}\rangle
\end{equation}
where $\lambda^{\sigma,0}=\lambda$ and 
$\lambda^{\sigma,i}=\lambda \backslash \{s_{\sigma(1)}, \ldots, s_{\sigma(i)}\}$
for $i=1, \ldots, n$. We set 
$\langle \lambda^{\sigma,i-1}; f_{1,l_i}; \lambda^{\sigma,i}\rangle=0$ 
if $\lambda^{\sigma,i-1}$ or $\lambda^{\sigma,i}$ 
is not a multipartition. We say that $\sigma$ is \emph{admissible} if 
$\lambda^{\sigma,1}, \ldots, \lambda^{\sigma,n-1}$
are all multipartitions.
If $n=1$ then we have
\begin{equation}
\langle \lambda; f_{1,l}; \mu \rangle=
\cc_1(\tau_{\mu,\lambda})^l 
\eu(N^*_{\lambda,\mu}-T^*_\lambda)=
\cc_1(\tau_{\mu,\lambda})^l a_{\mu,\lambda}. 
\end{equation}
Hence, if $\sigma$ is admissible then
$$\prod_{i=1}^n \langle \lambda^{\sigma,i-1};f_{1,l_i};\lambda^{\sigma,i}
\rangle=\cc_1(\tau_{s_{\sigma(1)}})^{l_1} \cdots 
\cc_1(\tau_{s_{\sigma(n)}})^{l_n} 
\eu\big( \sum_{i=1}^n N^*_{\lambda^{\sigma,i-1},\lambda^{\sigma,i}}-
T^*_{\lambda^{\sigma,i-1}}\big).$$
Now let $\sigma \in \mathfrak{S}_n$ be arbitrary. Using 
\eqref{HR:2.5}, \eqref{HR:3},
we get after a straightforward computation
\begin{equation}\label{Eq:prooflemfundclass}
\begin{split}
\sum_{i=1}^n \big(N^*_{\lambda^{\sigma,i-1},\lambda^{\sigma,i}}-
T^*_{\lambda^{\sigma,i-1}}\big)&=
\big( (1-q)(1-t)( \tau^*_{\mu,\lambda} \otimes \tau_{\lambda}) 
-\tau^*_{\mu,\lambda} \otimes W -nv^{-1}\big)-\\
& \quad -(1-q)(1-t)\sum_{i >j} 
\tau^*_{ \lambda^{\sigma,i},\lambda^{\sigma,i-1}} \otimes
\tau_{\lambda^{\sigma,j},\lambda^{\sigma,j-1}}.
\end{split}
\end{equation}
We have already seen in the proof of Proposition 
\ref{prop:torsionfreeWilson} that
$$a_{\mu,\lambda}=
\eu\big( (1-q)(1-t)( \tau^*_{\mu,\lambda} \otimes \tau_{\lambda}) 
-\tau^*_{\mu,\lambda} \otimes W -nv^{-1}\big)$$ is nonzero and well-defined.
A similar reasoning shows that 
\begin{equation}
\eu \big(-(1-q)(1-t)\sum_{i >j} \tau^*_{
\lambda^{\sigma,i},\lambda^{\sigma,i-1}} \otimes \tau_{
\lambda^{\sigma,j},\lambda^{\sigma,j-1}}\big)
\end{equation}
is well-defined and vanishes if $\sigma$ is not admissible.
Now note that
\begin{equation}\label{E:killop}
\eu \big(-(1-q)(1-t)\sum_{i >j} \tau^*_{
\lambda^{\sigma,i},\lambda^{\sigma,i-1}} \otimes
\tau_{
\lambda^{\sigma,j},\lambda^{\sigma,j-1}}\big)=g(\cc_1(\tau_{s_{\sigma(1)}}), \ldots, 
\cc_1(\tau_{s_{\sigma(n)}})).
\end{equation}
It follows that
\begin{equation*}
\begin{split}
\sum_{\sigma \in \mathfrak{S}_n}&\cc_1(\tau_{s_{\sigma(1)}})^{l_1} \cdots 
\cc_1(\tau_{s_{\sigma(n)}})^{l_n}
\eu \big(-(1-q)(1-t)\sum_{i >j} \tau^*_{\lambda^{\sigma,i-1} \backslash 
\lambda^{\sigma,i}} \otimes \tau_{\lambda^{\sigma,j-1} \backslash 
\lambda^{\sigma,j}} \big)=\\
&=\sum_{\sigma \in \mathfrak{S}_n} \cc_1(\tau_{s_{\sigma(1)}})^{l_1} \cdots 
\cc_1(\tau_{s_{\sigma(n)}})^{l_n}
g(\cc_1(\tau_{s_{\sigma(1)}}), \ldots, \cc_1(\tau_{s_{\sigma(n)}}))\\
&=\varpi_n(z_1^{l_1} \cdots z_n^{l_n})(\tau_{\mu,\lambda}).
\end{split}
\end{equation*}
 Lemma \ref{L:shuffleWilson} is an easy consequence.

\vspace{.2in}

\section{The Heisenberg subalgebra}
\label{app:D}

\vspace{.1in}

In this section we prove the formula in Section \ref{sec:heis}.

\vspace{.1in}

\begin{lem}\label{lem:D1}
For $k,l\geqslant 0$ we have $[D_{l,0},D_{k,0}]=[D_{-l,0}, D_{-k,0}]=0$.
\end{lem}

\begin{proof} 
Follows from Remark \ref{rem:com}.
\end{proof}

\vspace{.1in}

\begin{lem}\label{lem:D2}
For $l\geqslant 0$ we have $[D_{l,0},D_{-1,0}]=-E_0\,\delta_{l,1}$ and
$[D_{1,0},D_{-l,0}]=-E_0\,\delta_{l,1}$.
\end{lem}

\begin{proof}
The proof is by induction on $l$. From \eqref{E:rel3bis2} we have
$[D_{1,0},D_{-1,0}]=-E_0$.
Next, we have
\begin{equation}
\label{D3}
[D_{1,1}, D_{l,0}]=l D_{l+1,0}, \qquad [D_{-1,1},D_{-l,0}]=-l D_{-l-1,0}.
\end{equation}
Thus, using the induction hypothesis and \eqref{E:rel3bis2}, we get
\begin{equation}
\label{D4}
l[D_{l+1,0}, D_{-1,0}]=[[D_{1,1},D_{l,0}], D_{-1,0}]=
-[D_{l,0}, [D_{1,1},D_{-1,0}]]=[D_{l,0},E_1].
\end{equation}
Now, by  definition, $E_1$ is a central element.
Thus, we have
$$[D_{l,0},D_{-1,0}]=0,\qquad l\geqslant 2.$$
The second identity follows from the first 
one by applying the anti-involution $\pi$.
\end{proof}

\vspace{.1in}

\begin{lem}\label{lem:D5}
For $l\in\Z$ we have $[D_{0,1},D_{l,0}]=lD_{l,0}$.
\end{lem}

\begin{proof} 
From the faithful representation of $\SH^{+}$ in the Fock space 
(see Proposition~\ref{prop:rho+}) we get 
$$[D_{0,1}, D_{l,0}]=l D_{l,0},\qquad l\geqslant 1.$$ 
Now apply the anti-automorphism $\pi$
and Section \ref{sec:heis}.
\end{proof}

\vspace{.1in}

\begin{lem}\label{lem:D6}
For $l\geqslant 2$ we have $[D_{1,1},D_{-l,0}]=-\kappa lD_{1-l,0}$
and $[D_{-1,1},D_{l,0}]=\kappa lD_{l-1,0}$.
\end{lem}

\begin{proof}
The second relation follows from the first one and the anti-involution $\pi$ of $\SH^\cb$.
Let us concentrate on the first relation.
The proof is by induction on $l$.
By \eqref{E:rel3bis2} and the induction hypothesis we have
$$\begin{aligned}
{[D_{1,1},D_{-l,0}]}
&=[D_{1,1},[D_{-1,1},D_{1-l,0}]]/(1-l)\\
&=[[D_{1,1},D_{-1,1}],D_{1-l,0}]/(1-l)+[D_{-1,1},[D_{1,1},D_{1-l,0}]]/(1-l)\\
&=-[E_2,D_{1-l,0}]/(1-l)+\kappa[D_{-1,1},D_{2-l,0}].
\end{aligned}$$
Now, using \eqref{D7}, 
we get that the element $E_2-2\kappa D_{0,1}$ is central.
Therefore, by \eqref{D3} and Lemma \ref{lem:D5}, we have
$$\begin{aligned}
{[D_{1,1},D_{-l,0}]}
&=-2\kappa[D_{0,1},D_{1-l,0}]/(1-l)+\kappa(2-l)D_{1-l,0}\\
&=-2\kappa D_{1-l,0}+\kappa(2-l)D_{1-l,0}\\
&=-\kappa lD_{1-l,0}.
\end{aligned}$$
\end{proof}

\vspace{.1in}

We now prove formula \eqref{heis}, which is equivalent to
$(I_{l,k})$ below
\begin{equation*}
[D_{l,0},D_{-k,0}]=A_lE_0\,\delta_{l,k},\qquad 
A_l=-l\kappa^{l-1},\qquad
k,l\geqslant 1. \tag{$I_{l,k}$}
\end{equation*}
For $k=1$ or $l=1$ this is Lemma \ref{lem:D2}. 
Let us prove the formula $(I_{l,k+1})$, assuming that we have already 
proved $(I_{\bullet,k})$ and $(I_{l',k+1})$ for $l'\leqslant l$.
Using $(I_{l,k})$ and Lemma \ref{lem:D6} we get 
$$\begin{aligned}
{[D_{l+1,0},D_{-k,0}]}&=[[D_{1,1},D_{l,0}],D_{-k,0}]/l\\
&=-[D_{l,0},[D_{1,1},D_{-k,0}]]/l\\
&=\kappa k[D_{l,0},D_{1-k,0}]/l\\
&=\kappa(l+1)A_lE_0\delta_{l+1,k}/l.
\end{aligned}$$
We deduce that $A_{l+1}=\kappa (l+1)A_l/l$. Since $A_1=-1$ this shows that $A_l=-l\kappa^{l-1}$ as wanted.

\qed

\vspace{.2in}

\section{Relation to $W_{1+\infty}$}
\label{app:F}
Let $W_{1+\infty}$ be the 
universal central extension of the Lie 
algebra, over $\C$, of regular differential operators on the circle. 
To unburden the notation we'll abbreviate $\Wen=W_{1+\infty}$.
The aim of this section is to prove that the specialization 
at $\kappa=1$ of $\SH^{\cb}$ is 
isomorphic to the enveloping algebra of $\Wen$.

\subsection{The integral form of $\SH^{\cb}$} 
Let $\SH^{\cb}_{A}$
be the $A$-subalgebra of $\SH^\cb$ generated by the set
$\{\cb_{l}, D_{\pm 1,0}, D_{0,l}\;;\; l \geqslant 0\}$. 
From \eqref{E:rel4} and \eqref{E:defD1l} 
it follows that $\SH^{\cb}_{A}$ 
contains the elements $D_{l,0}$, $D_{\pm 1, l}$ for
any $l\geqslant 0$. Let $\SH^>_{A}$, $\SH^{\cb,0}_{A}$ and
$\SH^<_{A}$ be the $A$-subalgebras
generated by $\{D_{1,l}\;;\; l \geqslant 0\},$ 
$\{\cb_{l}, D_{0,l} \;;\; l \geqslant 0\}$
and $\{D_{-1,l}\;;\; l \geqslant 0\}$ respectively.
Replacing everywhere $A$ by $A_1$ we obtain the $A_1$-algebras
$\SH^>_{A_1}$, $\SH^{\cb,0}_{A_1}$, $\SH^<_{A_1}$ and $\SH_{A_1}$.

\begin{prop}\label{P:integralSH} 
(a) The $A_1$-module $\SH^{\cb}_{A_1}$ is  free  and we have
$\SH^{\cb}_{A_1} \otimes_{A_1} F=\SH^{\cb}$. 

(b) We have a triangular decomposition 
$\SH^{\cb}_{A_1}=
\SH^{>}_{A_1}\otimes_{A_1}\SH^{\cb,0}_{A_1}\otimes_{A_1}\SH^{<}_{A_1}.$
\end{prop}

\begin{proof}
We claim that
$\SH^>_{A_1},$ $ \SH^{\cb,0}_{A_1}$ and $\SH^{<}_{A_1}$ are free over $A_1$, 
and that
$$\SH^>_{A_1}\otimes_{A_1}F=\SH^>,
\qquad
\SH^{\cb,0}_{A_1}\otimes_{A_1}F=\SH^{\cb,0}
,\qquad
\SH^{<}_{A_1}\otimes_{A_1}F=\SH^{<}
.$$
Thus, we have an isomorphism
$$\big(\SH^{>}_{A_1} \otimes_{A_1} \SH^{\cb,0}_{A_1} 
\otimes_{A_1} \SH^{<}_{A_1}\big)
\otimes_{A_1}F= 
\SH^{>}\otimes\SH^{\cb,0}\otimes\SH^{<}.$$
Therefore, the multiplication map 
$$\SH^{>}_{A_1} \otimes_{A_1} \SH^{\cb,0}_{A_1} \otimes_{A_1} \SH^{<}_{A_1} \to 
\SH^{\cb}_{A_1},$$ being the restriction of
a similar map over $F$, it is injective 
by Proposition \ref{1.8:prop1}.
We only need to show its surjectivity. 
The proof is the same as for $\SH^{\cb}$ in Proposition \ref{1.8:prop1}.
It is based
on the fact that $D_{-1,l}, D_{1,l}\in \SH^{\cb}_{A_1}$ for
$l\geqslant 0$. 
Then, using the triangular decompostion, we get
that $\SH^{\cb}_{A_1}$ is free as an $A_1$-module and that
\begin{equation*}
\SH^{\cb}_{A_1} \otimes_{A_1} F= \SH^{\cb}.
\end{equation*}

Now, we prove the claim.
It is clear for $\SH^{\cb,0}_{A_1}$. 
The remaining two cases are similar, we only deal
with the first one. Recall that $\SH^>$ carries an $\N$-grading and an $\N$-filtration, with 
finite-dimensional pieces $\SH^>[r,\leqslant \!l]$. 
Consider the $A_1$-module
$$\SH^>_{A_1}[r,\leqslant \!l]=
\SH^>_{A_1} \cap \SH^>[r,\leqslant \!l].$$ 
Since the tensor product commutes with direct limits, 
it is enough to check that
\begin{equation}\label{intform1}
\SH^>_{A_1}[r,\leqslant \!l]\otimes_{A_1}F=
\SH^>[r,\leqslant \!l]
\end{equation}
and that the inclusion of $A_1$-modules
\begin{equation}\label{intform2}
\SH^>_{A_1}[r,<\!l]\subset
\SH^>_{A_1}[r,\leqslant \!l]
\end{equation}
is a direct summand.
Now, for $n$ large enough the map $\Phi_n$ yields an isomorphism
$$\SH^>[r,<\!l]\to \SH^>_n[r,<\!l].$$
By Remark \ref{rem:uno}, this map restricts to an isomorphism of $A_1$-modules
$$\SH^>_{A_1}[r,<\!l]\to \SH^>_{n,A_1}[r,<\!l].$$
In particular, the left hand side is finitely generated and 
torsion free. Hence it is free. 
Further, \eqref{intform1} holds by \eqref{1.4}.
Finally, to prove \eqref{intform2} it is enough to check
that the inclusion of $A_1$-modules
$$\SH^>_{n,A_1}[r,<\!l]\subset \SH^>_{n,A_1}[r,\leqslant \!l]$$
is a direct summand.
This follows from the fact that the inclusions of $A_1$-modules 
$$\H^>_{n,A_1}[r,<\!l]\subset\H^>_{n,A_1}[r,\leqslant \!l],
\qquad\SH^>_{n,A_1}[r,\leqslant\!l]\subset\H^>_{n,A_1}[r,\leqslant \!l]$$
are direct summands, by the PBW theorem and formula
$$\SH^>_{n,A_1}[r,\leqslant\!l]=\Sb\cdot\H^>_{n,A_1}[r,\leqslant \!l]\cdot\Sb.$$
 
\end{proof}

\vspace{.15in}

\subsection{The Lie algebra $W_{1+\infty}$}
The Lie algebra $\Wen$ has the basis
$\{C,w_{l,k}\; ;\;l\in\Z,\,k\in\N\}$ and, given formal variables $\alpha,$ $\beta,$ the relations 
are given by
\begin{equation}
\label{relW}
\begin{gathered}
w_{l,k}=t^lD^k,\\
[t^l\exp(\alpha D),t^k\exp(\beta D)]=
(\exp(k\alpha)-\exp(l\beta))\,t^{l+k}\exp(\alpha D+\beta D)+\\
+\delta_{l,-k}\frac{\exp(-l\alpha)-\exp(-k\beta)}{1-\exp(\alpha+\beta)}C.
\end{gathered}
\end{equation}

\vspace{.1in}

\begin{ex}
The elements $b_l=w_{l,0}$ with $l\in\Z$ satisfy the relations of the 
Heisenberg algebra
with central charge $C$, i.e., we have
$[b_{l}, b_{-k}]=l\delta_{l,k}C.$ Let $\Heis$ be this Lie subalgebra.
\end{ex}

\begin{ex}
For $\beta\in\C$,
the elements $L_l^\beta=-w_{l,1}-\beta (l+1)b_l$ with $l\in\Z$
satisfy the relations of the Virasoro algebra, i.e., 
we have
$$[L_{l}^\beta,L_{k}^\beta]=(l-k)L_{l+k}^\beta+\frac{l^3-l}{12}\,\delta_{l,-k}\,C_\beta,\quad
C_\beta=(-12\beta^2+12\beta-2)\,C.$$
In particular $\{L_l^{1/2}\,;\,l \in \Z\}$ generates a Virasoro algebra 
of central charge $C$ such that
$$[L^{1/2}_l,b_k]=-kb_{l+k}.$$
\end{ex}

\vspace{.1in}

\begin{ex}
We have the following formulas
$$[w_{0,2},w_{l,k}]=2lw_{l,k+1}+l^2w_{l,k},
\qquad
[b_1,w_{l,k}]=-\sum_{h=0}^{k-1}\begin{pmatrix}k\\h\end{pmatrix}
w_{l+1,h}+\delta_{l,-1}\delta_{k,0}C.$$
In particular, we have
$[w_{0,2},b_l]=-2lL^{1/2}_l-lb_l.$
\end{ex}

\vspace{.1in}

Let $\Wen^+,\Wen^>,\Wen^0\subset \Wen$ be the Lie subalgebras spanned by
$$\{C, w_{l,k}\; ;\;l,k\in\N\},\qquad
\{w_{l,k}\; ;\;l\geqslant 1,\, k\in\N\},\qquad
\{C,w_{0,l}\;;\;l\geqslant 0\}.$$
We define $\Wen^-$ and $\Wen^<$ in a similar fashion.
The enveloping algebra $U(\Wen)$ carries a $\Z$-grading, 
called the \emph{rank grading}, in which 
$w_{l,k}$ is placed in degree $l$ and $C$ is placed in degree zero, and an $\N$-filtration, called the \emph{order filtration},
in which $w_{l,k}$ is placed in degree $ \leqslant k$. 
The order filtration may alternatively be described as follows~: 
an element $u$ is of order $\leqslant k$ if 
$$\ad(p_1) \circ \cdots \circ \ad(p_k)(u)\in\C[C,w_{l,0}\;;\;l\in\Z],
\qquad\forall p_1, \ldots, p_k \in U(\Heis).$$ 
Let $U(\Wen)[r,\leqslant\! k]$ stands for the piece of degree $r$ 
and order $\leqslant k$.
The graded pieces $U(\Wen^>)[r,\leqslant \!k]$ and 
$U(\Wen^<)[r,\leqslant \!k]$ are finite-dimensional and the 
Poincar\'e polynomials of $U(\Wen^>)$ and $U(\Wen^<)$ 
with respect to this grading and filtration are given by
\begin{align}\label{E:poincarew}
P_{\Wen^>}(t,q)=\prod_{r > 0}\prod_{k \geqslant 0} \frac{1}{1-t^rq^k},\qquad
P_{\Wen^<}(t,q)=\prod_{r <0}\prod_{ k \geqslant 0} \frac{1}{1-t^rq^k}.
\end{align}

\vspace{.1in}

The proof of the following  result is left to the reader.

\vspace{.1in}

\begin{lem}\label{L:Wgenerate} The following holds

(a) $\Wen$ is generated by $b_{-1}$, $b_1$ and $w_{0,2}$,

(b) $\Wen^>$, $\Wen^{<}$ are generated by $\{w_{1,l}\;;\;l \geqslant 0\}$, 
$\{w_{-1,l}\;;\;l \geqslant 0\}$ respectively.
\end{lem}

\vspace{.15in}

\subsection{The Fock space of $W_{1+\infty}$} 
For $c,d\in\C$ 
we set 
$$U_{c,d}(\Wen)=U(\Wen)/(C-c,b_0-d),\qquad
U_{c,d}(\Heis)=U(\Heis)/(C-c,b_0-d).$$
Let $S_{c,d}$ be the 
\emph{irreducible vacuum module with level $(c,d)$},
see \cite[sec.~ 1]{FKRW}. It is the top of the \emph{Verma module}
$$M_{c,d}=\Ind_{\Wen^+}^{\Wen}(\C_{c,d}),$$ 
where $\C_{c,d}$ is the one-dimensional 
$\Wen^+$-module given by  
$$w_{l,k}\mapsto 0,\quad l,k \geqslant 0,\quad (l,k) \neq (0,0),\quad
C\mapsto c, \quad b_0\mapsto d.$$ 
We will mainly be interested in the pair $\eta$ given by $(c,d)=(1,-1/2)$.

\vspace{.1in}

\begin{prop}\label{prop:fideleW} 
(a) The restriction of $S_{\eta}$ to 
$\Heis$ is the level one Fock space of $\Heis$.

(b) The action of 
$U_{\eta}(\Wen)$ 
on $S_{\eta}$ is faithful.
\end{prop}

\begin{proof}
See  \cite[thm.~5.1]{FKRW} for $(a)$. Now, we prove part $(b)$.
Let $I \subset U_{\eta}(\Wen)$ be the annihilator of $S_{\eta}$. 
Since $U_{\eta}(\Heis)$ acts faithfully on 
$S_\eta$ we have $I \cap U_{\eta}(\Heis)=\{0\}$. The proposition
is a consequence of the following lemma.

\begin{lem} Let $I$ be an ideal of $U(\Wen)$ such that 
$I \cap U(\Heis) =\{0\}$. Then $I=\{0\}$.
\end{lem}

\begin{proof}
Let $I$ be as above, and let $I_0 \subseteq I_1 \subseteq \cdots$ 
be the filtration on $I$ induced from the order filtration on $U(\Wen)$. 
Assuming that $I \neq \{0\}$, 
let $n$ be minimal such that $I_n \neq \{0\}$. 
Since $I_0=I \cap U(\Heis) =\{0\}$, we have $n \geqslant 1$. 
Moreover, since 
$$\ad(b_l)\big(U(\Wen)[\leqslant\! n]\bigr)\subset U(\Wen)[<\! n],$$ we have
$[I_n,U(\Heis)]=0$. This contradicts the following claim.

\vspace{.05in}

\begin{claim}  The centralizer of $\Heis$ in $U(\Wen)$ is 
$\CC C \oplus \CC b_0$.
\end{claim}

\begin{proof} For  $l \in \Z$ we consider the map
$$\sigma_l=\ad (b_l): U(\Wen)[\leqslant\! n] / U(\Wen)[<\! n] \longrightarrow  
U(\Wen)[\leqslant\! n-1] / U(\Wen)[<\! n-1].$$
The space $U(\Wen)[\leqslant\! n] / U(\Wen)[<\! n]$ 
is identified with the degree $(\bullet,n)$ 
part of the polynomial ring
$\CC[\overline{w}_{h,k}\;;\;h,k]$.
One checks from the definition of 
$\Wen$ that $\sigma_l$ acts as 
the derivation satisfying
$$\sigma_l(\overline{w}_{h,k})=\begin{cases} -kl\overline{w}_{l+h,k-1} & 
\;\text{if}\; k \geqslant 1\\ 0 & \;\text{if}\; k=0\end{cases}$$
From this it is easy to check that 
$\bigcap_l \Ker (\sigma_l) =\{0\}$ if $n \geqslant 1$. This implies that 
the centralizer of $\Heis$ in $U(\Wen)$ is contained into
$U(\Wen)[\leqslant\! 0]=U(\Heis).$
The claim now follows from the
fact that the center of $U(\Heis)$ is $\CC C \oplus \CC b_0$.
\end{proof}
This finishes the proof of the lemma and of the proposition.
\end{proof}
\end{proof}

\vspace{.1in}

\begin{lem}\label{L:wFock}
The element $w_{0,2}/2$ acts in $S_{\eta}$ as the Laplace-Beltrami 
operator specialized at $\kappa=1$, i.e., we have
$$\rho(w_{0,2})= 2\,\square^1
=\sum_{k, l >0} \big( b_{-l}b_{-k}b_{l+k} + b_{-l-k}b_{l}b_{k}\big),$$
where $\rho: U_\eta(\Wen) \to \End(S_{\eta})=\End(\C[b_{l}\,;\,l<0])$ is the 
Fock space.
\end{lem}

\begin{proof} The free field formula for $\square^1$ is obtained by setting 
$\kappa=1$ in Proposition \ref{prop:LBinfty}.
Because $S_{\eta}$ is cyclic over $U(\Heis)$, the action of 
$w_{0,2}$ on $S_{\eta}$ is completely determined by the commutation relations 
of $w_{0,2}$ with $\{b_l\;;\;l \in \Z\}$ and by the
equation $w_{0,2}\cdot 1=0$. Hence it is enough to check that 
$[\rho(w_{0,2}), b_l]/2=[\square^1,b_l]$ for all $l$, 
because $\square^1 \cdot 1=0$. Likewise, the action of the
Virasoro operators $L_l^{1/2}$ on $S_{\eta}$ is fully determined 
by their commutation relation with the Heisenberg operators. 
More precisely, from the relations
$$[L^{1/2}_l,b_k]=-kb_{l+k},\qquad L_0^{1/2} \cdot 1=1/4$$ 
it follows that
$$\rho(L_0^{1/2})=\sum_{k \geqslant 0} b_{-k}b_k,\qquad \rho(L_l^{1/2})=
\sum_{k \in \Z} b_{l-k}b_k/2, \qquad l \neq 0.$$
Now, one checks by a direct computation that 
$$[\rho(w_{0,2}),b_l]/2=\rho\big([w_{0,2}, b_l]\big)/2 = 
-l\rho\big( L_{l}^{1/2} + b_l/2\big)=
-l\big( \sum_{k \in \Z} b_{l-k}b_k + b_l\big)/2=
[\square^1, b_l]$$
(recall that $b_0=-1/2$).
The lemma is proved.
\end{proof}

\vspace{.15in}

\subsection{The isomorphism at the level $1$}
\label{sec:Wchi}
Let $\SH_{A_1}^{(1)}$ be the specialization of
$\SH_{A_1}^\cb$ at $\cb=(1,0,0,\dots)$.
Recall the representation
$\rho: \SH^{(1)}_{A_1} \to \End(\Lambda_{A_1})$. We set
\begin{equation}
\SH^{(1)}_{1}=\SH_{A_1}^{(1)} \otimes_{A_1} \C,\qquad
\Lambda_1=\Lambda_{A_1}\otimes_{A_1} \C\end{equation}
where $A_1$ acts on $\C$ via $\kappa \mapsto 1$.
We identify  $S_{\eta}$  and  $\Lambda_1$ via the assignment 
\begin{equation}
b_{-l_1} \cdots b_{-l_r} \cdot 1 \mapsto p_{l_1} \cdots p_{l_r} \cdot 1, 
\qquad l_1, \ldots, l_r \geqslant 1.
\end{equation}
This identification intertwines the actions of the 
Heisenberg generators $b_l$ in 
$U_{\eta}(\Wen)$ with the 
Heisenberg generators $D_{-l,0}$ in $\SH_{1}^{(1)}$ for $l\in \Z$. It
intertwines also the action of $w_{0,2}/2$ 
with that of $D_{0,2}$ by Lemma~\ref{L:wFock} and 
Proposition~\ref{P:newrep}. Since, by Lemma~\ref{L:Wgenerate} and 
Proposition~\ref{prop:generators/involution2},
the algebras $U_{\eta}(\Wen)$ and $\SH^{(1)}_{1}$ 
are respectively generated by 
$\{b_{-1}, w_{0,2}, b_{1}\}$ and 
$\{D_{-1,0}, D_{0,2}, D_{1,0}\}$, and since by
Proposition~\ref{prop:fideleW} the representation on $S_{\eta}$ 
is faithful we obtain in this way a canonical 
surjective algebra homomorphism
\begin{equation}\Theta^1: \SH_{1}^{(1)} \to U_{\eta}(\Wen),\quad
D_{-l,0}\mapsto b_l,\quad D_{0,2}\mapsto w_{0,2}/2,\quad l\in\Z.
\end{equation}

\begin{prop}\label{P:thetachi} The map $\Theta^1$ is an algebra isomorphism.
\end{prop}

\begin{proof} 
Set $\SH^{(1),0}_{1}=\SH^0_{1}/(\cb-c)$. 
We first show that $\Theta^\chi$ restricts to an isomorphism
\begin{equation}\SH^{0}_{1} \to U_{\eta}(\Wen^0).
\end{equation}
By \eqref{E:rel1bis2} we have
$D_{1,l}=\ad(D_{0,2})^l (D_{1,0})$ for $l\geqslant 0$.
So, a direct computation proves that
\begin{equation}\Theta^1(D_{1,l})=
2^{-l} \ad(w_{0,2})^l(b_{-1}) \in 
(-1)^lw_{-1,l} \oplus \bigoplus_{k=0}^{l-1} \C w_{-1,k},\qquad
l\geqslant 0.
\end{equation}
Thus, since $E_l=[D_{-1,0},D_{1,l}]$, we get
\begin{equation}\Theta^1(E_l)=
[b_1,\Theta^1(D_{1,l})]
\in (-1)^{l+1}lw_{0,l-1}\oplus 
\bigoplus_{k=1}^{l-2} \C w_{0,k} \oplus \C,\qquad l\geqslant 0.
\end{equation}
Next, from \eqref{E:rel3bis2} and \eqref{D3}, we have the following formula
in $\SH^{(1)}_{1}$
\begin{equation}E_l\in l(l-1)D_{0,l-1} 
+\C[D_{0,1},\ldots,D_{0,l-2}],\qquad l\geqslant 2.
\end{equation}
It follows that 
\begin{equation}\Theta^1(D_{0,l}) \in 
(-1)^l w_{0,l}/l+\C[w_{0,1},\ldots,w_{0,l-1}],
\qquad l\geqslant 1.
\end{equation}
Thus $\Theta^\chi$ restricts to an isomorphism 
$\SH^{\chi,0}_{1} \to U_\eta(\Wen^0).$
Next, observe that 
\begin{equation}\Theta^1(\SH^>_{1}) \subset U(\Wen^<),
\qquad\Theta^1(\SH^<_{1}) \subset U(\Wen^>).
\end{equation}
Moreover, since 
\begin{equation}\SH^{(1)}_{1} = \SH^>_{1} \otimes 
\SH^{(1),0}_{1}\otimes \SH^<_{1},\quad
U_{\eta}(\Wen) = U(\Wen^<) \otimes U_{\eta}(\Wen^0) \otimes U(\Wen^>),
\end{equation}
by Proposition \ref{1.8:prop1}.
and the PBW theorem,
and since $\Theta^1$ is surjective we deduce that 
\begin{equation}\Theta^1 ~:\SH^>_{1} \to U(\Wen^<),\quad
\Theta^1~: \SH^<_{1} \to U(\Wen^>)
\end{equation} are surjective as well. 
It only remains to prove that they are isomorphisms. 
Both $\SH^>_{1}$ and $U(\Wen^<)$ carry a 
$\Z$-grading and a $\N$-filtration. 
The map $\Theta^1$ is compatible with these gradings and filtrations, 
i.e., we have
\begin{equation}\Theta^1(\SH^>_{1}[r,\leqslant\!l]) = U(\Wen^<)[-r,\leqslant\!l].
\end{equation}
But by Corolllary~\ref{P:poincareSH>} and 
\eqref{E:poincarew} these spaces have the same dimension. It follows that 
\begin{equation}\Theta^1~:\SH^>_{1} \to U(\Wen^<)
\end{equation}
is an isomorphism. The same holds for $\SH^<_{1}$. 
We are done.
 \end{proof}

\vskip3mm

\subsection{The isomorphism for a general level} 
Now we construct a $Z$-algebra isomorphism 
\begin{equation}\Theta~: \SH^{\cb}_{1} \to U(\Wen)\otimes Z,
\qquad Z=\C[\cb_l\,;\,l\geqslant 1].
\end{equation}
The construction of $\Theta$ is  inspired by $\Theta^\chi$. 
Recall that \eqref{E:rel1bis2}, \eqref{E:rel2bis2}
yield 
\begin{equation}D_{1,l}=\ad(D_{0,2})^l (D_{1,0}),\qquad
D_{-1,l}=(-1)^l\ad(D_{0,2})^l (D_{-1,0}),\qquad l\geqslant 0.
\end{equation}
Thus, by Proposition~\ref{P:thetachi}, the assignments
\begin{equation}D_{1,l} \mapsto 2^{-l} \ad(w_{0,2})^l (b_{-1}), \qquad 
D_{-1,l} \mapsto (-2)^{-l}\ad(w_{0,2})^l (b_{1}),\qquad l\geqslant 0
\end{equation}
extend to algebra isomorphisms 
\begin{equation}\label{theta1}\Theta~:\SH^>_{1} \to U(\Wen^<),\qquad
\Theta~:\SH^<_{1} \to U(\Wen^>).\end{equation}
They coincide with the restrictions of $\Theta^1$. 
Next, we lift  the map
\begin{equation}\Theta^1~: \SH^{(1),0}_{1} \to U_{\eta}(\Wen^0)
\end{equation}
to a $Z$-algebra isomorphism 
\begin{equation}\label{theta3}
\Theta~: \SH^{\cb,0}_{1} \to U(\Wen^0)\otimes Z.
\end{equation}
For $l\geqslant 2$ we have
\begin{equation}
\gathered
E_l \in l(l-1)D_{0,l-1}+Z[\cb_0,D_{0,1},\dots,D_{0,l-2}],\\
[b_{1},\ad(w_{0,2})^l (b_{-1})] \in -(-2)^{l}lw_{0,l-1} \oplus 
\bigoplus_{k=1}^{l-2} \C w_{0,k} \oplus \C C \oplus \C b_0.
\endgathered
\end{equation}
In particular, we have 
$\SH^{\cb,0}_{1}=Z[\cb_0,E_l\,;\,l\geqslant 1].$
Thus, there is a unique $Z$-algebra isomorphism $\Theta$ as in \eqref{theta3}
such that
\begin{equation}\label{theta2}
\gathered
\Theta(\cb_0)=C, \qquad 
%\Theta(\omega)=b_0+\cb_1+C/2,\qquad
\Theta(E_l)=2^{-l}[b_{1},\ad(w_{0,2})^l (b_{-1})],\qquad
l\geqslant 2.
\endgathered
\end{equation}
We claim that the maps \eqref{theta1}, \eqref{theta3} glue together into a
$Z$-algebra isomorphism 
\begin{equation}\Theta~: \SH^{\cb}_{1} \to U(\Wen)\otimes Z.
\end{equation}
By the triangular decomposition argument, it is enough to prove that $\Theta$ 
is an algebra morphism, i.e., that
relations \eqref{E:rel1bis2}-\eqref{E:rel3bis2} hold in $U(\Wen)\otimes Z$. 
This is clear for
\eqref{E:rel1bis2}, \eqref{E:rel2bis2}, because $\Theta^1$ is an algebra 
morphism and $\Theta$ is a lift of $\Theta^1$. The relation 
\eqref{E:rel3bis2} holds by construction, because
\begin{equation}\gathered
\Theta([D_{-1,0},D_{1,l}])=\Theta(E_l)=[b_{1},\Theta(D_{1,l})],\qquad l\geqslant 2,\\
\Theta([D_{-1,0},D_{1,0}])=\Theta(\cb_0)=C=[b_1,b_{-1}],\\
\Theta([D_{-1,0},D_{1,1}])=\Theta(-\cb_1)=b_0+C/2=
[b_1,L_{-1}^{1/2}+b_{-1}/2]=[b_{1},\Theta(D_{1,1})].
\endgathered
\end{equation}
Therefore, we have proved the following.

\vspace{.1in}

\begin{theo} There is a unique $Z$-algebra isomorphism 
$\Theta~:\SH^{\cb}_{1} \to U(\Wen)\otimes Z$ satisfying
\begin{equation}\Theta(\cb_0)=C,\qquad
%\Theta(\omega)=b_0+\cb_1+C/2,\quad
\Theta(D_{-l,0})=b_{l},\qquad 
\Theta(D_{0,2})=w_{0,2}/2,\qquad  l\neq 0.
\end{equation}
\end{theo}

\vspace{.2in}

\section{Complements on Section 9}
\label{App:G}

We freely use the notations of Appendices B, C and D. We begin by explicitly computing $f_{-1,d}(G)$.
By definition, we have $f_{-1,d} [M_{r,n}]=c[M_{r,n-1}]$
if and only if the quantity
\begin{equation}
c_{\mu}=\sum_{\substack{\lambda \supset \mu \\ |\lambda \backslash \mu |=1}} \eu_{\mu}\eu_{\lambda}^{-1} \langle \mu \;;\; f_{-1,d}\;;\; \lambda \rangle\end{equation}
is equal to $c$ for any $r$-partition $\mu$. Using (\ref{Er:31}) and (\ref{formH}) we have
\begin{equation}
\eu_{\mu}\eu_{\lambda}^{-1} \langle \mu\;;\; f_{-1,d} \;;\; \lambda \rangle=c_1(\tau_{\lambda,\mu})^d \eu(N^*_{\lambda,\mu}-T^*_{\lambda})=(xy)^{-1}c_1(\tau_{\lambda,\mu})^d \eu(\tau_{\lambda,\mu}^*H_{\mu}+1).
\end{equation}
Furthermore, by Lemma \ref{lem:toto1},
\begin{equation}
\tau^*_{\lambda,\mu}H_{\mu} +1=qt \sum_{\sigma \subset \mu} \tau^*_{\lambda,\mu}\tau_{\sigma,\mu} - \sum_{\substack{\lambda' \supset \mu\\ \lambda' \neq \lambda}} \tau^*_{\lambda,\mu}\tau_{\lambda',\mu},
\end{equation}
where in first sum $|\mu \backslash \sigma|=1$ while in the second $|\lambda' \backslash \mu|=1$. Setting 
$a_{\lambda}=c_{1}(\tau_{\lambda,\mu})$ and $b_{\sigma}=c_1(\tau_{\mu,\sigma}) + x + y$ we obtain
\begin{equation}
c_{\mu}=(xy)^{-1} \sum_{\lambda}a_{\lambda}^d \frac{\prod_{\sigma} 
(b_{\sigma}-a_{\sigma})}{\prod_{\lambda' \neq \lambda} (a_{\lambda'}-a_{\lambda})}.
\end{equation}

\vspace{.05in}

\begin{lem}\label{lem:G1} Let $m \geqslant 0, n=m+r$ and $d \geqslant 0$. 
Let $z_1, \ldots, z_n, y_1, \ldots, y_m$ be formal variables. Then
\begin{equation}\sum_{i}z_{i}^d \frac{\prod_{k} (y_{k}-z_{i})}{\prod_{j \neq i} (z_{j}-z_{i})}
=\begin{cases} 0 & \quad \text{if}\; d <r-1\\ (-1)^{r-1} 
& \quad \text{if}\; d=r-1 \\ (-1)^r\sum_k y_k -\sum_i z_i & \quad \text{if}\; d=r. \end{cases}
\end{equation}
\end{lem}

\begin{proof} Let $P_d(z,y)$ be the left hand side of the above expression. 
It is a rational function of degree $d-r+1$ with at most
simple poles along the divisors $z_i=z_j$. It is easy to see that the residue of $P_d(z,y)$ along each of these 
divisors is in fact equal to zero, so that $P_d(z,y)$ is a homogeneous polynomial of degree $d-r+1$. 
This proves that $P_d(z,y)=0$ if $d< r-1$. To compute
the scalar $P_{r-1}(z,y)$ we may set $y_k=0$ for all $k$ and let $z_1 \mapsto \infty$. 
To compute $P_r(z,y)$ we may likewise consider
the limits $P_r(z,y)/z_i, P_r(z,y)/y_k$ as $z_i \mapsto \infty$ and $y_k \mapsto \infty$ respectively.
\end{proof}

Using Lemma \ref{lem:G1} together with the fact that for a given $r$-partition $\mu$, 
\begin{equation}\sum_{\sigma} b_{\sigma} - \sum_\lambda a_{\lambda}= -(e_1 + \cdots + e_r)
\end{equation}
we deduce
\begin{equation}c_{\mu}=\begin{cases} 0 & \quad \text{if}\; d <r-1\\ (-1)^{r-1} 
& \quad \text{if}\; d=r-1 \\ (-1)^r(e_1 + \cdots + e_r) & \quad \text{if}\; d=r .\end{cases}
\end{equation}
and thus that, for $d<r-1$,
\begin{equation}
f_{-1,d}(G)=0, \quad f_{-1,r-1}(G)=-(1)^{r-1}(xy)^{-1} G, \quad f_{-1,r}(G)=(-1)^r(\sum_i e_i)(xy)^{-1} G.
\end{equation}

This proves (\ref{E:whittaker1}) for $l=1$, the first part of (\ref{E:whittaker2}) and (\ref{E:whittaker3}). Relation (\ref{E:whittaker1})
for $l \geqslant 1$ and the second part of (\ref{E:whittaker2}) follow since $D_{-l,d}$ is obtained from $D_{-1,d}$ and $D_{-1,d+1}$
by iterated commutators with $D_{-1,0}$ or $D_{-1,1}$. Proposition \ref{prop:Whittaker-SH} is proved.

\vspace{.1in}

\begin{rem} The operator
$\rho^{(r)}(x^{d-1+2l}y^{l}D_{-l,d})$
%\rho^{(r)}(x^ly^{2l}H_{l})$$ 
preserves the lattice $\Lb^{(r)}$ by 
Remark \ref{rem:int}, and it has the cohomological degree
$2(2l-rl+d-1)$
%$$2l(2-r)-2,\qquad 2l(2-r)$$ 
by Remark \ref{rem:homdeg}. 
Thus for $l\geqslant 1$ and $r\geqslant 2$ we may deduce directly that
\begin{equation}
\label{8.56}
\gathered
\rho^{(r)}(D_{-l,d})([M_{r,n}])=0,\qquad 
\rho^{(r)}(D_{-l,r-1})([M_{r,n}])\in K_r\,[M_{r,n-l}],
\qquad d<r-1.
%\rho^{(r)}(b_l)([M_{r,n}])=0,\qquad r\geqslant 2,\\
%\rho^{(r)}(H_l)([M_{r,n}])=0,\qquad r>2.
\endgathered
\end{equation}
\end{rem}

\newpage
 
\centerline{\textsc{Index of notations}}
 
\vspace{.1in}
 
\begin{tabular}{l @{\hspace{.3in}} r |  l @{\hspace{.3in}} r}
{\textbf{Rings, Groups}} & &{\textbf{Varieties, Bundles}} &\\
&&&\\

$F=\C(\kappa)$ & 1.1 & $\text{Hilb}_n$, $\text{Hilb}_{n,n+1}$ & 2.2, 2.4\\

$A=\C[\kappa]$ & 1.1 & $M_{r,n}, M_{r,n,n+1}$ & 3.1, 3.3\\

$\Lambda$, $\Lambda_n$&1.3
& $T_{\lambda}$ & 2.4, 3.2\\

$R=R_T=\C[x,y]$ & 2.7
& $T_{\lambda,\mu}, N_{\lambda,\mu}$&2.5, 3.3\\

$K=K_T=\C(x,y)$ & 2.7
& $\eu_{\lambda},$ $\eu_{\lambda,\mu}$ & 2.4, 3.3\\

 $K_r=K(\eps_1,\dots,\eps_r)$&1.8 
&$T_{\lambda},$ $N_{\mu,\lambda}$ & 3.3\\

$D$, $T$, $\widetilde D$&3.2
&$W=\chi_1^{-1} + \cdots + \chi_r^{-1}$ &3.4\\

& 
&$\tau_n,$ $\tau_{\lambda}$, $\tau_{\lambda,\mu}$ & 3.4\\

{\textbf{Algebras}} &&$C=C_\gen=C_E$ & 4.2\\
& & &\\

$\H_n, \H^\pm_n, \SH_n, \SH^\pm_{n}, \SH^{0}_n$ & 1.2 & \textbf{Others} &\\

$\SH^>_n, \SH^<_n$ & 1.5 & & \\
$\SH^{\pm}$& 1.5, 1.6 
& $\pi : \SH^+ \to \SH^-$& 1.2, 1.8\\

$\SH^>, \SH^<$& 1.7 
& $c(s)=x(s)-\kappa y(s)$ & 1.4\\

$\SHo, \SHoo$& 1.7 & $D_{0,l}^{(n)}, D_{0,l} $ & 1.4, 1.6
\\

$ \SH^{\cb}, \SH^{\cb,0}, \SH^{\cb}_{A}$& 1.8, F.1 & $D_{\pm 1, l}^{(n)}, D_{\pm 1,l}, E_l^{(n)}, E_l$ & 1.5, 1.7
\\   

$\SH^{(r)}_K$ & 1.8 
& $D_{l,0}^{(n)}, D_{l,0}$ & 1.5, 1.6\\

$\U^{(r)}_K,$ 
$\U^{(r),>}_K,$ $\U^{(r),<}_K,$ $\U^{(r),+}_K,$ $\U^{(r),-}_K$& 2.8, 3.6
& $\omega=D_{0,0}$ & 1.7\\

$\SC',$ $\mathbf{C}$, $\SC$,
$\SC'_K,$ $\mathbf{C}_K$, $\SC_K$& 4.4 
& $\square, \square_n$ & 1.4, 1.6 \\

$\mathbf{Sh}$& 4.5 
& $\xi=1-\kappa$ & 1.5\\

$\U^{(r),>},$ $\U^{(r),<}$& 6.1 
& $G_l(s), \varphi_l(s), \phi_l(s)$ & 1.5, 1.8\\

$\mathbb{H}_n, \mathbb{S}\mathbb{H}_n, \mathbb{S}\mathbb{H}^{\cm}$ & 7.1, 7.4
&$K(\kappa,\omega,s) $ & 1.5\\

$W_k(\gen\len_r)$ & 8.1 & $J_\lambda^{(n)}$, $J_\lambda$& 1.6\\

$\Uen(W_k(\gen\len_r)), \scrU(W_k(\gen\len_r))$ & 8.2, 8.4 
& $\epsilon : \SH^c \to F$ & 1.8\\

$\Uen(\SH^{(r)}_K)$ & 8.5 
& $D_{r,d}$ & 1.9 \\

$W_{1+\infty}$& F.2 
& $Y_{r,d}$ & 1.9 \\
 
& & $\bullet$ (Wilson operators) & 1.10, 3.7, 4.6\\
 
{\textbf{Maps}}&
& $b_l$, $H_l$ & 1.10\\

& & $q$, $t$, $x=\cc_1(q)$, $y=\cc_1(t)$ & 2.1 \\

$\pi_n : \Lambda \to \Lambda_n$, $\pi_{n+1,n}$ & 1.6 &$\kappa=-y/x$ & 2.8 \\

$\Phi_n : \SH^+ \to \SH^+_n$ & 1.6 & $\mathscr E$, $\mathscr E^+$, 
$\mathscr E^-$ &2.9 \\

$\Phi_n : \SHo \to \SH_n$ & 1.7 & $e_a=\cc_1(\chi_a)$ & 3.2 \\

$\Lambda_K\simeq\widetilde\Lb^{(1)}_K$&2.9  
& $v=q^{-1}t^{-1}$&3.2\\

$\Psi: \widetilde{\SH}^{(1)}_K \to \widetilde{\mathbf{U}}^{(1)}_K$ &  2.9
& $\eps_a=e_a/x$ & 3.6 \\

$\Psi: \SH^{(r)} \to \mathbf{U}^{(r)}_K$ & 3.6
& $c_a(s)=x(x(s)-\kappa y(s)-\eps_a)$ & 3.6\\

$\pi^{(1^r)} \to \mathbf{L}^{(1^r)}_K$&8.9
& $\SYM_n$&3.7 \\

$\Theta : \SH^{(r)}_K \to \scrU(W_k(\gen\len_r))$& 8.9
& $\theta_l$&4.5\\

&
& $a_l$&7.1 \\

\textbf{Representations}&
& $\alpha_l$&7.4 \\

& 
&  $H(z),$ $L(z)$, $b(z)$ & 8.10 \\

$\rho_n, \rho^+$ & 1.3, 1.6 & &\\

$\tilde{\rho}^{(1)} : \widetilde{\mathbf{U}}^{(1)}_K \to \text{End}(\widetilde{\Lb}^{(1)}_K)$ & 2.9 & & \\

$\rho^{(r)} : {\mathbf{U}}^{(r)}_K \to \text{End}({\Lb}^{(r)}_K)$ & 3.6 & &\\

$\rho^{(\nu)} : {\mathbf{U}}^{(|\nu|)}_K \to \text{End}({\Lb}^{(\nu)}_K)$ & 8.5 & & \\

$M_{\beta}, \pi_{\beta}$ & 8.3, 8.4 & & \\

$\pi^{(1)},$ $\pi^{(1^r)}$, $\pi^{(r)}$ & 8.7, 8.9, 8.11 & &\\

\end{tabular}

\newpage

\vspace{.3in}

\small{}

\vspace{4mm}

\noindent
O. Schiffmann, \texttt{olivier.schiffmann@math.u-psud.fr},\\
D\'epartement de Math\'ematiques, Universit\'e de Paris-Sud, B\^atiment 425
91405 Orsay Cedex, FRANCE.

\vspace{.1in}

\noindent
E. Vasserot, \texttt{vasserot@math.jussieu.fr},\\
D\'epartement de Math\'ematiques, Universit\'e de Paris 7, 175 rue du Chevaleret, 75013 Paris, FRANCE.

\end{document}